\documentclass[11pt]{amsart}

\usepackage{amsmath, amsthm, amssymb, url, color, hyperref, graphicx, esint}
\hypersetup{
    colorlinks,
    citecolor=blue,
    filecolor=black,
    linkcolor=black,    
    urlcolor=blac
}

\usepackage{todonotes}
\usepackage{comment}
\usepackage[normalem]{ulem}

\usepackage{latexsym, amsfonts}
\usepackage{graphicx, color,hyperref,dsfont,epsfig,caption,wrapfig,subfig}
\usepackage{pstricks}
\usepackage{movie15}
\usepackage{tikz}
\usepackage{bm}

\usetikzlibrary{shapes,decorations.markings}
\hypersetup{
    linktoc=page,
    linkcolor=red,          
    citecolor=blue,        
    filecolor=blue,      
    urlcolor=cyan,
    colorlinks=true           
}

\def\@rst #1 #2other{#1}

\numberwithin{equation}{section}

\newtheorem{theorem}{Theorem}[section]
\newtheorem{lemma}[theorem]{Lemma}
\newtheorem{proposition}[theorem]{Proposition}

\newtheorem{remark}[theorem]{Remark}
\newtheorem{definition}[theorem]{Definition}

\newcommand{\R}{{\mathbb{R}}}
\newcommand{\eps}{{\varepsilon}}

\newcommand{\LF}{\mathrm{LF}}
\newcommand{\bbH}{{\mathbb{H}}}
\newcommand{\cA}{{\mathcal{A}}}
\newcommand{\cL}{{\mathcal{L}}}
\newcommand{\cS}{{\mathcal{S}}}
\newcommand{\C}{{\mathbb{C}}}
\newcommand{\wt}{\widetilde}
\newcommand{\wh}{\widehat}
\newcommand{\cB}{{\mathcal{B}}}
\newcommand{\cM}{{\mathcal{M}}}
\newcommand{\cMtwo}{{\cM^{\textrm{disk}}_2}}
\newcommand{\disk}{{\mathrm{disk}}}

\renewcommand{\P}{\mathbb{P}}
\newcommand{\E}{\mathbb{E}}

\renewcommand{\lg}{\langle}
\newcommand{\rg}{\rangle}

\newcommand{\ii}{\mathbf i}

\newcommand{\HPT}{{H_{\mathrm{PT}}}}

\theoremstyle{definition}

\begin{document}

\title{Derivation of all structure constants for boundary Liouville CFT}

\date{\today}

\author{Morris Ang}
\author{Guillaume Remy} 
\author{Xin Sun} 
\author{Tunan Zhu}

\maketitle

\begin{abstract}
We prove that the probabilistic definition of the most general boundary three-point and bulk-boundary structure constants in Liouville conformal field theory (LCFT) agree respectively with the formula proposed by Ponsot-Techsner (2002) and by Hosomichi (2001). These formulas also respectively describe the fusion kernel and modular kernel of the Virasoro conformal blocks, which are important functions in various contexts of mathematical physics. As an intermediate step, we obtain the formula for the boundary reflection coefficient of LCFT proposed by Fateev-Zamolodchikov-Zamolodchikov (2000). 
Our proof relies on the boundary Belavin-Polyakov-Zamolodchikov differential equation recently proved by the first named author, and inputs from the coupling theory of Liouville quantum gravity (LQG) and Schramm Loewner evolution. Our results supply all the structure constants needed to perform the conformal bootstrap for boundary LCFT.  They also yield exact descriptions for the joint law of the area and boundary lengths of basic LQG surfaces, including quantum triangles and two-pointed quantum disks.
\end{abstract}

\section{Introduction}\label{sec:intro}
Liouville conformal field theory (LCFT) describes the law of the conformal factor of 
random surfaces in Liouville quantum gravity. Introduced by Polyakov \cite{Polyakov1981} in theoretical physics,  LCFT  was recently made rigorous in probability theory, first in the case of the Riemann sphere in \cite{dkrv-lqg-sphere}, and then in the case of a simply connected domain with boundary in \cite{hrv-disk}; see also \cite{LQG_tori, Remy_annulus, LQG_higher_genus} for the case of other topologies. Following the framework of~\cite{BPZ1984}, the correlation functions of LCFT can be solved by the conformal bootstrap program. In the probabilistic framework, this was recently carried out on surfaces without boundary  \cite{DOZZ_proof}, 
\cite{GKRV_bootstrap,GKRV-Segal}. The initial input of the conformal bootstrap are the \emph{structure constants}. For surfaces without boundary, it is the three-point correlation function on the sphere. It has an exact expression called the DOZZ formula which was proposed in \cite{DOZZ1,DOZZ2} and proved in \cite{DOZZ_proof}.

For LCFT on surfaces with boundary, the structure constants are the correlation functions on the disk
with three points on the boundary, or one point in the bulk and one on the boundary. The theory involves both the bulk and the boundary Liouville potentials (see the Liouville action \eqref{liouville_action}). When the bulk Liouville potential is absent, the structure constants were obtained by Remy and Zhu~\cite{rz-boundary}. When there is one bulk insertion and no boundary ones, the structure constant was obtained by Ang, Remy, and Sun in~\cite{ARS-FZZ}, confirming an earlier  proposal of Fateev, Zamolodchikov and Zamolodchikov (FZZ) in~\cite{FZZ}. The conformal bootstrap is also applicable in the boundary case and B. Wu~\cite{Wu-annulus} recently proved a bootstrap formula for the annulus with one insertion at each boundary. 

In this paper we obtain the exact formula for the boundary   three-point structure constant  proposed by Ponsot and Techsner~\cite{PT_boundary_3pt} and for the bulk-boundary structure constant proposed by Hosomichi~\cite{bulk-boundary}, in the most general case where both the boundary and bulk Liouville potentials are present. As an intermediate step, we obtain the formula for the boundary two-point function of LCFT --- also known as the boundary reflection coefficient --- which was proposed in~\cite{FZZ}.
This completes the derivation of all structure constants required for the conformal bootstrap of  boundary LCFT. We rely on a novel Belavin-Polyakov-Zamolodchikov (BPZ) differential equation for LCFT on the disk established in \cite{ang-lcft-zip} and
inputs from the coupling  theory  of Liouville quantum gravity  and Schramm Loewner evolution. See~Section~\ref{subsec:method} for a summary of our method.

The Liouville structure constants have intrinsic interest beyond the conformal bootstrap. 
The boundary three-point structure constant and the bulk-boundary structure constant agree modulo prefactors with the fusion kernel and the modular kernel of the Virasoro conformal blocks~\cite{PT_boundary_3pt}, an important special function with various interpretations in mathematical physics. Moreover, the boundary reflection coefficient gives the joint area and boundary length distribution of an important family of Liouville quantum gravity surfaces called two-pointed quantum disks~\cite{wedges,ahs-disk-welding};  the boundary three-point structure constant gives the corresponding information for a more general family of Liouville quantum gravity surfaces called quantum triangles~\cite{ASY-triangle,AHS-SLE-integrability}.

\subsection{Main results}\label{subsec:main}
We start by recalling the probabilistic construction of LCFT on the disk from \cite{hrv-disk}, which is adapted from the construction on the Riemann sphere performed in  \cite{dkrv-lqg-sphere}. By conformal invariance we will use the upper half plane $\bbH$ as the base domain. Our presentation is brief with more details provided in Section~\ref{sec:precise}. Fix the global constants
\begin{equation}
\gamma\in (0,2) \quad \textrm{ and } \quad Q=\frac{\gamma}{2}+\frac2{\gamma}.
\end{equation}
In physics LCFT is defined using a formal path integral. Fix $N$ points $z_i \in \mathbb{H}$ with associated weights $\alpha_i \in \mathbb{R}$ and similarly $M$ points $s_j \in \mathbb{R}$ with associated weights $\beta_j \in \mathbb{R}$. The correlation function associated to these points is given by the formal integral
\begin{align}\label{path_int_def}
\left \langle  \prod_{i=1}^N e^{\alpha_i \phi (z_i)}      \prod_{j=1}^M e^{\frac{\beta_i}2 \phi (s_i)} \right \rangle = 
\int_{\phi: \mathbb{H} \mapsto \mathbb{R}} D \phi    \prod_{i=1}^N e^{\alpha_i \phi(z_i)} \prod_{j=1}^M e^{\frac{\beta_j}{2} \phi(s_j)} e^{-S_L(\phi)},
\end{align}
where $S_L$ is the Liouville action given by:
\begin{align}\label{liouville_action}
S_L(\phi) = \frac{1}{4 \pi} \int_{\mathbb{H}} \left( \vert \partial^{\hat g} \phi \vert^2 + Q R_{\hat g} \phi + 4 \pi \mu e^{\gamma \phi} \right  ) d \lambda_{\hat g}  + \frac{1}{2 \pi} \int_{ \mathbb{R}} \left( Q K_{\hat g} \phi + 2 \pi \mu_{B} e^{\frac{\gamma}{2} \phi} \right) d \lambda_{\partial \hat g}.
\end{align}
Here $\hat g$ is a background metric with $R_{\hat g}$ and $\lambda_{\hat g} $ representing the curvature and  volume in the bulk, respectively, and $(K_{\hat g}, \lambda_{\partial \hat g})$ being their  boundary counterparts. 
The terms involving $\mu$ and $\mu_B$ are the bulk and boundary Liouville potentials, respectively. Here $\mu > 0$ and $\mu_B$ is a complex valued function on $\mathbb{R}$ which is piecewise constant in between boundary insertions. We assume the $s_j$'s are chosen in counterclockwise order on $\mathbb{R}$ and denote by $\mu_j $  the value of $\mu_B$ on the interval $ (s_{j-1}, s_{j})$, with convention $s_0 = -\infty$, $s_{M+1} = \infty$. We always assume $\Re(\mu_j) \geq 0$. Furthermore we frequently work under the following conditions on weights $\alpha_i, \beta_j$ known as the \emph{Seiberg bound}: 
\begin{equation}\label{eq:seiberg}
\sum_{i=1}^N \alpha_i + \sum_{j=1}^{M} \frac{\beta_j}{2} > Q, \quad \alpha_i < Q, \quad \textrm{and}\quad  \beta_j <Q. 
\end{equation}

We  now give a rigorous probabilistic meaning to \eqref{path_int_def} for $N=0$, $M=3$ and for $N=1$, $M=1$. Let $P_\bbH$ be the probability measure corresponding to the free-boundary Gaussian free field  on $\bbH$ normalized to have average zero on the semi-circle $\partial \mathbb{D} \cap \bbH$.
Let the infinite measure $\LF_\bbH(d\phi)$ be the law of $\phi(z) = h(z) - 2Q \log |z|_+ + \mathbf c$, where $|z|_+ := \max(|z|,1)$ and $(h, \mathbf c)$ is sampled according to $P_\bbH \times [e^{-Qc} \, dc]$. We call the field $\phi$ sampled from $\LF_{\bbH}$ a Liouville field on $\bbH$. This definition of $\LF_\bbH$ corresponds to choosing the background metric in \eqref{liouville_action} to be $\hat g(z) = |z|_+^{-4}$. We define the bulk and boundary Gaussian multiplicative chaos (GMC) measures  of $\phi$ as the limits (see e.g.~\cite{Ber_GMC,rhodes-vargas-review}):
\begin{equation*}
\cA_\phi =  \lim_{\eps \to 0} \epsilon^{\frac{\gamma^2}{2}} e^{\gamma \phi_{\epsilon}(z)} d^2z \textrm{ on }\bbH, \qquad \text{and} \qquad \cL_\phi = \lim_{\eps \to 0} \epsilon^{\frac{\gamma^2}{4}} e^{\frac{\gamma}{2} \phi_{\epsilon}(z)} dz \textrm{ on }\R.
\end{equation*}
Given three points $s_1,s_2,s_3$ lying counterclockwise on $ \R$, let $L_j$ be the $\cL_\phi$-length of the counterclock arc on $\partial \bbH$ from $s_{j-1}$ to $s_j$, namely $L_j = \cL_{\phi}(s_{j-1}, s_j)$, where we identify $s_0$ as $s_3$. Fix $\mu>0$ and $\mu_1, \mu_2,\mu_3\ge 0$.
For $\beta_1,\beta_2,\beta_3 \in\R$ satisfying  
the bounds of \eqref{eq:seiberg}, the \emph{boundary three-point correlation function} of LCFT on $\bbH$ with bulk cosmological constant $\mu>0$ and boundary cosmological constants $(\mu_j)_{1\le j \le 3}$ is defined by:
 \begin{equation}\label{eq:3-point}
\left \langle   \prod_{j=1}^3 e^{\frac{\beta_j}2 \phi (s_j)}   \right \rangle=  \int  \prod_{j=1}^3 e^{\frac{\beta_j}{2} \phi(s_j)} \cdot e^{-\mu \cA_\phi(\bbH)-\sum_{j=1}^3 \mu_j L_j}  \LF_\bbH(d\phi ).
 \end{equation}
Similarly, consider a point $z \in \mathbb{H}$ and a point $s \in \mathbb{R}$. Let the weights $\alpha, \beta \in \mathbb{R}$ obey \eqref{eq:seiberg}, and let $\mu, \mu_B >0$. We define the \emph{bulk-boundary correlation function} as:
\begin{equation}\label{eq:1-1-point}
\left \langle  e^{\alpha \phi (z)}  e^{\frac{\beta}2 \phi (s)}   \right \rangle= \int  e^{\alpha \phi (z)}  e^{\frac{\beta}2 \phi (s)} \cdot e^{-\mu \cA_\phi(\bbH)- \mu_B \cL_{\phi}(\mathbb{R})}  \LF_\bbH(d\phi ).
 \end{equation}
 Here although $\phi$ is only a generalized function, the factors $\prod_{j=1}^3 e^{\frac{\beta_j}{2}\phi(s_j)}$ and $e^{\alpha \phi (z)}  e^{\frac{\beta}2 \phi (s)}$  in the integrands can be made sense of by regularization and Girsanov's theorem.  Moreover, the Seiberg bounds \eqref{eq:seiberg} ensure that both integrations in \eqref{eq:3-point} and \eqref{eq:1-1-point} are finite. We will review this construction in full detail in 
Section~\ref{subsec:def-LF}.

Due to conformal symmetry, the 3-point boundary correlation function has the following form
\begin{equation}\label{c4}
\left \langle \prod_{j=1}^3 e^{\frac{\beta_j}2 \phi (s_j)}    \right \rangle = \frac{H \mu^{\frac{2Q - \beta_1 - \beta_2 - \beta_3}{2\gamma}}}{\vert s_1 - s_2 \vert^{\Delta_{\beta_1} + \Delta_{\beta_2} - \Delta_{\beta_3}} \vert s_1 - s_3 \vert^{\Delta_{\beta_1} + \Delta_{\beta_3} - \Delta_{\beta_2}} \vert s_2 - s_3 \vert^{\Delta_{\beta_2} + \Delta_{\beta_3} - \Delta_{\beta_1}} },
\end{equation}  
while the bulk-boundary correlation has the form:
\begin{equation}\label{c5}
\left \langle  e^{\alpha \phi (z)}  e^{\frac{\beta}2 \phi (s)}   \right \rangle = \frac{G \mu^{\frac{2Q -2 \alpha - \beta}{2\gamma}}}{\vert z - \overline{z} \vert^{2 \Delta_{\alpha} - \Delta_{\beta}} \vert z - s \vert^{2 \Delta_{\beta} }  }.
\end{equation}  
Here the conformal dimensions are given by:
$$\Delta_{\alpha} = \frac{\alpha}{2}(Q - \frac{\alpha}{2}), \quad\textrm{and}\quad  \Delta_{\beta} = \frac{\beta}{2}(Q - \frac{\beta}{2}).$$ The function $H$ only depends on $\beta_j$ and $\mu_j/\sqrt\mu$ for $1\le j\le 3$, while $G$ only depends on $\alpha, \beta$ and $\mu_B/\sqrt\mu$. We call $H$ the  \emph{boundary three-point structure constant} and $G$ the  \emph{bulk-boundary structure constant} for LCFT. For simplicity, in the rest of the paper we will set the bulk cosmological constant to be $\mu=1$ and write $H$ as  $H^{(\beta_1, \beta_2, \beta_3)}_{(\mu_1, \mu_2, \mu_3)}$ and $G$ as $G_{\mu_B}(\alpha, \beta)$.

The conjectural formulas in physics for  $H$ and $G$  requires a crucial change of variable:
\begin{equation}\label{eq:mu_to_sigma}
\mu_B(\sigma) := \sqrt{\frac{1}{\sin(\pi\frac{\gamma^2}{4})}}\cos \left(\pi\gamma(\sigma-\frac{Q}{2}) \right).
\end{equation}
For $ \Re (\mu_B(\sigma)) >0$, we choose $\Re (\sigma)  \in (- \frac{1}{2 \gamma} + \frac{Q}{2}, \frac{1}{2 \gamma} + \frac{Q}{2})$.
Ponsot-Teschner \cite{PT_boundary_3pt} proposed a remarkable formula for $H$ under the reparametrization:
\begin{equation}\label{eq:mu-sigma}
H
\begin{pmatrix}
\beta_1 , \beta_2, \beta_3 \\
\sigma_1,  \sigma_2,   \sigma_3 
\end{pmatrix}  : = H^{(\beta_1, \beta_2, \beta_3)}_{(\mu_B(\sigma_1), \mu_B(\sigma_2), \mu_B(\sigma_3))}.
\end{equation}
The  Ponsot-Teschner formula is expressed in terms of the Barnes' double Gamma function
$\Gamma_{\frac{\gamma}{2}}(x)$ and the double sine function $S_{\frac{\gamma}{2}}(x)=\frac{\Gamma_{\frac{\gamma}{2}}(x)}{\Gamma_{\frac{\gamma}{2}}(Q-x)}$, which are prevalent in LCFT.
Both functions admit a meromorphic extension on $\mathbb{C}$ with an explicit pole structure; see Section \ref{sec:double_gamma} for more details. Given a function $f$ that will be taken to be $\Gamma_{\frac{\gamma}{2}}$ or $S_{\frac{\gamma}{2}}$ below, we will use the shorthand notation: $$f(a \pm b) := f(a-b) f(a +b)\quad \textrm{and}\quad f(a \pm b\pm c) := f(a-b-c) f(a-b+c)f(a+b-c)f(a+b+c).$$
The  Ponsot-Teschner function  is given by:\footnote{Our formula contains a $2 \pi$ factor that was not present in \cite{PT_boundary_3pt}.  We believe that it was an error as it was claimed on~\cite[Page 8]{PT_boundary_3pt} that the residue of $H_{\mathrm{PT}}$ at $\beta_1 = 2Q - \beta_2 -\beta_3$ is 1, which is only true if the $2\pi$ factor is included. Moreover,} this factor is consistent with the checks performed in \cite{rz-boundary} and with the $2 \pi$ factor in  $G_{\mathrm{Hos}}$.

\begin{align}\label{formule_PT}
&\HPT
\begin{pmatrix}
\beta_1 , \beta_2, \beta_3 \\
\sigma_1,  \sigma_2,   \sigma_3 
\end{pmatrix} = 2\pi  \left(\frac{\pi  (\frac{\gamma}{2})^{2-\frac{\gamma^2}{2}} \Gamma(\frac{\gamma^2}{4})}{\Gamma(1-\frac{\gamma^2}{4})} \right)^{\frac{2Q- (\beta_1 + \beta_2 + \beta_3)}{2\gamma}}
\\
& \times  \frac{\Gamma_{\frac{\gamma}{2}}(Q - \frac{\beta_2}{2} \pm \frac{Q - \beta_1}{2} \pm \frac{Q - \beta_3}{2} ) }{S_{\frac{\gamma}{2}}(\frac{\beta_3}{2} - \sigma_1 + \frac{Q}{2} \pm ( \frac{Q}{2} - \sigma_3 ))S_{\frac{\gamma}{2}}(\frac{\beta_1}{2} + \sigma_1 - \frac{Q}{2} \pm ( \frac{Q}{2} - \sigma_2 )) \Gamma_{\frac{\gamma}{2}}(Q) \prod_{j=1}^3\Gamma_{\frac{\gamma}{2}}(Q-\beta_j)} \nonumber \\ 
&\times \int_{\mathcal{C}} \frac{S_{\frac{\gamma}{2}}(\frac{Q - \beta_2}{2} + \sigma_3 \pm ( \frac{Q}{2} - \sigma_2) +r)  S_{\frac{\gamma}{2}}(\frac{Q}{2} \pm \frac{Q - \beta_3}{2}  +\sigma_3-\sigma_1+r)}{S_{\frac{\gamma}{2}}(\frac{3Q}{2} \pm \frac{Q - \beta_1}{2}-\frac{\beta_2}{2}+\sigma_3-\sigma_1+r) S_{\frac{\gamma}{2}}(2\sigma_3 + r) S_{\frac{\gamma}{2}}(Q+r)} \frac{dr}{\ii}. \nonumber
\end{align}
Here $\mathcal{C}$ is a properly chosen contour  from $- \ii \infty$ to $ \ii \infty$ such that the integral is meromorphic - namely a ratio of two holomorphic functions - in all of the six variables in the entire complex plane. 
We provide more details on the function $\HPT$ and the contour $\mathcal{C}$ in Section~\ref{sec:PT_intro} below.
The first main result of our paper is the derivation of Ponsot-Teschner formula.
\begin{theorem}\label{thm:H} Let $\gamma \in (0, 2)$, $\sum_{j=1}^3 \beta_j > 2Q$, $ \beta_j < Q$, and $\mu_j=\mu_B(\sigma_j)>0$ for $j=1,2,3$. Then 	\begin{align}\label{main_th2_H}
	  H
	\begin{pmatrix}
	\beta_1 , \beta_2, \beta_3 \\
	\sigma_1,  \sigma_2,   \sigma_3 
	\end{pmatrix}
	=  \HPT
	\begin{pmatrix}
	\beta_1 , \beta_2, \beta_3 \\
	\sigma_1,  \sigma_2,   \sigma_3 
	\end{pmatrix}.
	\end{align}
\end{theorem}

Similarly the formula for the function $G$ will be expressed in terms of $\alpha, \beta, \sigma$ where $\sigma$ is related to $\mu_B$ still by the relation \eqref{eq:mu_to_sigma}. We thus write $ G(\alpha, \beta, \sigma) := G_{\mu_B}(\alpha, \beta)$. The exact formula $G_{\mathrm{Hos}}$ proposed for $G$ by Hosomichi \cite{bulk-boundary} is then given by:\footnote{For $G_{\mathrm{Hos}}$ we have an extra $2^{\Delta_{\beta} - 2 \Delta_{\alpha}}$ compared with \cite{bulk-boundary} which comes from a different convention for the renormalization of the bulk insertion. This factor is consistent with the FZZ formula of our previous work \cite{ARS-FZZ}.}
\begin{align*}
G_{\mathrm{Hos}}(\alpha, \beta, \sigma) =  & 2 \pi 2^{\Delta_{\beta} - 2 \Delta_{\alpha}} \left(\frac{\pi  (\frac{\gamma}{2})^{2-\frac{\gamma^2}{2}} \Gamma(\frac{\gamma^2}{4})}{\Gamma(1-\frac{\gamma^2}{4})} \right)^{ \frac{1}{ \gamma}(Q -\alpha - \frac{\beta}{2})} \frac{\Gamma_{\frac{\gamma}{2}}(2Q - \frac{\beta}{2}  -\alpha ) \Gamma_{\frac{\gamma}{2}}(\alpha-\frac{\beta}{2} )   \Gamma_{\frac{\gamma}{2}}(Q - \frac{\beta}{2}  )^3 }{  \Gamma_{\frac{\gamma}{2}}(Q -\alpha ) \Gamma_{\frac{\gamma}{2}}(Q - \beta ) \Gamma_{\frac{\gamma}{2}}(\alpha )\Gamma_{\frac{\gamma}{2}}(Q) \Gamma_{\frac{\gamma}{2}}(  \frac{\beta}{2}) }   \\
& \times \int_{ \mathcal{C}} \frac{d \sigma_2}{\ii} e^{2 \ii \pi (Q - 2 \sigma) \sigma_2} \frac{   S_{\frac{\gamma}{2}}( \frac{1}{2}(\alpha + \frac{\beta}{2} - Q) \pm \sigma_2 ) }{  S_{\frac{\gamma}{2}}( \frac{1}{2}(\alpha - \frac{\beta}{2} + Q) \pm \sigma_2 )   }.
\end{align*}

The contour integral in this expression for $G_{\mathrm{Hos}}$ goes from $- \ii \infty$ to $\ii \infty$ and is converging at $\pm \ii \infty$ if and only if the parameters $\alpha, \beta, \sigma$ obey the conditions $\frac{\Re(\beta)}{2} < Q$, $\Re(\sigma) \in \left[ \frac{\Re(\beta)}{4}, Q - \frac{ \Re(\beta)}{4} \right]$.
See also Section~\ref{sec:PT_intro} for more details on this function. Our second main result is thus:
\begin{theorem}\label{thm:G} Let $\gamma \in (0, 2)$, $\alpha + \frac{\beta}{2} > Q$, $ \alpha < Q$, $ \beta < Q$, and $\mu_B=\mu_B(\sigma)>0$. Then: 	\begin{align}\label{main_th2_G}
	  G(\alpha, \beta, \sigma) =G_{\mathrm{Hos}}(\alpha, \beta, \sigma).
   \end{align}
\end{theorem}

Our last main result concerns the  reflection coefficient for boundary LCFT. Formally speaking,  this quantity is defined as $R_{\mu_1,\mu_2}(\beta) :=\vert s_1 - s_2 \vert^{2 \Delta_{\beta}}\left \langle e^{\frac{\beta}{2} \phi (s_1)}  e^{\frac{\beta}2 \phi (s_2)} \right \rangle $ for $\beta\in \R$, which by conformal symmetry does not depend on $s_1,s_2\in \R$. The name reflection coefficient comes from the fact the values of $H$ at $\beta_1$ and at the reflected $2Q - \beta_1$ are related by a factor of $R$; see~\eqref{eq:reflection_id}.
Although this is in the same spirit as~\eqref{c4}, there are multiple subtleties in the rigorous definition of  $R_{\mu_1,\mu_2}(\beta)$. First, instead of integrating $e^{\frac{\beta}{2} \phi(s_1)} e^{\frac{\beta}{2} \phi(s_2)}\LF_\bbH(d\phi )$ following~\eqref{eq:3-point},  we need to integrate against $\cMtwo(\beta)$, which is the law of the two-pointed quantum disk with $\beta$-insertion. This is because   $\cMtwo(\beta)$ can be viewed as $e^{\frac{\beta}{2} \phi(s_1)} e^{\frac{\beta}{2} \phi(s_2)}\LF_\bbH(d\phi )$ modulo the redundant  conformal symmetries of $\bbH$ fixing $s_1,s_2$.  The measure $\cMtwo(\beta)$ was first introduced in~\cite{wedges} and implicitly used in~\cite{rz-boundary}. It describes the law of a quantum surface with two boundary marked points.  We  recall its precise definition  in Section~\ref{subsec:R-def}. 

Another subtlety in defining $R_{\mu_1,\mu_2}(\beta)$ is that we cannot directly integrate  $e^{- A-\mu_1L_1-\mu L_2}$ over $\cMtwo(\beta)$ as suggested by~\eqref{eq:3-point}, where $A$ and $L_1,L_2$ are the area and the two boundary lengths of a sample from $\cMtwo(\beta)$, respectively. In fact,  there is no $\beta$ such that this integration is finite. The same issue arises in~\cite{ARS-FZZ} where the LCFT correlation function on the disk with one bulk insertion is considered. In both cases, one needs to truncate the function $e^{-x}$ near $x=0$, after which the integral is finite for some range of $\beta$. Concretely,  for $\mu_1\ge 0, \mu_2\ge 0$, we define
\begin{equation}\label{eq:2-point}
R_{\mu_1,\mu_2}(\beta):=\frac{2(Q-\beta)}\gamma \int (e^{-A-\mu_1L_1-\mu_2 L_2} -1 ) \,  \mathrm d \cMtwo(\beta) \quad \textrm{for }\beta\in (\frac2{\gamma},Q).
\end{equation}
Then $R_{\mu_1,\mu_2}(\beta)$ is indeed finite. We will give the full detail of the integral in~\eqref{eq:2-point} in Section~\ref{subsec:R-def}, including its finiteness. Here we still use the convention that bulk cosmological constant $\mu=1$ as in the definition of $H^{(\beta_1, \beta_2, \beta_3)}_{(\mu_1, \mu_2, \mu_3)}$. 
The prefactor $\frac{2(Q-\beta)}\gamma$ is to match with \cite{rz-boundary}.

In the seminal work~\cite{FZZ}, Fateev, Zamolodchikov and Zamolodchikov proposed a formula for the boundary reflection coefficient  under the same reparametrization $\mu_B(\sigma)$ as in~\eqref{eq:mu-sigma}:
\begin{equation}\label{eq:R-mu-sigma}
R(\beta, \sigma_1, \sigma_2)   : =  R_{\mu_B(\sigma_1),\mu_B(\sigma_2)}(\beta).
\end{equation}  
Their formula is:
\begin{align}\label{eq:R-FZZ}
R_{\mathrm{FZZ}}(\beta, \sigma_1, \sigma_2) =  \left(\frac{\pi  (\frac{\gamma}{2})^{2-\frac{\gamma^2}{2}} \Gamma(\frac{\gamma^2}{4})}{\Gamma(1-\frac{\gamma^2}{4})} \right)^{\frac{Q-\beta}{\gamma}}  
\frac{\Gamma_{\frac{\gamma}{2}}(\beta - Q )  S_{\frac{\gamma}{2}}(Q \pm (Q-\sigma_1 -\sigma_2) -  \frac{\beta}{2}) }{     \Gamma_{\frac{\gamma}{2}}(Q- \beta )    S_{\frac{\gamma}{2}}(\frac{\beta}{2} \pm (\sigma_2- \sigma_1))}.
\end{align} 
Our last main result is the confirmation of their proposal.
\begin{theorem}\label{main_th2}  
Let $\gamma \in (0, 2)$, $\beta \in (\frac{2}{\gamma}, Q)$, and $\mu_j=\mu_B(\sigma_j)>0$ for $j=1,2$. Then
	
	\begin{align}\label{main_th2_R}
	R(\beta, \sigma_1, \sigma_2) =  	R_{\mathrm{FZZ}}(\beta, \sigma_1, \sigma_2).
	\end{align} 
\end{theorem}

 Figure 1   summarizes all the structure constants discussed above of the boundary LCFT, including for completeness the bulk one-point function $U_{\mu_B}(\alpha)$ which was proven in \cite{ARS-FZZ} to be equal to the FZZ formula. 

\begin{figure}[!htp]
\centering
\includegraphics[width = 0.9\linewidth]{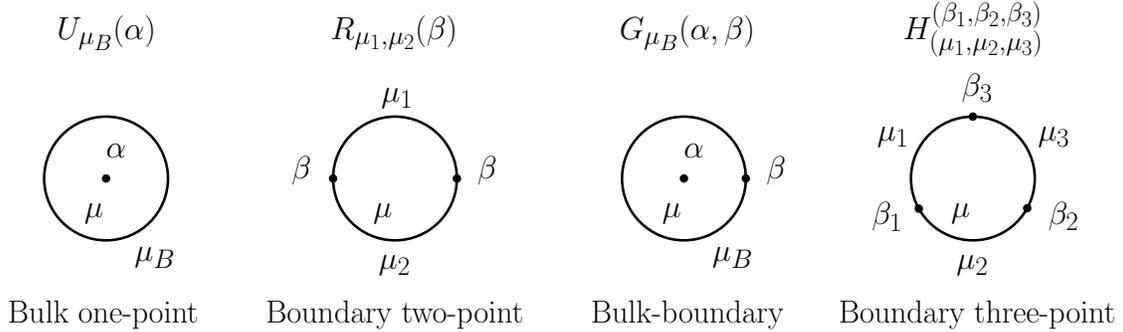}
\caption{Structure constants of boundary Liouville CFT.}
\label{figure}
\end{figure}

\pagebreak
\subsection{Overview of the proof}\label{subsec:method}
Our proof strategy can be summarized as follows.
 \begin{itemize}
\item[Step 1] Take as input the BPZ equation stated in Theorem~\ref{th_bpz_boundary} coming from \cite{ang-lcft-zip}.
\item[Step 2] In the case of one degenerate and three generic boundary points this equation reduces to a hypergeometric equation. Using the solution theory for such equations and the asymptotic analysis of Gaussian multiplicative chaos, one obtains a set of functional equations for the probabilistically defined $H$ and $R$ called \emph{shift equations}. This is the content of Section \ref{sec:reff}.
\item[Step 3] Prove that under certain conditions the solution to these shift equations is unique. Since both $R, R_{\mathrm{FZZ}}$ and $H , \HPT$ are respectively solutions to these shift equations and they obey the conditions for uniqueness, this implies Theorems~\ref{thm:H} and~\ref{main_th2}.
\item[Step 4] Using a relation between $G$ and $R$ recently proved  by \cite{wu2023sle}, we deduce Theorem~\ref{thm:G} from Theorem~\ref{main_th2}. This step and the previous one are performed in Section \ref{sec:main}.
 \end{itemize}  
 
There are several novelties in our proof to overcome new difficulties, as we highlight below:
\begin{itemize}
\item[(1).] In the previous works on boundary theory \cite{remy-fb-formula, rz-boundary},  a handy fact is that the Liouville correlation functions can be written as a moment of Gaussian multiplicative chaos times an explicit prefactor. Under this form the probabilistic definition can be extended meromorphically to a range where the shift equations make sense. This is no longer the case in our paper. We have to introduce proper truncations in the spirit of~\eqref{eq:2-point} to extend the probabilistic definitions of $H$ and $R$, and handle analytic issues coming with this complication. This is performed in Section \ref{sec:precise}. Along the same lines the fact that the Liouville correlations contain both area and boundary GMC measures complicates the proof of several results, including the limit giving $R$ from $H$ (see Section \ref{app:lim_H_R}), and the proof of the operator product expansions (see Appendix \ref{sec:ope}).

\item[(2).]  At several places in the proofs we use a blending of the techniques of the exact solvability of LCFT and  the LQG/SLE coupling, namely conformal welding and mating-of-trees theory. We now summarize the main places where the latter theory enters the picture.
Firstly, the BPZ equation from \cite{ang-lcft-zip} was proved using the mating-of-trees theorem.  Secondly,
in the derivation of the shift equations for $R$, we need the exact value of  $R(\gamma, \sigma_1, \sigma_2)$ coming from the description of the quantum disk in the mating-of-trees theory~\cite{wedges,ag-disk}. This is in contrast to \cite{rz-boundary} where the bulk potential is absent and the corresponding input was supplied by~\cite{RZ_interval}.   
Thirdly, the relation between $R$ and $G$ in our Step 4 relies on a conformal welding result of \cite{wu2023sle}, and the FZZ formula whose proof in~\cite{ARS-FZZ} again uses conformal welding and mating-of-trees.

\item[(3).] The fact that we are working with the complicate formulas $\HPT$ and $G_{\mathrm{Hos}}$ introduces several technicalities. The shift equations satisfied by $\HPT$ consist of three term relations, which is significantly more complicated than the two term relations obeyed by $R_{\mathrm{FZZ}}$ or used in \cite{DOZZ_proof}. The uniqueness argument for such relations is quite involved and is carried out in Section \ref{sec:proof_H}. Similarly, obtaining the formula $G_{\mathrm{Hos}}$ requires an involved computation and uses non-trivial identities of integrals of the functions $\Gamma_{\frac{\gamma}{2}}, S_{\frac{\gamma}{2}}$; see Section~\ref{sec:proof_G}. Finally, we need a number of technical computations on $\HPT, G_{\mathrm{Hos}}$ collected in Appendix~\ref{sec:special}.

\end{itemize}

\subsection{Connection with the fusion and modular kernels of Virasoro conformal blocks}\label{subsub:fusion}
Conformal blocks are fundamental analytic functions  that appear in the conformal bootstrap equations in CFT (see~\eqref{eq:disk_boot1eq}). There are two key linear transformations describing the symmetries of the space of conformal blocks. The two kernels appearing in these two transformations are called the fusion and modular kernels respectively. For rational CFT, the existence and properties of such kernels are well known in physics thanks to the work of Verlinde~\cite{Verlinde1989}, and Moore-Seiberg~\cite{Moore-Seiberg}. In mathematics they are understood in the context of modular tensor category; see e.g.~\cite{Bakalov-Kirillov}. For generic Virasoro conformal blocks with central charge $c>25$, the existence and explicit expressions of the two kernels were conjectured in~\cite{PT-fusion,PT-CMP} and~\cite{Teschner-modular}, respectively; see also~\cite{Teschner-revisited,Lorentz2023}. In a work in preparation, the second and the third named authors of this paper  will prove  these conjectures jointly with P. Ghosal and Y. Sun.

As argued by Ponsot-Teschner \cite{PT_boundary_3pt}, the fusion kernel is related to the boundary three-point Liouville correlation. By matching exact formulas, we found that  the modular kernel is related to the bulk-boundary Liouville correlation in a similar manner. Let us now explicitly write these relations. Below we will use the notation $\alpha_i'$ for the alpha parameters appearing in the fusion and modular kernel to clearly distinguish them from the alpha and beta parameters of the current paper. First we give  the relation for the fusion kernel,
under the parameter identification
\begin{align*}
\beta_1 = \alpha_2', \: \: \beta_2 = Q - \ii P', \: \: \beta_3 = \alpha_3', \: \: 2 \sigma_1 = Q + \ii P, \: \: 2 \sigma_2 = \alpha_1', \: \: 2 \sigma_3 = \alpha_4',
\end{align*}
we have
\begin{align}\label{eq:fusion-rel}
&\mathcal{M}_{\alpha'_1, \alpha'_2, \alpha'_3, \alpha'_4 }^{\mathsf{fusion}}(P,P') =  \frac{1}{2\pi}  \left(\frac{\pi  (\frac{\gamma}{2})^{2-\frac{\gamma^2}{2}} \Gamma(\frac{\gamma^2}{4})}{\Gamma(1-\frac{\gamma^2}{4})} \right)^{\frac{ \beta_1 + \beta_2 + \beta_3 - 2Q}{2\gamma}} \frac{\Gamma_{\frac{\gamma}{2}}(Q) \Gamma_{\frac{\gamma}{2}}(Q-\beta_1)\Gamma_{\frac{\gamma}{2}}(Q-\beta_3)}{ \Gamma_{\frac{\gamma}{2}}(\beta_2 - Q )}\\
&\times \frac{  \Gamma_{\frac{\gamma}{2}}(Q  \pm (2 \sigma_1 - Q))  \Gamma_{\frac{\gamma}{2}}( \frac{\beta_2}{2} \pm \frac{Q -2 \sigma_2}{2} \pm \frac{Q - 2 \sigma_3}{2} ) S_{\frac{\gamma}{2}}(  \frac{ Q + \beta_2 -\beta_3 }{2}  \pm \frac{Q -\beta_1}{2}  ) }{\Gamma_{\frac{\gamma}{2}}( Q - \frac{\beta_1}{2} \pm \frac{Q - 2\sigma_1}{2} \pm \frac{Q - 2 \sigma_2}{2}  ) \Gamma_{\frac{\gamma}{2}}( Q - \frac{ \beta_3}{2}  \pm \frac{Q -2 \sigma_1}{2}  \pm \frac{Q -2 \sigma_3}{2} )  } H
	\begin{pmatrix}
	\beta_1 , \beta_2, \beta_3 \\
	\sigma_1,  \sigma_2,   \sigma_3 
	\end{pmatrix}. \nonumber
\end{align}
Next we give the relation for the modular kernel.
Under the parameter identification
\begin{align*}
\alpha = Q + \ii P', \quad \beta = \alpha', \quad 2 \sigma = Q + \ii P,
\end{align*}
we have
\begin{align}\label{eq:modular-rel}
      \mathcal{M}_{\alpha'}^{\mathsf{modular}}(P,P') & = - \frac{\pi}{2} 2^{ 2 \Delta_{\alpha} - \Delta_{\beta} } \left( \frac{\gamma}{2} \right)^{( \frac{2}{\gamma} - \frac{\gamma}{2})(Q -\alpha) +1}  \left(\frac{\pi  (\frac{\gamma}{2})^{2-\frac{\gamma^2}{2}} \Gamma(\frac{\gamma^2}{4})}{\Gamma(1-\frac{\gamma^2}{4})} \right)^{ \frac{1}{ \gamma}(\alpha + \frac{\beta}{2} - Q)} \\
& \times   \frac{ \Gamma_{\frac{\gamma}{2}}(Q \pm (2 \sigma- Q))  \Gamma_{\frac{\gamma}{2}}(Q - \beta ) \Gamma_{\frac{\gamma}{2}}(Q)}{\Gamma(\frac{2}{\gamma}(\alpha - Q)) \Gamma(\frac{\gamma \alpha}{4} - \frac{\gamma^2}{4}) \Gamma_{\frac{\gamma}{2}}(Q - \frac{\beta}{2} \pm (2 \sigma- Q) )   \Gamma_{\frac{\gamma}{2}}(Q - \frac{\beta}{2}  )^2}    G(\alpha, \beta, \sigma).\nonumber
\end{align}

The relation~\eqref{eq:fusion-rel} between the Virasoro fusion kernel and boundary Liouville three-point function was inspired by a similar relation  in the context of Virasoro minimal models~\cite{Runkel}. 
We expect that the relation~\eqref{eq:modular-rel} we discovered between the modular kernel and the bulk-boundary correlation function has a similar interpretation, but this remains to be understood.

\subsection{Outlook and perspective}\label{subsec:outlook}

\subsubsection{Conformal bootstrap for boundary LCFT}\label{subsec:bootstrap}

Conformal bootstrap is a program for computing  correlation functions of a CFT on a Riemann surface with marked points, based on the decomposition of the marked surface into simpler pieces. The data in this program are the spectrum of its Hamiltonian and the structure constants. For LCFT on surfaces without boundary, the bootstrap program was completely carried out in~\cite{DOZZ_proof,GKRV_bootstrap,GKRV-Segal}, where the structure constant is the 3-point spherical correlation function given by the DOZZ formula~\cite{DOZZ_proof}.  

For the boundary LCFT, the structure constants needed for conformal bootstrap are precisely the boundary three-point and bulk-boundary correlation functions whose expressions we have determined in Theorems \ref{thm:H} and \ref{thm:G}. The FZZ formula will also appear as a special case of the bulk-boundary correlation with the boundary weight $\beta$ set to zero. A first case of the conformal bootstrap for a surface with boundary has recently been proved in~\cite{Wu-annulus} on the annulus with one insertion at each boundary. This bootstrap formula involves the bulk-boundary correlation $G$. The general conformal bootstrap for boundary LCFT, which is the counterpart of~\cite{GKRV-Segal}, is a work in progress by the authors of~\cite{GKRV-Segal,Wu-annulus}. As an example involving the boundary three-point function $H$, we state the bootstrap equation for the boundary four-point function:

\begin{align}\label{eq:disk_boot1eq}
\left \langle \prod_{j=1}^4 e^{\frac{\beta_j}{2} \phi(s_j)} \right \rangle  = \frac{1}{4 \pi} s^{( \frac{Q^2}{4} - \Delta_{\beta_1} - \Delta_{\beta_2}) } \int_{\mathbb{R}} H^{(\beta_1, \beta_2, Q + \ii P)}_{(\mu_1, \mu_2, \mu_3)} H^{(\beta_3, \beta_4, Q - \ii P)}_{(\mu_3, \mu_4, \mu_1)} s^{\frac{P^2}{4}} \mathcal{F}_{\beta_1, \beta_2, \beta_3, \beta_4 }(s,P) dP.
\end{align}

Here the marked points $s_j$ are at the locations $(0, s, 1, \infty)$. The correlation function depends on four cosmological constants $(\mu_j)_{j = 1, \dots, 4}$, and the function $\mathcal{F}_{\beta_1, \beta_2, \beta_3, \beta_4 }(s,P)$ is the conformal block.

\subsubsection{Implications for GMC measures}

In the spirit of the discussions in \cite[Section 1.3]{ARS-FZZ} and \cite[Sections 1.3 and 1.4]{rz-boundary}, we give here some consequences of the exact formulas we have derived for the GMC measures. First of all if one sets all boundary cosmological constants to $0$, the probabilistic expressions for $H$ and $G$ will reduce to a moment of the Gaussian multiplicative chaos area measure of $\mathbb{H}$. Thus our results give exact formulas for these 2d GMC measures. 

The function $R$ can also be used to give the tail expansion of Gaussian multiplicative chaos.  
In particular, we expect that when $\beta \in (\frac{\gamma}{2}, Q)$ 
\begin{align*}
\mathbb{P} \left( \int_{\mathbb{H} \cap \mathbb{D}} \frac{1}{ |x|^{\gamma \beta}  } e^{\gamma h(x)} d^2x > u^2\right) \underset{u \rightarrow \infty }{\sim} - \frac{1}{\Gamma(1 - \frac{2(Q -\beta)}{2})} \frac{R_{\mu_1=0,\mu_2=0}(\beta)}{u^{\frac{2}{\gamma}(Q - \beta)}}.
\end{align*}
Furthermore, the joint law of the area and boundary length distribution under $\cMtwo(\beta)$, which is  encoded by $R_{\mu_1,\mu_2}(\beta)$,  should describe the asymptotic behavior of  
\begin{align*}
\mathbb{P} \left( \int_{\mathbb{H} \cap \mathbb{D}} \frac{1}{ |x|^{\gamma \beta}  } e^{\gamma h(x)} d^2x > u^2,  \int_{-1}^{1} \frac{1}{ |r|^{\frac{\gamma \beta}{2}}  } e^{\frac{\gamma}{2} h(r)} d r > u \right).
\end{align*}
 See \cite{RV_reflection, cercle_toda_reflection} on the role of   reflection coefficients in the tail expansions of GMC measures.

\subsubsection{Interaction between the integrability of LCFT and the mating-of-trees theory.}\label{subsub:SLE}
We see from \eqref{eq:2-point} that the boundary reflection coefficient $R$ encodes the joint law of the length and area of the two-pointed quantum disk sampled from $\cMtwo(\beta)$. With Yu, the first and third named authors considered in~\cite{ASY-triangle} the LQG surfaces with three boundary marked points sampled from $ \prod_{j=1}^3 e^{\frac{\beta_j}{2}\phi(s_j)}\LF_\bbH(d\phi )$, which they called the quantum triangles. The boundary structure constant $H$ encodes the joint area and length distribution of these surfaces. 
As shown in~\cite{ahs-disk-welding,AHS-SLE-integrability,ARS-FZZ,ASY-triangle,wu2023sle,backbone},  quantum disks and triangles behave nicely under conformal welding and facilitate a mutually beneficial interaction between the integrability of LCFT and of the mating-of-trees theory. Results of this paper are not only  outcomes of this interaction, but also enrich it in several ways. For example, in a work in preparation with Yu, the first and third named authors will derive an exact formula for a quantity in mating-of-trees which is shown in~\cite{borga2023baxter} to encode  the expected proportion of inversions in a skew Brownian permuton recently introduced in~\cite{borga2021skew}. Our result on $H$ plays a crucial role there. Moreover, conformal welding results give various relations between the fusion kernel and modular kernel of Virasoro conformal blocks. We plan to investigate the possible connection between these relations and those which are well-known in the modular tensor category  context~\cite{Moore-Seiberg}.

\subsection{Outline of the paper}
The rest of our paper is organized as follows. In Section \ref{sec:precise} we give precise definitions of the functions $\HPT$ and $G_{\mathrm{Hos}}$, of the probabilistic expressions for $H$, $G$, and $R$, and prove the required analytic properties of these functions. In Section \ref{sec:reff} we prove the shift equations for  $H$ and $R$.
In Section \ref{sec:main} we complete the proof of our main theorems using these shift equations. In Section~\ref{sec:technical} we supply the proofs of two facts on $R$ that were used in Section~\ref{sec:reff}. 
In the appendix we provide several technical ingredients needed in the main text.

\medskip
\noindent\textbf{Acknowledgements.} 
The authors are grateful for the hospitality of the Institute for Advanced Study (IAS) at Princeton, where the final stage of the project took place. We would like to very warmly thank Lorenz Eberhardt who helped us fix the uniqueness argument for the boundary three-point function by suggesting to write the functional equations in matrix form. We also thank Baojun Wu and Da Wu for helpful discussions.
M.A.\ was partially supported by NSF grant DMS-1712862, and by the Simons Foundation as a Junior Fellow at the Simons Society of Fellows. G.R.\ was partially supported by an NSF mathematical sciences postdoctoral research
fellowship, NSF Grant DMS-1902804, and by a member fellowship from IAS.
X.S.\ was partially supported by the NSF grant DMS-2027986, the NSF Career award 2046514, a start-up grant from the University of Pennsylvania, and by a member fellowship from IAS.

\section{Definitions and meromorphicity  of $\HPT$, $G_{\mathrm{Hos}}$, $H$, $G$, and $R$ }\label{sec:precise}
In this section we start by giving a full definition of the functions $\HPT$ and $G_{\mathrm{Hos}}$ and the probabilistic definition of $H$ and $G$ under the Seiberg bounds \eqref{eq:seiberg}. We then given an extended definition of $H$ and the definition of $R$ using a truncation procedure. We also state and prove analyticity results for these functions. Throughout this section and the next we will frequently refer to a function of several complex variables as being meromorphic on some domain. By this terminology we simply mean that it can be expressed as the ratio of two holomorphic functions of several variables on that domain.

\subsection{The Ponsot-Teschner and Hosimichi functions}\label{sec:PT_intro}

We start by giving full details on the definition of $\HPT$ given by equation \eqref{formule_PT} and some of its properties. 
We first clarify how the contour $\mathcal{C}$ is chosen. For a function $f$, recall our notation $f(a \pm b) = f(a + b) f(a -b)$. 
Consider the integral
\begin{equation}\label{J_PT_app}
\mathcal{J}_{\mathrm{PT}} := \int_{\mathcal{C}} \frac{S_{\frac{\gamma}{2}}(\frac{Q - \beta_2}{2} + \sigma_3 \pm ( \frac{Q}{2} - \sigma_2) +r)  S_{\frac{\gamma}{2}}(\frac{Q}{2} \pm \frac{Q - \beta_3}{2}  +\sigma_3-\sigma_1+r)}{S_{\frac{\gamma}{2}}(\frac{3Q}{2} \pm \frac{Q - \beta_1}{2}-\frac{\beta_2}{2}+\sigma_3-\sigma_1+r) S_{\frac{\gamma}{2}}(2\sigma_3 + r) S_{\frac{\gamma}{2}}(Q+r)} \frac{dr}{\ii},
\end{equation}
so that  $\HPT$ and $\mathcal{J}_{\mathrm{PT}}$ are then related by an explicit prefactor containing the functions $\Gamma_{\frac{\gamma}{2}}$ and $S_{\frac{\gamma}{2}}$ as in~\eqref{formule_PT}.
To describe how to choose the contour of integration $\mathcal{C}$, we must look at the locations of the poles in $r$ of the integrand. Recall that the function $S_{\frac{\gamma}{2}}$ has poles at $x = -n\frac{\gamma}{2}-m\frac{2}{\gamma}$ and zeros at $x = Q+n\frac{\gamma}{2}+m\frac{2}{\gamma}$ for any $n,m \in \mathbb{N}$ (here and throughout the paper $\mathbb{N}$ contains $0$). We will group the poles of the integrand into two groups, the ones coming from the poles of the numerator and the ones coming from the zeros of the denominator. More precisely consider the set of poles in the lattices extending to the left
\begin{align*}
P^{-}_{\mathrm{PT}} = \left \{ \xi - \frac{n \gamma}{2} - \frac{2 m}{\gamma} \Bigg \vert  n,m \in \mathbb{N}, \xi \in  \{  \frac{\beta_2}{2} - \sigma_2 - \sigma_3,  \frac{\beta_2}{2} - Q +\sigma_2 - \sigma_3, - \frac{\beta_3}{2} - \sigma_3 +\sigma_1, \frac{\beta_3}{2} - Q - \sigma_3 +\sigma_1 \} \right \},
\end{align*}
and the set of poles in the lattices extending to the right:
\begin{align*}
P^{+}_{\mathrm{PT}} = \left \{ \xi + \frac{n \gamma}{2} + \frac{2 m}{\gamma} \bigg \vert  n,m \in \mathbb{N}, \xi \in  \{ - \frac{\beta_1}{2} + \frac{\beta_2}{2} - \sigma_3 + \sigma_1, - Q + \frac{\beta_1}{2} + \frac{\beta_2}{2} - \sigma_3 + \sigma_1, - 2\sigma_3 + Q, 0\} \right \}.
\end{align*}

Assume the parameters $\beta_i, \sigma_i$ are chosen so that $P^{-}_{\mathrm{PT}} \cap P^{+}_{\mathrm{PT}} = \emptyset$. Then the contour $\mathcal{C}$ in $\mathcal{J}_{\mathrm{PT}}$ goes from $- \ii \infty$ to $ \ii \infty$ passing to the right to all the poles in $P^{-}_{\mathrm{PT}}$ and to the left of all the poles in $P^{+}_{\mathrm{PT}}$. In the region of large or very negative imaginary part, namely away from the region containing the poles, the contour $\mathcal{C}$ is chosen simply to be a vertical line on the axis $ \ii \mathbb{R}$. Figure 2 below provides an illustration of the contour $\mathcal{C}$ and the lattices of poles $P^{-}_{\mathrm{PT}}$ and $P^{+}_{\mathrm{PT}}$.
The following lemma verifies that this integral is converging at $\pm \ii \infty$.
\begin{figure}[!htp]
\centering
\includegraphics[width=0.6\linewidth]{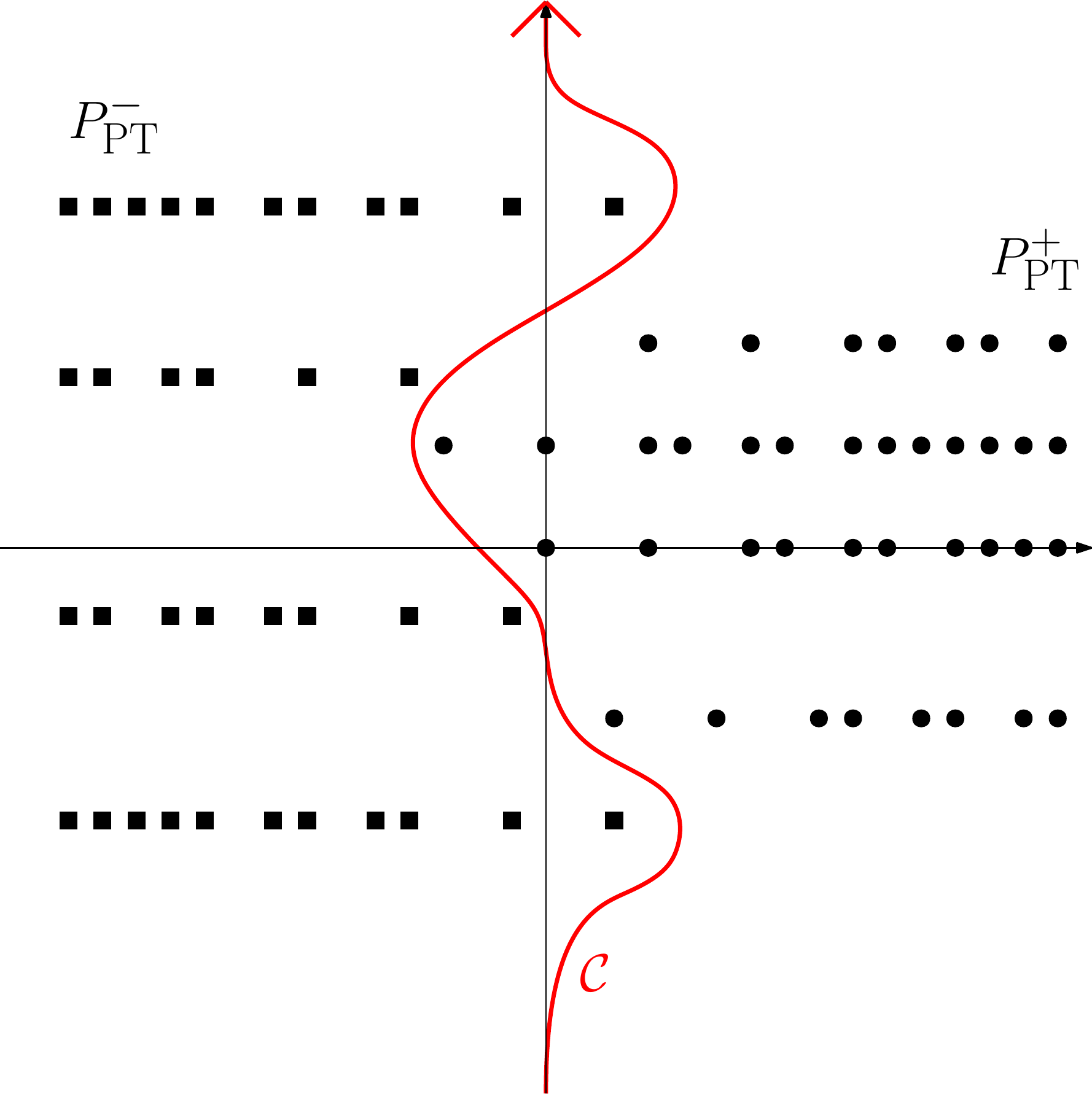}
\caption{The contour of integration $\mathcal{C}$ in the function $\HPT$.}
\label{figure}
\end{figure} 

\begin{lemma}
The integral in \eqref{J_PT_app} is absolutely converging near both $\ii \infty$ and $- \ii \infty$.
\end{lemma}

\begin{proof}
	We need to determine the asymptotic in $r$ of the integrand of the contour integral as $r$ goes to $\pm \ii \infty$. For this we simply need to use the asymptotic of $S_{\frac{\gamma}{2}}$ given by equation \eqref{eq:lim_S}.
	Using this fact, one obtains that the integrand as $r \rightarrow + \ii \infty$ is equivalent to $c_1 e^{3 \ii \pi Q r}$ and as $r \rightarrow - \ii \infty$ to $c_2 e^{-3 \ii \pi Q r}$ for some constants $c_1, c_2 \in \mathbb{C}$ independent of $r$. Since $Q >0$, the integral is absolutely convergent.
\end{proof}
When the parameters $\beta_i, \sigma_i$ are such that $P^{-}_{\mathrm{PT}} \cap P^{+}_{\mathrm{PT}} \neq \emptyset$, the function $\mathcal{J}_{\mathrm{PT}}$   have a pole which can be determined by a residue computation; see the following Lemma \ref{lem:J}.

\begin{lemma}\label{lem:J}
	The poles of the function $\mathcal{J}_{\mathrm{PT}}$ occur when the set  $ \{n \frac{\gamma}{2} + m \frac{2}{\gamma}:n, m \in \mathbb{N}\}$ contains any of the following
	$$
	\begin{array}{llll}
		\frac{\beta_1}{2} - \sigma_1 - \sigma_2, & \quad Q - \frac{\beta_1}{2} - \sigma_1 - \sigma_2, & \quad \frac{\beta_2}{2} - \sigma_2 + \sigma_3 -Q, & \quad \frac{\beta_2}{2} - \sigma_2 - \sigma_3, \\[0.2cm]
	- Q + \sigma_2 + \frac{\beta_1}{2} - \sigma_1, & \quad \sigma_2  - \sigma_1 - \frac{\beta_1}{2}, & \quad -2 Q + \frac{\beta_2}{2} + \sigma_3 + \sigma_2, & \quad -Q + \frac{\beta_2}{2} - \sigma_3 + \sigma_2, \\[0.2cm]
	\frac{\beta_1}{2} - \frac{\beta_2}{2} - \frac{\beta_3}{2}, & \quad Q - \frac{\beta_1}{2} - \frac{\beta_2}{2} - \frac{\beta_3}{2}, & \quad - Q - \frac{\beta_3}{2} + \sigma_3 + \sigma_1, & \quad - \frac{\beta_3}{2} - \sigma_3 + \sigma_1, \\[0.2cm]
	- Q + \frac{\beta_1}{2} - \frac{\beta_2}{2} + \frac{\beta_3}{2}, & \quad \frac{\beta_3}{2} - \frac{\beta_1}{2} - \frac{\beta_2}{2}, & \quad - 2Q + \frac{\beta_3}{2} + \sigma_3 + \sigma_1, & \quad - Q + \frac{\beta_3}{2} - \sigma_3 + \sigma_1.
	\end{array}
	$$
\end{lemma}

\begin{proof} The proof follows exactly the same steps as the one of \cite[Lemma 5.11]{rz-boundary}, except that in our case we have four $S_{\frac{\gamma}{2}}$ functions in the numerator and denominator of the integrand of the contour integral of $\mathcal{J}_{\mathrm{PT}}$ (instead of three in the case of \cite[Lemma 5.11]{rz-boundary}). We nonetheless recall below the steps of the proof and explain how to derive the list of poles.

As long as the parameters $\beta_i, \sigma_i$ are chosen such that $P^{-}_{\mathrm{PT}} \cap P^{+}_{\mathrm{PT}}  = \emptyset$, it is always possible to choose the contour $\mathcal{C}$ so that it goes from $- \ii \infty$ to $ \ii \infty$, passing to the right of the poles of $P^{-}_{\mathrm{PT}}$ and to the left of the poles of $P^{+}_{\mathrm{PT}}$. On the other hand, 
it is impossible to choose such a contour when   $P^{-}_{\mathrm{PT}} \cap P^{+}_{\mathrm{PT}}  \neq \emptyset$, since at least one pole from $P^{-}_{\mathrm{PT}}$ collapses with a pole from $P^{+}_{\mathrm{PT}}$, 
This situation will lead to a pole for the function $\mathcal{J}_{\mathrm{PT}}$. To properly identify this pole, one can start by deforming the contour $\mathcal{C}$ to cross one of the two collapsing poles before they collapse and pick up a residue contribution; see the proof of Lemma \ref{lem:secH2} in Appendix \ref{subsec:PT} for an instance of such deformation of contour. This residue term viewed as a function of $\sigma_i, \beta_i$  have a pole at the value of parameters $\sigma_i, \beta_i$ that caused one pole of $P^{+}_{\mathrm{PT}}$ to collapse with a pole of $P^{-}_{\mathrm{PT}}$. Therefore the poles of $\mathcal{J}_{\mathrm{PT}}$ are at the values of $\beta_i, \sigma_i$ where $P^{-}_{\mathrm{PT}} \cap P^{+}_{\mathrm{PT}}  \neq \emptyset$ which is precisely the list given in the lemma.
\end{proof}

From here it is immediate that $\HPT$ is a meromorphic function on $\mathbb{C}^6$.
\begin{lemma}\label{lem:mero_HPT}
	The function $\HPT$ is meromorphic on $\mathbb{C}^6$, namely a ratio of two holomorphic functions on $\mathbb{C}^6$ of the six parameters $\beta_1, \beta_2, \beta_3, \sigma_1, \sigma_2, \sigma_3$. 
\end{lemma}

\begin{proof}
	In Lemma \ref{lem:J} we have established that $\mathcal{J}_{\mathrm{PT}}$ is meromorphic on $\mathbb{C}^6$ with prescribed poles. Since $\HPT$ is equal to $\mathcal{J}_{\mathrm{PT}}$ times an explicit prefactor containing $\Gamma_{\frac{\gamma}{2}}$ and $S_{\frac{\gamma}{2}}$ functions which are meromorphic, this implies the claim.
\end{proof}

We now repeat a similar analysis for the contour integral 
\begin{align*}
\mathcal{J}_{\mathrm{Hos}} := \int_{ \mathcal{C} } \frac{d \sigma_2}{\ii} e^{2 \ii \pi (Q - 2 \sigma) \sigma_2} \frac{   S_{\frac{\gamma}{2}}( \frac{1}{2}(\alpha + \frac{\beta}{2} - Q) - \sigma_2 ) S_{\frac{\gamma}{2}}( \frac{1}{2}(\alpha + \frac{\beta}{2} - Q) + \sigma_2 ) }{  S_{\frac{\gamma}{2}}( \frac{1}{2}(\alpha - \frac{\beta}{2} + Q) + \sigma_2 )       S_{\frac{\gamma}{2}}(\frac{1}{2}(\alpha - \frac{\beta}{2} + Q) - \sigma_2)}.
\end{align*}
defining the function $G_{\mathrm{Hos}}$. There is one notable difference, which is that there will be a condition on the parameters for the contour integral to converge at $ \pm \ii \infty$. In fact, 
using again the asymptotic of $S_{\frac{\gamma}{2}}$ in~\eqref{eq:lim_S}, the
contour integral is converging at $\pm \ii \infty$ if and only if:
\begin{align*}
\frac{\Re(\beta)}{2} < Q, \quad\textrm{and}\quad  \Re(\sigma) \in \left[ \frac{\Re(\beta)}{4}, Q - \frac{ \Re(\beta)}{4} \right].
\end{align*}
Despite this condition it is still possible to analytically continue the function $\mathcal{J}_{\mathrm{Hos}}$ to a meromorphic function of its three parameters $\alpha, \beta, \sigma$ on $\mathbb{C}^3$.  We do not pursue this here, see \cite{Lorenz_notes} for more details on this fact. For the choice of the contour of integration, similarly as for $\mathcal{J}_{\mathrm{PT}}$, it has to pass to the left of the lattice $P_{\mathrm{Hos}}^{+}$ of poles extending to the right,  and to the right of the lattice $P_{\mathrm{Hos}}^{-}$ of poles extending to the left, where
\begin{align}\label{eq:G-range1}
P_{\mathrm{Hos}}^{-} &= \left \{ \xi - \frac{n \gamma}{2} - \frac{2 m}{\gamma} \quad  \Bigg \vert \quad    n,m \in \mathbb{N}, \quad  \xi \in  \{  \frac{1}{2}(Q - \alpha - \frac{\beta}{2}), \frac{1}{2}(\alpha - \frac{\beta}{2} - Q) \} \right \},\\
P_{\mathrm{Hos}}^{+} &= \left \{ \xi + \frac{n \gamma}{2} + \frac{2 m}{\gamma} \quad \bigg \vert \quad  n,m \in \mathbb{N}, \quad  \xi \in  \{  \frac{1}{2}( \alpha + \frac{\beta}{2} - Q), \frac{1}{2}(Q - \alpha + \frac{\beta}{2} ) \} \right \}.\label{eq:G-range2}
\end{align}

\subsection{Definition of $H$ and $G$ under the Seiberg bound}\label{subsec:def-LF}
Let $h$ be the free boundary Gaussian free field  on   the upper half plane $\bbH=\{ z \in \C: \Im (z)>0 \}$ with covariance kernel  
\begin{equation}\label{covariance}
	\E[h(x)h(y)]= G_\bbH(x,y) :=\log \frac{1}{|x-y||x-\bar{y}|} + 2 \log |x|_+ + 2 \log |y|_+,
\end{equation}
where  $|x|_+ := \max(|x|,1)$ and in the sense that  $\E[(h,f) (h,g)]= \iint f(x)\E[h(x)h(y)]g(y) dxdy$, for smooth test functions $f$ and $g$.
Let $P_\bbH$ be the law of $h$, so that  $P_\bbH$ is a probability measure on  the negatively indexed Sobolev space $H^{-1}(\bbH)$ (\cite{shef-gff,dubedat-coupling}).  The particular covariance kernel corresponds to requiring the field to have average 0 on the unit circle.

Following~\cite{AHS-SLE-integrability}, we introduce the Liouville field on $\bbH$, possibly with boundary insertions.
\begin{definition}[Liouville field]\label{def-LF-rz}
	Let $(h, \mathbf c)$ be sampled from $P_\bbH\times [e^{-Qc}dc]$ and set $\phi =  h(z) -2Q \log |z|_+ +\mathbf c$. 
	We write  $\LF_{\bbH}$ as the law of $\phi$, and call  a sample from   $\LF_{\bbH}$ a \emph{Liouville field on $\bbH$}.
\end{definition}

\begin{definition}[Liouville field with insertions]\label{def_LCFT}
		Let $(\beta_j, s_j) \in  \R \times \partial \bbH$ for $j = 1, \dots, M$, where $M \geq 1$ and the $s_j$ are pairwise distinct. Sample $(h, \mathbf c)$ from $C_{\bbH}^{(\beta_j, s_j)_j} P_\bbH\times [e^{(\frac12\sum_j \beta_j - Q)c}dc]$ where
		\[C_{\bbH}^{(\beta_j, s_j)_j}
		= \prod_{j=1}^M |s_j|_+^{-\beta_j(Q-\frac{\beta_j}2)} \prod_{1\leq j<k \leq M} e^{\frac{\beta_j\beta_k}4 G_\bbH(s_j,s_k)}. \]
		Let $\phi(z) = h(z) - 2Q \log |z|_+ + \sum_{j=1}^M \frac{\beta_j}2 G_\bbH(z, s_j) + \mathbf c$. We write $\LF_\bbH^{(\beta_j, s_j)_j}$ for the law of $\phi$ and call a sample from $\LF_\bbH^{(\beta_j, s_j)_j}$
 the \emph{Liouville field on $\bbH$ with insertions $(\beta_j, s_j)_{1 \leq j \leq M}$}.
 \end{definition}
We can also define Liouville fields with an insertion at $\infty$. We will need the  case $\LF_\bbH^{(\beta_1,0), (\beta_2,1), (\beta_3, \infty)}$, which can be defined by $\lim_{s\to\infty} |s|^{2\Delta_3}  \LF_\bbH^{(\beta_1,0), (\beta_2,1), (\beta_3,s)}$ 
with $\Delta_3=\frac{\beta_3}{2}(Q-\frac{\beta_3}{2})$. Here we give a more explicit definition which is equivalent to the limiting procedure; see \cite[Lemma 2.9]{AHS-SLE-integrability}.
\begin{definition}\label{def:LCFT-infty}
Fix $\beta_1,\beta_2,\beta_3\in \R$. Set $s_1=0,s_2=1,s_3=\infty$ and  $G_\bbH(z,\infty)=2\log |z|_+$.
Sample $(h, \mathbf c)$ from $P_\bbH\times [e^{(\frac12\sum_j \beta_j - Q)c}dc]$. 
We write  $\LF_\bbH^{(\beta_1,0), (\beta_2,1), (\beta_3, \infty)}$ for the law of $\phi$ where $\phi(z) = h(z) - 2Q \log |z|_+ + \sum_{j=1}^3 \frac{\beta_j}2 G_\bbH(z, s_j) + \mathbf c$. 
\end{definition}

Given a sample $h$ from $P_\bbH$, the associated quantum area and length measure are defined by  
\begin{equation}\label{eq:AL}
\cA_h =  \lim_{\eps \to 0} \epsilon^{\frac{\gamma^2}{2}}   e^{\gamma h_{\epsilon}(z)} d^2z, \quad \text{and} \quad \cL_h  = \lim_{\eps \to 0} \epsilon^{\frac{\gamma^2}{4}}   e^{\frac{\gamma}{2} h_{\epsilon}(z)} dz,
\end{equation}
where the limits hold in probability in the weak topology of measures.  The existence of limits is well-known from Gaussian multiplicative chaos (GMC); see e.g.~\cite{Ber_GMC,rhodes-vargas-review}.

\begin{remark}
Note that it is also possible to define the GMC measures by renormalizing by the variance. The difference between the two conventions is given by:
\begin{align*}
\mathcal{A}_h = \frac{ |z|_+^{2 \gamma^2}}{|z - \bar z|^{\frac{\gamma^2}{2}}} \lim_{\eps \to 0}   e^{\gamma h_{\epsilon}(z) - \frac{\gamma^2}{2} \mathbb{E}[h_{\epsilon}(z)^2] } d^2z, \quad \text{and} \quad \mathcal{L}_h  = |z|_+^{\frac{\gamma^2}{2}}\lim_{\eps \to 0}   e^{ \frac{\gamma}{2} h_{\epsilon}(z) - \frac{\gamma^2}{8} \mathbb{E}[h_{\epsilon}(z)^2]} dz.
\end{align*}
\end{remark}

Given a sample $\phi$ from $\LF_\bbH$ or $\LF_\bbH^{(\beta_j, s_j)_j}$, we similarly define $\cA_\phi$ and $\cL_\phi$ as in~\eqref{eq:AL} with $h$ replaced by $\phi$. 
This allows us to rigorously define the function $H$ in~\eqref{c4} for a certain range of parameters.
\begin{definition} \label{def:corr}
Let $\mu_i\ge 0$ for $i=1,2,3$.   Suppose $\beta_1,\beta_2,\beta_3\in \R$ satisfy the Seiberg bound: 
\begin{equation}\label{eq:Seiberg-3pt}
\sum_{i=1}^{3} \beta_i> 2Q\quad \textrm{and}\quad  \beta_i <Q. 
\end{equation} 
Then  the boundary three-point structure constant is  defined by 
\begin{equation}\label{def:H-rig}
H^{(\beta_1 , \beta_2, \beta_3)}_{(\mu_1,\mu_2, \mu_3)} :=\int  
e^{- \cA_\phi(\bbH)-\mu_1  \cL_\phi(-\infty,0) -\mu_2 \cL_\phi(0,1)- \mu_3  \cL_\phi(1,+\infty)}   \;
\LF_\bbH^{(\beta_1,0), (\beta_2,1), (\beta_3, \infty)} (\mathrm d \phi).
\end{equation}
\end{definition}

In a similar way we can give the probabilistic definition of $G$. We start with the definition:
\begin{definition}\label{def:LF_G_function}
Fix $\alpha,\beta \in \R$.
Sample $(h, \mathbf c)$ from $2^{-\frac{\alpha^2}{2}} P_\bbH\times [e^{(\alpha + \frac{ \beta}{2} - Q)c}dc]$. 
We write  $\LF_\bbH^{(\alpha, \ii), (\beta,0)}$ for the law of $\phi$ where $\phi(z) = h(z) - 2Q \log |z|_+ + \alpha G_\bbH(z, \ii) + \frac{\beta}{2} G_\bbH(z, 0) + \mathbf c$.
\end{definition}

We can now give the precise definition of the structure constant $G_{\mu_B}(\alpha, \beta)$.

\begin{definition} \label{def:corr_G}
Let $\mu_B \ge 0$. Suppose $\alpha, \beta \in \R$ satisfy the Seiberg bound: 
\begin{equation}
\alpha + \frac{\beta}{2} > Q, \quad \alpha < Q, \quad \textrm{and}\quad  \beta <Q. 
\end{equation} 
Then  the bulk-boundary structure constant is  defined by: 
\begin{equation}\label{def:G-rig}
G_{\mu_B}(\alpha, \beta) :=\int  
e^{- \cA_\phi(\bbH)-\mu_B  \cL_\phi(\mathbb{R})}   \;
\LF_\bbH^{(\alpha, \ii), (\beta,0) } (\mathrm d \phi).
\end{equation}
\end{definition}

\subsection{Meromorphic extension of $H$}\label{subsec:extend}
To make sense of shift equations for $H$   we  meromorphically extend the range of its definition via an integration by parts, based on the following lemma.

\begin{lemma}\label{lem:def_H_trunc}
In the range of parameters as in Definition~\ref{def:corr}, set $s = \frac12\sum\beta_i - Q$ and
	\begin{equation}\label{eq-hatH}
		\hat H^{(\beta_1 , \beta_2, \beta_3)}_{(\mu_1,\mu_2, \mu_3)} := \int (\gamma (s+\frac\gamma2) A + \frac{\gamma^2}{2} A (\sum_i \mu_i L_i) + \frac{\gamma^2}{4}  (\sum_i \mu_i L_i)^2 ) e^{-A - \sum_i \mu_i L_i} \, \LF_\bbH^{(\beta_1,0), (\beta_2,1), (\beta_3, \infty)} (\mathrm d \phi).
	\end{equation}
	where  $A = \cA_\phi(\bbH), L_1 = \cL_\phi(-\infty,0), L_2 = \cL_\phi(0,1)$ and $L_3 = \cL_\phi(1,\infty)$. Then
	\[s(s+\frac\gamma2)H^{(\beta_1 , \beta_2, \beta_3)}_{(\mu_1,\mu_2, \mu_3)} = 	\hat H^{(\beta_1 , \beta_2, \beta_3)}_{(\mu_1,\mu_2, \mu_3)}.\]
\end{lemma}
\begin{proof}
Let $h$ be a Gaussian free field and $\wt h := h - 2Q \log |\cdot|_+ + \sum \frac{\beta_i}2 G_\bbH(\cdot, s_i)$, where $(s_1,s_2,s_3) = (0,1,\infty)$. Write $\wt A = \cA_{\wt h}(\bbH)$ and $\wt L_1 = \cL_{\wt h}(-\infty, 0), \wt L_2 = \cL_{\wt h}(0,1), \wt L_3 = \cL_{\wt h}(1, +\infty)$. Since $A = e^{\gamma c} \wt A$ and $L_i = e^{\gamma c/2} \wt L_i$, by Definition~\ref{def:corr}, we have
\[
H^{(\beta_1 , \beta_2, \beta_3)}_{(\mu_1,\mu_2 \mu_3)} =\E\left[\int e^{sc}   \cdot 
e^{-e^{\gamma c}\wt A- e^{\gamma c/2}\sum_{i=1}^3\mu_i \wt L_i}  \, dc \right].\]	
For $s>0$,  $a>0$ and $ \ell\in \C$ with $\Re \ell > 0$, applying integration by parts twice we have 
	\begin{align*}
		&\int e^{sc} \cdot e^{-e^{\gamma c}a - e^{\gamma c/2}\ell} \, dc = - \int (\frac1s e^{sc})((-\gamma e^{\gamma c}a - \frac\gamma2 e^{\gamma c/2} \ell)e^{-e^{\gamma c}a - e^{\gamma c/2}\ell})\, dc \\
		&= \frac\gamma s \int e^{sc} (e^{\gamma c} a) e^{-e^{\gamma c}a - e^{\gamma c/2}\ell} \, dc + \frac{\gamma}{2s} \int e^{sc}(e^{\gamma c/2}\ell) e^{-e^{\gamma c}a - e^{\gamma c/2}\ell}\, dc  \\
		&= \frac\gamma s \int e^{sc} (e^{\gamma c}a) e^{-e^{\gamma c}a - e^{\gamma c/2}\ell} \, dc + \frac{\gamma}{2s} \frac{\gamma}{s+\frac\gamma2} \int e^{sc} (e^{\gamma c}a) (e^{\gamma c/2}\ell) e^{-e^{\gamma c}a - e^{\gamma c/2}\ell}\, dc 
		\\&+ \frac{\gamma}{2s}\frac{\frac\gamma2}{s+\frac\gamma2} \int e^{sc}(e^{\gamma c/2}\ell)^2 e^{-e^{\gamma c}a - e^{\gamma c/2}\ell}\, dc . 
	\end{align*}  
	Now setting  $s = \frac12\sum \beta_i - Q$, $a = \wt A$ and $\ell = \sum_i \mu_i \wt L_i$ we get the desired equality:
	\begin{equation*}
		H^{(\beta_1 , \beta_2, \beta_3)}_{(\mu_1,\mu_2 \mu_3)} = \int (\frac\gamma s A + \frac{\gamma}{2s} \frac\gamma{s+\frac\gamma2} A (\sum_i \mu_i L_i) + \frac\gamma{2s} \frac{\frac\gamma2}{s+\frac\gamma2} (\sum_i \mu_i L_i)^2 ) e^{-A - \sum_i \mu_i L_i} \, \LF_\bbH^{(\beta_1,0), (\beta_2,1), (\beta_3, \infty)} (\mathrm d \phi).
	\end{equation*}
\end{proof}
The function $\hat H$  has a better analytic property as shown in the following proposition. 
\begin{proposition}\label{prop:H-def}
Set  $V=\{ (\beta_1,\beta_2,\beta_3):   Q - \frac12 \sum \beta_i < \gamma \wedge \frac{2}{\gamma} \wedge \min_i (Q -\beta_i) \textrm{ and } \beta_i<Q \textrm{ for each }i \}$. 
When $(\beta_1,\beta_2,\beta_3)\in V$ and $\Re \mu_i\ge 0$ for $i=1,2,3$,  the integral in~\eqref{eq-hatH}  defining $\hat H$ is absolutely convergent.
Moreover, for each $(\beta_1,\beta_2,\beta_3)\in V$, the function $\hat H$ is holomorphic on  $\{(\mu_1,\mu_2,\mu_3): \Re \mu_i> 0\}$ and continuous on $\{(\mu_1,\mu_2,\mu_3): \Re \mu_i\ge 0\}$. 
Finally, for each $(\mu_1,\mu_2,\mu_3)$ satisfying $ \Re \mu_i> 0$, the function $\hat H$ can be analytically extended  on a complex neighborhood in $\C^3$of $V$.
\end{proposition}
We defer the proof of Proposition~\ref{prop:H-def} to Section~\ref{subsec-anal-H} and proceed to extend the definition of $H$.
\begin{definition}\label{def:H-general}
	Given $(\beta_1,\beta_2,\beta_3)\in V$ in Proposition~\ref{prop:H-def} and $\Re \mu_i\ge 0$ for $i=1,2,3$, define
\begin{equation}
			\label{eq-H-general}
	H^{(\beta_1 , \beta_2, \beta_3)}_{(\mu_1,\mu_2, \mu_3)} := \frac1{s(s+\frac\gamma2)} \hat H^{(\beta_1 , \beta_2, \beta_3)}_{(\mu_1,\mu_2, \mu_3)}.
\end{equation}
\end{definition}

Recall that $H$ from Theorem~\ref{thm:H} is expressed in the $\sigma$ variable via a change of variable. 
We now transfer  Proposition~\ref{prop:H-def} in terms of the $\sigma$ variable.
\begin{proposition}\label{prop:H-def-sigma}
For $V,s$,  and $H^{(\beta_1 , \beta_2, \beta_3)}_{(\mu_1,\mu_2, \mu_3)}$  as in Definition~\ref{def:H-general}, we define
\begin{align}  \label{eq:H-mu-sigma}
H
\begin{pmatrix}
\beta_1 , \beta_2, \beta_3 \\
\sigma_1,  \sigma_2,   \sigma_3 
\end{pmatrix} : = H^{(\beta_1, \beta_2, \beta_3)}_{(\mu_B(\sigma_1), \mu_B(\sigma_2), \mu_B(\sigma_3))}\quad \textrm{where }\mu_B(\sigma) = \sqrt{\frac{1}{\sin(\pi\frac{\gamma^2}{4})}}\cos \left(\pi\gamma(\sigma-\frac{Q}{2}) \right).
\end{align}
Let \( \cB= (-\frac{1}{2 \gamma} + \frac{Q}{2}, \frac{1}{2 \gamma} + \frac{Q}{2} ) \times \mathbb{R}\) and  \( \overline{\cB}= [-\frac{1}{2 \gamma} + \frac{Q}{2}, \frac{1}{2 \gamma} + \frac{Q}{2} ]\times \mathbb{R}\). 
Then for each 	$(\beta_1,\beta_2,\beta_3)\in V$, the function $(\sigma_1,\sigma_2,\sigma_3)\mapsto H$ is holomorphic on $\mathcal B^3$ and  continuous on $\overline{\mathcal B}^3$. 
Moreover, for each $(\sigma_1,\sigma_2,\sigma_3)\in \mathcal B^3$, the function $(\beta_1,\beta_2,\beta_3)\mapsto s(s+\frac{\gamma}{2})H$ is analytic on a complex neighborhood in $\C^3$of $V$.
\end{proposition} 
\begin{proof}
The function $\mu_B(\sigma)$ is a  holomorphic bijection between $\{ \mu\in \mathbb C: \Re \mu >0 \}$ and $\cB^3$. Moreover, $\mu_B$ is continuous  on   $\overline{\cB}^3$.  Therefore  Proposition~\ref{prop:H-def} yields  Proposition~\ref{prop:H-def-sigma}. 
\end{proof}
 
\subsection{Probabilistic definition of $R$ via truncations}\label{subsec:R-def}
We first recall the definition of two pointed quantum disk  $\cMtwo(\beta)$  used in the definition of $R$. 
Following~\cite[Section 2]{AHS-SLE-integrability}, we describe it using the horizontal strip $\cS=\R\times (0,\pi)$.
Let  $h_\cS(z)=h_\bbH (e^z)$,  where $h_\bbH$ is sampled from $P_\bbH$.
We call  $h_\cS$ a free-boundary GFF on $\cS$. 
The field $h_\cS$ can be written as $h_\cS=h^{\mathrm c}+h^{\ell}$, where 
$h^{\mathrm c}$ is constant on each vertical line, and $h^{\ell}$  has mean zero on all such
lines~\cite[Section 4.1.6]{wedges}. We call $h^{\ell}$ the \emph{lateral component} of the free-boundary GFF on $\cS$.
\begin{definition}
\label{def-thick-disk}
Fix $\beta<Q$. Let 
\[Y_t =
\left\{
\begin{array}{ll}
B_{2t} - (Q -\beta)t  & \mbox{if } t \geq 0 \\
\wt B_{-2t} +(Q-\beta) t & \mbox{if } t < 0
\end{array}
\right. , \]
where $(B_s)_{s \geq 0}$ is a standard Brownian motion  conditioned on $B_{2s} - (Q-\beta)s<0$ for all $s>0$,\footnote{Here we condition on a zero probability event. This can be made sense of via a limiting procedure.}  and $(\wt B_s)_{s \geq 0}$ is an independent copy of $(B_s)_{s \geq 0}$.
Let $h^1(z) = Y_{\Re z}$ for each $z \in \cS$.
Let $h^2_\cS$ be independent of $h^1$  and have the law  the  lateral component of the free-boundary GFF on $\cS$.
Let $\psi = h^1+h^2_\cS$.
Let  $\mathbf c$ be a real number  sampled from $\frac\gamma2 e^{(\beta-Q)c}dc$ independent of $\psi$ and $\phi=\psi +\mathbf c$.
Let $\cMtwo(\beta)$ be the infinite measure  describing the law of $\phi$. 
\end{definition}

For $\phi$ in Definition~\ref{def-thick-disk},  define \(\cA_\phi =  \lim_{\eps \to 0} \epsilon^{\frac{\gamma^2}{2}}   e^{\gamma \phi_{\epsilon}(z)} d^2z\) on $\cS$ 
 and \( \cL_\phi  = \lim_{\eps \to 0} \epsilon^{\frac{\gamma^2}{4}}   e^{\frac{\gamma}{2} \phi_{\epsilon}(z)} dz\) on $\partial \cS$. Write $A=\cA_\phi (\cS)$ as the total $\cA_\phi$-area,  and write $L_1$, $L_2$ as the $\cL_\phi$-length of top and bottom boundary arc of $\cS$, respectively. It appears that $	\int e^{-A-\mu_1L_1-\mu L_2}  \,  \mathrm d \cMtwo(\beta)=\infty$. 
For $\beta\in (\frac2{\gamma},Q)$, the correct definition is via truncation and we have the following counterpart of Lemma~\ref{lem:def_H_trunc}.
\begin{lemma}\label{lem:def_R_trunc}
Let $\beta\in (\frac2{\gamma},Q)$ and write $A, L_1,L_2$ for the quantum area and boundary arc lengths of a sample from $\cMtwo(\beta)$ as above. Then the following integral defining $R_{\mu_1,\mu_2}(\beta)$ is finite:
\begin{equation}\label{eq:2-point-sec2}
R_{\mu_1,\mu_2}(\beta):=\frac{2(Q-\beta)}\gamma \int (e^{-A-\mu_1L_1-\mu_2 L_2} -1 ) \,  \mathrm d \cMtwo(\beta).
\end{equation}
Furthermore set $s = \beta - Q$ and
\begin{equation}\label{eq-hatR}
 		\hat R_{\mu_1,\mu_2}(\beta) := \frac{-2s}\gamma\int (\gamma (s+\frac\gamma2) A + \frac{\gamma^2}{2} A (\sum_i \mu_i L_i) + \frac{\gamma^2}{4}  (\sum_i \mu_i L_i)^2 ) e^{-A - \sum_i \mu_i L_i} \, d\cMtwo(\beta).
 	\end{equation}Then $s(s+\frac\gamma2)R_{\mu_1,\mu_2}(\beta)= \hat R_{\mu_1,\mu_2}(\beta)$.
\end{lemma}
\begin{proof}
    Recall the field $\psi$ used in Definition~\ref{def-thick-disk} and write $\wt A = A_{\psi}(\cS)$ and $L_1 = \cL_\psi(\R \times \{\pi\}), L_2 = \cL_\psi(\R)$. By definition we have
    \[ R_{\mu_1, \mu_2}(\beta) = -\frac{2s}\gamma \E[\int \frac\gamma2 e^{sc} \cdot (e^{-e^{\gamma c} \wt A - e^{-\gamma c/2} \sum_{i=1}^2 \mu_i \wt L_i }-1)\, dc].\]
    
    We have $s \in (-\frac \gamma2, 0)$. For $a>0$ and $ \ell\in \C$ with $\Re \ell > 0$, applying integration by parts twice as in the proof of Lemma~\ref{lem:def_H_trunc}, the expression $\int e^{sc} \cdot (e^{-e^{\gamma c}a - e^{\gamma c/2}\ell} -1) \, dc$ equals 
    \begin{align*}
    &\frac\gamma s \int e^{sc} (e^{\gamma c}a) e^{-e^{\gamma c}a - e^{\gamma c/2}\ell} \, dc + \frac{\gamma}{2s} \frac{\gamma}{s+\frac\gamma2} \int e^{sc} (e^{\gamma c}a) (e^{\gamma c/2}\ell) e^{-e^{\gamma c}a - e^{\gamma c/2}\ell}\, dc 
		\\&+ \frac{\gamma}{2s}\frac{\frac\gamma2}{s+\frac\gamma2} \int e^{sc}(e^{\gamma c/2}\ell)^2 e^{-e^{\gamma c}a - e^{\gamma c/2}\ell}\, dc . \end{align*}
  Applying this with $a = \wt A$ and $\ell = \sum_{i=1}^2 \mu_i \wt L_i$ gives the desired result.
 
\end{proof}
The following proposition extends the definition of $R$ in the same way as for $H$ in Proposition~\ref{prop:H-def}. We defer its proof to Appendix~\ref{app:GMC} as it is similar to that of Proposition~\ref{prop:H-def}. 
\begin{proposition}\label{prop:R-def}
For $\beta \in ( (\frac2\gamma - \frac\gamma2)  \vee \frac\gamma2, Q)$ and $\Re \mu_1, \Re \mu_2 \geq 0$,  the integral in~\eqref{eq-hatR}  defining $\hat R$ is absolutely convergent. 
Moreover, for each  $\beta \in ( (\frac2\gamma - \frac\gamma2)  \vee \frac\gamma2, Q)$, the function $(\mu_1,\mu_2)\mapsto \hat R$ is holomorphic on 	$\{\Re \mu_1> 0, \Re \mu_2> 0\}$ and continuous on $\{\Re \mu_1 \ge  0, \Re \mu_2\ge 0\}$. 
Finally, for $\Re \mu_1 >  0, \Re \mu_2 > 0$, the function $\beta\mapsto \hat R$ can be analytically extended on a complex neighborhood of  $( (\frac2\gamma - \frac\gamma2)  \vee \frac\gamma2, Q)$.
\end{proposition}

\begin{definition}\label{def-R-general}
 	For $\beta \in ( (\frac2\gamma - \frac\gamma2)  \vee \frac\gamma2, Q)$ and $\Re \mu_1, \Re \mu_2 \geq 0$, writing  $s = \beta - Q$, we define 
 	\begin{equation}
 		\label{eq-R-general}
 		R_{\mu_1,\mu_2}(\beta):= \frac1{s(s+\frac\gamma2)} \hat R_{\mu_1,\mu_2}(\beta).
 	\end{equation}
 \end{definition}
We have the following counterpart of Proposition~\ref{prop:H-def-sigma} on the  change of variable. 
\begin{proposition}\label{prop:R-def-sigma}
Recall $\mu_B(\sigma)$ and \( \cB= (-\frac{1}{2 \gamma} + \frac{Q}{2}, \frac{1}{2 \gamma} + \frac{Q}{2} ) \times \mathbb{R}\) in Proposition~\ref{prop:H-def-sigma}. Let  $R(\beta, \sigma_1, \sigma_2)   : =  R_{\mu_B(\sigma_1),\mu_B(\sigma_2)}(\beta)$. For each $\beta \in ( (\frac2\gamma - \frac\gamma2)  \vee \frac\gamma2, Q)$,
 the function $(\sigma_1,\sigma_2)\mapsto R$ is holomorphic on $\mathcal B^2$ and  continuous on $\overline{\mathcal B}^2$. 
 For each $(\sigma_1,\sigma_2) \in \mathcal B^2$, the function $\beta \mapsto s(s+\frac{\gamma}{2})R$ is analytic on a complex neighborhood  of $( (\frac2\gamma - \frac\gamma2)  \vee \frac\gamma2, Q)$.
\end{proposition} 
\begin{proof}
This follows from Proposition~\ref{prop:R-def} similarly as in the proof of Proposition~\ref{prop:H-def-sigma}.
\end{proof}

 	\begin{remark}\label{rem-H}
The meromorphic extension of $H$ can also be done by truncations in the spirit of~\eqref{eq:2-point-sec2}. Indeed,
		suppose $\mu_1,\mu_2,\mu_3 \in \C$ satisfy $\Re \mu_i \geq 0$ for $i=1,2,3$, and  $(\beta_1,\beta_2,\beta_3)$ satisfy $0<Q - \frac12 \sum \beta_i <  \frac\gamma2 \wedge \min_i (Q -\beta_i)$, then 
 		\begin{equation}\label{def:H-rig-1}
 		H^{(\beta_1 , \beta_2, \beta_3)}_{(\mu_1,\mu_2 ,\mu_3)} =\int  
 		\left(e^{- \cA_\phi(\bbH)-\sum_{i=1}^3\mu_i L_i} -1\right)  \;
 		\LF_\bbH^{(\beta_1,0), (\beta_2,1), (\beta_3, \infty)} (\mathrm d \phi)
 		\end{equation}
 		If instead
 		$(\beta_1,\beta_2,\beta_3)$ satisfy 
 		$
 		\frac{\gamma}{2}<Q - \frac12 \sum \beta_i < 
 		\gamma\wedge \frac{2}{\gamma} \wedge \min_i (Q -\beta_i)
 		$, then
 		\begin{equation}\label{def:H-rig-2}
 		H^{(\beta_1 , \beta_2, \beta_3)}_{(\mu_1,\mu_2 ,\mu_3)} =\int  
 		\left( e^{- \cA_\phi(\bbH)-\sum_{i=1}^3\mu_i L_i}  -1 +\sum_{i=1}^3 \mu_i L_i\right) \;
 		\LF_\bbH^{(\beta_1,0), (\beta_2,1), (\beta_3, \infty)} (\mathrm d \phi).
 		\end{equation}
Both~\eqref{def:H-rig-1} and~\eqref{def:H-rig-2} follow from integration by parts. One can further extend the range of $H$ via further truncations of $e^{-x}$. We leave these details to the interested readers as we will not need them in this paper.
Similarly, we record without a detailed proof that when	$\beta \in ((Q - \gamma) \vee \frac\gamma2, \frac2\gamma)$,
		\begin{equation*}
		R_{\mu_1,\mu_2}(\beta) =\frac{2(Q-\beta)}\gamma \int  
		\left( e^{- A-\sum_{i=1}^2\mu_i L_i}  -1 +\sum_{i=1}^2 \mu_i L_i\right) \;
		d\cM_2^\disk(\beta).
		\end{equation*}
	\end{remark}

\subsection{Analytic continuation of $H$: proof of Proposition~\ref{prop:H-def}}\label{subsec-anal-H}
We first prove Proposition~\ref{prop:analytic-appendix} which will immediately imply Proposition~\ref{prop:H-def}. 
Our argument follows the same strategy as in~\cite{DOZZ_proof} for the sphere case, with the additional difficulty that 
the correlation functions are no longer  GMC moments due to the presence of both bulk and boundary Liouville potentials.

\begin{proposition}\label{prop:analytic-appendix}
Recall the domain $V$ from Definition~\ref{lem:def_H_trunc} and Proposition~\ref{prop:H-def}. Fix $\mu_i$ with $\Re \mu_i > 0$ for $i=1,2,3$.
Then for each of 
	\begin{align*}
	(\beta_1, \beta_2, \beta_3) &\mapsto \int A e^{-A - \sum \mu_i L_i} \, \LF_\bbH^{(\beta_1,0) (\beta_2, 1), (\beta_3, \infty)}(d\phi),\\
	(\beta_1, \beta_2, \beta_3) &\mapsto \int A (\sum \mu_i L_i) e^{-A - \sum \mu_i L_i} \, \LF_\bbH^{(\beta_1,0) (\beta_2, 1), (\beta_3, \infty)}(d\phi),\\
	(\beta_1, \beta_2, \beta_3) &\mapsto \int (\sum \mu_i L_i)^2 e^{-A - \sum \mu_i L_i} \, \LF_\bbH^{(\beta_1,0) (\beta_2, 1), (\beta_3, \infty)}(d\phi),
	\end{align*}
the integral converges  when $(\beta_1,\beta_2,\beta_3)\in V$, and the function 
extends analytically to a neighborhood of $V$ in $\C^3$. Here, we write $A = \cA_\phi(\bbH)$, $L_1 = \cL_\phi(-\infty,0)$, $L_2 = \cL_\phi(0,1)$ and $L_3 = \cL_\phi(1,\infty)$.   
\end{proposition}
We only prove the result for the second function in detail as the other two can be proved similarly with the same inputs. 
To simplify notation, we will assume that $\mu_1=\mu_2=\mu_3=1$  but the same argument works for the general case. By the change of coordinates in \cite[Lemma 2.20, Proposition 2.16]{AHS-SLE-integrability}, modulo an explicit prefactor
the second function becomes
	\[f(\beta_1, \beta_2, \beta_3) = \int \cA_\phi(\bbH) \cL_\phi(\R) e^{-\cA_\phi(\bbH) -  \cL_\phi(\R)} \, \LF_\bbH^{(\beta_j,p_j)_j}(d\phi)\quad \textrm{with }(p_1, p_2, p_3) = (0,2,-2).\]
Therefore, it suffices to prove that $f$ extends analytically to a neighborhood of $V$ in $\C^3$. 
To show this, we will approximate $f$ using truncations of $\bbH$ away from the insertions.  
\begin{lemma}\label{lem:claim1}
For $r \geq 1$, let $\bbH_r := \bbH\backslash \bigcup_i B_{e^{-r}}(p_i)$ and $\R_r := \R \backslash \bigcup_i B_{e^{-r}}(p_i)$.
	For a Gaussian free field   $h $  sampled from $P_\bbH$, write $A_r := \cA_{h-2Q\log |\cdot|_+}(\bbH_r)$ and $L_r := \cL_{h-2Q\log |\cdot|_+}(\R_r)$. For $x \in \R$, write $h_r(x)$ for the average of $h$ on $\partial B_{e^{-r}}(x) \cap \bbH$. 
	For $r\geq 1$, let 
	\[f_r(\beta_1, \beta_2, \beta_3) := \int_\R dc \, e^{(\frac12\sum\beta_i - Q)c} \E\left[ \prod_i e^{\frac{\beta_i}2 h_r(p_i) - \frac{\beta_i^2}4r} (e^{\gamma c} A_r) (e^{\frac\gamma2 c} L_r) \exp(-e^{\gamma c} A_r -  e^{\frac\gamma2 c} L_r ) \right]\]
Then each  point in  $V$ has a neighborhood in $\C^3$ on which $f_r$ converges uniformly as $r \to \infty$.    
Moreover, $\lim_{r \to \infty} f_r(\beta_1, \beta_2, \beta_3) = f(\beta_1, \beta_2, \beta_3)$ for all $(\beta_1, \beta_2, \beta_3) \in V$.
\end{lemma}
Proposition~\ref{prop:analytic-appendix} follows immediately from Lemma~\ref{lem:claim1}.
\begin{proof}[Proof of Proposition~\ref{prop:analytic-appendix}]
Since $f_r$ is analytic,
by Lemma~\ref{lem:claim1}, there is a neighborhood of $V$ on which $\lim_{r \to \infty} f_r$ is holomorphic. Moreover, this limit agrees with $f$ on $V$, hence is the desired analytic continuation of $f$.
\end{proof}

\begin{proof}[Proof of Lemma~\ref{lem:claim1}]
	If we condition on $h|_{\bbH_r}$, then $(h_{r+1}(p_i) - h_{r}(p_i))_{i=1,2,3}$ is a triple of conditionally independent Gaussians. Therefore
	\[f_r(\beta_1, \beta_2, \beta_3) = \int_\R dc \, e^{(\frac12\sum\beta_i - Q)c} \E\left[ \prod_{i=1}^3 e^{\frac{\beta_i}2 h_{r+1}(p_i) - \frac{\beta_i^2}4(r+1)} (e^{\gamma c} A_r) (e^{\frac\gamma2 c} L_r) \exp(-e^{\gamma c} A_r -  e^{\frac\gamma2 c} L_r ) \right].\]
Write $g(a,\ell) := a e^{-a} \ell e^{- \ell}$. We have $f_{r+1}(\beta_1, \beta_2, \beta_3) -f_r(\beta_1, \beta_2, \beta_3)$ equals
\begin{align*}
\int_\R dc \, e^{(\frac12\sum\beta_i - Q)c} \E\left[ \prod_i e^{\frac{\beta_i}2 h_{r+1}(p_i) - \frac{\beta_i^2}4 (r+1)} ( g(e^{\gamma c} A_{r+1}, e^{\frac\gamma2 c} L_{r+1}) - g(e^{\gamma c} A_r, e^{\frac\gamma2 c} L_r) )\right] .
\end{align*}
Fix a point in $V$ and let $O$ be a neighborhood of the point. We will require $O$ to be sufficiently small as explained below. For  $(\beta_1,\beta_2,\beta_3) \in O$ write $\beta_j = u_j +  v_j \ii$ where $u_j,v_j \in \R$. {Set $\tilde s := \frac{1}{2} \sum_i u_i -Q $.}
By the definition of $V$, we can choose $O$ small  enough such that  $\tilde s> -\gamma$.
Write $\wt A_r := \cA_{\wt h}(\bbH_r)$ and $\wt L_r := \cL_{\wt h}(\R_r)$ where $\wt h = h + \sum \frac{u_i}2 G_\bbH(\cdot, p_i) -2Q\log |\cdot|_+$. By Girsanov's theorem,
	\begin{align}
	&|f_{r+1}(\beta_1, \beta_2, \beta_3) -f_r(\beta_1, \beta_2, \beta_3)| \nonumber\\
	&\leq   \int_\R dc \, e^{\tilde s c} \E\left[ \left|\prod_i e^{\frac{\beta_i}2 h_{r+1}(p_i) - \frac{\beta_i^2}4 (r+1)} ( g(e^{\gamma c} A_{r+1}, e^{\frac\gamma2 c} L_{r+1}) - g(e^{\gamma c} A_r, e^{\frac\gamma2 c} L_r) )\right| \right] \nonumber \\
	&\leq C_0 e^{\frac {r+1}4 \sum v_i^2} \int_\R dc \, e^{\tilde s c} \E \left[ |g(e^{\gamma c} \wt A_{r+1}, e^{\frac\gamma2 c} \wt L_{r+1}) - g(e^{\gamma c} \wt A_{r}, e^{\frac\gamma2 c} \wt L_r)| \right], \label{eq-C0}
	\end{align}
 where $C_0 = \sup_{O} \E[ \prod_i e^{\frac {u_i}2 h_{r+1}(p_i) - \frac{u_i^2}4 (r+1)}]$ is a constant not depending on $r$, which is finite if $O$ is sufficiently small.
 Now, using the triangle inequality, and the fact that $\ell \mapsto \ell e^{- \ell}$ is bounded and Lipschitz, we have for any  $a_{r+1} > a_r > 0$ and $\ell_{r+1} > \ell_r > 0$ that
	\[|g(a_{r+1}, \ell_{r+1}) - g(a_r, \ell_r)| \leq C_1( (\ell_{r+1} - \ell_r) a_{r+1} e^{-a_{r+1}} + (a_{r+1} - a_r) e^{-a_{r+1}} + a_r(e^{-a_r} - e^{-a_{r+1}})),\]
 where $C_1$ is a universal constant.
 This gives three terms to bound. 
Now using the identity $\int  e^{tc} e^{-e^{\gamma c} a}\, \mathrm dc = \frac1\gamma \Gamma(\frac {t}\gamma) a^{-\frac {t}\gamma}$ for $t,\ell >0$, we have
	\begin{align*}
			\int_\R dc \, e^{(\tilde s+\gamma + \frac\gamma2)c} \E[(\wt L_{r+1}-\wt L_r) \wt A_{r+1} e^{-e^{\gamma c} \wt A_{r+1}}] &= \frac1\gamma \Gamma(\frac{\tilde s}\gamma + \frac32) \E[(\wt L_{r+1} - \wt L_r) \wt A_{r+1}^{-\frac{\tilde s}\gamma - \frac12}], \\
	\int_\R dc \, e^{(\tilde s+\gamma)c} \E[(\wt A_{r+1} - \wt A_r)e^{-e^{\gamma c} \wt A_{r+1}}] &= \frac1\gamma\Gamma(\frac{\tilde s}{\gamma} + 1) \E[(\wt A_{r+1} - \wt A_r) \wt A_{r+1}^{-\frac{\tilde s}{\gamma} - 1}],  \\
	\int_\R dc \, e^{( \tilde s+\gamma)c} \E[\wt A_r(e^{-e^{\gamma c} \wt A_{r}} - e^{-e^{\gamma c} \wt A_{r+1}} )] &= \frac1\gamma \Gamma(\frac{\tilde s}\gamma + 1) \E[\wt A_r(\wt A_r^{-\frac{\tilde s}{\gamma} - 1} - \wt A_{r+1}^{-\frac{\tilde s}{\gamma} - 1})].
\end{align*} 
Now Lemma~\ref{lem:claim1} follows from Lemma~\ref{lem-exp-bound-holo}, which we state right below.
 \begin{lemma}\label{lem-exp-bound-holo}
	For each point in $V$ there is a neighborhood $U \subset \R^3$ and a constant $C>0$ such that the following holds. For $(u_1,u_2,u_3) \in U$ and $r \geq 1$, with $\tilde s$, $\wt A_r$ and $\wt L_r$ defined above,  
	\[ \E[(\wt A_{r+1} - \wt A_r)\wt A_{r+1}^{-\frac{\tilde s}{\gamma} - 1}], \quad \E[\wt A_r(\wt A_r^{-\frac{\tilde s}{\gamma} - 1} - \wt A_{r+1}^{-\frac{\tilde s}{\gamma} - 1})], \quad 
	\E[(\wt L_{r+1} - \wt L_r) \wt A_{r+1}^{-\frac{\tilde s}{\gamma} - \frac12}] \]
	are all bounded by $C e^{-r/C}$.  
	Moreover, write $\wt A_\infty:= \cA_{\wt h}(\bbH)$ and $\wt L_\infty:= \cL_{\wt h}(\R)$. Then the same holds with $\wt A_\infty,\wt L_\infty $ in place of $\wt A_{r+1},\wt L_{r+1} $.
\end{lemma}
	By  Lemma~\ref{lem-exp-bound-holo}, we can choose  $O$ small enough and $C$ large enough so that $|f_{r+1}(\beta_1, \beta_2, \beta_3) -f_r(\beta_1, \beta_2, \beta_3)| < C e^{(\frac14 \sum v_i^2 - C^{-1})r}$ uniformly in $O$. By possibly shrinking $O$ further so that $\sum v_i^2 - 1/C < -1/2C$ for all $(\beta_1, \beta_2, \beta_3) \in O$, we see that the $f_r$ uniformly  converge in $O$ as desired. 

It remains to prove  $\lim_{r \to \infty} f_r(\beta_1, \beta_2, \beta_3) = f(\beta_1, \beta_2, \beta_3)$ for $(\beta_1, \beta_2, \beta_3) \in V$. 
In this case since $\beta_1,\beta_2,\beta_3$ are real we have
 \[|f(\beta_1, \beta_2, \beta_3) - f_r(\beta_1, \beta_2, \beta_3)| \leq C_0 \int_\R dc \, e^{(\frac12\sum\beta_i - Q)c} \E[| g(e^{\gamma c} \wt A_\infty, e^{\frac\gamma2 c} \wt L_\infty) -   g(e^{\gamma c} \wt A_r, e^{\frac\gamma2 c} \wt L_r)|].\] 
 Now by the same argument as before with the $\wt A_\infty,\wt L_\infty$ case of Lemma~\ref{lem-exp-bound-holo} being applied, we get $\lim_{r \to \infty} f_r(\beta_1, \beta_2, \beta_3) = f(\beta_1, \beta_2, \beta_3)$. 
 \end{proof}

\begin{proof}[Proof of Lemma~\ref{lem-exp-bound-holo}]
Fix a choice of $(\tilde \beta_1, \tilde \beta_2, \tilde \beta_3) \in V$; we will find a suitably small neighborhood $U$ of this point.
We will only explain the $A_{r+1},L_{r+1}$ case in detail since the  $A_{\infty},L_{\infty}$ follows from  the same argument. 
Moreover, we will focus on a fixed $(u_1,u_2, u_3) \in U$,
but all inputs in the argument vary continuously in $(u_1,u_2, u_3)$ which gives the desired uniform bounds in  $U$. 
We now prove the three bounds one by one. We write $x \lesssim y$ to denote that $x \leq k y$ for some constant $k$ not depending on $r$ (but which may depend on $U$).
\medskip 
	
	\noindent \textbf{First bound.}
For $\eps>0$, we have 
$\E[\wt A_{r+1}^{- \frac {\tilde s} \gamma - 1}  (\wt A_{r+1} - \wt A_r)]\leq \eps \E[\wt A_{r+1}^{- \frac {\tilde s} \gamma - 1}] + \E[1_{\wt A_{r+1} - \wt A_r>\eps}\wt A_{r+1}^{- \frac {\tilde s} \gamma}]$. 
By H{\"o}lder inequality,  $\E[1_{\wt A_{r+1} - \wt A_r>\eps}\wt A_{r+1}^{- \frac {\tilde s} \gamma}]\le \E[\wt A_{r+1}^{-\frac{\tilde s}\gamma p}]^{1/p} \P[\wt A_{r+1} - \wt A_r > \eps]^{1/q}$ 
for  $p,q>1$ satisfying $\frac1p+\frac1q = 1$. By the definition of $V$, we can choose $U$ and $p$ small enough such that both $-\frac{\tilde s}{\gamma} - 1$ and $  -\frac{\tilde s}\gamma p  $ lie in $ (-\infty, \frac2{\gamma^2} \wedge \frac1\gamma \min(Q-\tilde \beta_i) )$. By \cite[Corollary 6.11]{hrv-disk}, for $\lambda < \frac2{\gamma^2} \wedge \frac1\gamma \min(Q-\tilde \beta_i)$
we have $\E[\wt A_{r+1}^{\lambda}]< \E[\wt A_1^{\lambda}]+\E[\wt A_\infty^{\lambda}] <\infty$. Therefore $\E[\wt A_{r+1}^{- \frac {\tilde s} \gamma - 1}  (\wt A_{r+1} - \wt A_r)]\lesssim \eps  + \P[\wt A_{r+1} - \wt A_r > \eps]^{1/q}$.
 Now, for  $m>0$ small enough  so that $Q-\max_i \tilde \beta_i - m\gamma>0$, 
 by the multifractal spectrum of LQG (e.g.\ \cite[Theorem 3.23]{berestycki-lqg-notes} with  minor adaptation for the boundary setting) we have 
	\begin{equation}\label{eq-decay-A}
		\P[\wt A_{r+1} - \wt A_r > \eps]\leq \eps^{-m} \E[(\wt A_{r+1} - \wt A_r)^m]  \lesssim \eps^{-m} e^{-m\gamma(Q - \max_i \tilde \beta_i - m\gamma)r}
	\end{equation} 
	where the term $e^{m \frac\gamma2 r \max_i \tilde \beta_i}$ comes from the log singularities added to the field. Choosing $\eps = e^{-r/C}$ for large enough $C$  yields 
$\E[\wt A_{r+1}^{- \frac {s} \gamma - 1}  (\wt A_{r+1} - \wt A_r)]\leq  Ce^{-r/C}$.

	\medskip 
	
	\noindent \textbf{Second bound.} Choose  $U$ small enough so that $-\frac {\tilde s}\gamma - 1 < 0$. 
 For $a_{r+1}>a_r>0$ we have $a_r(a_r^{-\frac s\gamma - 1} - a_{r+1}^{-\frac s\gamma - 1}) \lesssim (a_{r+1}-a_r) a_r^{-\frac s\gamma - 1}$. As in the proof of the first equality,  for $\eps>0$ we have 
 \[\E[\wt A_r(\wt A_r^{-\frac s\gamma - 1} - \wt A_{r+1}^{-\frac s\gamma - 1})] \lesssim \eps \E[\wt A_r^{-\frac s\gamma - 1}] + \E[1_{\wt A_{r+1} - \wt A_r > \eps} \wt A_r^{-\frac s\gamma}] \lesssim \eps + \E[\wt A_{r}^{-\frac{s}\gamma p}]^{1/p} \P[\wt A_{r+1} - \wt A_r > \eps]^{1/q} . \]
Now the same argument   gives $\E[\wt A_r(\wt A_r^{-\frac s\gamma - 1} - \wt A_{r+1}^{-\frac s\gamma - 1})] \le Ce^{-r/C}$ for large enough $C$.

	\medskip 
	
	\noindent \textbf{Third bound.} Arguing as in the first case, we have
	\[\E[(\wt L_{r+1} - \wt L_r) \wt A_{r+1}^{-\frac s\gamma - \frac12}] \lesssim \eps + \E[(\wt L_{r+1} \wt A_{r+1}^{-\frac s\gamma - \frac12})^p]^{1/p} \P[\wt L_{r+1} - \wt L_r > \eps]^{1/q}.\]
	We now show $\E[(\wt L_{r+1} \wt A^{-\frac s\gamma - \frac12})^p]$ is uniformly bounded in $r$, from which the same argument gives $\E[(\wt L_{r+1} - \wt L_r) \wt A_{r+1}^{-\frac s\gamma - \frac12}]\le Ce^{-r/C}$ for large enough $C$. 
	
	Since $\frac12 + (-\frac s\gamma -\frac12) = -\frac s\gamma < \frac2{\gamma^2}\wedge \min_i  \frac1\gamma (Q-\beta_i)$ we can choose $\lambda_1, \lambda_2$ satisfying $\frac1{\lambda_1} + \frac1{\lambda_2}=1$ such that $\frac12 \lambda_1$ and $(-\frac s\gamma -\frac12)\lambda_2$ are each bounded above by $\frac2{\gamma^2}\wedge \frac1\gamma \min_i   (Q-\beta_i)$.  H\"older's inequality gives $\E[(\wt L_{r+1} \wt A^{-\frac s\gamma - \frac12})^p] \leq \E[\wt L_{r+1}^{\lambda_1 p}]^{1/\lambda_1} \E[(\wt A_{r+1}^{-\frac s\gamma - \frac12})^{p \lambda_2}]^{1/\lambda_2}$.
	Choose $p>1$ sufficiently small and recall 
	 our moment bound for $\wt A_{r+1}$ in the first case. By \cite[Corollary 3.10]{hrv-disk} the analogous bound  $\E[\wt L_{r+1}^\lambda] < \E[\wt L_{1}^\lambda] +\E[\wt L_{\infty}^\lambda]< \infty$ holds for $\lambda < \frac4{\gamma^2} \wedge \frac2\gamma \min (Q-\beta_i)$.  Therefore $\E[(\wt L_{r+1} \wt A_{r+1}^{-\frac s\gamma - \frac12})^p]$ is uniformly bounded as needed. 	
\end{proof}

\begin{proof}[Proof of Proposition~\ref{prop:H-def}]
The analyticity in $\mu$ follows from Morera's theorem. The meromorphicity  in $\beta$ follows from  	Proposition~\ref{prop:analytic-appendix}, where we proved the desired properties for each term in $\hat H$.
\end{proof}

\section{Shift equations for $H$ and $R$}\label{sec:reff}
Our goal is to  derive a set of shift equations (Theorems~\ref{full_shift_H} and~\ref{full_shift_R})  that will completely specify $H$ and $R$. 
We will always assume $\gamma \neq \sqrt2$ --- this is a technical limitation inherited from Theorem~\ref{th_bpz_boundary} from \cite{ang-lcft-zip}. In the next section we will use these shift equations to first derive the exact formulas for $H$ and $R$ when $\gamma^2$ is irrational, and then use continuity in $\gamma$ to remove this restriction. In particular the assumption $\gamma \neq \sqrt2$ can be removed from Theorems~\ref{full_shift_H} and~\ref{full_shift_R} once these exact formulas are proved. For concision we will not explicitly state the $\gamma \neq \sqrt2$ assumption in the lemmas and propositions of this section.

To state the shift equations we introduce the following two  notations:
\begin{equation}\label{eq:g-chi}
q = \frac{2Q-\beta_1-\beta_2-\beta_3 + \chi}{\gamma} \quad  \textrm{and}\quad g_\chi(\sigma) =  \left(\sin(\frac{\pi \gamma^2}{4})\right)^{-\chi/\gamma} \cos \left(2\pi\chi(\sigma-\frac{Q}{2}) \right).
\end{equation}
Recalling \eqref{eq:mu_to_sigma}, note that for $ \chi = \frac{\gamma}{2} $ one has $g_{\frac{\gamma}{2}}(\sigma) = \mu_B(\sigma) $. 
We will again refer to a function of several complex variables as being meromorphic on a domain if it can be expressed as the ratio of two holomorphic functions of several variables on the domain.

\begin{theorem}[Shift equations for $H$]\label{full_shift_H} Assume $\gamma \neq \sqrt2$,	let $\chi = \frac{\gamma}{2}$ or $\frac{2}{\gamma}$ and fix $\sigma_3 \in [-\frac{1}{2 \gamma} + \frac{Q}{2}, \frac{1}{2 \gamma} + \frac{Q}{2} - \frac{\gamma}{4}] \times \mathbb{R} $.
The function $H $ can be jointly meromorphically extended to a complex neighborhood of $\mathbb{R}^3$ in $(\beta_1, \beta_2, \beta_3)$ and to $\mathbb{C}^2$ in $(\sigma_1, \sigma_2)$. Under this extension, it obeys   
 
	\begin{align}\label{shift three point 1}
	&H
\begin{pmatrix}
\beta_1 , \beta_2 - \chi, \beta_3 \\
\sigma_1,  \sigma_2,   \sigma_3 
\end{pmatrix} = \frac{\Gamma(\chi(\beta_1-\chi)) \Gamma(1 - \chi\beta_2 + \chi^2)}{\Gamma(\chi(\beta_1-\chi+q\frac{\gamma}{2})) \Gamma(1 - \chi\beta_2 + \chi^2 -q\frac{\gamma\chi}{2}) } H
\begin{pmatrix}
\beta_1 -\chi, \beta_2, \beta_3 \\
\sigma_1,  \sigma_2 + \frac{\chi}{2},   \sigma_3 
\end{pmatrix} \\ \nonumber 
	& +   \frac{\chi^2 \pi^{\frac{2\chi}{\gamma}}  }{\Gamma(1-\frac{\gamma^2}{4})^{\frac{2\chi}{\gamma}} } \frac{\Gamma(1-\chi\beta_1) \Gamma(1 - \chi\beta_2 + \chi^2 ) \left(  g_{\chi}(\sigma_1)- g_{\chi}(\sigma_2+\frac{\beta_1}{2})   \right)}{\sin(\pi\chi(\chi-\beta_1)) \Gamma(1 + \frac{q \gamma\chi}{2}) \Gamma(2 - \chi(\beta_1 + \beta_2-2\chi+q\frac{\gamma}{2}))   } H
\begin{pmatrix}
\beta_1 +\chi, \beta_2, \beta_3 \\
\sigma_1,  \sigma_2 + \frac{\chi}{2},   \sigma_3 
\end{pmatrix},
	\end{align}
	and
	\begin{align}\label{shift three point 2}
	&  \frac{\chi^2 \pi^{\frac{2\chi}{\gamma} -1 }  }{\Gamma(1-\frac{\gamma^2}{4})^{\frac{2\chi}{\gamma}}}  \Gamma(1-\chi\beta_2)\left( g_{\chi}(\sigma_3) - g_{\chi}(\sigma_2+\frac{\beta_2}{2}) \right) 
 H
	\begin{pmatrix}
	\beta_1 , \beta_2 + \chi, \beta_3 \\
	\sigma_1,  \sigma_2 + \frac{\chi}{2},   \sigma_3 
	\end{pmatrix}\\
	&=\frac{\Gamma(\chi(\beta_1-\chi))}{\Gamma( - q \frac{\gamma\chi}{2} ) \Gamma(-1 + \chi(\beta_1 + \beta_2-2\chi+q\frac{\gamma}{2}) ) } H
	\begin{pmatrix}
	\beta_1 - \chi, \beta_2, \beta_3 \\
	\sigma_1,  \sigma_2,   \sigma_3 
	\end{pmatrix} \nonumber\\
	&+   \frac{\chi^2 \pi^{\frac{2\chi}{\gamma}} }{\Gamma(1-\frac{\gamma^2}{4})^{\frac{2\chi}{\gamma}}}  \frac{   \left(  g_{\chi}(\sigma_1) - g_{\chi}(\sigma_2-\frac{\beta_1}{2}+\frac{\chi}{2} ) \right) \Gamma(1-\chi\beta_1) }{\sin(\pi\chi(\chi-\beta_1))\Gamma(1- \chi(\beta_1-\chi+q\frac{\gamma}{2} )) \Gamma( \chi\beta_2 - \chi^2 +q\frac{\gamma\chi}{2}) }  H
	\begin{pmatrix}
	\beta_1 + \chi, \beta_2, \beta_3 \\
	\sigma_1,  \sigma_2,   \sigma_3 
	\end{pmatrix}. \nonumber
	\end{align}
\end{theorem}

\begin{theorem}[Shift equations for $R$]\label{full_shift_R} 
Assume $\gamma \neq\sqrt2$ and	let $\chi = \frac{\gamma}{2}$ or $\frac{2}{\gamma}$. The function $R(\beta, \sigma_1,\sigma_2)$ can be jointly meromorphically extended to a complex neighborhood of $\mathbb{R}$ in $\beta$ and to $\mathbb{C}^2$ in $(\sigma_1, \sigma_2)$. Under this extension, it obeys 
\begin{align}
\frac{R(\beta, \sigma_1,\sigma_2)}{R(\beta + \chi, \sigma_1- \frac{\chi}{2}, \sigma_2)} 
&= c_{\chi}(\gamma)
\Gamma(-1+ \chi \beta- \chi^2) 
	\Gamma(1-  \chi \beta)  
\left(   g_{\chi}( \sigma_2  ) - g_{\chi} (\sigma_1-\frac{\beta}2)       \right), \label{eq:R-shift1} \\
\frac{R(\beta,\sigma_1,\sigma_2)}{R(\beta+ \chi, \sigma_1 + \frac{\chi}{2},\sigma_2)} 
&= c_{\chi}(\gamma)   
 \Gamma(-1+ \chi \beta - \chi^2) \Gamma(1- \chi \beta ) 
\left(   g_{\chi}( \sigma_2  ) - g_{\chi} (\sigma_1 + \frac{\beta}2)       \right), \label{eq:R-shift2}
\end{align}
where $c_{\frac{\gamma}{2}}(\gamma) = - \frac{1}{\Gamma(- \frac{\gamma^2}{4})}$ and $c_{\frac{2}{\gamma}}(\gamma) = \frac{4}{ \gamma^2} \pi^{\frac{4}{\gamma^2} -1}   \Gamma(1 - \frac{\gamma^2}{4})^{- \frac{4}{\gamma^2}}$.
\end{theorem}

The proofs of Theorems \ref{full_shift_H} and \ref{full_shift_R} are based on the Belavin-Polyakov-Zamolodchikov (BPZ) equations and the operator product expansions (OPE), a strategy first used in~\cite{DOZZ_proof} and carried out for boundary LCFT in the absence of bulk Liouville potential in \cite{rz-boundary}.  In our case, the high level idea is to consider two special cases of four-point correlation functions as deformations of the $H$ function which satisfy the BPZ equations proved in~\cite{ang-lcft-zip}. They reduce to hypergeometric equations after a change of variable. The OPE is the asymptotic expansion by which we relate the coefficients of the solution space of hypergeometric equations to $H$. The shift equations on $H$ are then simply the connection formulas on the solution  space.  The shift equations on $R$ are obtained by taking a suitable limit of the equations on $H$.

Although the same shift equations as in Theorems~\ref{full_shift_H} and~\ref{full_shift_R} were derived in \cite{rz-boundary} except that the function $g_{\chi}$ is replaced by $e^{2\chi \pi \ii (\sigma-\frac{Q}{2})}$, there are two additional difficulties that we must overcome. 
The first is that in \cite{rz-boundary} $ H$ and $R$ (denoted by $\overline{H}$ and $\overline{R}$ in \cite{rz-boundary}) can reduce  to moments of GMC up to a prefactor. This is not possible in our case. We are forced to work with the more involved Definitions \ref{def:H-general} and \ref{def-R-general} which introduces extra constraints on the parameters. The second is that \cite{rz-boundary} used the value of $R(\beta,\sigma_1,\sigma_2)$ at $\mu_1=0$ as an input, which was known from \cite{RZ_interval}. For us this is not available. Instead we use the mating-of-trees theory to get the needed explicit formula for $R(\gamma,\sigma_1,\sigma_2)$; see Lemma~\ref{lem-mot-match}. 
Finally, we have to first prove Theorem~\ref{full_shift_R} based on some partial results towards  Theorem~\ref{full_shift_H}, and then completely prove  Theorem~\ref{full_shift_H}.

The rest of this section is organized as follows. In Section \ref{subsec:deform} we recall the BPZ equation from~\cite{ang-lcft-zip}. In Sections \ref{subsec:shift-H} and \ref{subsec:shift-R} we establish respectively Theorems \ref{full_shift_H} and \ref{full_shift_R} in the case $\chi = \frac{\gamma}{2}$ and in a limited range of parameters. In Section \ref{subsec:ext_H_R} we prove the reflection principle for $H$ and $R$ and show they can be meromorphically extended to the full range of parameters claimed in Theorems \ref{full_shift_H} and \ref{full_shift_R}. 
In Section~\ref{subsec:shift2_R} we prove  Theorem \ref{full_shift_R} in the case $\chi = \frac{2}{\gamma}$, which is the most technically involved and novel part of this section.
Finally in Section~\ref{subsec:shift2_H} we prove Theorem~\ref{full_shift_H} for $\chi = \frac{2}{\gamma}$. We provide in Appendix \ref{sec:special} the necessary background on hypergeometric equations and several facts on the Ponsot-Teschner formula $\HPT$. We defer the proofs of two facts on the function $R$ to Section~\ref{sec:technical} and the proofs of the OPE lemmas to Appendix~\ref{sec:ope}. 

\subsection{The BPZ equation for a deformation of $H$}\label{subsec:deform}
Before stating the BPZ differential equation, let us give a generic definition of the boundary four-point function, presented after having applied the Girsanov theorem to all the insertions. This definition was first proposed in \cite{hrv-disk}, modulo the fact that \cite{hrv-disk} does not consider the varying boundary cosmological constant.

\begin{definition} (Boundary four-point function) \label{def:corr-exp}
	Consider four points $s_1 < s_2 < s_3 < s_4 $ in $\mathbb{R}$ with associated weights $\beta_1, \beta_2, \beta_3, \beta_4 \in \mathbb{R}$ satisfying the conditions $\beta_i < Q$ and  $\sum_i \beta_i > 2Q$. Consider also $\mu_1, \mu_2, \mu_3, \mu_4$ satisfying the condition $\Re(\mu_i) \geq 0$. Then define
	\begin{align*}
&\lg \prod_{j=1}^4 B_{\beta_j}^{\sigma_{j},\sigma_{j+1}}(s_j) \rg = \prod_{j < j'}|s_j - s_{j'}|^{-\frac{\beta_j\beta_{j'}}{2}} \int_{\mathbb{R}} dc  \,e^{- \frac{\gamma \widetilde{q} c}{2}} \E \Bigg[ \exp\Bigg( - e^{\gamma c}\int_{\mathbb{H}} \frac{\hat g(x)^{\frac{\gamma^2}{4}(\widetilde{q} + 1)}}{ \prod_{j=1}^4 |x-s_j|^{\gamma\beta_j} } e^{\gamma h(x) }   d^2x\\
&-e^{\frac{\gamma c}{2}}\int_{\mathbb{R}} \frac{\hat g(r)^{\frac{\gamma^2 \widetilde{q}}{8}}}{\prod_{j=1}^4 |r-s_j|^{\frac{\gamma\beta_j}{2}}} e^{\frac{\gamma}{2}h(r) }   d\mu_{\partial} (r)  \Bigg) \Bigg],
\end{align*}
where $\widetilde{q} = \frac{2Q-\beta_1-\beta_2-\beta_3 - \beta_4}{\gamma}$, $\hat g(x) = |x|_+^{-4}$ and where the boundary measure is given by: 
	\begin{equation}
	d\mu_{\partial}(r)/dr = \sum_{j=1}^{3} \mu_{j+1} \mathbf{1}_{s_j<r < s_{j+1}} + \mu_1 \mathbf{1}_{r \notin (s_1, s_4)}.
	\end{equation}
\end{definition}

We now state parameter ranges for which the four-point function with the degenerate insertion will be well-defined. Recall \( \cB= (-\frac{1}{2 \gamma} + \frac{Q}{2}, \frac{1}{2 \gamma} + \frac{Q}{2} ) \times \mathbb{R}\) and  \( \overline{\cB}= [-\frac{1}{2 \gamma} + \frac{Q}{2}, \frac{1}{2 \gamma} + \frac{Q}{2} ]\times \mathbb{R}\). For $\chi = \frac{\gamma}{2}$ or $\frac{2}{\gamma}$, consider the following ranges of parameters
\begin{equation}\label{eq:para_H_chi}
\beta_i \in (-\infty, Q), \: \: \beta_1 + \beta_2 + \beta_3 > 2Q + \chi, \quad \sigma_1 - \frac{\chi}{2} \in \overline{\cB}, \: \: \sigma_2 - \frac{\chi}{2} \in \overline{\cB},  \: \:  \sigma_1 \in \overline{\cB}, \: \:  \sigma_2 \in \overline{\cB}, \: \: \sigma_3  \in \overline{\cB}, 
\end{equation}
and:
\begin{equation}\label{eq:para_H_chi2}
\beta_i \in ( -\infty, Q), \: \: \beta_1 + \beta_2 + \beta_3 > 2Q + \chi, \quad \sigma_1 + \frac{\chi}{2} \in \overline{\cB}, \: \: \sigma_2 + \frac{\chi}{2} \in \overline{\cB}, \: \:  \sigma_1 \in \overline{\cB}, \: \:  \sigma_2 \in \overline{\cB}, \: \: \sigma_3  \in \overline{\cB}. 
\end{equation}

Note that each set of constraints specifies a nonempty parameter range. Indeed, for \eqref{eq:para_H_chi}, if all three $\beta_i$ are close to $Q$ then $\beta_1 + \beta_2 + \beta_3 > 2Q + \chi$ is true. For the $\sigma_i$, the band with its boundary included $\overline{\cB}$ has width $\frac{1}{\gamma}$. Therefore it is possible to have both $\sigma_1 - \frac{\chi}{2}$ and $\sigma_1$ in $\overline{\cB}$, but in the case $\chi = \frac{2}{\gamma}$ this means that $\sigma_1 - \frac{\chi}{2}$ and $\sigma_1$ are exactly on the boundary of $\overline{\cB}$. The same is true for $\sigma_2$. Then the parameter range \eqref{eq:para_H_chi2} works the same. Under these constraints we have the following BPZ equation obtained in \cite{ang-lcft-zip}.

\begin{theorem}[\cite{ang-lcft-zip}]\label{th_bpz_boundary} Suppose $\gamma \neq \sqrt2$. Consider parameters $(\beta_i, \sigma_i)_{i=1,2,3}$ obeying \eqref{eq:para_H_chi} or \eqref{eq:para_H_chi2}. Consider disjoint points $s, s_1, s_2, s_3 \in \mathbb{R}$ with associated weights $-\chi, \beta_1, \beta_2, \beta_3$ where $\chi = \frac{\gamma}{2}$ or $\frac{2}{\gamma}$. Define the following differential operator:
\begin{equation}
\mathcal{D} := \frac{1}{\chi^2} \partial_{ss}  +   \sum_{j=1}^3 \frac{1}{s - s_j} \partial_{s_j} + \sum_{j=1}^3 \frac{\Delta_{\beta_j}}{(s - s_j)^2}.
\end{equation}
Then for $s_1 < s < s_2 < s_3$ we have
	\begin{align}
	\mathcal{D} 
	\left \lg B_{\beta_1}^{\sigma_{1} \pm \frac{\chi}{2}, \sigma_{2} \pm \frac{\chi}{2} }(s_1) B_{-\chi}^{\sigma_{2} \pm \frac{\chi}{2}, \sigma_{2} }(s) B_{\beta_2}^{\sigma_{2} , \sigma_{3} }(s_2) B_{\beta_3}^{\sigma_{3} , \sigma_{1} \pm \frac{\chi}{2} }(s_3)  \right \rg = 0,\label{eq:bpz-form}
	\end{align}
 and for $s < s_1 < s_2 < s_3$:
 \begin{align}
	\mathcal{D} 
	\left \lg B_{-\chi}^{\sigma_{1} \pm \frac{\chi}{2}, \sigma_{1} }(s) B_{\beta_1}^{\sigma_{1} , \sigma_{2} }(s_1) B_{\beta_2}^{\sigma_{2} , \sigma_{3} }(s_2) B_{\beta_3}^{\sigma_{3} , \sigma_{1} \pm \frac{\chi}{2} }(s_3)  \right \rg = 0.\label{eq:bpz-form}
	\end{align}
 Above when $\pm \frac{\chi}{2}$ equals $-\frac{\chi}{2}$ the parameters must obey \eqref{eq:para_H_chi} and  when it equals $  \frac{\chi}{2}$ must obey \eqref{eq:para_H_chi2}.
\end{theorem} 
We point out that very recently Cercl\'e~\cite{Cercle-BPZ} gave a new proof of Theorem~\ref{th_bpz_boundary} without using mating of trees. Moreover, Baverez and Wu~\cite{bw-bpz} suggested a proof strategy; see \cite[Theorem 1.3]{bw-bpz} and the discussion afterwards.

We now use conformal invariance to perform a change of variable that will place the boundary spectator insertions $s_1, s_2, s_3$ at $0, 1, \infty$ and the degenerate insertion at $s$ will be mapped to the cross ratio $t$ of $s_1, s_2, s_3, s$. The concrete conformal map we use for this change of variable is thus:
\begin{align}
\psi(u) = \frac{(u-s_1)(s_2 - s_3)}{(u-s_3)(s_2 - s_1)}.
\end{align}
By setting $t = \psi(s)$
we have  $t \in (0,1)$ when $s_1 < s < s_2$ and $t \in (-\infty, 0) $ when $s < s_1$. Now under the constraint of \eqref{eq:para_H_chi}, consider the functional
\begin{align}\label{eq:def_H_chi}
H_{\chi}(t) 
=  \int_{\mathbb{R}} dc  \,e^{- \frac{\gamma q c}{2} }\E \Bigg[ \exp\Bigg( - e^{\gamma c} &\int_{\mathbb{H}} \frac{|x-t|^{\gamma\chi} \hat g(x)^{\frac{\gamma^2}{4}(q+1)}}{ |x|^{\gamma\beta_1}|x-1|^{\gamma\beta_2}} e^{\gamma h(x)} d^2x\\
&-e^{\frac{\gamma c}{2}}\int_{\mathbb{R}}  \frac{|r-t|^{\frac{\gamma\chi}{2}} \hat g(r)^{\frac{\gamma^2 q}{8}}}{|r|^{\frac{\gamma\beta_1}{2}}|r-1|^{\frac{\gamma\beta_2}{2}}}  e^{\frac{\gamma}{2} h(r)}  d\mu_{\partial}^t (r)  \Bigg) \Bigg], \nonumber
\end{align}
where the boundary measure $d\mu_{\partial}^t(r)$ is defined by: 
\begin{align*}
d\mu_{\partial}^t(r)/dr =
\left\{
\begin{array}{ll}
g_{\frac{\gamma}{2}}(\sigma_1 - \frac{\chi}{2}) \mathbf{1}_{r < 0}  + g_{\frac{\gamma}{2}}(\sigma_2  - \frac{\chi}{2}) \mathbf{1}_{0<r<t} +  g_{\frac{\gamma}{2}}(\sigma_2) \mathbf{1}_{t<r<1} + g_{\frac{\gamma}{2}}(\sigma_3) \mathbf{1}_{r>1},  & \mbox{for } t \in (0,1), \\
g_{\frac{\gamma}{2}}(\sigma_1 - \frac{\chi}{2}) \mathbf{1}_{r < t}  + g_{\frac{\gamma}{2}}(\sigma_1) \mathbf{1}_{t<r<0} +  g_{\frac{\gamma}{2}}(\sigma_2) \mathbf{1}_{0<r<1} + g_{\frac{\gamma}{2}}(\sigma_3) \mathbf{1}_{r>1}, & \mbox{for } t < 0.
\end{array}
\right.
\end{align*}

Since the change of $\sigma$ on the left and right or $t$ can be either $+ \frac{\chi}{2}$ or $- \frac{\chi}{2}$, we similarly consider the function $\tilde{H}_{\chi}(t)$ with the other choice under the constraints of \eqref{eq:para_H_chi2}, which is given by the same expression as $H_{\chi}(t)$ except one needs to replace $d\mu_{\partial}^t(r)$ by the following measure
\begin{align*}
d \tilde{\mu}_{\partial}^t(r)/dr =
\left\{
\begin{array}{ll}
g_{\frac{\gamma}{2}}(\sigma_1 + \frac{\chi}{2}) \mathbf{1}_{r < 0}  + g_{\frac{\gamma}{2}}(\sigma_2  + \frac{\chi}{2}) \mathbf{1}_{0<r<t} +  g_{\frac{\gamma}{2}}(\sigma_2) \mathbf{1}_{t<r<1} + g_{\frac{\gamma}{2}}(\sigma_3) \mathbf{1}_{r>1},  & \mbox{for } t \in (0,1), \\
g_{\frac{\gamma}{2}}(\sigma_1 + \frac{\chi}{2}) \mathbf{1}_{r < t}  + g_{\frac{\gamma}{2}}(\sigma_1) \mathbf{1}_{t<r<0} +  g_{\frac{\gamma}{2}}(\sigma_2) \mathbf{1}_{0<r<1} + g_{\frac{\gamma}{2}}(\sigma_3) \mathbf{1}_{r>1}, & \mbox{for } t < 0.
\end{array}
\right.
\end{align*}

We could also of course define in a similar manner $H_{\chi}(t)$, $\tilde{H}_{\chi}(t)$  for $t\in (1,\infty)$, but this will not be needed.
By a direct change of variable applied to the differential operator of Theorem \ref{th_bpz_boundary}, $H_{\chi}(t)$ and $\tilde{H}_{\chi}(t)$ obey the hypergeometric equation. We state this in the next proposition.

\begin{proposition}\label{prop:hyper} Under the parameter constraint of \eqref{eq:para_H_chi}, the function $H_{\chi}$ obeys
	\begin{equation}
	t(1-t)\partial_t^2 H_{\chi} + (C-(A+B+1)t)\partial_t H_{\chi}-ABH_{\chi} = 0 \quad \textrm{for } t\in (-\infty,0)\cup (0,1),
	\end{equation}
	with $A,B,C$ given by:
	\begin{equation}
	A =-q\frac{\gamma\chi}{2}, \quad  B = -1+ \chi(\beta_1+\beta_2-2\chi + q\frac{\gamma}{2}), \quad C = \chi(\beta_1 - \chi).
	\end{equation}
	The exact same result holds for $\tilde{H}_{\chi}$ under the parameter constraints \eqref{eq:para_H_chi2}.
\end{proposition}

The following analyticity result for $H_{\frac{\gamma}{2}}$ and $\tilde{H}_{\frac{\gamma}{2}}$ will be used in the proof of Lemma \ref{reflection_H} below.

\begin{lemma}\label{lem:ext-H-chi}
Set $\chi = \frac{\gamma}{2}$. Fix the $\sigma_i$ in the parameter range \eqref{eq:para_H_chi} but with the boundary of the band being excluded, namely $\sigma_1 - \frac{\gamma}{4} \in \cB$,  $\sigma_2 - \frac{\gamma}{4} \in \cB$, $ \sigma_1, \sigma_2, \sigma_3 \in \cB$. Then the function $H_{\chi}$ is meromorphic in all three $\beta_1, \beta_2, \beta_3$ in a complex neighborhood of the subdomain of $\mathbb{R}^3$ given by the constraints $\beta_i \in (-\infty, Q), \: \: \beta_1 + \beta_2 + \beta_3 > 2Q + \chi$. The same claim holds for $\tilde{H}_{\chi}$, except this time the $\sigma_i$ need to obey $\sigma_1 + \frac{\gamma}{4} \in \cB$,  $\sigma_2 + \frac{\gamma}{4} \in \cB$, $ \sigma_1, \sigma_2, \sigma_3 \in \cB$.
\end{lemma}

\begin{proof}
This result follows from a direct adaptation of the proof of the claim of analyticity in the $\beta_i$ for $H$ of Proposition \ref{prop:H-def}. The fact that we have an extra insertion on the boundary poses no additional difficulty. Here we are also in the range of parameters where the Seiberg bounds are satisfied and thus no truncation procedure is required.
\end{proof}

Before further analysis,  we recall here the range of parameters where the functions $H$ and $R$ have been defined in Section \ref{sec:precise}. For $H$ the range on the $\beta_i$ and $\sigma_i$ is given in Definition~\ref{def:H-general}:
\begin{equation}\label{sec4_proba_H}
\left\{ (\beta_1, \beta_2, \beta_3, \sigma_1, \sigma_2, \sigma_3) \: : \:  \beta_i < Q, \: \: Q - \frac{1}{2} \sum \beta_i < \gamma \wedge \frac{2}{\gamma} \wedge \min_i (Q -\beta_i), \: \: \sigma_i \in \overline{\cB} \right \}.
\end{equation}
For $R$ the range of $\beta$ and $\sigma_i$ is given in  Definition~\ref{def-R-general}:
\begin{equation}\label{sec4_proba_R}
\left\{ (\beta, \sigma_1, \sigma_2) \: : \: \frac{\gamma}{2} \vee (\frac{2}{\gamma} - \frac{\gamma}{2}) < \beta < Q, \: \: \sigma_i \in \overline{\cB} \right \}.
\end{equation}

We have kept the $\beta_i$ to be real parameters above but by the results of Section \ref{sec:precise} we have shown analytic continuation of $H$ and $R$ in the $\beta_i$ to a complex neighborhood of the real interval where they are defined.

\subsection{The $\frac{\gamma}{2}$-shift equations for $H$}\label{subsec:shift-H}

Using the functions $H_{\frac{\gamma}{2}}(t)$ and $\tilde{H}_{\frac{\gamma}{2}}(t)$, we prove the shift equations on $H$ in Theorem~\ref{full_shift_H}  in the case when $\chi = \frac{\gamma}{2}$ and the parameter range is chosen such that the $\beta_i$ and $\sigma_i$ parameters of each $H$ in the shift equations belong to the domain \eqref{sec4_proba_H}.

\begin{lemma}\label{Shift1_H_proba} 
Set $\chi = \frac{\gamma}{2}$. The shift equation~\eqref{shift three point 1} holds in the parameter range
	\begin{equation}\label{eq:lem_shift_H_1}
	\beta_1 <  \frac{2}{\gamma}, \quad \beta_2, \beta_3 < Q,  \quad Q - \frac{1}{2} (\beta_1 + \beta_2  + \beta_3 -\frac{\gamma}{2}) < \gamma \wedge \frac{2}{\gamma} \wedge \min_i (Q - \beta_i), \quad \sigma_1, \sigma_2, \sigma_3 \in \cB, \quad \sigma_2 + \frac{\gamma}{4} \in \cB,
	\end{equation}	
and the shift equation ~\eqref{shift three point 2} holds in parameter range:
\begin{equation}\label{eq:lem_shift_H_2}
\beta_1, \beta_2 <  \frac{2}{\gamma}, \quad \beta_3 <Q, \quad Q - \frac{1}{2} (\beta_1 + \beta_2 + \beta_3  -\frac{\gamma}{2}) < \gamma \wedge \frac{2}{\gamma} \wedge \min_i (Q - \beta_i),  \quad \sigma_1, \sigma_2, \sigma_3 \in \cB, \quad \sigma_2 + \frac{\gamma}{4} \in \cB.
\end{equation}
\end{lemma}

\begin{proof} 
Here we are always assuming $\chi = \frac{\gamma}{2}$.
We first take a look at the parameter ranges given by \eqref{eq:lem_shift_H_1} and \eqref{eq:lem_shift_H_2}. For $\chi =\frac{\gamma}{2}$, the shift equation \eqref{shift three point 1} contains the following three $H$ functions:
\begin{align*}
H
	\begin{pmatrix}
	\beta_1 , \beta_2 - \frac{\gamma}{2}, \beta_3 \\
	\sigma_1,  \sigma_2,   \sigma_3 
	\end{pmatrix}, \quad  H
	\begin{pmatrix}
	\beta_1 - \frac{\gamma}{2}, \beta_2, \beta_3 \\
	\sigma_1,  \sigma_2 + \frac{\gamma}{4},   \sigma_3 
	\end{pmatrix}, \quad H
	\begin{pmatrix}
	\beta_1 + \frac{\gamma}{2}, \beta_2, \beta_3 \\
	\sigma_1,  \sigma_2 + \frac{\gamma}{4},   \sigma_3 
	\end{pmatrix}.
\end{align*}
The $\beta_i$ and $\sigma_i$ parameters of each $H$ must be in the range \eqref{sec4_proba_H}. This gives \eqref{eq:lem_shift_H_1}. Similarly, the second shift equations \eqref{shift three point 2} contains the $H$ functions
\begin{align*}
 H
	\begin{pmatrix}
	\beta_1 , \beta_2 + \frac{\gamma}{2}, \beta_3 \\
	\sigma_1,  \sigma_2 + \frac{\gamma}{4},   \sigma_3 
	\end{pmatrix}, \quad H
	\begin{pmatrix}
	\beta_1 - \frac{\gamma}{2}, \beta_2, \beta_3 \\
	\sigma_1,  \sigma_2,   \sigma_3 
	\end{pmatrix}, \quad   H
	\begin{pmatrix}
	\beta_1 + \frac{\gamma}{2}, \beta_2, \beta_3 \\
	\sigma_1,  \sigma_2,   \sigma_3 
	\end{pmatrix},
\end{align*}
and the $\beta_i$ and $\sigma_i$ then need to obey the constraint \eqref{eq:lem_shift_H_2}. It is easy to check that these parameter constraints are non-empty. The condition of \eqref{eq:lem_shift_H_1} is weaker than \eqref{eq:lem_shift_H_2}, and \eqref{eq:lem_shift_H_2} can be satisfied if all the $\beta_i$ parameters are chosen in the interval $(\frac{2}{\gamma} -\epsilon, \frac{2}{\gamma})$ for a small $\epsilon >0$.

Let us now derive the first shift equation \eqref{shift three point 1}.  We assume $\beta_i, \sigma_i$ obey \eqref{eq:para_H_chi} in order for $H_{\frac{\gamma}{2}}$ to be well-defined, plus the following extra constraint on $\beta_1$:
\begin{equation}\label{para_range_301}
\frac{\gamma}{2} <  \beta_1 < \frac{2}{\gamma}.
\end{equation}
By Proposition \ref{prop:hyper}, the function $t \mapsto H_{\frac{\gamma}{2}}(t)$ obeys the hypergeometric equation for $ t \in (0,1)$. Using the basis of solutions of the hypergeometric equation recalled in Section \ref{sec:hyp_eq}, we can write the following solutions around $t=0$ and $t= 1$, under the assumption that neither $C$, $C-A-B$, or $A-B$ are integers.\footnote{The values excluded here are recovered by a simple continuity argument in $\gamma$.}
 For $t \in (0,1)$:
\begin{align}
H_{\frac{\gamma}{2}}(t) &= C_1 F(A,B,C,t) + C_2^+ t^{1 - C} F(1 + A-C, 1 +B - C, 2 -C, t) \\
 &= B_1 F(A,B,1+A+B- C, 1 -t) + B_2^- (1-t)^{C -A -B} F(C- A, C- B, 1 + C - A - B , 1 -t). \nonumber 
\end{align}
Here, $C_1, C_2^+, B_1, B_2^-$ are the real constants that parametrize the solution space around $t=0$ and $t=1$. 
We now express $C_1, C_2^+, B_1$  using $H\begin{pmatrix}
\cdot, \cdot, \cdot \\
\cdot,  \cdot,  \cdot 
\end{pmatrix}$ and then relate them.

First,  by sending $t\to 0$ and inspecting the definition~\eqref{eq:def_H_chi} of $H_{\chi}(t)$, we have:
\begin{equation}\label{eq:C1}
C_1 =\lim_{t\to 0} H_{\chi}(t) = H
\begin{pmatrix}
\beta_1 - \chi, \beta_2, \beta_3 \\
\sigma_1 - \frac{\chi}{2},  \sigma_2,   \sigma_3 
\end{pmatrix}.
\end{equation} 
Next to find $C_2^+$ we go at higher order in the $t \rightarrow 0_+$ limit. For this we use the asymptotic expansion of $H_{\frac{\gamma}{2}}(t) - H_{\frac{\gamma}{2}}(0)$ given by the operator product expansion (OPE) from Lemma \ref{lem:ope1}. Under the parameter constraints  \eqref{eq:para_H_chi} and \eqref{para_range_301}, Lemma \ref{lem:ope1} gives 
\begin{align}
H_{\frac{\gamma}{2}}(t) - H_{\frac{\gamma}{2}}(0)
= C_2^+ t^{1-C} + o(t^{1-C}), 
\end{align}
where  
\begin{align}\label{eq_ope_gluing2}
C_2^+  =  & -\frac{\Gamma(-1+\frac{\gamma \beta_1}{2} - \frac{\gamma^2}{4} ) \Gamma(1 - \frac{\gamma \beta_1}{2})}{\Gamma(- \frac{\gamma^2}{4})}  \left( g_{\frac{\gamma}{2}}(\sigma_1-\frac{\gamma}{4}) - g_{\frac{\gamma}{2}}(\sigma_2+\frac{\beta_1}{2}-\frac{\gamma}{4}) \right)  H
\begin{pmatrix}
\beta_1 + \frac{\gamma}{2}, \beta_2, \beta_3 \\
\sigma_1 - \frac{\gamma}{4},  \sigma_2,   \sigma_3 
\end{pmatrix}.
\end{align}
To align with \eqref{shift three point 1}, we can write~\eqref{eq_ope_gluing2} as
\begin{equation}\label{eq:C2+}
C_2^+ = \chi^2 \pi^{\frac{2\chi}{\gamma}-1}   \frac{\Gamma(-1+\chi\beta_1-\chi^2) \Gamma(1-\chi\beta_1)}{\Gamma(1-\frac{\gamma^2}{4})^{\frac{2\chi}{\gamma}}} \left( g_{\chi}(\sigma_1-\frac{\chi}{2})-  g_{\chi}(\sigma_2+\frac{\beta_1}{2}-\frac{\chi}{2})   \right) H
\begin{pmatrix}
\beta_1 + \chi, \beta_2, \beta_3 \\
\sigma_1 - \frac{\chi}{2},  \sigma_2,   \sigma_3 
\end{pmatrix}.    
\end{equation}

Similarly as for~\eqref{eq:C1}, by sending $t\to 1$ we get:
\begin{align}
B_1  =\lim_{t\to 1} H_{\chi}(t)= H
\begin{pmatrix}
\beta_1 , \beta_2 - \chi, \beta_3 \\
\sigma_1 - \frac{\chi}{2},  \sigma_2 - \frac{\chi}{2},   \sigma_3 
\end{pmatrix}.
\end{align}
By the connection formula \eqref{connection1}, with $\chi = \frac{\gamma}{2}$ we have
	\begin{align}\label{eq:final_shift_H}
	B_1 &= \frac{\Gamma(\chi(\beta_1-\chi)) \Gamma(1 - \chi\beta_2 + \chi^2)}{\Gamma(\chi(\beta_1-\chi+q\frac{\gamma}{2}) \Gamma(1 - \chi\beta_2 + \chi^2 -q\frac{\gamma\chi}{2}) } C_1 + \frac{\Gamma(2  - \chi\beta_1+\chi^2) \Gamma(1 - \chi\beta_2 + \chi^2)}{\Gamma(1 + \frac{q \gamma\chi}{2}) \Gamma(2 - \chi(\beta_1 + \beta_2-2\chi+q\frac{\gamma}{2}))   } C_2^+,
	\end{align}
This implies the shift equation 
\eqref{shift three point 1} for $\chi=\frac{\gamma}2$
 in the range of parameters constraint by \eqref{eq:para_H_chi} and \eqref{para_range_301}, after performing furthermore the replacement $\sigma_1 \rightarrow \sigma_1 + \frac{\gamma}{4} $ and $\sigma_2 \rightarrow \sigma_2 + \frac{\gamma}{4} $ (which also shifts the domain where $\sigma_1, \sigma_2$ belong).  To lift these constraints we then invoke the analyticity of $H$ as a function of its parameters given by Proposition \ref{prop:H-def-sigma}. We have thus shown that 
 \eqref{shift three point 1} holds for $\chi=\frac{\gamma}2$ in the parameter range given by \eqref{eq:lem_shift_H_1}.

Now we repeat these steps with $\tilde{H}_{\frac{\gamma}{2}}$ to obtain the shift equation with the opposite phase.  For $t \in (0,1)$ we have 
\begin{align}
\tilde{H}_{\frac{\gamma}{2}}(t) &= \tilde{C}_1 F(A,B,C,t) + \tilde{C}_2^+ t^{1 - C} F(1 + A-C, 1 +B - C, 2 -C, t) \\ \nonumber
 &= \tilde{B}_1 F(A,B,1+A+B- C, 1 -t) + \tilde{B}_2^- (1-t)^{C -A -B} F(C- A, C- B, 1 + C - A - B , 1 -t).
\end{align}
Similarly as for  $C_1, C_2^+, B_1$, we get the following expressions for $ \tilde{C}_1, \tilde{C}_2^+,  \tilde{B}_2^-$ in terms of $H$: 
\begin{equation}\label{eq:C1tilde}
 \tilde{C}_1 =\lim_{t\to 0}\tilde{H}_{\chi}(t)= H
\begin{pmatrix}
\beta_1 - \chi, \beta_2, \beta_3 \\
\sigma_1 + \frac{\chi}{2},  \sigma_2,   \sigma_3 
\end{pmatrix}.
\end{equation}
\begin{equation}\label{C2+tilde}
	\tilde{C}_2^+ = \chi^2 \pi^{\frac{2\chi}{\gamma}-1}  \frac{\Gamma(-1+\chi\beta_1-\chi^2) \Gamma(1-\chi\beta_1)}{\Gamma(1-\frac{\gamma^2}{4})^{\frac{2\chi}{\gamma}}} \left( g_{\chi}(\sigma_1+\frac{\chi}{2} ) - g_{\chi}(\sigma_2-\frac{\beta_1}{2}+\frac{\chi}{2}) \right) H \begin{pmatrix}
\beta_1 + \chi, \beta_2, \beta_3 \\
\sigma_1 + \frac{\chi}{2},  \sigma_2,   \sigma_3 
\end{pmatrix}.    
\end{equation}
\begin{equation}\label{eq:B2-tilde}
    \tilde{B}_2^- = \chi^2 \pi^{\frac{2\chi}{\gamma}-1}  \frac{\Gamma(-1+\chi\beta_2-\chi^2) \Gamma(1-\chi\beta_2)}{\Gamma(1-\frac{\gamma^2}{4})^{\frac{2\chi}{\gamma}}} \left( g_{\chi}(\sigma_3) - g_{\chi}(\sigma_2+\frac{\beta_2}{2}) \right) H \begin{pmatrix}
\beta_1 , \beta_2 + \chi, \beta_3 \\
\sigma_1 + \frac{\chi}{2},  \sigma_2 + \frac{\chi}{2} ,    \sigma_3 
\end{pmatrix}.
\end{equation}

Again by the connection formula \eqref{connection1}, with $\chi = \frac{\gamma}{2}$ we have
	\begin{align}\label{eq:final_shift_H2}
	\tilde{B}_2^- = \frac{\Gamma(\chi(\beta_1-\chi)) \Gamma(-1 + \chi\beta_2 - \chi^2)}{\Gamma( - q \frac{\gamma\chi}{2} ) \Gamma(-1 + \chi(\beta_1 + \beta_2-2\chi+q\frac{\gamma}{2}) ) } \tilde{C}_1 + \frac{\Gamma(2  - \chi \beta_1+\chi^2)) \Gamma(-1 + \chi\beta_2 - \chi^2)}{\Gamma(1- \chi(\beta_1-\chi+q\frac{\gamma}{2} ) \Gamma( \chi\beta_2 - \chi^2 +q\frac{\gamma\chi}{2}) } \tilde{C}_2^+.
	\end{align}
Therefore \eqref{shift three point 2} holds for $\chi=\frac{\gamma}2$ in the parameter range given by \eqref{eq:lem_shift_H_2} as before.
\end{proof}

\begin{remark}
The two shift equations we derive using respectively $H_{\chi}$ and $\tilde{H}_{\chi}$ do not have exactly the same shape. The reason is that in the proof above we used the connection formula between the coefficients $C_1, C_2^+, B_1$ to derive \eqref{shift three point 1}, and used the connection formula between $ \tilde{C}_1, \tilde{C}_2^+, \tilde{B}_2^-$ to derive \eqref{shift three point 2}. One may wonder why we did not use the connection formula between $ \tilde{C}_1, \tilde{C}_2^+, \tilde{B}_1$ in the case of $\tilde{H}_{\chi}$ as well. The main reason is that with our choice it is possible to perform a linear combination of \eqref{shift three point 1} and \eqref{shift three point 2} to obtain a shift equation of $H$ that only shifts the parameter $\beta_1$. This will be crucial to establish a uniqueness statement for the solutions to \eqref{shift three point 1} and \eqref{shift three point 2}, see the proof of Theorem \ref{thm:H} in Section~\ref{sec:proof_H}.
\end{remark}

\subsection{The $\frac{\gamma}{2}$-shift equations for $R$}\label{subsec:shift-R} The next step is to derive the $\frac{\gamma}{2}$-shift equation for the reflection coefficient $R$.
	The key idea is to take a suitable limit of the $\frac{\gamma}{2}$-shift equations for $H$ that we have established in Lemma \ref{Shift1_H_proba} to make $R$ appear from $H$. For this purpose we give the following result expressing $R$ as a limit of $H$.	
	\begin{lemma}\label{lem:HlimR} 
		Suppose the $\beta_i$ are in the range given by \eqref{sec4_proba_H} and the $\sigma_i$ in $\cB$. Suppose $Q> \beta_1 > \beta_2 \vee \frac{\gamma}{2}$, $\beta_2>0$ and $\beta_1 - \beta_2 < \beta_3 < Q$. Then the following limit holds: 
		\begin{align*}
		\lim_{\beta_3 \downarrow \beta_1 - \beta_2 } (\beta_2 + \beta_3 - \beta_1) H
\begin{pmatrix}
\beta_1 , \beta_2, \beta_3 \\
\sigma_1,  \sigma_2,   \sigma_3 
\end{pmatrix}  = 2 R(\beta_1, \sigma_1, \sigma_2).
		\end{align*}
	\end{lemma}
\noindent We defer the proof of Lemma~\ref{lem:HlimR} to Section \ref{app:lim_H_R} and proceed to  prove the $\frac{\gamma}{2}$-shift equation for $R$.

	\begin{lemma}\label{lem_shift_eq_R} Consider $\beta_1 \in (\frac{\gamma}{2} \vee (\frac{2}{\gamma} - \frac{\gamma}{2}), \frac{2}{\gamma})$ and  $\sigma_1, \sigma_2 \in \mathbb{C}$ such that  $\sigma_1, \sigma_2, \sigma_2 + \frac{\gamma}{4}$ all belong to $\cB$. Then $R(\beta_1, \sigma_1, \sigma_2)$ obeys
	\begin{align}\label{equation R1}
	R(\beta_1, \sigma_1, \sigma_2) =-\frac{\Gamma(-1+\frac{\gamma \beta_1}{2} - \frac{\gamma^2}{4} ) \Gamma(1 - \frac{\gamma \beta_1}{2}) }{ \Gamma(- \frac{\gamma^2}{4}) }
	  	\left(	g_{\frac{\gamma}2}(\sigma_1)   -g_{\frac{\gamma}2}(\sigma_2+\frac{\beta_1}{2} ) \right)
	 R(\beta_1 + \frac{\gamma}{2}, \sigma_1,   \sigma_2 + \frac{\gamma}{4}).
	\end{align}
	Similarly for $\beta_1 \in (0 \vee (\frac{2}{\gamma} - \gamma), \frac{2}{\gamma} - \frac{\gamma}{2})$ and the same constraint on $\sigma_1, \sigma_2$ as previously,
	\begin{align}\label{equation R2}
	R(\beta_1+\frac{\gamma}{2}, \sigma_1, \sigma_2 + \frac{\gamma}{4}) =-\frac{\Gamma(-1+\frac{\gamma \beta_1}{2} ) \Gamma(1 - \frac{\gamma \beta_1}{2}-\frac{\gamma^2}{4}) }{ \Gamma(- \frac{\gamma^2}{4}) } 	  	\left(	g_{\frac{\gamma}2}(\sigma_1)   -g_{\frac{\gamma}2}(\sigma_2-\frac{\beta_1}{2} ) \right)
	 R(\beta_1 + \gamma, \sigma_1, \sigma_2). 
	\end{align}
\end{lemma}

\begin{proof}
We derive the first shift equation \eqref{equation R1} by taking a limit of \eqref{shift three point 1} established in Lemma~\ref{Shift1_H_proba}. Fix a $\beta_1 \in (\frac{\gamma}{2} \vee (\frac{2}{\gamma} - \frac{\gamma}{2}), \frac{2}{\gamma}) $. Consider two parameters $\epsilon, \eta >0$ chosen small enough and set $ \beta_2 = \beta_1  - \epsilon $, $\beta_3 = \beta_1 - \beta_2 +\frac{\gamma}{2} + \eta = \frac{\gamma}{2} + \epsilon +\eta$. For this parameter choice the condition \eqref{eq:lem_shift_H_1} required in Lemma \ref{Shift1_H_proba} is satisfied. By  Lemma \ref{lem:HlimR} we get:

	\begin{align*}
	&\lim_{\beta_3 \downarrow \beta_1 - \beta_2 + \frac{\gamma}{2}} (\beta_2 + \beta_3 - \beta_1 - \frac{\gamma}{2}) H
\begin{pmatrix}
\beta_1 , \beta_2 - \frac{\gamma}{2}, \beta_3 \\
\sigma_1,  \sigma_2,   \sigma_3 
\end{pmatrix}  = 2 R(\beta_1, \sigma_1, \sigma_2), \\
	&\lim_{\beta_3 \downarrow \beta_1 - \beta_2 + \frac{\gamma}{2}} (\beta_2 + \beta_3 - \beta_1 - \frac{\gamma}{2})  \frac{\Gamma(\frac{\gamma \beta_1}{2} - \frac{\gamma^2}{4}) \Gamma(1 - \frac{\gamma \beta_2}{2} + \frac{\gamma^2}{4})}{\Gamma(\frac{\gamma \beta_1}{2} +(q-1) \frac{\gamma^2}{4}) \Gamma(1 - \frac{\gamma}{2}\beta_2 - (q-1)\frac{\gamma^2}{4})) } H
\begin{pmatrix}
\beta_1 -\frac{\gamma}{2}, \beta_2, \beta_3 \\
\sigma_1,  \sigma_2 + \frac{\gamma}{4},   \sigma_3 
\end{pmatrix} = 0,  \\
	& \lim_{\beta_3 \downarrow \beta_1 - \beta_2 + \frac{\gamma}{2}} (\beta_2 + \beta_3 - \beta_1 - \frac{\gamma}{2})  \Bigg[ \frac{ \Gamma(2 + \frac{\gamma^2}{4} - \frac{\gamma \beta_1}{2}) \Gamma(1 - \frac{\gamma \beta_2}{2} + \frac{\gamma^2}{4})\Gamma(-1+\frac{\gamma \beta_1}{2} - \frac{\gamma^2}{4} ) \Gamma(1 - \frac{\gamma \beta_1}{2}) }{\Gamma(1 + \frac{\gamma^2}{4}) \Gamma(2 - \frac{\gamma}{2}(\beta_1 + \beta_2)  - (q-2)\frac{\gamma^2}{4} )  ) \Gamma(- \frac{\gamma^2}{4}) } \\
	& \quad \quad \quad  \quad \quad \quad  \quad \quad \quad \times 	  	\left(	g_{\frac{\gamma}2}(\sigma_1)   -g_{\frac{\gamma}2}(\sigma_2+\frac{\beta_1}{2} ) \right)
H
\begin{pmatrix}
\beta_1 + \frac{\gamma}{2}, \beta_2, \beta_3 \\
\sigma_1,  \sigma_2 + \frac{\gamma}{4},   \sigma_3 
\end{pmatrix}  
	 \Bigg] \\
	& =  2 \frac{ \Gamma(-1+\frac{\gamma \beta_1}{2} - \frac{\gamma^2}{4} ) \Gamma(1 - \frac{\gamma \beta_1}{2}) }{ \Gamma(- \frac{\gamma^2}{4}) } 
		  	\left(	g_{\frac{\gamma}2}(\sigma_1)   -g_{\frac{\gamma}2}(\sigma_2+\frac{\beta_1}{2} ) \right) R(\beta_1 + \frac{\gamma}{2}, \sigma_1,   \sigma_2 + \frac{\gamma}{4}).\\
	\end{align*}
	In the above three limits the second one is actually trivial, only the first and third involve using Lemma \ref{lem:HlimR}. Putting these limits together using \eqref{shift three point 1} leads to~\eqref{equation R1}.
	By using the alternative function $\tilde{H}_{\frac{\gamma}{2}}(t)$ along the same lines we obtain the desired relation \eqref{equation R2} between $R(\beta_1+\frac{\gamma}{2}, \sigma_1,  \sigma_2 + \frac{\gamma}{4})$ and $R(\beta_1 + \gamma, \sigma_1,   \sigma_2)$. 
 \end{proof}

\subsection{Analytic continuation for $H$ and $R$}\label{subsec:ext_H_R}

In this subsection we will analytically extend the functions $H$ and $R$ to a larger domain of parameters than the one we are currently working with. For this purpose, we must first derive the so-called reflection principle for $H$ and $R$ which will be obtain by performing the operator product expansion (OPE) with reflection in the case $\chi = \frac{\gamma}{2}$.  We rely extensively on Lemma \ref{lem_reflection_ope2} giving the Taylor expansion using the reflection coefficient. Finally we also extend the validity of Lemma \ref{lem:HlimR} to a larger range of parameters.

\begin{lemma}[Reflection principle for $H$] \label{reflection_H} Consider parameters $\sigma_1, \sigma_2, \sigma_3 $, $\beta_1, \beta_2, \beta_3$ satisfying  \eqref{sec4_proba_H} and \eqref{sec4_proba_R} so that $H$ and $R$ are well-defined. One can meromorphically extend $\beta_1 \mapsto  H$ beyond the point $\beta_1 = Q$ by the following relation:
	\begin{align}\label{eq:reflection_id}
	 H
\begin{pmatrix}
\beta_1 , \beta_2, \beta_3 \\
\sigma_1,  \sigma_2,   \sigma_3 
\end{pmatrix} = R(\beta_1,\sigma_1,\sigma_2)  H
\begin{pmatrix}
2Q - \beta_1 , \beta_2, \beta_3 \\
\sigma_1,  \sigma_2,   \sigma_3 
\end{pmatrix}.
	\end{align}
	The quantity $ H
\begin{pmatrix}
2Q - \beta_1 , \beta_2, \beta_3 \\
\sigma_1,  \sigma_2,   \sigma_3 
\end{pmatrix}$ is thus well-defined as long as the parameters of $ H
\begin{pmatrix}
\beta_1 , \beta_2, \beta_3 \\
\sigma_1,  \sigma_2,   \sigma_3 
\end{pmatrix}  $ and $R(\beta_1,\sigma_1,\sigma_2)$ respectively obey the constraints of \eqref{sec4_proba_H} and \eqref{sec4_proba_R}.
	Similarly, for $\sigma_1, \sigma_2 \in \cB $, we can analytically extend $\beta_1 \mapsto R(\beta_1,\sigma_1,\sigma_2)$ to the range $((\frac2\gamma - \frac\gamma2) \vee \frac{\gamma}{2}, Q + (\frac{2}{\gamma} \wedge \gamma))$  thanks to the relation:
	\begin{align}\label{equation R reflect}
	R(\beta_1,\sigma_1,\sigma_2) R(2Q-\beta_1,\sigma_1,\sigma_2) = 1.
	\end{align}
\end{lemma}

\begin{proof}
 Throughout the proof we keep the same notations as used in the proof of Lemma \ref{Shift1_H_proba} for the solution space of the hypergeometric equation satisfied by $H_{\frac{\gamma}{2}}(t)$ for $t \in(0,1)$. We assume the parameters $\beta_i$, $\sigma_i$ obey the condition \eqref{eq:para_H_chi} in order for $H_{\frac{\gamma}{2}}(t)$ to be well-defined. We also add the condition $\beta_1 \in (Q - \beta_0, Q )$ so that we can apply the result of Lemma \ref{lem_reflection_ope2} to get:
\begin{equation}\label{eq_ope_gluing}
C_2^+ =  R(\beta_1, \sigma_1- \frac{\gamma}{4}, \sigma_2 - \frac{\gamma}{4}) H
\begin{pmatrix}
2Q - \beta_1 - \frac{\gamma}{2} , \beta_2, \beta_3 \\
\sigma_1 - \frac{\gamma}{4},  \sigma_2,   \sigma_3 
\end{pmatrix}.
\end{equation}
The key argument is to observe that since by Lemma \ref{lem:ext-H-chi} $ \beta_1 \mapsto H_{\frac{\gamma}{2}}(t)$ is complex analytic in $\beta_1$ so is the coefficient $C_2^+$. 
By using this combined with the analyticity of $R$ and $H$, we can extend the range of validity of equation \eqref{eq_ope_gluing} from $\beta_1 \in (Q - \beta_0, Q )$ to $\beta_1 \in (\frac{2}{\gamma}, Q )$, still under the constraint of  \eqref{eq:para_H_chi}. Now equation \eqref{eq_ope_gluing2} derived in the the proof of Lemma \ref{Shift1_H_proba} gives us an alternative expression for $C_2^+$, which is valid for $\beta_1 \in (\frac{\gamma}{2}, \frac{2}{\gamma})$. The analyticity of $\beta_1 \mapsto C_2^+$ in a complex neighborhood of $\frac{2}{\gamma}$ then implies that one can ``glue" together the two expressions for $C_2^+$. More precisely,  by setting equal the right hand sides of \eqref{eq_ope_gluing2} and \eqref{eq_ope_gluing} and after performing the parameter replacement $\sigma_i \rightarrow \sigma_i + \frac{\gamma}{4}$ for $i=1,2,3$, one obtains the equality
\begin{align}
 &R(\beta_1, \sigma_1, \sigma_2) H
\begin{pmatrix}
2Q - \beta_1 - \frac{\gamma}{2}, \beta_2, \beta_3 \\
\sigma_1,  \sigma_2 + \frac{\gamma}{4},   \sigma_3 + \frac{\gamma}{4} 
\end{pmatrix}  \\
 &=\frac{\Gamma(-1+\frac{\gamma \beta_1}{2} - \frac{\gamma^2}{4} ) \Gamma(1 - \frac{\gamma \beta_1}{2})}{\Gamma(- \frac{\gamma^2}{4})} 
	  	\left(	g_{\frac{\gamma}2}(\sigma_1)   -g_{\frac{\gamma}2}(\sigma_2+\frac{\beta_1}{2} ) \right)
H
\begin{pmatrix}
\beta_1 + \frac{\gamma}{2}, \beta_2, \beta_3 \\
\sigma_1,  \sigma_2 + \frac{\gamma}{4},   \sigma_3 + \frac{\gamma}{4} 
\end{pmatrix}, \nonumber
\end{align} 
which provides the desired analytic continuation of $H$. To land on the form of the reflection equation given in the lemma one needs to replace $\beta_1$ by $\beta_1 - \frac{\gamma}{2}$. This transforms $R(\beta_1, \sigma_1, \sigma_2)$ into $R(\beta_1 - \frac{\gamma}{2}, \sigma_1, \sigma_2)$ which we can shift back to $R(\beta_1, \sigma_1, \sigma_2 + \frac{\gamma}{4})$ using the shift equation \eqref{equation R1}. Lastly we perform the parameter replacement $ \sigma_2 +\frac{\gamma}{4}$ to $\sigma_2$ and $ \sigma_3 + \frac{\gamma}{4}$ to $\sigma_3$. This implies the reflection principle~\eqref{eq:reflection_id} for $H$. The relation~\eqref{equation R reflect} for $R$ follows as an immediate consequence. 
\end{proof} 

We now use the shift equations we have derived to analytically continue $H$ and $R$ both in the parameters $\beta_i$ and $\sigma_i$. The analytic continuations will be defined in a larger range of parameters than \eqref{sec4_proba_H} and \eqref{sec4_proba_R} required for the GMC expressions to be well-defined. The proofs follow closely the ones of \cite{rz-boundary}. We start with the case of $R$ which is very straightforward.

\begin{lemma}\label{analycity_R} (Analytic continuation of $ R(\beta_1, \sigma_1, \sigma_2)$)
	For all $ \sigma_1, \sigma_2 \in \cB $, the meromorphic function $\beta_1 \mapsto R(\beta_1, \sigma_1, \sigma_2)$ originally defined on the interval $((\frac2\gamma - \frac\gamma2) \vee \frac{\gamma}{2}, Q)$ extends to a meromorphic function defined in a complex neighborhood of $\mathbb{R}$ and satisfying the shift equation:
	\begin{align}\label{R_first_shift} 
	&\frac{R(\beta_1, \sigma_1, \sigma_2)}{R(\beta_1 + \gamma, \sigma_1, \sigma_2)} = (\frac{\gamma}{2})^4  \frac{\Gamma(-1+\frac{\gamma\beta_1}{2}-\frac{\gamma^2}{4}) \Gamma(1-\frac{\gamma\beta_1}{2}-\frac{\gamma^2}{4}) \Gamma(1-\frac{\gamma\beta_1}{2}) \Gamma(-1+\frac{\gamma\beta_1}{2})}{\sin(\pi\frac{\gamma^2}{4}) \Gamma(1-\frac{\gamma^2}{4})^2}  \\
	&\times 4\sin(\frac{\gamma\pi}{2}(\frac{\beta}{2}-\sigma_1-\sigma_2+Q)) \sin(\frac{\gamma\pi}{2}(\frac{\beta}{2}+\sigma_1+\sigma_2-Q)) \sin(\frac{\gamma\pi}{2}(\frac{\beta}{2}+\sigma_2-\sigma_1)) \sin(\frac{\gamma\pi}{2}(\frac{\beta}{2}+\sigma_1-\sigma_2)) . \nonumber
	\end{align}
	Furthermore, for a fixed $\beta_1$ in the above complex neighborhood of $\mathbb{R}$, the function $ R(\beta_1, \sigma_1, \sigma_2)$ extends to a meromorphic function of $(\sigma_1, \sigma_2)$ on $\mathbb{C}^2$.
\end{lemma}
\begin{proof} This proof is performed in \cite{rz-boundary} but we sketch it below as it is quite short.
Fix first the parameters $\sigma_1, \sigma_2$ in $\cB$. Thanks to Proposition \ref{prop:R-def}, $\beta_1 \mapsto R(\beta_1, \sigma_1, \sigma_2) $ is meromorphic in a complex neighborhood of the real interval $((\frac2\gamma - \frac\gamma2) \vee \frac{\gamma}{2}, Q)$, and thanks to the reflection principle given by Lemma \ref{reflection_H}, it is meromorphic in a complex neighborhood of $((\frac2\gamma - \frac\gamma2) \vee \frac{\gamma}{2}, Q + (\frac{2}{\gamma} \wedge \gamma) )$. Note that the length of this interval is strictly greater than $ \gamma$. 

Combining the two shift equations of Lemma \ref{lem_shift_eq_R}, we obtain the shift equation \eqref{R_first_shift} relating $R(\beta_1 + \gamma, \sigma_1, \sigma_2) $ and $R(\beta_1, \sigma_1, \sigma_2) $, which does not shift the $\sigma_1, \sigma_2$ parameters. Since $R(\beta_1, \sigma_1, \sigma_2) $ is defined in a complex neighborhood of a real interval of length strictly bigger than $\gamma$, the shift equation \eqref{R_first_shift} can be used to meromorphically extend $ \beta_1 \mapsto R(\beta_1, \sigma_1, \sigma_2) $ to a complex neighborhood of the whole real line.

Now for any $\beta_1$ fixed in this complex neighborhood of $\mathbb{R}$, we perform the analytic continuation to $\mathbb{C}^2$ in the variables $\sigma_1, \sigma_2$. For this one can simply apply either of the shift equation of Lemma \ref{lem_shift_eq_R}. This is possible since the band $\cB$ has width strictly greater than $\frac{\gamma}{4}$. Hence the result.
\end{proof}

\begin{lemma}\label{analycity_H}(Analytic continuation of $H$)  Fix $\sigma_3 \in [-\frac{1}{2 \gamma} + \frac{Q}{2}, \frac{1}{2 \gamma} + \frac{Q}{2} - \frac{\gamma}{4}] \times \mathbb{R} $ and fix $\sigma_1, \sigma_2 \in \cB $. Then the function $(\beta_1, \beta_2, \beta_3) \mapsto H
$ originally defined in the parameter range given by \eqref{sec4_proba_H} extends to a meromorphic function of the three variables in a small complex neighborhood of $\mathbb{R}^3$. Now fix $\beta_1, \beta_2, \beta_3$ in this complex neighborhood of $\mathbb{R}^3$, keeping $\sigma_3$ still fixed in $[-\frac{1}{2 \gamma} + \frac{Q}{2}, \frac{1}{2 \gamma} + \frac{Q}{2} - \frac{\gamma}{4}] \times \mathbb{R} $. The function $(\sigma_1, \sigma_2) \mapsto H
$ then extends to a meromorphic function of $\mathbb{C}^2$.
\end{lemma}

\begin{proof} 
This proof also follows exactly the proof of \cite[Lemma 3.6]{rz-boundary}, with one notable difference highlighted below that prevents us from performing an analytic continuation in $\sigma_3$. The main idea of \cite{rz-boundary} is to use the shift equations on $H$ to successively extend the domain where $H$ is defined to larger and larger regions until it has been extended to the region claimed in the lemma. Keeping the notations of \cite{rz-boundary}, we choose as the starting domain the region given by:

\begin{equation}
\mathcal{E}_1 := \left \{  \beta_i \in (\frac{2}{\gamma} - \delta, Q) \times [-\nu, \nu], \: \sigma_i \in (-\frac{1}{2\gamma} + \frac{Q}{2}, \frac{1}{2\gamma} + \frac{Q}{2}) \times \mathbb{R} \right \}.
\end{equation}

By $\beta_i \in (\frac{2}{\gamma} - \delta, Q)\times [-\nu, \nu]$ we mean $\Re(\beta_i) \in (\frac{2}{\gamma} - \delta, Q) $ and $\Im(\beta_i) \in [-\nu, \nu] $, the same convention is used for the domain of $\sigma_i$. We have introduced $\delta, \nu >0$ chosen small enough so that \eqref{sec4_proba_H} holds for $\beta_i \in (\frac{2}{\gamma} - \delta, Q) $ and one can apply Proposition \ref{prop:H-def} to show analyticity of $H$ in all of its variables on $\mathcal{E}_1$. The condition on $\delta$ is $\delta < \frac{1}{\gamma} - \frac{\gamma}{4}$. By applying the shift equations on $H$ as performed in the proof of \cite[Lemma 3.6]{rz-boundary}, one can meromorphically extend $H$ to the region:

\begin{equation}
\mathcal{E}_7 := \left \{ \beta_i \in \mathbb{R} \times [-\nu, \nu],  \: (\sigma_1,\sigma_2) \in \mathbb{C}^2, \: \sigma_3 \in (-\frac{1}{2\gamma} + \frac{Q}{2}, \frac{1}{2\gamma} + \frac{Q}{2} -\frac{\gamma}{4} ) \times \mathbb{R} \right \}.
\end{equation}

Note here that we have performed a cyclic permutation with respect to the indices $1,2,3$ of the parameters in comparison to the domain $\mathcal{E}_7$ written in \cite{rz-boundary}. This thus gives the desired domain of analytic continuation of $H$ claimed in the statement of the lemma. 

Contrarily to what is done in \cite{rz-boundary}, we are not able to lift the constraint on $\sigma_3$. Indeed, in the case $\mu= 0$ when $H$ has an expression reducing to a moment of a boundary GMC measure, if one adds a global constant to each of the $\sigma_i$ the function $H$ is simply changed by a global phase. In our case this property does not hold which is why we are not able to perform the extension in the variable $\sigma_3$. Notice though that once Theorem \ref{thm:H} is established, the exact formula $\HPT$ implies that $H$ extends meromorphically in all its parameters to $\mathbb{C}^6$.
\end{proof}

We finish this subsection by extending Lemma \ref{lem:HlimR} to a larger range of parameters that will be required for the next subsection.
This is a novel difficulty that was not present in \cite{rz-boundary} because in our case we had to perform a truncation procedure to define $H$ and $R$ while in \cite{rz-boundary} they simply reduce to moments of GMC.

\begin{lemma}\label{lem:HlimR_ext}
Suppose $\beta_1, \beta_2, \beta_3$ satisfy $Q - \frac{1}{2} \sum_{i} \beta_i < \frac{2}{\gamma} \wedge \min_i (Q -\beta_i)$,  $Q> \beta_1 > \beta_2 \vee \frac{\gamma}{2}$ and $\beta_1 - \beta_2 < \beta_3 < Q$. Let also $\sigma_1 \in \overline{\cB}$ and $\sigma_2, \sigma_3 \in \cB$. Then the following limit holds
\begin{align*}
	&\lim_{\beta_3 \rightarrow \beta_1 - \beta_2 } (\beta_2 + \beta_3 - \beta_1) H
\begin{pmatrix}
\beta_1 , \beta_2, \beta_3 \\
\sigma_1,  \sigma_2,   \sigma_3 
\end{pmatrix} = 2 R(\beta_1, \sigma_1, \sigma_2),
	\end{align*}
where the function $R$ should be understood by the analytic continuation given in Lemma \ref{analycity_R}.
\end{lemma}

\begin{proof}  
By the meromorphicity  of $H$ function from Lemma \ref{analycity_H} and the residue theorem, we have \begin{align*}
\lim_{\beta_3 \rightarrow \beta_1 - \beta_2 } (\beta_2 + \beta_3 - \beta_1) H
\begin{pmatrix}
\beta_1 , \beta_2, \beta_3 \\
\sigma_1,  \sigma_2,   \sigma_3 
\end{pmatrix} =  \frac{1}{2 \pi \ii} \int_{\mathcal{C}_{\epsilon}(\beta_1- \beta_2) }
 H
\begin{pmatrix}
\beta_1 , \beta_2, \beta_3 \\
\sigma_1,  \sigma_2,   \sigma_3 
\end{pmatrix} d \beta_3  
\end{align*}
where above $\mathcal{C}_{\epsilon}(\beta_1- \beta_2)$ is a circle of radius $\epsilon$ around $\beta_1 -\beta_2$. By Lemma \ref{lem:HlimR}, the equality
\begin{align*}
\frac{1}{2 \pi \ii} \int_{\mathcal{C}_{\epsilon}(\beta_1- \beta_2) }
 H
\begin{pmatrix}
\beta_1 , \beta_2, \beta_3 \\
\sigma_1,  \sigma_2,   \sigma_3 
\end{pmatrix} d \beta_3 =  2 R (\beta_1, \sigma_1, \sigma_2)
\end{align*}
holds in the parameter range given by Lemma \ref{lem:HlimR}. 
By the analytic property of $H$ and $R$, this equality  extends to the larger parameter range in Lemma~\ref{lem:HlimR_ext}. This concludes the proof.
\end{proof}

\subsection{The $\frac2{\gamma}$-shift equations for $R$}\label{subsec:shift2_R}
We now derive the $\frac{2}{\gamma}$-shift equations on $ R(\beta_1, \sigma_1, \sigma_2)$.
\begin{lemma}\label{lem:sec4.5}  For all $ \sigma_1, \sigma_2 \in \mathbb{C}$, the meromorphic function $\beta_1 \mapsto R(\beta_1, \sigma_1, \sigma_2)$ defined in a complex neighborhood of $\mathbb{R}$ satisfies the  shift equations in Theorem~\ref{full_shift_R}  with $\chi= \frac2{\gamma}$.
\end{lemma}

To prove Lemma~\ref{lem:sec4.5}  we work exclusively with $\chi = \frac{2}{\gamma}$.  There are two steps that will each require their own range of parameters. In the first step, we place ourselves in the following range that will allow us to apply the OPE with reflection in Lemma \ref{lem_reflection_ope2} around the $\beta_1$ insertion at $0$: 
\begin{equation}
t \in (0,1), \: \: \beta_1, \beta_2, \beta_3 \in (Q -\epsilon, Q),   \: \: \sigma_1, \sigma_2 \in \ii \mathbb{R} + \frac{1}{2\gamma} + \frac{Q}{2} , \:\:  \sigma_3 \in \cB.
\end{equation}

In the above $\epsilon $ is chosen small enough, smaller than the constant $\beta_0$ required to apply Lemma \ref{lem_reflection_ope2}. We can expand $H_{\frac{2}{\gamma}}(t)$ on the basis, for $t \in (0,1)$
\begin{align}
H_{\frac{2}{\gamma}}(t) &= C_1 F(A,B,C,t) + C_2^+ t^{1 - C} F(1 + A-C, 1 +B - C, 2 -C, t) \\ \nonumber
 &= B_1 F(A,B,1+A+B- C, 1 -t) + B_2^- (1-t)^{C -A -B} F(C- A, C- B, 1 + C - A - B , 1 -t),
\end{align}
where again $C_1, C_2^+, B_1, B_2^-$ are parametrizing the solution space around the points $0$ and $1$. As before by sending $t$ to $0$ and to $1$ one obtains:
\begin{equation}
C_1 = H_{\frac{2}{\gamma}}(0) = H
\begin{pmatrix}
\beta_1 - \frac{2}{\gamma}, \beta_2, \beta_3 \\
\sigma_1 - \frac{1}{\gamma},  \sigma_2,   \sigma_3 
\end{pmatrix}, \quad B_1 = H_{\frac{2}{\gamma}}(1) = H
\begin{pmatrix}
\beta_1 , \beta_2 - \frac{2}{\gamma}, \beta_3 \\
\sigma_1 - \frac{1}{\gamma},  \sigma_2 - \frac{1}{\gamma},   \sigma_3 
\end{pmatrix}.
\end{equation}
Since the condition required for Lemma \ref{lem_reflection_ope2}, $\beta_1 \in (Q-\beta_0, Q)$, is satisfied one then derives:
\begin{equation}
 C_2^+ = R(\beta_1, \sigma_1 - \frac{1}{\gamma}, \sigma_2-\frac{1}{\gamma}) H
\begin{pmatrix}
2Q- \beta_1 - \frac{2}{\gamma}, \beta_2, \beta_3 \\
\sigma_1 - \frac{1}{\gamma},  \sigma_2,   \sigma_3 
\end{pmatrix}.
\end{equation}
Similarly, we can apply Lemma \ref{lem_reflection_ope2} around $t=1$ and get:
\begin{equation}
 B_2^- = R(\beta_2, \sigma_2,\sigma_3) H
\begin{pmatrix}
\beta_1 , 2Q - \beta_2 - \frac{2}{\gamma}, \beta_3 \\
\sigma_1 - \frac{1}{\gamma},  \sigma_2 - \frac{1}{\gamma},   \sigma_3 
\end{pmatrix}.
\end{equation}
The quantities $C_1, B_2^-, C_2^+$ identified above are then related by the connection formula \eqref{connection1}:

\begin{equation}
B_2^- = \frac{\Gamma(C)\Gamma(A+B-C)}{\Gamma(A)\Gamma(B)} C_1 + \frac{\Gamma(2-C)\Gamma(A+B-C)}{\Gamma(A-C+1)\Gamma(B-C+1)} C_2^+.
\end{equation}

We repeat the same procedure to identify the coefficients $D_2^+, C_2^-$ using the same parameter ranges as before except with $t \in (- \infty,0)$:
\begin{equation}
t \in (-\infty, 0), \: \: \beta_1, \beta_2, \beta_3 \in (Q -\epsilon, Q),   \: \: \sigma_1, \sigma_2 \in \ii \mathbb{R} + \frac{1}{2\gamma} + \frac{Q}{2} , \:\:  \sigma_3 \in \cB.
\end{equation}

For $t \in (-\infty,0)$, $H_{\frac{2}{\gamma}}(t)$ can be expanded on the basis:
\begin{align}
H_{\frac{2}{\gamma}}(t) &= C_1 F(A,B,C,t) + C_2^- t^{1 - C} F(1 + A-C, 1 +B - C, 2 -C, t) \\ \nonumber
 &= D_1 e^{\ii \pi A} t^{-A} F(A,1+A-C,1+A-B,t^{-1}) + D_2^+ e^{\ii \pi B}t^{-B} F(B, 1 +B - C, 1 +B - A, t^{-1}).
\end{align}
The coefficients $C_2^-, D_2^+$ have expressions given by:
\begin{align}
C_2^- &= e^{- \ii \pi(1-\frac{2\beta_1}{\gamma}+\frac{4}{\gamma^2})} R(\beta_1, \sigma_1, \sigma_2) H
\begin{pmatrix}
2Q - \beta_1 - \frac{2}{\gamma}, \beta_2, \beta_3 \\
\sigma_1 - \frac{1}{\gamma},  \sigma_2,   \sigma_3 
\end{pmatrix},\\
 D_2^+ &= R(\beta_3, \sigma_1 - \frac{1}{\gamma}, \sigma_3) H
\begin{pmatrix}
\beta_1 , \beta_2, 2Q - \beta_3 - \frac{2}{\gamma} \\
\sigma_1,  \sigma_2,   \sigma_3 
\end{pmatrix}.
\end{align}
Using the connection formula \eqref{hpy1} we can write this time: 
\begin{align}
D_2^+ = \frac{\Gamma(C)\Gamma(A-B)}{\Gamma(A)\Gamma(C-B)}C_1 + e^{ \ii \pi (1-C)}\frac{\Gamma(2-C)\Gamma(A-B)}{\Gamma(1-B)\Gamma(A-C+1)}C_2^-.
\end{align}
By eliminating the coefficient $C_1$ we obtain the relation:
\begin{align}\label{equation 3.21}
\frac{\Gamma(B)}{\Gamma(A+B-C)}B_2^- -\frac{\Gamma(C-B)}{\Gamma(A-B)}D_2^+ = \frac{\Gamma(2-C)}{\Gamma(A-C+1)}\left(  \frac{\Gamma(B)}{\Gamma(B-C+1)} C_2^+ -\frac{e^{\ii \pi (1-C)}\Gamma(C-B)}{\Gamma(1-B)} C_2^- \right).
\end{align}
We state~\eqref{equation 3.21} as a lemma, where the constants $B_2^-$, $D_2^+$, $C_2^+$, and $C_2^-$ are replaced by their explicit expressions in terms of $H$ and $R$.

\begin{lemma}
The following identity holds:
\begin{align}\label{eq:lem_BCD}
&\frac{\Gamma(B)}{\Gamma(A+B-C)}R(\beta_2, \sigma_2,\sigma_3) H
\begin{pmatrix}
\beta_1 , 2Q - \beta_2 - \frac{2}{\gamma}, \beta_3 \\
\sigma_1 - \frac{1}{\gamma},  \sigma_2 - \frac{1}{\gamma},   \sigma_3 
\end{pmatrix}    -\frac{\Gamma(C-B)}{\Gamma(A-B)} R(\beta_3, \sigma_1 - \frac{1}{\gamma}, \sigma_3) H
\begin{pmatrix}
\beta_1 , \beta_2, 2Q - \beta_3 - \frac{2}{\gamma} \\
\sigma_1,  \sigma_2,   \sigma_3 
\end{pmatrix}  \\
 &= \frac{\Gamma(2-C)}{\Gamma(A-C+1)} H
\begin{pmatrix}
2Q - \beta_1 - \frac{2}{\gamma}, \beta_2, \beta_3 \\
\sigma_1 - \frac{1}{\gamma},  \sigma_2,   \sigma_3 
\end{pmatrix}   \Bigg(  \frac{\Gamma(B)}{\Gamma(B-C+1)} R(\beta_1, \sigma_1 - \frac{1}{\gamma}, \sigma_2-\frac{1}{\gamma}) \nonumber \\
 & \quad \quad \quad  -\frac{e^{\ii \pi (1-C)}\Gamma(C-B)}{\Gamma(1-B)} e^{- \ii \pi(1-\frac{2\beta_1}{\gamma}+\frac{4}{\gamma^2})} R(\beta_1, \sigma_1, \sigma_2)  \Bigg). \nonumber
\end{align}

This identity is originally derived in the range of parameters
\begin{equation}
\beta_1, \beta_2, \beta_3 \in (Q -\epsilon, Q),  \: \: \sigma_1, \sigma_2 \in \ii \mathbb{R} + \frac{1}{2\gamma} + \frac{Q}{2} , \:\:  \sigma_3 \in \cB,
\end{equation}
but it can be viewed as an identity for $H$ and $R$ under the meromorphic extension  provided by Lemmas \ref{analycity_H} and \ref{analycity_R}.
\end{lemma}

In the second step we derive the following crucial intermediate relation on $R$ from~\eqref{eq:lem_BCD}.
\begin{lemma}\label{lem:R-relation}
Set \(\beta \in (\frac{\gamma}{2}, \frac{2}{\gamma})\), \(\sigma_1 \in \ii \mathbb{R} + \frac{1}{2\gamma} + \frac{Q}{2}\), and \(\sigma_2, \sigma_3 \in \cB.\)
Let $\mathcal{R}(\beta,\sigma_1,\sigma_2) = \frac{R(\beta,\sigma_1,\sigma_2)}{R(\beta+\frac{2}{\gamma}, \sigma_1-\frac{1}{\gamma}, \sigma_2)}$. Then
    \begin{align}\label{equation R intermediary}
\frac{\mathcal{R}(\beta,\sigma_1,\sigma_2) - \mathcal{R}(\beta,\sigma_1,\sigma_3 - \frac{\gamma}{4})}{\Gamma(-1+\frac{2\beta}{\gamma}-\frac{4}{\gamma^2})\Gamma(1-\frac{2\beta}{\gamma})}
=\frac{2}{\gamma\Gamma(-\frac{4}{\gamma^2})}\lim_{\eta \to 0}\eta  R(\frac{\gamma}{2}+\eta, \sigma_2, \sigma_3). 
\end{align}
\end{lemma}
\begin{proof}
We prove~\eqref{equation R intermediary} by applying the $H\to R$ limit in Lemma \ref{lem:HlimR_ext} to~\eqref{eq:lem_BCD}. For this purpose, set
\(\beta_1 = \beta \in (\frac{\gamma}{2}, \frac{2}{\gamma}), \quad \beta_2 = \frac{\gamma}{2}+\eta, \quad \beta_3 = Q-\beta\).
With this choice of parameters we have
\begin{equation}
q = \frac{4}{\gamma^2}-\frac{\eta}{\gamma}, \quad A = -\frac{4}{\gamma^2}+\frac{\eta}{\gamma}, \quad  B = \frac{2\beta}{\gamma}-\frac{4}{\gamma^2}+\frac{\eta}{\gamma}, \quad C = \frac{2\beta}{\gamma}-\frac{4}{\gamma^2},
\end{equation}
and the two $H$ functions that we are going to apply the Lemma \ref{lem:HlimR_ext} to are
\begin{align*}
&
H
\begin{pmatrix}
2Q - \beta - \frac{2}{\gamma}, \frac{\gamma}{2} + \eta, Q - \beta \\
\sigma_1 - \frac{1}{\gamma},  \sigma_2,   \sigma_3 
\end{pmatrix}\quad \text{and} \quad  H
\begin{pmatrix}
\beta , \frac{\gamma}{2} + \eta, \frac{\gamma}{2} + \beta \\
\sigma_1,  \sigma_2,   \sigma_3 
\end{pmatrix}.
\end{align*}

We compute the following limits, the last one is trivial and does not require using Lemma \ref{lem:HlimR_ext}.
\begin{align*}
\lim_{\eta \to 0} \eta D_2^+ &=   2 R(Q-\beta, \sigma_1-\frac{1}{\gamma}, \sigma_3) R(\beta+\frac{\gamma}{2}, \sigma_1, \sigma_3) = \frac{2R(\beta+\frac{\gamma}{2},\sigma_1,\sigma_3)}{R(\beta+Q, \sigma_1-\frac{1}{\gamma}, \sigma_3)}, \\
\lim_{\eta \to 0} \eta  C_2^- &= 2e^{- \ii \pi(1-\frac{2\beta}{\gamma}+\frac{4}{\gamma^2})}   R(\beta, \sigma_1, \sigma_2) R(2Q-\beta-\frac{2}{\gamma}, \sigma_1-\frac{1}{\gamma}, \sigma_2) =  \frac{2e^{- \ii \pi(1-\frac{2\beta}{\gamma}+\frac{4}{\gamma^2})} R(\beta, \sigma_1, \sigma_2)}{ R(\beta+\frac{2}{\gamma}, \sigma_1-\frac{1}{\gamma}, \sigma_2)} ,\\
\lim_{\eta \to 0}\eta^2   B_2^-  &= -4 \lim_{\eta \to 0}\eta   R(\frac{\gamma}{2}+\eta, \sigma_2, \sigma_3).
\end{align*}
Putting all these into \eqref{eq:lem_BCD}, we get: 
\begin{align*}
-4\frac{\Gamma(\frac{2\beta}{\gamma}-\frac{4}{\gamma^2})}{\Gamma(-\frac{4}{\gamma^2})}\lim_{\eta \to 0}\eta  R(\frac{\gamma}{2}+\eta, \sigma_2, \sigma_3) + \frac{2\gamma}{\Gamma(-\frac{2\beta}{\gamma})}\frac{(1-\frac{\gamma\beta}{2}+\frac{\gamma^2}{4})}{-\frac{\gamma\beta}{2}}\frac{R(\beta,\sigma_1,\sigma_3-\frac{\gamma}{4})}{R(\beta+\frac{2}{\gamma},\sigma_1-\frac{1}{\gamma}, \sigma_3-\frac{\gamma}{4})}\\
=  \frac{2\gamma (1-\frac{2\beta}{\gamma} + \frac{4}{\gamma^2})}{\Gamma(1-\frac{2\beta}{\gamma})} \frac{R(\beta, \sigma_1, \sigma_2)}{ R(\beta+\frac{2}{\gamma}, \sigma_1-\frac{1}{\gamma}, \sigma_2)}.
\end{align*}
After simplifications one obtains~\eqref{equation R intermediary}.
\end{proof}

We now determine the right hand side of~\eqref{equation R intermediary}. By the $\frac{\gamma}{2}$-shift equation \eqref{equation R2} on $R$ we have:
\begin{align*}
&R(\frac{\gamma}{2} + \eta, \sigma_2, \sigma_3) = -\frac{\Gamma(-1+\frac{\gamma \eta}{2}  ) \Gamma(1 - \frac{\gamma^2}{4} - \frac{\gamma \eta}{2}) }{ \Gamma(- \frac{\gamma^2}{4}) } \left( g_{\frac{\gamma}{2}}(\sigma_2 ) - g_{\frac{\gamma}{2}}(\sigma_3 - \frac{\gamma}{4} + \frac{\eta}{2}) \right)  R(\gamma +\eta, \sigma_2,   \sigma_3 -  \frac{\gamma}{4}).
\end{align*}
Simplifying the limit, one gets:

\begin{align}\label{equation R intermediary2}
&\lim_{\eta\to 0} \eta    R(\frac{\gamma}{2}+\eta, \sigma_2, \sigma_3)  = \frac{2}{\gamma} \frac{  \Gamma(1 - \frac{\gamma^2}{4} ) }{ \Gamma(- \frac{\gamma^2}{4}) }   \left( g_{\frac{\gamma}{2}}(\sigma_2) - g_{\frac{\gamma}{2}}(\sigma_3 - \frac{\gamma}{4}) \right)  R(\gamma, \sigma_2,   \sigma_3 -  \frac{\gamma}{4}) .
\end{align}

We use an extra input form the mating-of-trees framework to evaluate $R(\gamma, \sigma_2,   \sigma_3 -  \frac{\gamma}{4})$.
\begin{lemma}\label{lem-mot-match}
	Suppose $\frac{4}{\gamma^2} \not \in \mathbb Z$ and set $c = \sqrt{1/\sin(\frac{\pi \gamma^2}4)}$. Define $f: (-c, c) \to \mathbb R$ by $f(\mu) = \cos(\frac4{\gamma^2} \arccos( \mu/c))$; that is, if $\sigma \in \mathcal B$ satisfies $\mu = c \cos (\pi \gamma(\sigma - \frac Q2))$, then $f(\mu) = \cos(\frac{4\pi}\gamma(\sigma - \frac Q2))$. 
For a sample from $\cM_2^\disk(\gamma)$, let $(L_1, L_2)$ be the quantum lengths and $A$ the quantum area. Then for all $k> \frac4{\gamma^2}$ we have 
	\[f^{(k)}(\mu) = \mathcal{C}_1 \cM_2^\disk (\gamma)[ (\mu (-L_1)^k + k (-L_1)^{k-1}) e^{-A - \mu L_1} ] \quad\text{ for all } \mu \in (-c, c).\]
 where $\mathcal{C}_1 = \frac{4}{\gamma^2} \sin( \frac{4 \pi}{\gamma^2}) \Gamma( \frac{4}{\gamma^2}) ( 1 - \frac{\gamma^2}{4} ) \left( \Gamma(\frac{\gamma^2}{4})^2 \sin(\frac{\pi \gamma^2}{4} ) \right)^{ - \frac{2}{\gamma^2} } $.
\end{lemma}

We postpone the proof of Lemma~\ref{lem-mot-match} to Section~\ref{sec:Rgamma} and proceed to solve $R(\gamma, \sigma, \sigma')$ using it.
\begin{lemma}\label{lem:R_gamma}
As equality of meromorphic functions of $(\sigma, \sigma') \in \mathbb{C}^2$:
\[R(\gamma, \sigma, \sigma') = \left(\frac{\pi  \Gamma(\frac{\gamma^2}{4})}{\Gamma(1-\frac{\gamma^2}{4})} \right)^{\frac{2}{\gamma^2} - \frac{1}{2} }  \frac{\Gamma(1 - \frac{4}{\gamma^2})}{\Gamma(1 -\frac{\gamma^2}{4})} \frac{\cos(\frac{4\pi}\gamma (\sigma - \frac Q2)) - \cos(\frac{4\pi}\gamma (\sigma' - \frac Q2))}{\cos(\gamma\pi (\sigma - \frac Q2)) - \cos(\gamma\pi (\sigma' - \frac Q2))}.\]
\end{lemma} 

\begin{proof}
Fix arbitrary $\sigma_1$ such that $\sigma_1, \sigma_1 - \frac1\gamma \in \overline{\mathcal B}$.
Let us introduce the notations $\mu = g_{\frac{\gamma}{2}}(\sigma_2)$, $\mu' = g_{\frac{\gamma}{2}}(\sigma_3 - \frac{\gamma}{4})$, and $h(\mu) = - \frac{1}{\Gamma(-1+\frac{2\beta}{\gamma}-\frac{4}{\gamma^2})\Gamma(1-\frac{2\beta}{\gamma})} \mathcal{R}(\beta,\sigma_1,\sigma_2)$. 
Since $\lim_{\eta\to 0} \eta R(\frac{\gamma}{2}+\eta, \sigma_2, \sigma_2 - \frac{\gamma}{4})  =0$,  taking $\sigma_3 = \sigma_2 -\frac{\gamma}{4}$ in \eqref{equation R intermediary} yields that $\mathcal{R}(\beta,\sigma_1,\sigma_2)$ is $\frac{\gamma}{2}$-periodic in $\sigma_2$. Thus $h$ is $\frac{\gamma}{2}$-periodic in $\sigma_2$.
Combining   \eqref{equation R intermediary} and \eqref{equation R intermediary2} we get $(\mu - \mu')R(\gamma, \mu, \mu') = h(\mu) - h(\mu')$ for $\mu, \mu'$ in the domain $\{ z \in \C \: :\: \Re z > 0\}$ on which $h$ is holomorphic. Taking derivatives with respect to $\mu$, we get:
	\[h^{(k)}(\mu) = \mu \frac{\partial^k}{\partial\mu^k} R(\gamma, \mu, \mu')  + k \frac{\partial^{k-1}}{\partial\mu^{k-1}} R(\gamma, \mu, \mu') \quad \textrm{for }k\ge 1. \]

 For $\beta \in (\frac2\gamma, Q)$ and $k \geq 1$, using the probabilistic definition of $R$ we have 
 \begin{equation}\label{eq-mot-dR}
     \frac{{\partial^k}}{{\partial \mu^k}}R(\beta, \mu, \mu') = \frac{2(Q-\beta)}{\gamma} \cMtwo(\beta)[(-L_1)^k e^{-A - \mu L_1 - \mu' L_2}].
 \end{equation}
 The right hand side is well defined for $\beta \in (\frac\gamma2, Q)$, and when $k > \frac4{\gamma^2}$, as a function of $\beta$ it  extends holomorphically to a neighborhood of $(\frac\gamma2, Q)$ in $\C$ by Lemma~\ref{lem-hol-dR}. Since the left hand side is meromorphic in $\beta$, we conclude~\eqref{eq-mot-dR} holds for $k > \frac4{\gamma^2}$, $\beta \in (\frac\gamma2, Q)$. Choosing $\beta = \gamma$ and sending $\mu' \to 0$, we get for large $k$ that 	
 \[h^{(k)}(\mu) = \frac{2(Q - \beta)}{\gamma} \cM_2^\disk (\gamma)[ (\mu (-L_1)^k + k (-L_1)^{k-1}) e^{-A - \mu L_1} ].\]
 
 By Lemma~\ref{lem-mot-match}, using the notation $f(\mu) = \cos(\frac4{\gamma^2} \arccos( \mu/c))$ with $c = \sqrt{1/\sin(\frac{\pi \gamma^2}4)}$ and $\mathcal{C}_1$ as in Lemma~\ref{lem-mot-match}, we conclude $f^{(k)}(\mu) = \frac{\gamma}{2(Q - \gamma)} \mathcal{C}_1 h^{(k)}(\mu)$. Since $f, h$ are both analytic in a suitable domain, for some $a_0,\cdots, a_{k-1}\in \mathbb C$   we have:
	\[ \frac{\gamma}{2(Q - \gamma)} \mathcal{C}_1 h(\mu) = f(\mu) + \sum_{i=0}^{k-1} a_i \mu^i.\]
	Thus with $c_{\gamma} = ( \frac{4}{\gamma^2} - 1) \frac{1}{\mathcal{C}_1} \sqrt{\sin(\frac{\pi \gamma^2}{4})}$
and $p(\sigma) = \sum_{i=1}^{k -1} a_i(c \cos(\pi \gamma(\sigma - \frac Q2)))^i$ we have 
 \[R(\gamma, \sigma, \sigma') = c_\gamma \frac{\cos(\frac{4\pi}\gamma (\sigma - \frac Q2)) - \cos(\frac{4\pi}\gamma (\sigma' - \frac Q2)) + p(\sigma) - p(\sigma')}{\cos(\gamma\pi (\sigma - \frac Q2)) - \cos(\gamma\pi (\sigma' - \frac Q2))}.\]
	Since $f$ and $h$ are $\frac\gamma2$-periodic, we conclude that $p$ is $\frac\gamma2$-periodic. By its definition, it's also clear that $p$ is $\frac2\gamma$-periodic. Thus, when $\gamma^2$ is irrational, then the function has two periods which differ by an irrational factor, hence $p$ must be constant (and thus $p(\sigma) - p(\sigma') =0$). Lastly we compute 
\begin{align*}
c_{\gamma} = \frac{1}{\pi} \Gamma(1 - \frac{4}{\gamma^2})  \sqrt{\sin(\frac{\pi \gamma^2}{4})} \left( \frac{\pi \Gamma( \frac{\gamma^2}{4} ) }{\Gamma( 1 - \frac{\gamma^2}{4} )} \right)^{\frac{2}{\gamma^2}}= \frac{ \Gamma(1 - \frac{4}{\gamma^2}) }{\Gamma(1 - \frac{\gamma^2}{4})}  \left( \frac{\pi \Gamma( \frac{\gamma^2}{4} ) }{\Gamma( 1 - \frac{\gamma^2}{4} )} \right)^{\frac{2}{\gamma^2} - \frac{1}{2}} \textrm{ as desired}.\quad \quad \qedhere
\end{align*}
\end{proof}

\begin{proof}[Proof of Lemma \ref{lem:sec4.5}] 
By Lemma \ref{lem:R_gamma}, we  have
\begin{align*}
&\lim_{\eta\to 0} \eta R(\frac{\gamma}{2}+\eta, \sigma_2, \sigma_3) = \frac{2}{\gamma} \frac{  \Gamma(1 - \frac{\gamma^2}{4} ) }{ \Gamma(- \frac{\gamma^2}{4}) \sqrt{\sin(\frac{\pi \gamma^2}{4})} }   \left( \cos(\pi\gamma(\sigma_2 - \frac{Q}{2})) -\cos(\pi\gamma(\sigma_3 - \frac{\gamma}{4} -\frac{Q}{2})) \right)  R(\gamma, \sigma_2,   \sigma_3 - \frac{\gamma}{4})\\
& =  \frac{\gamma}{2 \pi}  \Gamma(1 - \frac{4}{\gamma^2} )  \left( \frac{\pi \Gamma( \frac{\gamma^2}{4} ) }{\Gamma( 1 - \frac{\gamma^2}{4} )} \right)^{\frac{2}{\gamma^2}}  \left( \cos( \frac{4 \pi}{\gamma} (\sigma_2 - \frac{Q}{2})) -\cos( \frac{4 \pi }{\gamma} (\sigma_3 - \frac{\gamma}{4} -\frac{Q}{2})) \right).
\end{align*} 
Plugging this to \eqref{equation R intermediary}, we get
\begin{align*}
&\mathcal{R}(\beta,\sigma_1,\sigma_2) - \mathcal{R}(\beta,\sigma_1,\sigma_3 - \frac{\gamma}{4})\\ \nonumber & \quad = - \frac{4}{\pi \gamma^2} \left( \frac{\pi \Gamma( \frac{\gamma^2}{4} ) }{\Gamma( 1 - \frac{\gamma^2}{4} )} \right)^{\frac{2}{\gamma^2}}  \Gamma(-1+\frac{2\beta}{\gamma}-\frac{4}{\gamma^2})\Gamma(1-\frac{2\beta}{\gamma})  \left( \cos(\frac{2\pi}{\gamma} (2 \sigma_2  - \frac{2}{\gamma})) - \cos(\frac{2 \pi}{\gamma} (2 \sigma_3 - \frac{\gamma}{2} - \frac{2}{\gamma}) ) \right).
\end{align*}
By setting $\sigma_3$ to any fixed value, the above equation implies the following claim
\begin{equation*}
\mathcal{R}(\beta,\sigma_1,\sigma_2) = - \frac{4}{\pi \gamma^2} \left( \frac{\pi \Gamma( \frac{\gamma^2}{4} ) }{\Gamma( 1 - \frac{\gamma^2}{4} )} \right)^{\frac{2}{\gamma^2}}  \Gamma(-1+\frac{2\beta}{\gamma}-\frac{4}{\gamma^2})\Gamma(1-\frac{2\beta}{\gamma})   \cos(\frac{2\pi}{\gamma} (2 \sigma_2  - \frac{2}{\gamma})) + u(\sigma_1, \beta, \gamma),
\end{equation*}
where $u(\sigma_1, \beta, \gamma)$ is a  function   not depending on $\sigma_2$. We  evaluate $u$ using the shift equation \eqref{equation R1}. Since the right side of \eqref{equation R1} is zero when $\sigma_2 = \sigma_1 - \frac{\beta}{2}$
we get  $\mathcal{R}(\beta,\sigma_1, \sigma_1 - \frac{\beta}{2}) =0$, hence
\begin{align*}
u(\sigma_1, \beta, \gamma) = \frac{4}{\pi \gamma^2} \left( \frac{\pi \Gamma( \frac{\gamma^2}{4} ) }{\Gamma( 1 - \frac{\gamma^2}{4} )} \right)^{\frac{2}{\gamma^2}}  \Gamma(-1+\frac{2\beta}{\gamma}-\frac{4}{\gamma^2})\Gamma(1-\frac{2\beta}{\gamma})   \cos(\frac{2\pi}{\gamma} (2 \sigma_1 - \beta  - \frac{2}{\gamma})).
\end{align*}
Setting now $c_{\frac{2}{\gamma}}(\gamma) = \frac{4}{ \gamma^2} \pi^{\frac{4}{\gamma^2} -1}   \Gamma(1 - \frac{\gamma^2}{4})^{- \frac{4}{\gamma^2}}$, we have thus shown that:
\begin{align}
\frac{R(\beta, \sigma_1,\sigma_2)}{R(\beta+\frac{2}{\gamma}, \sigma_1-\frac{1}{\gamma}, \sigma_2)} = c_{\frac{2}{\gamma}}(\gamma)  \Gamma(-1+\frac{2\beta}{\gamma}-\frac{4}{\gamma^2}) \Gamma(1-\frac{2\beta}{\gamma}) \left(  g_{\frac{2}{\gamma}}(\sigma_2) - g_{\frac{2}{\gamma}}( \sigma_1-\frac{\beta}{2} ) \right).
\end{align}
Similarly, by working with auxiliary function $\tilde{H}_{\chi}(t)$ yields the other desired shift equation:
\begin{align*}
\frac{R(\beta+\frac{2}{\gamma},\sigma_1-\frac{1}{\gamma},\sigma_2)}{R(\beta+\frac{4}{\gamma}, \sigma_1,\sigma_2)} = c_{\frac{2}{\gamma}}(\gamma)  \Gamma(-1+\frac{2\beta}{\gamma}) \Gamma(1-\frac{2\beta}{\gamma}-\frac{4}{\gamma^2}) \left( g_{\frac{2}{\gamma}}(\sigma_2) - g_{\frac{2}{\gamma}}(\sigma_1+\frac{\beta}{2}) \right). \quad \quad  \qedhere 
\end{align*} 
\end{proof}

\begin{proof}[Proof of Theorem~\ref{full_shift_R}]
The statement on the meromorphic property of $R$ is proved in Lemma \ref{analycity_R}. The shift equations for $\chi = \frac{\gamma}{2}$ is proved in Lemma~\ref{Shift1_H_proba}. Here~\eqref{equation R2} is equivalent to~\eqref{eq:R-shift2} after a parameter change and  the same holds for ~\eqref{equation R1} and~\eqref{eq:R-shift1}. The shift equations for $\chi = \frac{2}{\gamma}$ is proved in Lemma \ref{lem:sec4.5}.
\end{proof}
 
\subsection{The $\frac{2}{\gamma}$-shift equations for $H$}\label{subsec:shift2_H}
\begin{proof}[Proof of Theorem~\ref{full_shift_H}] The statement on the meromorphic property of $H$ is proved in Lemma \ref{analycity_H} and the shift equations for $\chi = \frac{\gamma}{2}$ is proved in Lemma~\ref{Shift1_H_proba}. To treat the case $\chi = \frac2{\gamma}$, we recall the connection formulas~\eqref{eq:final_shift_H} and~\eqref{eq:final_shift_H2} in the proof of Lemma~\ref{Shift1_H_proba}.  For $\chi = \frac{2}{\gamma}$, we identify $B_1, C_1, C_2^+, \tilde{B}_2^-, \tilde{C}_1, \tilde{C}_2^+ $ as in the proof of Lemma~\ref{Shift1_H_proba}. By the OPE Lemma \ref{lem_reflection_ope2}, we get  expressions for $C_2^+, \tilde{B}_2^-, \tilde{C}_2^+ $  with both $H$ and $R$ involved.  For instance:
$$C_2^+ = R(\beta_1, \sigma_1 - \frac{1}{\gamma}, \sigma_2-\frac{1}{\gamma}) H
\begin{pmatrix}
2Q - \beta_1 -\frac{2}{\gamma}, \beta_2, \beta_3 \\
\sigma_1 - \frac{1}{\gamma},  \sigma_2,   \sigma_3 
\end{pmatrix}.$$
The goal is to get expressions for $C_2^+, \tilde{B}_2^-, \tilde{C}_2^+$ that only involve $ H$ without $R$. To do this for $C_2^+$, we apply the shift equation on $R$ in Theorem \ref{full_shift_R} to simplify the ratio $\frac{R(\beta_1, \sigma_1-\frac{1}{\gamma}, \sigma_2-\frac{1}{\gamma})}{R(\beta_1+\frac{2}{\gamma}, \sigma_1-\frac{1}{\gamma}, \sigma_2)}$ and then the reflection principle given by equation \eqref{eq:reflection_id}. The same strategy can be applied to $\tilde{C}_2^+$ and $\tilde{B}_2^-$. This allows us to write $C_2^+,	\tilde{C}_2^+$, and $ \tilde{B}_2^- $   respectively as in~\eqref{eq:C2+}, \eqref{C2+tilde}, and ~\eqref{eq:B2-tilde} with $\chi=\frac{2}{\gamma}$.  Putting all these into \eqref{eq:final_shift_H} and \eqref{eq:final_shift_H2} proves the desired shift equations as in the  $\chi = \frac{\gamma}{2}$ case.
\end{proof}

\section{Proof of the main theorems}\label{sec:main}

In this section we successively prove Theorem~\ref{main_th2}, Theorem~\ref{thm:H} and Theorem~\ref{thm:G}. The first two rely on the shift equations of Theorems~\ref{full_shift_H} and~\ref{full_shift_R}. The last proof uses Theorem~\ref{main_th2} combined with an expression of $G_{\mu_B}(\alpha, \beta)$ from \cite{wu2023sle} in terms of  the two pointed quantum disk, whose area and boundary length distribution is prescribed by Theorem~\ref{main_th2}.

\subsection{Proof of $R = R_{\mathrm{FZZ}}$}

\begin{proof}[Proof of Theorem~\ref{main_th2}] 
First assume $\gamma^2 \not \in \mathbb Q$.
By combining both shift equations of Theorem~\ref{full_shift_R} we obtain for $ \chi \in \{ \frac{\gamma}{2}, \frac{2}{\gamma} \} $
\begin{align}
	&\frac{R(\beta_1, \sigma_1, \sigma_2)}{R(\beta_1 + 2 \chi, \sigma_1, \sigma_2)} =  \tilde{c}_{\chi}(\gamma) \Gamma(-1+ \chi \beta_1 - \chi^2 ) \Gamma(1- \chi \beta_1- \chi^2) \Gamma(1- \chi \beta_1 ) \Gamma(-1+ \chi \beta_1 )  \\
	&\times 4\sin(  \pi \chi (\frac{\beta}{2}-\sigma_1-\sigma_2+Q)) \sin( \pi \chi (\frac{\beta}{2}+\sigma_1+\sigma_2-Q)) \sin( \pi \chi (\frac{\beta}{2}+\sigma_2-\sigma_1)) \sin( \pi \chi  (\frac{\beta}{2}+\sigma_1-\sigma_2)), \nonumber
	\end{align}
where $\tilde{c}_{\frac{\gamma}{2}}(\gamma) = \frac{1}{\sin(\pi \frac{\gamma^2}{4}) \Gamma(- \frac{\gamma^2}{4})^2 }$ and where $\tilde{c}_{\frac{2}{\gamma}}(\gamma) = \frac{16}{ \gamma^4} \pi^{\frac{8}{\gamma^2} -2}   \Gamma(1 - \frac{\gamma^2}{4})^{- \frac{8}{\gamma^2}} \sin(\frac{\pi \gamma^2}{4})^{- \frac{4}{\gamma^2}}$.

Consider now the ratio $f(\beta) = \frac{R(\beta)}{R_{\mathrm{FZZ}}(\beta)}$. By using the shift equations \eqref{eq:shift_G1}, \eqref{eq:shift_G2}, \eqref{eq:shift_S1} satisfied by $\Gamma_{\frac{\gamma}{2}}$ and $S_{\frac{\gamma}{2}}$, one can show that $R_{\mathrm{FZZ}}(\beta)$ satisfies the same shift equations for both values of $\chi$ satisfied by $R$. From this we obtain that:
\begin{align}\label{eq:f-shift}
f(\beta + \gamma) = f(\beta) \quad \textrm{and}\quad f(\beta + \frac{4}{\gamma}) = f(\beta).
\end{align} 
Since $\beta \rightarrow f(\beta)$ is meromorphic and $\gamma^2 \notin \mathbb{Q}$, 
\eqref{eq:f-shift} implies $f$ is constant on a dense subset of $\mathbb R$ and thus is constant everywhere.
One can then show $f(\beta) = 1$ by using the fact that $R (Q) = R_{\mathrm{FZZ}}(Q) = -1$. The fact that $ R_{\mathrm{FZZ}}(Q) = -1$ can be checked directly on the exact formula simply by using the relation $S_{\frac{\gamma}{2}}(Q -x) = (S_{\frac{\gamma}{2}}(x))^{-1}$. To see that $R(Q) = -1 $, one can  combine the result of Lemma \ref{lem:R_gamma} with the shift equations on $R$ given by Theorem \ref{full_shift_R}.  Lastly, the case $\gamma^2 \in \mathbb{Q}$ is recovered by a  continuity argument (as in \cite{DOZZ_proof}).
\end{proof}

\subsection{Proof of $H = H_{\mathrm{PT}}$}\label{sec:proof_H}

We first perform a residue computation for $H$ and $\HPT$ in the following lemma, and then use it in a uniqueness argument to prove Theorem \ref{thm:H}.

\begin{lemma}\label{lem:res_H}
Set $\overline{\beta} = \beta_1 + \beta_2 + \beta_3$. For $\beta_i, \sigma_i$ in the domain of Lemma \ref{analycity_H} where $H$ has been analytically extended, $H$ satisfies the following properties:
\begin{align*}
&\lim_{\beta_1\to 2Q-\beta_2-\beta_3}(\frac{\bar{\beta}}{2}-Q) H
\begin{pmatrix}
\beta_1 , \beta_2, \beta_3 \\
\sigma_1,  \sigma_2,   \sigma_3 
\end{pmatrix} = 1,\\
&\lim_{\beta_1\to 2Q-\beta_2-\beta_3 - \gamma}(\frac{\bar{\beta}}{2}-Q + \frac{\gamma}{2}) H
\begin{pmatrix}
\beta_1 , \beta_2, \beta_3 \\
\sigma_1,  \sigma_2,   \sigma_3 
\end{pmatrix} =  - \frac{1}{\pi} \sqrt{ \frac{1}{\sin \frac{\pi \gamma^2}{4}}}\Gamma(1 - \frac{\gamma \beta_2}{2}) \Gamma(1 - \frac{\gamma \beta_3}{2}) \Gamma(\frac{\gamma \beta_2}{2} + \frac{\gamma \beta_3}{2} - 1)   \\
& \times \left( \cos(\pi \gamma( \sigma_1 - \frac{Q}{2})) \sin(\frac{\pi \gamma \beta_2}{2}) + \cos(\pi \gamma( \sigma_2 - \frac{Q}{2})) \sin(\frac{\pi \gamma \beta_3}{2}) - \cos(\pi \gamma ( \sigma_3 - \frac{Q}{2})) \sin(\frac{\pi \gamma (\beta_2 + \beta_3)}{2}) \right).
\end{align*}
The same limits hold with $H$ replaced with $\HPT$.
\end{lemma}

\begin{proof}
We will focus on checking the second limit, the first one following the same steps but with much simpler computations.
To study the limit $\beta_1\to 2Q-\beta_2-\beta_3 - \gamma$, we assume that 
\begin{equation}\label{eq:para-range}
    \textrm{$\beta_i < \frac{2}{\gamma} -\eta $ for $i = 1,2,3$ and $Q - \frac{1}{2} \sum_i \beta_i < \frac{\gamma}{2} +\eta$ for a small $\eta>0$.}
\end{equation}
For $\eta$ small enough, \eqref{eq:para-range} implies the condition of  Definition~\ref{def:H-general}: $\beta_i < Q$ and $Q - \frac{1}{2} \sum \beta_i < \gamma \wedge \frac{2}{\gamma} \wedge \min_i (Q -\beta_i)$. 
Hence we can represent $H$ using the probabilistic expression. 

Now recall from Lemma \ref{lem:def_H_trunc} that $s(s+\frac\gamma2)H^{(\beta_1 , \beta_2, \beta_3)}_{(\mu_1,\mu_2, \mu_3)} = 	\hat H^{(\beta_1 , \beta_2, \beta_3)}_{(\mu_1,\mu_2, \mu_3)}$ where
\begin{equation*}
		\hat H^{(\beta_1 , \beta_2, \beta_3)}_{(\mu_1,\mu_2, \mu_3)} = \int (\gamma (s+\frac\gamma2) A + \frac{\gamma^2}{2} A (\sum_i \mu_i L_i) + \frac{\gamma^2}{4}  (\sum_i \mu_i L_i)^2 ) e^{-A - \sum_i \mu_i L_i} \, \LF_\bbH^{(\beta_1,0), (\beta_2,1), (\beta_3, \infty)} (\mathrm d \phi),
	\end{equation*}
and where $s = \frac{1}{2} \sum_i \beta_i - Q$,  $A = \cA_\phi(\bbH), L_1 = \cL_\phi(-\infty,0), L_2 = \cL_\phi(0,1)$ and $L_3 = \cL_\phi(1,\infty)$. Note the limit we are interested in corresponds to $s \rightarrow - \frac{\gamma}{2}$. We can now write out explicitly the integration over the zero mode. For this let $h$ be a Gaussian free field and $\wt h := h - 2Q \log |\cdot|_+ + \sum \frac{\beta_i}2 G_\bbH(\cdot, s_i)$, where $(s_1,s_2,s_3) = (0,1,\infty)$. Write $\wt A = \cA_{\wt h}(\bbH)$ and $\wt L_1 = \mu\cL_{\wt h}(-\infty, 0), \wt L_2 = \cL_{\wt h}(0,1), \wt L_3 = \cL_{\wt h}(1, +\infty)$, $\wt L = \sum_i \wt L_i$. Since $A = e^{\gamma c} \wt A$ and $L_i = e^{\gamma c/2} \wt L_i$:
\begin{align*}
\hat H^{(\beta_1 , \beta_2, \beta_3)}_{(\mu_1,\mu_2, \mu_3)} &= \int_{\mathbb{R}} dc e^{s c} \mathbb{E} \left[ \left(\gamma (s+\frac\gamma2)e^{\gamma c} \wt A + \frac{\gamma^2}{2} e^{\frac{3\gamma c}{2} } \wt A \wt L + \frac{\gamma^2}{4} e^{\gamma c}  \wt L^2 \right) e^{- e^{\gamma c} \wt A - e^{\frac{\gamma c}{2}} \wt L} \right] \\
&\underset{s \rightarrow -\frac{\gamma}{2}}{\rightarrow} \frac{\gamma}{2} \int_{\mathbb{R}} dc \mathbb{E} \left[ \wt L \left( \gamma e^{\gamma c} \wt A + \frac{\gamma}{2} e^{\frac{\gamma c}{2}}  \wt L \right) e^{- e^{\gamma c} \wt A - e^{\frac{\gamma c}{2}} \wt L} \right] = \frac{\gamma}{2} \int_0^{+ \infty} du \mathbb{E}[\wt L e^{-u}]  =  \frac{\gamma}{2} \mathbb{E}[\wt L].
\end{align*}
In the last line we have used the change of variable $u = e^{\gamma c} \wt A +  e^{\frac{\gamma c}{2}} \wt L$. The expectation $\mathbb{E}[\wt L]$ of the boundary GMC measure simply reduces to an integral over $\mathbb{R}$:
\begin{align*}
& \mathbb{E}[\wt L] = \mu_1 \int_{-\infty}^0 \frac{dx}{|x|^{\frac{\gamma \beta_1}{2}} |x-1|^{\frac{\gamma \beta_2}{2}} } +  \mu_2 \int_0^1 \frac{dx}{|x|^{\frac{\gamma \beta_1}{2}} |x-1|^{\frac{\gamma \beta_2}{2}} }
+ \mu_3 \int_1^{\infty} \frac{dx}{|x|^{\frac{\gamma \beta_1}{2}} |x-1|^{\frac{\gamma \beta_2}{2}} } \\
& = \mu_1 \frac{\Gamma(\frac{\gamma \beta_1}{2} + \frac{\gamma \beta_2}{2} -1) \Gamma(1 - \frac{\gamma \beta_1}{2})}{\Gamma( \frac{\gamma \beta_2}{2})}
+
\mu_2 \frac{\Gamma(1 - \frac{\gamma \beta_1}{2}) \Gamma(1 - \frac{\gamma \beta_2}{2})}{\Gamma(2 - \frac{\gamma \beta_1}{2} - \frac{\gamma \beta_2}{2})} + \mu_3 \frac{\Gamma(\frac{\gamma \beta_1}{2} + \frac{\gamma \beta_2}{2} -1) \Gamma(1 - \frac{\gamma \beta_2}{2})}{\Gamma( \frac{\gamma \beta_1}{2})}\\
& =\frac{1}{\pi} \sqrt{ \frac{1}{\sin \frac{\pi \gamma^2}{4}}} \Gamma(1 - \frac{\gamma \beta_2}{2}) \Gamma(1 - \frac{\gamma \beta_3}{2}) \Gamma(\frac{\gamma \beta_2}{2} + \frac{\gamma \beta_3}{2} - 1)\\
& \times \left( \cos(\pi \gamma( \sigma_1 - \frac{Q}{2})) \sin(\frac{\pi \gamma \beta_2}{2}) + \cos(\pi \gamma( \sigma_2 - \frac{Q}{2})) \sin(\frac{\pi \gamma \beta_3}{2}) - \cos(\pi \gamma ( \sigma_3 - \frac{Q}{2})) \sin(\frac{\pi \gamma (\beta_2 + \beta_3)}{2}) \right).
\end{align*}
Note that the integral above converges provided $\beta_1, \beta_2 < \frac{2}{\gamma}$, $\beta_1 + \beta_2 > \frac{2}{\gamma}$, which are implied by~\eqref{eq:para-range}. Since $\lim_{s \to -\frac\gamma2} (s + \frac\gamma2) H^{(\beta_1 , \beta_2, \beta_3)}_{(\mu_1,\mu_2, \mu_3)} = \lim_{s \to -\frac\gamma2} s^{-1} \hat H^{(\beta_1 , \beta_2, \beta_3)}_{(\mu_1,\mu_2, \mu_3)} = - \mathbb E[\wt L]$, we conclude that the second limit holds under the condition~\eqref{eq:para-range}.
Using the Cauchy residue formula as in the proof of Lemma~\ref{lem:HlimR_ext}, the condition~\eqref{eq:para-range} can be relaxed to the domain in Lemma \ref{analycity_H} where $H$ is analytically extended.
The first limit is simpler, which uses the GMC expression for $H$ with one less truncation. 

Lastly the claim for $\HPT$ can be deduced from Lemma \ref{lem:secH2} which computes the desired limits for the contour integral appearing in $\HPT$, namely the function  $\mathcal{J}_{\mathrm{PT}}$ defined by equation \eqref{J_PT_app}. Computing the limit of the explicit prefactor relating  $\HPT$ and $\mathcal{J}_{\mathrm{PT}}$ allows to deduce the desired claim for $\HPT$.
\end{proof}

\begin{proof}[Proof of Theorem \ref{thm:H}] First assume $\gamma^2 \not \in \mathbb Q$. It is less straightforward to see than in the case of $R$ that the shift equations on $H$ completely determine its value, since they contain three terms instead of two. Instead of working directly with $H, \HPT$, we will match the following two functions:
\begin{align*}
&\mathcal{J}_{\mathrm{PT}}
\begin{pmatrix}
\beta_1 , \beta_2, \beta_3 \\
\sigma_1,  \sigma_2,   \sigma_3 
\end{pmatrix} := \int_{\mathcal{C}} \frac{S_{\frac{\gamma}{2}}(\frac{Q - \beta_2}{2} + \sigma_3 \pm ( \frac{Q}{2} - \sigma_2) +r)  S_{\frac{\gamma}{2}}(\frac{Q}{2} \pm \frac{Q - \beta_3}{2}  +\sigma_3-\sigma_1+r)}{S_{\frac{\gamma}{2}}(\frac{3Q}{2} \pm \frac{Q - \beta_1}{2}-\frac{\beta_2}{2}+\sigma_3-\sigma_1+r) S_{\frac{\gamma}{2}}(2\sigma_3 + r) S_{\frac{\gamma}{2}}(Q+r)} \frac{dr}{\ii},\\
& \mathcal{J}_{H}
\begin{pmatrix}
\beta_1 , \beta_2, \beta_3 \\
\sigma_1,  \sigma_2,   \sigma_3 
\end{pmatrix} := H
\begin{pmatrix}
\beta_1 , \beta_2, \beta_3 \\
\sigma_1,  \sigma_2,   \sigma_3 
\end{pmatrix} \times \frac{1}{2\pi}  \left(\frac{\pi  (\frac{\gamma}{2})^{2-\frac{\gamma^2}{2}} \Gamma(\frac{\gamma^2}{4})}{\Gamma(1-\frac{\gamma^2}{4})} \right)^{\frac{ (\beta_1 + \beta_2 + \beta_3)- 2 Q}{2\gamma}}\\
& \times
  \frac{S_{\frac{\gamma}{2}}(\frac{\beta_3}{2} - \sigma_1 + \frac{Q}{2} \pm ( \frac{Q}{2} - \sigma_3 ))S_{\frac{\gamma}{2}}(\frac{\beta_1}{2} + \sigma_1 - \frac{Q}{2} \pm ( \frac{Q}{2} - \sigma_2 )) \Gamma_{\frac{\gamma}{2}}(Q) \prod_{i=1}^3\Gamma_{\frac{\gamma}{2}}(Q-\beta_i)}{\Gamma_{\frac{\gamma}{2}}(Q - \frac{\beta_2}{2} \pm \frac{Q - \beta_1}{2} \pm \frac{Q - \beta_3}{2} ) }.
\end{align*}
Namely we have removed the prefactor in front of the contour integral of $\HPT$ and done the same multiplication for $H$. Now thanks to Theorem~\ref{full_shift_H} and Lemma~\ref{lem:secH1}, both $\mathcal{J}_{H}$ and $\mathcal{J}_{\mathrm{PT}}$ obey the same shift equations which have the following form:  for $\chi = \frac{\gamma}{2}$ or $\frac{2}{\gamma}$,
\begin{align}
& \mathcal{J}
\begin{pmatrix}
\beta_1 , \beta_2 - \chi, \beta_3 \\
\sigma_1,  \sigma_2,   \sigma_3 
\end{pmatrix} = f_1(\beta_1, \beta_2) \mathcal{J}
\begin{pmatrix}
\beta_1 - \chi , \beta_2, \beta_3 \\
\sigma_1,  \sigma_2 + \frac{\chi}{2},   \sigma_3 
\end{pmatrix} + f_2(\beta_1, \beta_2) \mathcal{J}
\begin{pmatrix}
\beta_1 +\chi, \beta_2, \beta_3 \\
\sigma_1,  \sigma_2 + \frac{\chi}{2},   \sigma_3 
\end{pmatrix}, \label{shift_3pt_sec4_1} \\ 
 & \mathcal{J}
\begin{pmatrix}
\beta_1 , \beta_2 +\chi, \beta_3 \\
\sigma_1,  \sigma_2 + \frac{\chi}{2},   \sigma_3 
\end{pmatrix} = f_3(\beta_1, \beta_2) \mathcal{J}
\begin{pmatrix}
\beta_1 - \chi, \beta_2, \beta_3 \\
\sigma_1,  \sigma_2,   \sigma_3 
\end{pmatrix}   + f_4(\beta_1, \beta_2) \mathcal{J}
\begin{pmatrix}
\beta_1 +\chi, \beta_2, \beta_3 \\
\sigma_1,  \sigma_2,   \sigma_3 
\end{pmatrix}, \label{shift_3pt_sec4_2}
\end{align}
where $\mathcal{J} = \mathcal{J}_{H}$ or $\mathcal{J}_{\mathrm{PT}}$ and the functions $f_1, f_2, f_3, f_4$ are given by:
\begin{align*}
&f_{1}(\beta_1, \beta_2) = \frac{1}{2 \sin(\pi \chi(\beta_1 - \chi))},\\
&f_{2}(\beta_1, \beta_2) = \frac{ 2 \sin(\pi\chi(\frac{\beta_1}{2}+\sigma_1+\sigma_2-Q)) \sin(\pi\chi(\frac{\beta_1}{2}-\sigma_1+\sigma_2))}{\sin(\pi\chi(\chi-\beta_1)) },\\
& f_{3}(\beta_1, \beta_2) = \frac{\sin(\pi\chi( \frac{3\chi}{2} - \frac{\bar{\beta}}{2})) \sin(\frac{\pi\chi}{2}(\beta_1 - \chi+\beta_2-\beta_3))}{2 \sin(\pi \chi (\beta_1 -\chi)) \sin(\pi\chi(\frac{\beta_2}{2}+\sigma_2+\sigma_3-Q)) \sin(\pi\chi(\frac{\beta_2}{2}+\sigma_2-\sigma_3)) },\\
& f_{4}(\beta_1, \beta_2) = \frac{2 \sin(\frac{\pi\chi}{2}(\beta_3 - \chi \pm( \beta_1-\beta_2)  ))  \sin(\pi\chi(\frac{\beta_1 - \chi}{2}-\sigma_1-\sigma_2+Q)) \sin(\pi\chi(\frac{\beta_1- \chi}{2}+\sigma_1-\sigma_2))  }{\sin(\pi \chi (\beta_1-\chi)) \sin(\pi\chi(\frac{\beta_2}{2}+\sigma_2+\sigma_3-Q)) \sin(\pi\chi(\frac{\beta_2}{2}+\sigma_2-\sigma_3)) }.
\end{align*}

The advantage of this form of the shift equations is that the four functions $f_i$ all contain only sine functions (and no more gamma functions) and therefore will have a periodicity property which will be useful.
Now consider the shift equation \eqref{shift_3pt_sec4_1} with the parameter replacement $\beta_1 \rightarrow \beta_1 +  \chi $,  $\beta_2 \rightarrow \beta_2 + \chi $:
\begin{align*}
&\mathcal{J}
\begin{pmatrix}
\beta_1 + \chi, \beta_2, \beta_3 \\
\sigma_1,  \sigma_2,   \sigma_3 
\end{pmatrix} = f_1(\beta_1 + \chi, \beta_2 + \chi) \mathcal{J}
\begin{pmatrix}
\beta_1 , \beta_2 +\chi, \beta_3 \\
\sigma_1,  \sigma_2 + \frac{\chi}{2},   \sigma_3 
\end{pmatrix} + f_2(\beta_1 + \chi, \beta_2 + \chi) \mathcal{J}
\begin{pmatrix}
\beta_1 +2 \chi, \beta_2 +\chi, \beta_3 \\
\sigma_1,  \sigma_2 + \frac{\chi}{2},   \sigma_3 
\end{pmatrix}.
\end{align*}
In this equation the two $\mathcal{J}$ functions appearing on the right hand side can be expressed in terms of $\mathcal{J}$ functions involving only shifts on $\beta_1$ using twice equation \eqref{shift_3pt_sec4_2}, once as it is and once with the parameter replacement $\beta_1 \rightarrow \beta_1 + 2 \chi $. Performing one more global parameter replacement $\beta_1$ to $\beta_1 + \chi$, we get the following shift equation which is only on the variable $\beta_1$:
\begin{align}\label{eq:key-shift}
0 &=   f_1(\beta_1 + 2 \chi, \beta_2 + \chi )  f_3(\beta_1 + \chi, \beta_2) \mathcal{J}
\begin{pmatrix}
\beta_1 , \beta_2, \beta_3 \\
\sigma_1,  \sigma_2,   \sigma_3 
\end{pmatrix} \\ \nonumber
&  + \left( -1+  f_1(\beta_1 + 2 \chi, \beta_2 + \chi )  f_4(\beta_1 + \chi, \beta_2) +  f_2(\beta_1 + 2 \chi, \beta_2 + \chi )  f_3(\beta_1 + 3 \chi, \beta_2) \right) \mathcal{J}
\begin{pmatrix}
\beta_1 +2 \chi, \beta_2, \beta_3 \\
\sigma_1,  \sigma_2,   \sigma_3 
\end{pmatrix}   \\ \nonumber
& + f_2(\beta_1 + 2 \chi , \beta_2 + \chi )  f_4(\beta_1 + 3 \chi, \beta_2) \mathcal{J}
\begin{pmatrix}
\beta_1 +4 \chi, \beta_2, \beta_3 \\
\sigma_1,  \sigma_2,   \sigma_3 
\end{pmatrix}.
\end{align}
Now we fix the five parameters $\beta_2, \beta_3, \sigma_1, \sigma_2, \sigma_3$ to some generic values and view all of the functions below as functions of the single parameter $\beta_1$. To lighten notations we write simply $\mathcal{J}(\beta_1)$ . 
The above shift equations~\eqref{eq:key-shift} can be put into the form:
\begin{align}\label{eq:shift_J}
 \mathcal{J}(\beta_1 + 4 \chi) + a_{\chi}(\beta_1) \mathcal{J}(\beta_1 + 2\chi) + b_{\chi}(\beta_1)  \mathcal{J}(\beta_1) =0,
\end{align}
where 
\begin{align*}
&a_{\chi}(\beta_1) = \frac{ - 1+  f_1(\beta_1 + 2 \chi, \beta_2 + \chi )  f_4(\beta_1 + \chi, \beta_2) +  f_2(\beta_1 + 2 \chi, \beta_2 + \chi )  f_3(\beta_1 + 3 \chi, \beta_2) }{f_2(\beta_1 + 2 \chi , \beta_2 + \chi )  f_4(\beta_1 + 3 \chi, \beta_2)},\\
& b_{\chi}(\beta_1) = \frac{f_1(\beta_1 + 2 \chi, \beta_2 + \chi )  f_3(\beta_1 + \chi, \beta_2)}{f_2(\beta_1 + 2 \chi , \beta_2 + \chi )  f_4(\beta_1 + 3 \chi, \beta_2)}.
\end{align*}
To complete the proof, we will first show that $\mathcal J_{H}(\beta_1)=\mathcal J_{\mathrm{PT}}(\beta_1)$ agree for four special values of $\beta_1$ (as in~\eqref{eq:match_res}), then use~\eqref{eq:shift_J} to extend to all values of $\beta_1$. Consider the matrices:
\begin{align*}
M_{1}(\beta_1) = \begin{pmatrix}
\mathcal{J}_H(\beta_1) & \mathcal{J}_H(\beta_1 - \frac{4}{\gamma})  \\
\mathcal{J}_H(\beta_1 - \gamma )& \mathcal{J}_H(\beta_1 - 2Q)
\end{pmatrix}, \quad \textrm{and}\quad M_{2}(\beta_1) = \begin{pmatrix}
\mathcal{J}_{\mathrm{PT}}(\beta_1) & \mathcal{J}_{\mathrm{PT}}(\beta_1 - \frac{4}{\gamma} )  \\
\mathcal{J}_{\mathrm{PT}}(\beta_1 - \gamma )& \mathcal{J}_{\mathrm{PT}}(\beta_1 - 2Q)
\end{pmatrix}. 
\end{align*}
Set $\beta_0 := 2Q - \beta_2 - \beta_3$. We will derive the consequence of Lemma \ref{lem:res_H} for these matrices. Note that both functions $\mathcal{J}_H, \mathcal{J}_{\mathrm{PT}}$ obey an exact reflection identity, namely $\mathcal{J}_H(\beta_1) = \mathcal{J}_H(2Q - \beta_1)$ and $\mathcal{J}_{\mathrm{PT}}(\beta_1) = \mathcal{J}_{\mathrm{PT}}(2Q - \beta_1)$. For $\mathcal{J}_{\mathrm{PT}}$ this fact is obvious from the definition. For $\mathcal{J}_H$ it is a consequence of Lemma \ref{reflection_H}. Now Lemma \ref{lem:res_H} implies that:
\begin{align*}
&\lim_{\beta_1\to \beta_0}(\beta_1 - \beta_0)\mathcal{J}_H(\beta_1) = \lim_{\beta_1\to \beta_0}(\beta_1 - \beta_0)\mathcal{J}_{\mathrm{PT}}(\beta_1),\\
&\lim_{\beta_1\to \beta_0}(\beta_1 - \beta_0)\mathcal{J}_H(\beta_1 - \gamma) = \lim_{\beta_1\to \beta_0}(\beta_1 - \beta_0)\mathcal{J}_{\mathrm{PT}}(\beta_1 - \gamma).
\end{align*}
Since $\mathcal{J}_H(\beta_1) = \mathcal{J}_H(2Q - \beta_1)$ we can reformulate the first limit as:
\begin{align*}
\lim_{\beta_1\to \beta_0}(\beta_1 - \beta_0)\mathcal{J}_H(\beta_1) = \lim_{\beta_1\to \beta_0}(\beta_1 - \beta_0)\mathcal{J}_H(2Q - \beta_1) = \lim_{\beta_1\to - \beta_2 - \beta_3}(\beta_1 + \beta_2 + \beta_3)\mathcal{J}_H(- \beta_1).  
\end{align*}
Since everything we derive works for generic values of $\beta_2, \beta_3$, we can replace $(\beta_2, \beta_3)$ by $(- \beta_2, - \beta_3)$. Therefore we obtain the equality of limits:
\begin{align*}
&\lim_{\beta_1\to \beta_2 + \beta_3}(\beta_1 - \beta_2 - \beta_3)\mathcal{J}_H(-\beta_1) = \lim_{\beta_1\to \beta_2 + \beta_3}(\beta_1 - \beta_2 - \beta_3)\mathcal{J}_{\mathrm{PT}}(-\beta_1)\\
\Leftrightarrow &\lim_{\beta_1\to \beta_0}(\beta_1 - \beta_0)\mathcal{J}_H(\beta_1 - 2Q) = \lim_{\beta_1\to \beta_0}(\beta_1 - \beta_0)\mathcal{J}_{\mathrm{PT}}(\beta_1 - 2Q).
\end{align*}
By applying a similar argument to $\mathcal{J}_H(\beta_1 - \gamma)$ we can show that the limits also match for $\mathcal{J}_H(\beta_1 - \frac{4}{\gamma})$. Therefore the conclusion of the above is that:
\begin{equation}\label{eq:match_res}
\lim_{\beta_1\to \beta_0}(\beta_1 - \beta_0) M_1(\beta_1) = \lim_{\beta_1\to \beta_0}(\beta_1 - \beta_0) M_2(\beta_1).
\end{equation}
Let us write down the shift equations satisfied by the matrices $M_1, M_2$. One has
\begin{align*}
M_1(\beta_1 + \gamma) = A_{\frac{\gamma}{2}}(\beta_1) M_1(\beta_1), \quad M_1(\beta_1 + \frac{4}{\gamma}) =  M_1(\beta_1) A_{\frac{2}{\gamma}}(\beta_1)^\top, \quad A_{\chi}(\beta_1) := \begin{pmatrix}
- a_{\chi}(\beta_1) & - b_{\chi}(\beta_1)  \\
1 & 0
\end{pmatrix},
\end{align*}
and the same relations for $M_2$.
The fact that these first order shift equations on the matrices $M_1, M_2$ hold uses \eqref{eq:shift_J} combined with the fact that the function $a_{\chi}(\beta_1), b_{\chi}(\beta_1)$ are both $2/\chi$-periodic for both values of $\chi$.
Consider the determinant: $$D_2(\beta_1) = \det M_2(\beta_1) = \mathcal{J}_{\mathrm{PT}}(\beta_1) \mathcal{J}_{\mathrm{PT}}(\beta_1- 2 Q)- \mathcal{J}_{\mathrm{PT}}(\beta_1 -\gamma) \mathcal{J}_{\mathrm{PT}}(\beta_1 - \frac{4}{\gamma}).$$
From the result of Lemma \ref{lem:secH2} it is easy to see that the residue of $D_2(\beta_1)$ at $\beta_1 = \beta_0$ is not zero. More precisely the residues of $\mathcal{J}_{\mathrm{PT}}(\beta_1)$ and  $\mathcal{J}_{\mathrm{PT}}(\beta_1 -\gamma)$ are related by an explicit meromorphic function written in Lemma \ref{lem:secH2}, and the residues of $\mathcal{J}_{\mathrm{PT}}(\beta_1 - 2Q)$ and  $\mathcal{J}_{\mathrm{PT}}(\beta_1 -\frac{4}{\gamma})$ are related by the same function with $\beta_2, \beta_3$ replaced by $-\beta_2, - \beta_3$.

Since $D_2(\beta_1)$ is a nonzero meromorphic function, it must have isolated and at most countably many zeros. Therefore the matrix inverse $M_2(\beta_1)^{-1}$ is a well-defined 2-by-2 matrix whose entries are meromorphic functions of $\beta_1$.
Now consider $M(\beta_1) := M_1(\beta_1) M_2(\beta_1)^{-1}$. Then we have:
\begin{equation}\label{eq:M-shift}
    M(\beta_1 + \gamma ) = A_{\frac{\gamma}{2}}(\beta_1) M(\beta_1) A_{\frac{\gamma}{2}}(\beta_1)^{-1}, \quad \textrm{and}\quad  M(\beta_1 + \frac{4}{\gamma}) =  M(\beta_1).
\end{equation}
Similarly as for $M_2$, since $\det A_{\frac{\gamma}{2}}(\beta_1) = b_{\frac{\gamma}{2}}(\beta_1)$   is nonzero  meromorphic function,  the matrix $A_{\frac{\gamma}{2}}(\beta_1)^{-1}$ is a well-defined 2-by-2 matrix with meromorphic entries. Thanks to equation \eqref{eq:match_res}, we know that $M(\beta_1)$ is the identity matrix for $\beta_1 = \beta_0$. By the same standard argument as below~\eqref{eq:f-shift}, when $ \gamma^2 \notin \mathbb{Q}$, the two shift equations in~\eqref{eq:M-shift} imply $M(\beta_1)$ is the identity matrix for all $\beta_1$. The same holds for $ \gamma^2 \in \mathbb{Q}$ by continuity. Therefore $\mathcal{J}_H = \mathcal{J}_{\mathrm{PT}}$ and hence $H = \HPT$.
\end{proof}

\subsection{Proof of $G = G_{\mathrm{Hos}}$}\label{sec:proof_G}

We first perform all the computations up to a global constant  depending on $\gamma, \beta$ but not on $\alpha$ or $\mu_B$.  This will determine $G_{\mu_B}(\alpha,\beta)$ up to a $(\gamma, \beta)$-dependent factor, which we pin down at the end. We use $c_{\gamma, \beta}$ and $c_{\gamma}$ to denote $(\gamma, \beta)$-dependent and $\gamma$-dependent constants, respectively, which can vary from line to line.

We first recall three formulas from \cite{rz-boundary}. We add a $0$ subscript to all the formulas to indicate that they all correspond to having the bulk cosmological constant set to zero, namely $\mu=0$. 
\begin{equation}
\overline{U}_0(\alpha)  =  \left(\frac{2^{-\frac{\gamma\alpha}{2}} 2\pi}{\Gamma(1-\frac{\gamma^2}{4})} \right)^{\frac{2}{\gamma}(Q-\alpha)}\Gamma(\frac{\gamma \alpha}{2} - \frac{\gamma^2}{4}).
\end{equation}

\begin{equation}
\overline{G}_0(\alpha, \beta)= c_{\gamma, \beta} \left( \frac{2^{\frac{\gamma}{2}(\frac{\beta}{2} - \alpha)} 2\pi}{\Gamma(1-\frac{\gamma^2}{4})} \right)^{\frac{2}{\gamma}(Q-\alpha-\frac{\beta}{2})} \frac{ \Gamma(\frac{\gamma\alpha}{2}+\frac{\gamma\beta}{4}-\frac{\gamma^2}{4}) \Gamma_{\frac{\gamma}{2}}(\alpha \pm \frac{\beta}{2} )  }{  \Gamma_{\frac{\gamma}{2}}(\alpha )^2 }.
\end{equation}

\begin{align}
\overline{R}_0(\beta, \mu_1, \mu_2) = c_{\gamma, \beta}  \frac{  e^{\ii \pi(\sigma_1^0+\sigma_2^0 -Q)(Q-\beta)}}{  S_{\frac{\gamma}{2}}(\frac{\beta}{2} + \sigma_2^0- \sigma_1^0)S_{\frac{\gamma}{2}}(\frac{\beta}{2} + \sigma_1^0- \sigma_2^0) }, \quad \text{where} \quad \mu_i = e^{\ii \pi \gamma (\sigma_i^0 -\frac{Q}{2})}.
\end{align} 

We prove $G = G_{\mathrm{Hos}}$ using \cite[Proposition 1.7]{wu2023sle}. To recall it properly, for a fixed  $\beta \in (0, \gamma)$, we let $\Psi_{\beta}$ be a random variable whose moment generating function satisfies for each $\alpha\in (Q-\frac{\beta}2, Q )$
\begin{align}\label{eq:CR}
& \E[\Psi_\beta^{2 \Delta_{\alpha} -2}  ]   =  \frac{\overline{G}_0(\alpha, \gamma) \overline{G}_0(\gamma, \beta)}{\overline{G}_0(\alpha, \beta) \overline{G}_0(\gamma, \gamma)} \frac{ \Gamma( \frac{4}{\gamma^2} ) \Gamma( \frac{2}{\gamma}( \frac{\beta}{2} +\gamma - Q ) )}{  \Gamma( \frac{2}{\gamma}(Q -\alpha) +1 ) \Gamma( \frac{2}{\gamma}( \frac{\beta}{2} +\alpha - Q ) )} \\ \nonumber
& \quad  \times \left( \int_0^{\infty} \mu_2^{ \frac{2}{\gamma} (Q -\alpha)  } \frac{\partial}{\partial \mu_2}  R_0( Q + \frac{\beta}{2}, \mu_2, 1)  d \mu_2 \right) \left( \int_0^{\infty} \mu_2^{ \frac{2}{\gamma} (Q -\gamma) } \frac{\partial}{\partial \mu_2} R_0( Q + \frac{\beta}{2}, \mu_2, 1)  d \mu_2 \right)^{-1}.
\end{align}

By~\cite[Proposition 7.13]{wu2023sle},  the law of $\Psi_{\beta}$ describes a quantity 
related to the so-called SLE$_\kappa(\rho)$ bubble measure (see~\cite{Zhan-bubble,wu2023sle}).  We do not recall its definition since it is not needed for the proof of $G=G_{\mathrm {Hos}}$.

For $\mu_B, \mu_2 \geq 0$, let $\mu_B = \mu_B(\sigma), \mu_2 = \mu_B(\sigma_2)$ as in \eqref{eq:mu_to_sigma}.
We also need the following measure $\mathcal M$ defined on a measurable space $(\Omega,\mathcal F)$ with three positive measurable functions $L_1,L_2,A$ satisfying 
\begin{align}\label{eq:M2W}
\mathcal M \left[ e^{- \mu_B L_1 - \mu_2 L_2 -A}  \right] = - \frac{\gamma}{\beta} R(Q -\frac{\beta}{2}, \sigma, \sigma_2)^{-1}.
\end{align} 
The measure $\mathcal M $ in~\cite{wu2023sle} is realized by the so-called thin quantum disk with weight $\frac\gamma2(\gamma - \beta)$, and   $L_1,L_2,A$ are its left/right boundary lengths and its area, respectively. Again we do not recall its definition since it is not needed; 
see~\cite[Remark 7.19]{wu2023sle} for more detail. We are now ready to state  \cite[Proposition 1.7]{wu2023sle}.
\begin{proposition} (\cite[Proposition 1.7]{wu2023sle}) \label{prop:Wu1}
Assume $\alpha, \beta$ obey the Seiberg bounds $\alpha + \frac{\beta}{2} > Q$, $\alpha < Q$, $\beta < Q$. Assume further that $\beta \in (0, \gamma)$. Then 
\begin{align}\label{eq:weldG}
G_{\mu_B}(\alpha, \beta) &= \frac{c_{\gamma, \beta}}{\E[ \Psi_{\beta}^{2 \Delta_{\alpha} -2}]}      \times \frac{2^{- \frac{\alpha^2}{2}} \overline{U}_0(\alpha)}{\Gamma(\frac{2}{\gamma}(Q -\alpha))} \left( \frac{1}{2} \sqrt{\frac{1}{\sin(\pi  \frac{\gamma^2}{4})}} \right)^{ \frac{2}{\gamma}(Q -\alpha) } \\
&\times \mathcal M \left[ e^{- \mu_B L_1 -  A} K_{\frac{2}{\gamma}(Q -\alpha)} \left( L_2 \sqrt{\frac{1}{\sin \frac{\pi \gamma^2}{4}}} \right) \right], \nonumber
\end{align}
where $K_{\nu}(x) = \int_0^{\infty} e^{-x \cosh t} \cosh(\nu t) dt$ is the modified Bessel function, $c_{\gamma, \beta}$ a constant of $\gamma, \beta$. 
\end{proposition} 

Proposition~\ref{prop:Wu1} was proved using the conformal welding result \cite[Theorem 1.1]{wu2023sle}, which itself builds on our earlier work \cite{ARS-FZZ}. For the range of $(\alpha,\beta)$ specified in Proposition~\ref{prop:Wu1}, the proof of the conformal welding result was also reviewed in~\cite[Appendix C.1]{backbone}.

We now derive $G = G_{\mathrm{Hos}}$ from Proposition~\ref{prop:Wu1}. We first  simplify the expression of
$\E[\Psi_{\beta}^{2 \Delta_{\alpha} -2}]$.

\begin{lemma}\label{lem:CR}
In the setting of Proposition~\ref{prop:Wu1}, we have
\begin{align*}
\E[\Psi_{\beta}^{2 \Delta_{\alpha} -2}  ] =   \frac{ c_{\gamma, \beta} \Gamma(\frac{\gamma \alpha}{2} - \frac{\gamma^2}{4}) \Gamma_{\frac{\gamma}{2}}(\alpha ) }{ \Gamma(\frac{\gamma\alpha}{2}+\frac{\gamma\beta}{4}-\frac{\gamma^2}{4}) \Gamma_{\frac{\gamma}{2}}(\alpha \pm \frac{\beta}{2} )   }  \frac{ (\alpha - Q)   }{  \Gamma( \frac{2}{\gamma}(Q -\alpha) +1 ) \Gamma( \frac{2}{\gamma}( \frac{\beta}{2} +\alpha - Q ) )      }   \frac{ \Gamma_{\frac{\gamma}{2}}(Q -\alpha )}{S_{\frac{\gamma}{2}}(2Q - \frac{\beta}{2}  -\alpha )}.
\end{align*}
\end{lemma}
\begin{proof}
Since $\overline{G}_0(\alpha, \gamma) = \frac{1}{\pi} \overline{U}_0(\alpha)$ and $R_0(\beta, \mu_1, \mu_2) = - \Gamma(1 - \frac{2}{\gamma}(Q -\beta)) \overline{R}_0(\beta, \mu_1, \mu_2)$, we have
\begin{align*}
\frac{\overline{G}_0(\alpha, \gamma) \overline{G}_0(\gamma, \beta)}{\overline{G}_0(\alpha, \beta) \overline{G}_0(\gamma, \gamma)} =  c_{\gamma, \beta}  \frac{ \Gamma(\frac{\gamma \alpha}{2} - \frac{\gamma^2}{4})   \Gamma_{\frac{\gamma}{2}}(\alpha )^2  }{  \Gamma(\frac{\gamma\alpha}{2}+\frac{\gamma\beta}{4}-\frac{\gamma^2}{4})  \Gamma_{\frac{\gamma}{2}}(\alpha \pm \frac{\beta}{2} )  }.
\end{align*}
Now we look at the integral containing $R_0$ in~\eqref{eq:CR}. We first write:
\begin{align*}
&  \int_0^{\infty} \mu_2^{ \frac{2}{\gamma} (Q -\alpha)  } \frac{\partial}{\partial \mu_2}  R_0( Q + \frac{\beta}{2}, \mu_2, 1)  d \mu_2 = \frac{2}{\gamma} (Q -\alpha) \Gamma(1 + \frac{\beta}{\gamma}) \int_0^{\infty} \mu_2^{ \frac{2}{\gamma} (Q -\alpha) -1 }  \overline{R}_0( Q + \frac{\beta}{2}, \mu_2, 1)  d \mu_2\\
& = c_{\gamma, \beta}  ( Q -\alpha)  \int_0^{\infty}  d \mu_2  
\,\mu_2^{ \frac{2}{\gamma} (Q -\alpha) - 1 } \frac{  e^{ \ii \pi(\sigma_2^0 - \frac{Q}{2})(- \frac{\beta}{2})}}{S_{\frac{\gamma}{2}}( \frac{\beta}{4} + \sigma_2^0 )S_{\frac{\gamma}{2}}( Q + \frac{\beta}{4} - \sigma_2^0) }. 
\end{align*}
We now compute this last integral, the relation between $\mu_2$ and $\sigma_2^0$ is $\mu_2 = e^{ \ii \pi \gamma(\sigma_2^0 - \frac{Q}{2})}$. We get:
\begin{align*}
&\int_0^{+ \infty} d \mu_2 \,\mu_{2}^{\frac{2(Q-\alpha)}{\gamma} -1 } \frac{  e^{ \ii \pi(\sigma_2^0 - \frac{Q}{2} )(- \frac{\beta}{2})}}{ S_{\frac{\gamma}{2}}( \frac{\beta}{4} + \sigma_2^0)S_{\frac{\gamma}{2}}( Q + \frac{\beta}{4} - \sigma_2^0) }  =  \pi \gamma \int_{\frac{Q}{2} + \ii \mathbb{R}} \frac{d \sigma_2^0}{\ii}  \frac{ e^{2 \ii \pi(Q -\alpha ) (\sigma_2^0 -\frac{Q}{2}) }  e^{ \ii \pi(\sigma_2^0 - \frac{Q}{2})(- \frac{\beta}{2})}  }{ S_{\frac{\gamma}{2}}( \frac{\beta}{4} + \sigma_2^0)S_{\frac{\gamma}{2}}( Q+ \frac{\beta}{4} - \sigma_2^0) }\\
& =  \pi \gamma e^{2 \ii \pi(Q -\alpha  - \frac{\beta}{4})( \frac{Q}{2} - \frac{\beta}{4} ) } \int_{ \ii \mathbb{R}} \frac{d \sigma_2^0}{ \ii} e^{2 \ii \pi(Q -\alpha  - \frac{\beta}{4})\sigma_2^0  }  \frac{S_{\frac{\gamma}{2}}(Q -  \frac{\beta}{2}  + \sigma_2^0) }{ S_{\frac{\gamma}{2}}(Q + \sigma_2^0) }.
\end{align*}
We will now use the following identity derived in \cite[Lemma 15]{PT-CMP},
\begin{align*}
 \int_{ \ii \mathbb{R}} d \tau e^{2 \pi \ii \tau \beta'} e^{ \ii \pi \tau ( \alpha'  -  Q )} \frac{S_{\frac{\gamma}{2}}(\tau + \alpha')}{S_{\frac{\gamma}{2}}(\tau +Q)}  = \ii e^{\frac{ \ii \pi}{2} \alpha'(Q -\alpha')} e^{- \ii \pi \alpha' \beta'} \frac{S_{\frac{\gamma}{2}}(\alpha') S_{\frac{\gamma}{2}}(\beta')}{S_{\frac{\gamma}{2}}(\alpha' + \beta')},
\end{align*}
with $\alpha' = Q - \frac{\beta}{2}$, $\beta' = Q -\alpha $.
This gives us:
\begin{align*}
&\int_0^{+ \infty} d \mu_2 \mu_{2}^{\frac{2(Q-\alpha)}{\gamma} -1} \frac{  e^{ \ii \pi(\sigma_2^0 - \frac{Q}{2})(- \frac{\beta}{2})}}{ S_{\frac{\gamma}{2}}( \frac{\beta}{4} + \sigma_2^0)S_{\frac{\gamma}{2}}( Q + \frac{\beta}{4} - \sigma_2^0) } \\
& = \pi \gamma  e^{2 \ii \pi(Q -\alpha  - \frac{\beta}{4})( \frac{Q}{2} - \frac{\beta}{4} ) }  e^{\frac{\ii \pi}{2} \frac{\beta}{2} (Q - \frac{\beta}{2})} e^{- \ii \pi (Q - \frac{\beta}{2}) (Q -\alpha )} \frac{S_{\frac{\gamma}{2}}(Q - \frac{\beta}{2}) S_{\frac{\gamma}{2}}(Q -\alpha )}{S_{\frac{\gamma}{2}}(2Q - \frac{\beta}{2}  -\alpha )}  = \pi \gamma  \frac{S_{\frac{\gamma}{2}}(Q - \frac{\beta}{2}) S_{\frac{\gamma}{2}}(Q -\alpha )}{S_{\frac{\gamma}{2}}(2Q - \frac{\beta}{2}  -\alpha )}.
\end{align*}
Putting everything together we conclude the proof.
\end{proof} 

Coming now back to the  equation \eqref{eq:weldG}, we first compute:
\begin{align}\label{eq:mid}
\frac{2^{- \frac{\alpha^2}{2}} \overline{U}_0(\alpha)}{\Gamma(\frac{2}{\gamma}(Q -\alpha))} \left( \frac{1}{2} \sqrt{\frac{1}{\sin(\pi  \frac{\gamma^2}{4})}} \right)^{ \frac{2}{\gamma}(Q -\alpha) }
= c_{\gamma} \; 2^{ \frac{\alpha^2}{2} - \alpha Q} \frac{\Gamma(\frac{\gamma \alpha}{2} - \frac{\gamma^2}{4}) }{\Gamma( \frac{2}{\gamma}(Q -\alpha) )} \left( \frac{\pi \Gamma( \frac{\gamma^2}{4} )}{\Gamma(1 - \frac{\gamma^2}{4})} \right)^{ \frac{1}{\gamma}(Q -\alpha)}.
\end{align}

We finally turn to the part with the expectation over $\mathcal M$.
\begin{lemma}\label{lem:M}
In the setting of Proposition~\ref{prop:Wu1}, we have
\begin{align*}
\mathcal M \left[ e^{- \mu_B L_1 -  A} K_{\frac{2}{\gamma}(Q -\alpha)} \left( L_2 \sqrt{\frac{1}{\sin \frac{\pi \gamma^2}{4}}} \right) \right]  = c_{\gamma, \beta} \int_{ \ii \mathbb{R}} \frac{d \sigma_2}{\ii} e^{2 \ii \pi (Q - 2 \sigma_1) \sigma_2} \frac{   S_{\frac{\gamma}{2}}( \frac{1}{2}(\alpha + \frac{\beta}{2} - Q) \pm \sigma_2 )  }{  S_{\frac{\gamma}{2}}( \frac{1}{2}(\alpha - \frac{\beta}{2} + Q) \pm \sigma_2 )  }.
\end{align*}
\end{lemma}

\begin{proof}
Record following parameters and relations:
\begin{align*}
\mu_B = \frac{1}{\sqrt{\sin \frac{\pi \gamma^2}{4}}} \cos (\pi \gamma (\sigma - \frac{Q}{2})), \: \hat{\mu}_{B}(t) = \frac{1}{\sqrt{\sin \frac{\pi \gamma^2}{4}}} \cosh(t), \:  t = \ii \pi \gamma (\sigma_2 - \frac{Q}{2}), \: \cosh t = \cos (\pi \gamma (\sigma_2 - \frac{Q}{2})).
\end{align*}
Recall from \eqref{equation R reflect} that $R(Q + \frac{\beta}{2}, \sigma, \sigma_2)R(Q -\frac{\beta}{2}, \sigma, \sigma_2)=1$. By~\eqref{eq:M2W} and the expression of $K_{\nu}$,
\begin{align*}
& \mathcal M \left[ e^{- \mu_B L_1 -  A} K_{\frac{2}{\gamma}(Q -\alpha)} \left( L_2 \sqrt{\frac{1}{\sin \frac{\pi \gamma^2}{4}}} \right) \right]  = \frac{1}{2} \int_{\mathbb{R}} dt \cosh( \frac{2}{\gamma}(Q -\alpha)t) \mathcal M \left[ e^{- \mu_B L_1 - \hat{\mu}_{B}(t)   L_2 -A}  \right]\\
& = - \frac{\pi \gamma}{2} \frac{\gamma}{\beta} \int_{\frac{Q}{2} + \ii \mathbb{R}} \frac{d \sigma_2}{\ii} \cos( 2 \pi (Q -\alpha)(\sigma_2 - \frac{Q}{2}) ) R(Q + \frac{\beta}{2}, \sigma, \sigma_2).
\end{align*}
Plugging in the exact formula $R_{\mathrm{FZZ}}$ we have for $R$, 
the above integral then becomes:
\begin{align*}
& - \frac{\pi \gamma}{2  } \frac{\gamma}{\beta} \int_{\frac{Q}{2} + \ii \mathbb{R}} \frac{d \sigma_2}{\ii} \cos( 2 \pi (Q -\alpha)(\sigma_2 - \frac{Q}{2}) ) R(Q + \frac{\beta}{2}, \sigma, \sigma_2)\\
& = - \frac{\pi \gamma^2}{2  \beta} \left(\frac{\pi  (\frac{\gamma}{2})^{2-\frac{\gamma^2}{2}} \Gamma(\frac{\gamma^2}{4})}{\Gamma(1-\frac{\gamma^2}{4})} \right)^{- \frac{\beta}{2 \gamma}} \frac{\Gamma_{\frac{\gamma}{2}}( \frac{\beta}{2}) }{\Gamma_{\frac{\gamma}{2}}(- \frac{\beta}{2}) }  \int_{\frac{Q}{2} + \ii \mathbb{R}} \frac{d \sigma_2}{\ii} \cos( 2 \pi (Q -\alpha)(\sigma_2 - \frac{Q}{2}) ) \frac{  S_{\frac{\gamma}{2}}(\frac{Q}{2} - \frac{\beta}{4} \pm (Q-\sigma -\sigma_2) ) }{        S_{\frac{\gamma}{2}}(\frac{Q}{2} + \frac{\beta}{4} \pm (\sigma_2- \sigma))}.
\end{align*}
Let us now take only the contour integral above and write the contour integral over the line $\ii \mathbb{R}$:
\begin{align*}
 &\int_{ \ii \mathbb{R}} d \sigma_2 \cos( 2 \pi (Q -\alpha)\sigma_2  ) \frac{  S_{\frac{\gamma}{2}}(Q  - \frac{\beta}{4}   -\sigma -\sigma_2 ) S_{\frac{\gamma}{2}}( - \frac{\beta}{4} + \sigma + \sigma_2 ) }{        S_{\frac{\gamma}{2}}(Q + \frac{\beta}{4} + \sigma_2- \sigma) S_{\frac{\gamma}{2}}(\frac{\beta}{4}  + \sigma - \sigma_2)} = \int_{ \ii \mathbb{R}} d \sigma_2 e^{2 \ii \pi (Q -\alpha) \sigma_2} \frac{   S_{\frac{\gamma}{2}}( - \frac{\beta}{4}  +  \sigma \pm \sigma_2 )  }{  S_{\frac{\gamma}{2}}( \frac{\beta}{4}   + \sigma \pm \sigma_2 )     }.
\end{align*}
In this last line we have used the symmetry $\sigma_2 \rightarrow - \sigma_2$ of the integrand coming from the identity $S_{\frac{\gamma}{2}}(x) = \frac{1}{S_{\frac{\gamma}{2}}(Q -x)} $.
We apply the following identity coming from \cite[Equation (D.34a)]{TesVar2013}:
\begin{align*}
 & \int_{ \ii \mathbb{R}} d \sigma_2 e^{2 \ii \pi (Q -\alpha) \sigma_2} \frac{   S_{\frac{\gamma}{2}}( - \frac{\beta}{4}  +  \sigma \pm \sigma_2 )  }{  S_{\frac{\gamma}{2}}( \frac{\beta}{4}   + \sigma \pm \sigma_2 )     } = \frac{S_{\frac{\gamma}{2}}( Q - \frac{\beta}{2})}{S_{\frac{\gamma}{2}}(  \frac{\beta}{2})} \int_{ \ii \mathbb{R}} d \sigma_2 e^{2 \ii \pi (Q - 2 \sigma) \sigma_2} \frac{   S_{\frac{\gamma}{2}}( \frac{1}{2}(\alpha + \frac{\beta}{2} - Q) \pm \sigma_2 )  }{  S_{\frac{\gamma}{2}}( \frac{1}{2}(\alpha - \frac{\beta}{2} + Q) \pm \sigma_2 )  }.
\end{align*}
Putting everything together  we conclude the proof.
\end{proof}
\begin{proof}[Proof of Theorem~\ref{thm:G}]
We first assume that $\alpha,\beta$ are in the range specified by Proposition~\ref{prop:Wu1}. In the formula~\eqref{eq:weldG} for $G_{\mu_B}(\alpha, \beta)$, we substitute the identities  from Lemma~\ref{lem:CR}, \eqref{eq:mid} and Lemma~\ref{lem:M} to conclude that
for some $\tilde c_{\gamma, \beta}$ to be determined later, we have:
\begin{align*}
&G_{\mu_B}(\alpha, \beta) = \tilde c_{\gamma, \beta} 2^{ \frac{\alpha^2}{2} - \alpha Q} \frac{\Gamma(\frac{\gamma \alpha}{2} - \frac{\gamma^2}{4}) }{\Gamma( \frac{2}{\gamma}(Q -\alpha) )} \left( \frac{\pi \Gamma( \frac{\gamma^2}{4} )}{\Gamma(1 - \frac{\gamma^2}{4})} \right)^{ \frac{1}{\gamma}(Q -\alpha)}  \frac{ \Gamma(\frac{\gamma\alpha}{2}+\frac{\gamma\beta}{4}-\frac{\gamma^2}{4}) \Gamma_{\frac{\gamma}{2}}(\alpha \pm \frac{\beta}{2} )    }{ \Gamma(\frac{\gamma \alpha}{2} - \frac{\gamma^2}{4}) \Gamma_{\frac{\gamma}{2}}(\alpha ) } \\
& \times  \frac{  \Gamma( \frac{2}{\gamma}(Q -\alpha) +1 ) \Gamma( \frac{2}{\gamma}( \frac{\beta}{2} +\alpha - Q ) )      }{ (\alpha - Q)   }   \frac{S_{\frac{\gamma}{2}}(2Q - \frac{\beta}{2}  -\alpha )}{ \Gamma_{\frac{\gamma}{2}}(Q -\alpha )}  \int_{ \ii \mathbb{R}} \frac{d \sigma_2}{\ii} e^{2 \ii \pi (Q - 2 \sigma) \sigma_2} \frac{   S_{\frac{\gamma}{2}}( \frac{1}{2}(\alpha + \frac{\beta}{2} - Q) \pm \sigma_2 )  }{  S_{\frac{\gamma}{2}}( \frac{1}{2}(\alpha - \frac{\beta}{2} + Q) \pm \sigma_2 )  }\\
&= \tilde c_{\gamma, \beta} 2^{ \frac{\alpha^2}{2} - \alpha Q}  \left( \frac{\pi \Gamma( \frac{\gamma^2}{4} )}{\Gamma(1 - \frac{\gamma^2}{4})} \right)^{ \frac{1}{\gamma}(Q -\alpha)}   \frac{ \Gamma(\frac{\gamma\alpha}{2}+\frac{\gamma\beta}{4}-\frac{\gamma^2}{4}) \Gamma_{\frac{\gamma}{2}}(\alpha \pm \frac{\beta}{2} )    }{   \Gamma_{\frac{\gamma}{2}}(\alpha ) }    \Gamma( \frac{2}{\gamma}( \frac{\beta}{2} +\alpha - Q ) )          \\
&\times \frac{S_{\frac{\gamma}{2}}(2Q - \frac{\beta}{2}  -\alpha )}{ \Gamma_{\frac{\gamma}{2}}(Q -\alpha )} \int_{ \ii \mathbb{R}} \frac{d \sigma_2}{\ii} e^{2 \ii \pi (Q - 2 \sigma) \sigma_2} \frac{   S_{\frac{\gamma}{2}}( \frac{1}{2}(\alpha + \frac{\beta}{2} - Q) - \sigma_2 ) S_{\frac{\gamma}{2}}( \frac{1}{2}(\alpha + \frac{\beta}{2} - Q) + \sigma_2 ) }{  S_{\frac{\gamma}{2}}( \frac{1}{2}(\alpha - \frac{\beta}{2} + Q) + \sigma_2 )       S_{\frac{\gamma}{2}}(\frac{1}{2}(\alpha - \frac{\beta}{2} + Q) - \sigma_2)}.
\end{align*}

To pin down the value of $\tilde c_{\gamma, \beta}$, we first note that
\begin{align}\label{eq:limit-G-beta}
\underset{\alpha \rightarrow Q - \frac{\beta}{2}}{\lim} (\alpha + \frac{\beta}{2} - Q) G_{\mu_B}(\alpha, \beta)= 2^{-\frac{1}{2}(Q - \frac{\beta}{2})^2},
\end{align} 
which can be deduced by following the exact same steps as for the residues of $H$ given in Lemma \ref{lem:res_H}.
The only difference is the power of $2$ above which comes from the way we normalize the $\alpha$ insertion of $G$ at $\ii$. Namely the $2^{- \frac{\alpha^2}{2}}$ of Definition \ref{def:LF_G_function} becomes $2^{- \frac{1}{2}(Q -\frac{\beta}{2})^2 }$ in the limit.
Coming back to the expression above for $G$ involving $\tilde c_{\gamma, \beta}$, we now evaluate its $\alpha \rightarrow Q - \frac{\beta}{2}$ limit. First, the limit of the prefactor in front of the contour integral is given by:
\begin{align*}
 &\tilde c_{\gamma, \beta} 2^{ - \frac{1}{2}(Q - \frac{\beta}{2})(Q + \frac{\beta}{2}) }  \left( \frac{\pi \Gamma( \frac{\gamma^2}{4} )}{\Gamma(1 - \frac{\gamma^2}{4})} \right)^{ \frac{\beta}{2 \gamma}} \frac{ \Gamma_{\frac{\gamma}{2}}(Q - \beta )  \Gamma_{\frac{\gamma}{2}}(Q  )^2  }{   \Gamma_{\frac{\gamma}{2}}(Q - \frac{\beta}{2} ) \Gamma_{\frac{\gamma}{2}}( \frac{\beta}{2} ) }    \underset{\alpha \rightarrow Q - \frac{\beta}{2}}{\lim}   \frac{\Gamma( \frac{2}{\gamma}( \frac{\beta}{2} +\alpha - Q ) ) }{ \Gamma_{\frac{\gamma}{2}}(\alpha + \frac{\beta}{2} - Q )}\\
& = \tilde c_{\gamma, \beta}  2^{ - \frac{1}{2}(Q - \frac{\beta}{2})(Q + \frac{\beta}{2}) }  \left( \frac{\pi \Gamma( \frac{\gamma^2}{4} )}{\Gamma(1 - \frac{\gamma^2}{4})} \right)^{ \frac{\beta}{2 \gamma}} \frac{ \Gamma_{\frac{\gamma}{2}}(Q - \beta )  \Gamma_{\frac{\gamma}{2}}(Q  )^2  }{   \Gamma_{\frac{\gamma}{2}}(Q - \frac{\beta}{2} ) \Gamma_{\frac{\gamma}{2}}( \frac{\beta}{2} ) } \frac{\sqrt{2 \pi}}{\Gamma_{\frac{\gamma}{2}}(\frac{2}{\gamma}) } (\frac{\gamma}{2})^{1/2}.
\end{align*}
The limit of the contour integral  is given by Lemma \ref{lem:comp1_G}, which is $(2 \pi      S_{\frac{\gamma}{2}}(Q - \frac{\beta}{2} )^2)^{-1} $.
Therefore \begin{align*}
&\underset{\alpha \rightarrow Q - \frac{\beta}{2}}{\lim} (\alpha + \frac{\beta}{2} - Q) G_{\mu_B}(\alpha, \beta) = \tilde c_{\gamma, \beta}   2^{ - \frac{1}{2}(Q - \frac{\beta}{2})(Q + \frac{\beta}{2}) } \frac{\gamma}{2} \left( \frac{\pi \Gamma( \frac{\gamma^2}{4} )}{\Gamma(1 - \frac{\gamma^2}{4})} \right)^{ \frac{\beta}{2 \gamma}} \frac{ \Gamma_{\frac{\gamma}{2}}(Q - \beta )  \Gamma_{\frac{\gamma}{2}}(Q  ) \Gamma_{\frac{\gamma}{2}}( \frac{\beta}{2} ) }{\Gamma_{\frac{\gamma}{2}}(Q - \frac{\beta}{2} )^3},
\end{align*} where we have used $ \Gamma_{\frac{\gamma}{2}}(Q  ) (\Gamma_{\frac{\gamma}{2}}(\frac{2}{\gamma} ))^{-1} = \sqrt{2 \pi} (\frac{\gamma}{2})^{1/2} $ to simplify.
Comparing with~\eqref{eq:limit-G-beta}, we get
\begin{align*}
\tilde c_{\gamma, \beta} = 2^{\frac{\beta}{2}(Q - \frac{\beta}{2})}  \frac{2}{\gamma} \left( \frac{\pi \Gamma( \frac{\gamma^2}{4} )}{\Gamma(1 - \frac{\gamma^2}{4})} \right)^{ -\frac{\beta}{2 \gamma}} \frac{   \Gamma_{\frac{\gamma}{2}}(Q - \frac{\beta}{2} )^3  }{ \Gamma_{\frac{\gamma}{2}}(Q - \beta )  \Gamma_{\frac{\gamma}{2}}(Q  ) \Gamma_{\frac{\gamma}{2}}( \frac{\beta}{2} ) }.
\end{align*}
Substituting this formula for $\tilde c_{\gamma, \beta}$ into the first equation, we obtain
\begin{align*}
&G_{\mu_B}(\alpha, \beta) =  2^{ \frac{\alpha^2}{2} - \alpha Q} 2^{\frac{\beta}{2}(Q - \frac{\beta}{2})}  \frac{2}{\gamma}  \left( \frac{\pi \Gamma( \frac{\gamma^2}{4} )}{\Gamma(1 - \frac{\gamma^2}{4})} \right)^{ \frac{1}{\gamma}(Q -\alpha - \frac{\beta}{2})} \frac{ \Gamma(\frac{\gamma\alpha}{2}+\frac{\gamma\beta}{4}-\frac{\gamma^2}{4}) \Gamma_{\frac{\gamma}{2}}(\alpha \pm \frac{\beta}{2} )  \Gamma_{\frac{\gamma}{2}}(Q - \frac{\beta}{2} )^3 }{   \Gamma_{\frac{\gamma}{2}}(\alpha ) \Gamma_{\frac{\gamma}{2}}(Q - \beta )  \Gamma_{\frac{\gamma}{2}}(Q  ) \Gamma_{\frac{\gamma}{2}}( \frac{\beta}{2} ) }        \\
&\times   \Gamma( \frac{2}{\gamma}( \frac{\beta}{2} +\alpha - Q ) )     \frac{S_{\frac{\gamma}{2}}(2Q - \frac{\beta}{2}  -\alpha )}{ \Gamma_{\frac{\gamma}{2}}(Q -\alpha )} \int_{ \ii \mathbb{R}} \frac{d \sigma_2}{\ii} e^{2 \ii \pi (Q - 2 \sigma) \sigma_2} \frac{   S_{\frac{\gamma}{2}}( \frac{1}{2}(\alpha + \frac{\beta}{2} - Q) \pm \sigma_2 )  }{  S_{\frac{\gamma}{2}}( \frac{1}{2}(\alpha - \frac{\beta}{2} + Q) \pm \sigma_2 )    }.
\end{align*}
Using equation \eqref{eq:shift_G1} and \eqref{eq:shift_G2} on the double gamma function we can simplify:
\begin{align*}
\frac{\Gamma( \frac{2}{\gamma}( \frac{\beta}{2} +\alpha - Q ) ) \Gamma(\frac{\gamma\alpha}{2}+\frac{\gamma\beta}{4}-\frac{\gamma^2}{4})   \Gamma_{\frac{\gamma}{2}}(\alpha+\frac{\beta}{2}  )  }{    \Gamma_{\frac{\gamma}{2}}( \frac{\beta}{2}  +\alpha - Q )} = 2 \pi \left( \frac{\gamma}{2} \right)^{( \frac{\gamma}{2} - \frac{2}{\gamma})(\alpha + \frac{\beta}{2} - Q) + 1}.
\end{align*}
This gives the  desired expression  $G_{\mathrm{Hos} }$ of $G_{\mu_B}(\alpha, \beta) $:
\begin{align*}
   & 2 \pi 2^{- \alpha Q + \frac{\alpha^2}{2}}  2^{\frac{\beta}{2}(Q - \frac{\beta}{2})}         \left(\frac{\pi  (\frac{\gamma}{2})^{2-\frac{\gamma^2}{2}} \Gamma(\frac{\gamma^2}{4})}{\Gamma(1-\frac{\gamma^2}{4})} \right)^{ \frac{1}{ \gamma}(Q -\alpha - \frac{\beta}{2})}  \frac{\Gamma_{\frac{\gamma}{2}}(2Q - \frac{\beta}{2}  -\alpha ) \Gamma_{\frac{\gamma}{2}}(\alpha-\frac{\beta}{2} )   \Gamma_{\frac{\gamma}{2}}(Q - \frac{\beta}{2}  )^3 }{  \Gamma_{\frac{\gamma}{2}}(Q -\alpha ) \Gamma_{\frac{\gamma}{2}}(Q - \beta ) \Gamma_{\frac{\gamma}{2}}(\alpha )\Gamma_{\frac{\gamma}{2}}(Q) \Gamma_{\frac{\gamma}{2}}(  \frac{\beta}{2}) } \\
& \times \int_{ \ii \mathbb{R}} \frac{d \sigma_2}{\ii} e^{2 \ii \pi (Q - 2 \sigma) \sigma_2} \frac{   S_{\frac{\gamma}{2}}( \frac{1}{2}(\alpha + \frac{\beta}{2} - Q) \pm \sigma_2 ) }{  S_{\frac{\gamma}{2}}( \frac{1}{2}(\alpha - \frac{\beta}{2} + Q) \pm \sigma_2 )},
\end{align*}
where we note that since 
$\alpha$ and $\beta$ are in the range specified by Proposition~\ref{prop:Wu1}, the contour $\ii\mathbb R$ in the above integral satisfies the requirement for $G_{\mathrm{Hos} }$  described above~\eqref{eq:G-range1} and~\eqref{eq:G-range2}. Thus, we have shown $G_{\mu_B}= G_{\mathrm{Hos}}$ when $\alpha,\beta$ are in the range specified by Proposition~\ref{prop:Wu1}. 

To extend this identity to the Seiberg bound $\alpha + \frac{\beta}{2} > Q$, $ \alpha < Q$, $ \beta < Q$ as in Theorem~\ref{thm:G}, it suffices to note that for $\alpha\in (\frac{Q}{2},Q)$, both $G_{\mathrm{Hos}}$ and $G_{\mu_B}$ are meromorphic in $\beta$ in a complex neighborhood of $(2Q-2\alpha, Q)$. The meromorphicity of $G_{\mathrm{Hos}}$ follows from the same consideration as for $H_{\mathrm{PT}}$ in Lemma~\ref{lem:mero_HPT}. The meromorphicity of  $G_{\mu_B}$  follows from the same consideration as for $H$ in Section~\ref{subsec-anal-H}.
\end{proof}
\begin{remark}
One can check that $G_{\mathrm{Hos} }(\alpha, 0)$  recovers the FZZ formula of \cite[Equation (1.4)]{ARS-FZZ}; see Remark 4.4 from the first arXiv version of this paper for the computation.
\end{remark}

\section{Two lemmas on the reflections coefficients}\label{sec:technical}
We prove Lemmas~\ref{lem:HlimR} and~\ref{lem-mot-match} on the reflections coefficients in Sections~\ref{app:lim_H_R} and~\ref{sec:Rgamma}, respectively.

\subsection{Reflection coefficient as the limit of $H$: proof of Lemma~\ref{lem:HlimR}}\label{app:lim_H_R}
The proof is analogous to that of \cite[Proposition 8.1]{DOZZ_proof}, which obtains the LCFT sphere reflection coefficient from the three-point function.
Our high level argument is similar, but with the additional difficulty that correlation functions are not GMC moments due to the presence of both bulk and boundary Liouville potentials. 

Throughout this section we consider $\beta_i,\sigma_i,\mu_i$ with $i=1,2,3$ as given in Lemma~\ref{lem:HlimR}.
Similarly as in Definition~\ref{lem:def_R_trunc}, it will be more convenient to work with the horizontal strip $\cS= \R \times (0,\pi)$. 
Let $P_\cS$ be the law of a free-boundary GFF on $\cS$. Namley, $P_\cS$ is the law of 
 $h_\cS=h_\bbH \circ \exp$ where $h_\bbH$ is sampled from $P_\bbH$. 
Let $\psi$ be the field from Definition~\ref{def-thick-disk} with $\beta$ replaced by $\beta_1$.
In particular, the lateral component of $\psi$ has the same law as that of  $h_\cS$.
For $c\in \R$, write  $A' = \cA_{\psi+c}(\cS), L_1' = \cL_{\psi+c}(\R)$, $ L_2' = \cL_{\psi+c}(\R\times \{\pi\})$. Define
\begin{equation}\label{eq:Rpsi}
 \mathfrak R_{\psi+c}:=  (Q-{\beta_1}) \left(\gamma ({\beta_1}-Q+\frac\gamma2) A' + \frac{\gamma^2}2 A' (\sum_{i=1}^2 
\mu_i L_i') + \frac{\gamma^2}4 (\sum_{i=1}^2 \mu_i L_i')^2 \right) e^{-A' - \sum_{i=1}^2 \mu_i L_i'}.   
\end{equation}
Recall $\hat R(\beta_1, \sigma_1, \sigma_2)$ from  Definition~\ref{lem:def_R_trunc}. We have 
\(\hat R(\beta_1, \sigma_1, \sigma_2)=\int_{-\infty}^{\infty} e^{(\beta_1-Q)c} \E[\mathfrak R_{\psi + c}] \, dc\).

We now express $H$ in terms of a field on $\cS$. Sample  $(h, \mathbf c)$ from $P_\cS \times [ e^{(\frac12\sum \beta_i - Q)c}\,dc]$, and let 
	\begin{equation}\label{eq-wth}
	\hat h(z)= h(z) - (Q-\beta_1) (0 \vee \Re z) - (Q-\beta_2) (0 \vee (-\Re z)) + \frac{\beta_3}2 G_\cS(z, \ii \pi),
	\end{equation} where $G_\cS(z, w)= G_\bbH (e^z,e^w)$.
Let $\phi = \hat h + \mathbf c$ and  $\LF_\cS^{\beta_1, \beta_2, \beta_3}$ be the law of $\phi$.
Let  $\hat A = \cA_{\phi}(\cS)$, $ \hat  L_1 = \cL_{\phi}(0,\infty)$, $L_2 = \cL_{\phi}(\R \times \{\pi\}), \hat L_3 = \cL_\phi(-\infty, 0)$. 
In light of $\hat H^{(\beta_1 , \beta_2, \beta_3)}_{(\mu_1,\mu_2, \mu_3)}$ from \eqref{eq-hatH} and  Definition~\ref{def:H-general}, we define 
\begin{equation}\label{eq-Hphi}
\mathfrak H_\phi := (\gamma (s+\frac\gamma2) \hat A + \frac{\gamma^2}2 \hat A (\sum_{i=1}^3 \mu_i \hat L_i) + \frac{\gamma^2}4 (\sum_{i=1}^3 \mu_i \hat L_i)^2)e^{-\hat A - \sum_{i=1}^3 \mu_i \hat L_i} \quad \textrm{with }s = \frac12\sum\beta_i - Q.    
\end{equation}
By  \cite[Lemma 2.9]{ASY-triangle}, the LQG observables of $\LF_\cS^{\beta_1, \beta_2, \beta_3}$ and $\LF_\bbH^{(\beta_1,0), (\beta_2, 1), (\beta_3, \infty)}$ agree in law, under the coordinate change $z\mapsto e^{-z}$.  In particular,  $\hat H^{(\beta_1, \beta_2, \beta_3)}_{(\sigma_1, \sigma_2, \sigma_3)}  = \LF_\cS^{\beta_1,\beta_2,\beta_3}[\mathfrak H_\phi] $.
Therefore Lemma~\ref{lem:HlimR} can be rephrased as 
\begin{equation}\label{eq-limit-HR-strip}
		\lim_{\beta_3 \downarrow \beta_1-\beta_2} (\beta_2+\beta_3-\beta_1) \LF_\cS^{\beta_1,\beta_2,\beta_3}[\mathfrak H_\phi] = 2\int_{-\infty}^{\infty} e^{(\beta_1-Q)c} \E[\mathfrak R_{\psi + c}] \, dc .
\end{equation}
We first prove the following variant of~\eqref{eq-limit-HR-strip} where we condition on the field average maximum.

\begin{proposition}\label{prop-phi-cond-max}
For $\beta_1 \in (\frac\gamma2 \vee (Q-\gamma),  Q)$, $\beta_2 \in (0 ,\beta_1)$ and $\beta_3 > \beta_1 - \beta_2$,  
 sample $\phi$ from $\LF_\cS^{\beta_1, \beta_2, \beta_3}$. 
Let $M_\phi$ be the supremum over $t>0$ of the average of $\phi$ on $\{t \} \times (0,\pi)$.
Fix $m \in \R$. Write $T_{\phi - m}=(\cA_{\phi - m} (\cS), \cL_{\phi - m}((0,\infty) \times \{0\}), \cL_{\phi - m}(\R \times \{\pi\}), \cL_{\phi - m}((-\infty,0) \times \{0\}))$.
Let $f:\R^4 \to \R$ be a bounded continuous function. Then \begin{equation}\label{eq:weak}
  \lim_{\beta_3 \downarrow \beta_1-\beta_2} \LF_\cS^{\beta_1,\beta_2,\beta_3}[f(T_{\phi - m}) \mid M_\phi = m]  = \E[f(\cA_\psi(\cS), \cL_\psi(\R), \cL_\psi(\R \times \{\pi\}),0)]  
\end{equation}
where $\psi$ on the right hand side is the field from Definition~\ref{def-thick-disk} with $\beta$ replaced by $\beta_1$.
\end{proposition}
We prove Proposition~\ref{prop-phi-cond-max} by proving the following two lemmas.
\begin{lemma}\label{lem-pinch-h}
Sample $h$ from $P_\cS$ and define $\wt h$ in terms of $h$ as  $\hat h$ in~\eqref{eq-wth}.
For $M>0$, let $\wt T_{\wt h - M} = (\cA_{\wt h - M}(\cS), \cL_{\wt h - M}(0,\infty), \cL_{\wt h - M} (\R \times \{\pi\}), \cL_{\wt h - M}(-\infty, 0))$.  
Let $M_{\wt h}$ be the supremum over $t>0$ of the average of $\wt h$ on $\{t \} \times (0,\pi)$. 
Let $f:\R^4 \to \R$ be a bounded continuous function. Then
\begin{equation}\label{eq:Ttilde}
    \lim_{M \to \infty}\E[f(\wt T_{\wt h - M})) \mid M_{\wt h} = M] = \E[f(\cA_\psi(\cS), \cL_\psi(\R ), \cL_\psi(\R \times \{\pi\}),0)].
\end{equation} 
Moreover, the convergence is uniform in $\beta_3\in [\beta_1 - \beta_2,\beta_1 - \beta_2+\eps]$ for $\eps>0$ small enough. 
\end{lemma}
\begin{proof}
 Consider $\wt h$ conditioned on $M_{\wt h} = M$. 
	Let $X_t^{\wt h}$ be the average of $\wt h$ on $\{t\}\times(0,\pi)$, and define $X_t^\psi$ likewise. Let
\( \tau = \inf \{ t \: : \: X^{\wt h}_t >M -  \sqrt M\}\) and \( \sigma = \inf \{ t \: : \: X^{\psi}_t > -\sqrt M\}\).
We first show that we can couple $\wt h$ and $\psi$ so that with probability $1-o_M(1)$, 
 \begin{equation}\label{eq:couple}
(\wt h(\cdot - \tau)-M)|_{\cS_+}=\psi(\cdot - \sigma)|_{\cS_+}, \quad  \textrm{where } \cS_+ = (0, \infty) \times (0,\pi).
 \end{equation}

	We can couple $(X^{\wt h}_{t + \tau}-M)_{t \geq 0}$ and $(X^\psi_{t + \sigma})_{t \geq 0}$ to agree almost surely; indeed they have the law of $(B_{2t} - (Q-\beta_1)t - \sqrt M)_{t \geq 0}$ conditioned to have maximum value $0$, where $B_t$ is standard Brownian motion.	
	Let $h^2, \wt h^2, \psi^2$ denote the lateral components of $h, \wt h, \psi$, i.e.\ the projections to the space of distributions with mean zero on $\{t\} \times (0,\pi)$ for all $t$.
	We have $\tau \to \infty$ in probability as $M \to \infty$, so with probability $1-o_M(1)$ the Dirichlet energy of $\gamma G_\cS(\cdot, \ii \pi)$ on $(\tau, \infty)\times (0,\pi)$ is close to zero. Thus the law of $(h^2 + \gamma G_\cS(\cdot, i\pi))|_{(\tau, \infty) \times (0,\pi)}$ is within $o_M(1)$ in total variation distance to the law of $h^2|_{(\tau, \infty) \times (0,\pi)}$ \cite[Proposition 2.9]{ig4}, so we can further couple $\wt h^2(\cdot + \tau)|_{\cS_+}$ and $\psi^2(\cdot + \sigma)|_{\cS_+}$ to agree with probability $1-o_M(1)$. This gives the desired coupling~\eqref{eq:couple}.

To prove~\eqref{eq:Ttilde}, it suffices to show that as $M\to \infty$ the quantum areas and lengths of $(\wt h(\cdot - \tau)-M)|_{\cS \backslash \cS_+}$ and $\psi(\cdot - \sigma)|_{\cS \backslash \cS_+}$ tend to zero in probability; this holds because $(X_{t + \tau}^{\wt h} - M)_{ t \leq 0}$ and $(X_{t + \sigma}^\psi)_{t \leq 0}$ are bounded above by $-\sqrt M$ with probability $1-o_M(1)$. 
The uniform convergence in $\beta_3$ can be checked by inspecting the argument above.
\end{proof}
\begin{lemma}\label{lem:Mphi} 
Sample  $(h, \mathbf c)$ from the measure $\mathbb F:=P_\cS \times [ e^{(\frac12\sum \beta_i - Q)c}\,dc]$ and let $\hat h$ be as in ~\eqref{eq-wth}. 
Write  $\phi=\hat h+c$ so that the law of $\phi$ is $\LF_\cS^{\beta_1, \beta_2, \beta_3}$. Let $M_\phi$ be the supremum over $t>0$ of the average of $\phi$ on $\{t \} \times (0,\pi)$.
Let $\delta = \frac12(\beta_2 + \beta_3 - \beta_1)$.  Fix $m\in \R$. Then under $\mathbb F$ the conditional law of $\mathbf c$  conditioning on $\{M_\phi = m\}$ is $1_{c < m} \delta e^{\delta  (c-m)}\, dc$. Moreover, the law of $M_\phi$ is
\begin{equation}\label{eq-M-law}
    \LF_\cS^{\beta_1, \beta_2, \beta_3}[M_\phi \in dm]  = \frac{2(Q-\beta_1)}{\beta_2+\beta_3-\beta_1} e^{(\frac12\sum \beta_i - Q)m}.
\end{equation}
\end{lemma} 
\begin{proof} For a standard Brownian motion $(B_t)_{t \geq 0}$ we have $\P[ \sup_{t\geq0} B_{2t} - bt \geq m] = \min(e^{-bm}, 1)$ for $b >0$.
Therefore, with $M_{\wt h}$ defined in Lemma~\ref{lem-pinch-h}, we see that $ \mathbb F[M_\phi \geq m \text{ and } \mathbf c \in dc]$ equals 
\[
e^{(\frac12\sum \beta_i - Q)c} P_\cS[M_{\wt h} \geq m-c] = e^{(\frac12\sum \beta_i - Q)c}(1_{m > c} e^{- (Q - \beta_1)(m-c)} + 1_{m \leq c}).
\]
Therefore  $\mathbb F [M_\phi \in  dm \text{ and } \mathbf c \in dc] = 1_{m > c} (Q-\beta_1) e^{(\frac12\sum \beta_i - Q)c - (Q-\beta_1)(m-c)}$, hence 
\[
 \mathbb F[M_\phi \in dm] = \int 1_{m>c} (Q-\beta_1) e^{(\frac12\sum \beta_i - Q)c - (Q-\beta_1)(m-c)}\, dc= \frac{2(Q-\beta_1)}{\beta_2+\beta_3-\beta_1} e^{(\frac12\sum \beta_i - Q)m},
\]
which gives~\eqref{eq-M-law}. By  Bayes rule we have \( \mathbb F [\mathbf c \in dc \mid M_\phi = m] = 1_{c < m} \delta e^{\delta  (c-m)}\, dc\) as desired.
\end{proof}
\begin{proof}[Proof of Proposition~\ref{prop-phi-cond-max}] By Lemma~\ref{lem:Mphi}, with $\delta = \frac12(\beta_2 + \beta_3 - \beta_1)$ we have
	\[
	\LF_\cS^{\beta_1,\beta_2,\beta_3}[f(T_{\phi - m}) \mid M_\phi = m] = \int_{-\infty}^m \delta e^{\delta  (c-m)}\E[f(\wt T_{ \wt h - (m-c)}) \mid M_{\wt h} = m - c] \, dc.
	\]
	As $\beta_3 \downarrow \beta_1 - \beta_2$ we have $\delta \downarrow 0$. By Lemma~\ref{lem-pinch-h} we get  the desired~\eqref{eq:weak}.
	\end{proof}

 We now prove~\eqref{eq-limit-HR-strip} hence Lemma~\ref{lem:HlimR} using the  following lemma.
\begin{lemma}\label{lem-noncompact-bound}
Fix $\beta_1 \in (\frac\gamma2 \vee (Q-\gamma),  Q)$ and $\beta_2 \in (0 ,\beta_1)$. Suppose $\beta_3 > \beta_1 - \beta_2$, $\beta_3<2 Q-\beta_1-\beta_2$, and $\beta_3 < \beta_1+\beta_2$.
	Let $\Re \mu_i > 0$ for $i=1,2,3$. Recall $M_\phi$ from Proposition~\ref{prop-phi-cond-max} and 
$\mathfrak H_\phi$ from~\eqref{eq-Hphi}. We have \(\lim_{C \to \infty} (\beta_2 + \beta_3 - \beta_1) \LF_\cS^{\beta_1,\beta_2,\beta_3}[| \mathfrak H_\phi | 1_{|M_\phi| > C}  ] = 0.\)
Moreover, the convergence is uniform in $\beta_3\in (\beta_1 - \beta_2,\beta_1 - \beta_2+\eps]$ for $\eps>0$ small enough. 
\end{lemma}

\begin{proof}
In this proof let $K$ denote a deterministic constant whose value may change from line to line, but which does not depend on $\beta_3$.
Set $s=\frac12 \sum \beta_i - Q$.	Choose $q \in (-\frac s\gamma, \frac2{\gamma^2})$ such that
$q < \frac1\gamma (Q - \beta_j)$ for $j=2,3$. This choice is possible due to the conditions placed on $\beta_1,\beta_2,\beta_3$. The upper bounds for $q$ imply that $\E[\cA_{\wt h}((-\infty, 1) \times (0,\pi))^q]$ and $\E[\cL_{\wt h}((-\infty, 1)\times \{0,\pi\})^{2q}] $ are finite; see \cite[Corollaries 3.8 and 3.10]{hrv-disk}. Combining with Lemma~\ref{lem-AL-mixed} gives
	\[\LF_\cS^{\beta_1,\beta_2,\beta_3}[ \cA_\phi(\cS)^ q\mid M_\phi] \leq K e^{q\gamma M_\phi}  \quad \textrm{and}\quad \LF_\cS^{\beta_1,\beta_2,\beta_3}[\cL_\phi(\partial \cS)^{2q} \mid M_\phi] \leq K e^{q\gamma M_\phi}. \]
 Writing $L = \sum \mu_i L_i$, we now show that $\lim_{C \to \infty} \LF_\cS^{\beta_1, \beta_2, \beta_3}[1_{|M_\phi|>C}\mid AL e^{-A - L}|]= 0$. 
	Clearly we have $|AL e^{-A-\mu L}| \leq K A^q$. Then the law of $M_\phi$ from~\eqref{eq-M-law} and the moment bounds give
	\begin{align*}
		&(\beta_2 + \beta_3 - \beta_1) \LF_\cS^{\beta_1, \beta_2, \beta_3}[|ALe^{-A-\mu L}| 1_{|M_\phi| > C} ] \\
		&\leq K\int_{-\infty}^{-C} \LF_\cS^{\beta_1, \beta_2, \beta_3}[ A^q \mid M_\phi=m] e^{sm} \, \mathrm dm + K\int_C^\infty \LF_\cS^{\beta_1, \beta_2, \beta_3}[A e^{-A} \mid M_\phi = m] e^{sm} \, \mathrm dm \\
		&\leq K\int_{-\infty}^{-C} e^{q \gamma m + sm} \, \mathrm dm + K\int_C^\infty e^{sm}\, dm \leq K(e^{-(\gamma q + s)C} + e^{sC})\xrightarrow{C \to \infty} 0.		
	\end{align*}
 Here, we are using  $\gamma q + s > 0 > s$ which follows from our conditions on $q$ and $s$. This proves $\lim_{C \to \infty}  (\beta_2 + \beta_3 - \beta_1)\LF_\cS^{\beta_1, \beta_2, \beta_3}[1_{|M_\phi|>C}|AL e^{-A - L}|]= 0$. The analogous terms for $Ae^{-A-L}$ and $L^2 e^{-A-L}$ are similarly bounded, so the triangle inequality yields the result. 
 The uniform convergence follows by inspecting the proof.
\end{proof}
\begin{proof}[Proof of Lemma~\ref{lem:HlimR}]
Recall $R_{\psi+m}$ from~\eqref{eq:Rpsi} and compare with $\mathfrak H_\phi$. By Proposition~\ref{prop-phi-cond-max} we have
\(\lim_{\beta_3 \downarrow \beta_1-\beta_2}\LF_\cS^{\beta_1,\beta_2,\beta_3}[\mathfrak H_\phi \mid M_\phi = m]=  \frac1{Q-\beta_1} \E[\mathfrak R_{\psi+m}].
\)
 Moreover, for each  $C>0$ the convergence is uniform in $m \in (-C,C)$.
Using the law of $M_\phi$ from~\eqref{eq-M-law}, we get
	\[\lim_{\beta_3 \downarrow \beta_1-\beta_2} (\beta_2+\beta_3-\beta_1) \LF_\cS^{\beta_1,\beta_2,\beta_3}[ 1_{|M_\phi|\leq C} \mathfrak H_\phi] = 2 \int_{-C}^{C} e^{(\beta_1-Q)m} \E[\mathfrak R_{\psi + m}] \, dm .\]
 Sending $C \to \infty$ we get the desired~\eqref{eq-limit-HR-strip} from Lemma~\ref{lem-noncompact-bound}.
\end{proof}

\subsection{Two-point quantum disk from mating of trees: proof of Lemma~\ref{lem-mot-match}} \label{sec:Rgamma}

In the setting of  Lemma~\ref{lem-mot-match}, we have the following two lemmas.

\begin{lemma}\label{lem-MOT-input-0}
Let $c = \sqrt{1/\sin(\frac{\pi\gamma^2}4)}$ and $C = \frac{(2\pi)^{\frac4{\gamma^2}-1}}{(1-\frac{\gamma^2}4) \Gamma(1 - \frac{\gamma^2}4)^{\frac4{\gamma^2}}} \cdot \frac2{\Gamma(\frac4{\gamma^2})} 2^{-\frac4{\gamma^2}}$. Then
\[ \cM_2^\disk {(\gamma)}[ L_1^m e^{-A}] = Cc^{\frac4{\gamma^2}}\iint_0^\infty  \ell_1^m \cdot  (\ell_1 + \ell_2)^{-1} K_{\frac4{\gamma^2}}(c(\ell + r) ) \, d\ell_1 \, d\ell_2.\]
\end{lemma}
\begin{proof}
    \cite[Proposition 5.2]{AHS-SLE-integrability} gives the density function of $(L_1, L_2)$ as
    \[1_{\ell_1, \ell_2>0}\frac{(2\pi)^{\frac4{\gamma^2}-1}}{(1-\frac{\gamma^2}4) \Gamma(1 - \frac{\gamma^2}4)^{\frac4{\gamma^2}}} (\ell_1 + \ell_2)^{-\frac4{\gamma^2}-1} \, d\ell_1 \, d\ell_2.\]
    Next, the conditional law of $A$ given $L_1, L_2$ was obtained in \cite[Theorem 1.2]{ag-disk} modulo the mating-of-trees variance which was later computed in \cite[Theorem 1.3]{ARS-FZZ}, from which we get:
    \[\cM_2^\disk(\gamma)[e^{- A} \mid L_1, L_2] = \frac2{\Gamma(\frac4{\gamma^2})} \left( \frac{1}{4\sin(\frac{\pi\gamma^2}4)}\right)^{\frac2{\gamma^2}} (L_1+L_2)^{\frac4{\gamma^2}} K_{\frac4{\gamma^2}}\left((L_1+L_2)\sqrt{\frac1{\sin (\frac{\pi \gamma^2}4)}}\right). \]
    Combining the two indented equations gives the result. 
\end{proof}

\begin{lemma}\label{lem-mot-input}
 Suppose $k\in \mathbb N$ satisfies $k+1 > \frac4{\gamma^2}$ and $\mu \in (0, c)$. Then with $C' = \frac{1}{2} c^ {\frac{4}{\gamma^2} - 1 } C$:
	\[\cM_2^\disk(\gamma) [(-L_1)^k e^{-A - \mu L_1}] = C' \sum_{i=0}^\infty \frac{(-2/c)^{k+i}}{i! (k+i+1)} \Gamma(\frac12(k+i+1-\frac4{\gamma^2})) \Gamma(\frac12(k+i+1+\frac4{\gamma^2})) \mu^i. \]
\end{lemma}

\begin{proof}
	Suppose $m+1 > \frac4{\gamma^2}$. 
	Then
	\begin{align*}
	\cM_2^\disk (\gamma)[ L_1^m e^{-A}] &= Cc^{\frac4{\gamma^2}}\iint_0^\infty  \ell_1^m \cdot  (\ell_1 + \ell_2)^{-1} K_{\frac4{\gamma^2}}(c(\ell_1 + \ell_2) ) \, d\ell_1 \, d\ell_2\\
	&= Cc^{\frac4{\gamma^2}}\int_0^\infty\int_0^x \ell_1^m \cdot x^{-1} K_{\frac4{\gamma^2}}(cx) \, d\ell_1 \, dx\\
	&= Cc^{\frac4{\gamma^2}}\int_0^\infty \frac1{m+1} \cdot x^{m} K_{\frac4{\gamma^2}}(cx)  \, dx\\
	&= Cc^{\frac4{\gamma^2} - m - 1} \frac{2^{m-1}}{m+1} \Gamma(\frac12(m+1 - \frac4{\gamma^2})) \Gamma(\frac12(m+1 + \frac4{\gamma^2})).
	\end{align*}
 The first step used Lemma~\ref{lem-MOT-input-0}. The integral in the last step was evaluated via \cite[(10.43.19)]{NIST:DLMF}, and the condition $m+1 > \frac4{\gamma^2}$ is necessary for this step. Now, using the Taylor expansion $e^{-\mu L} = \sum_{i=0}^\infty \frac{(-\mu L)^i}{i!}$ gives the result. 	
\end{proof}

\begin{proof}[Proof of Lemma~\ref{lem-mot-match}]
	By Lemma~\ref{lem-mot-input}, gathering equal powers of $\mu$ and simplifying gives
	\[\cM_2^\disk(\gamma)[(\mu(-L_1)^k + k (-L_1)^{k-1}) e^{-A - \mu L_1}] = C' \sum_{i=0}^\infty \frac{(-2/c)^{k+i-1} }{i! } \Gamma(\frac12(k+i-\frac4{\gamma^2})) \Gamma(\frac12(k+i+\frac4{\gamma^2})) \mu^i, \]
 with the constant $C'$ here being the same as in Lemma~\ref{lem-mot-input}.
	Now the result is immediate from the following identity, (see e.g.\ \cite[Lemma 4.15]{ARS-FZZ}), 
	\[\cos (a \cos^{-1} (x)) = \frac {a\sin(\pi a)}{2\pi}  \sum_{n=0}^\infty \frac{(-2)^{n-1}}{n!} \Gamma(\frac12(n+a))\Gamma(\frac12(n-a))x^n   \textrm{ for }a \in\mathbb R\backslash \mathbb Z \textrm{ and }x \in (-1, 1),\]
which by setting $a = \frac{4}{\gamma^2}$ and $x = \frac{\mu}{c} $ implies:
\[ f^{(k)}(\mu) = \frac{2}{\pi \gamma^2 c} \sin(\frac{4 \pi}{\gamma^2})   \sum_{n=0}^\infty \frac{(-2/c)^{n + k -1}}{n!} \Gamma(\frac12(n + k + \frac{4}{\gamma^2}))\Gamma(\frac12(n + k - \frac{4}{\gamma^2})) \mu^n.\]
The constant $\mathcal{C}_1$ is given by: \(\mathcal{C}_1 = \frac{2}{\pi \gamma^2 c} \sin(\frac{4 \pi}{\gamma^2}) \frac{1}{C'} = \frac{\pi(\frac{4}{\gamma^2} -1)}{\Gamma(1 - \frac{4}{\gamma^2})} \left( \frac{\pi \Gamma( \frac{\gamma^2}{4} ) }{\Gamma( 1 - \frac{\gamma^2}{4} )} \right)^{ - \frac{2}{\gamma^2}}.\)
\end{proof}

\appendix

\section{Background on special functions}\label{sec:special}

\subsection{The hypergeometric equation}\label{sec:hyp_eq}

Here we recall some facts we have used on the hypergeometric equation and its solution space.
For $A>0$ let $\Gamma(A) = \int_0^{\infty} t^{A-1} e^{-t} dt $ denote the standard Gamma function which can then be analytically extended to $\mathbb{C} \setminus \{ - \mathbb{N} \} $. Record the following properties:
\begin{align}\label{prop_gamma}
\Gamma(A+1) = A \Gamma(A), \quad \Gamma(A) \Gamma(1-A) = \frac{\pi}{\sin(\pi A)}, \quad \Gamma(A) \Gamma(A + \frac{1}{2}) =  \sqrt{\pi} 2^{1 - 2A}  \Gamma(2A).
\end{align}
Let $(A)_n : = \frac{\Gamma(A +n)}{\Gamma(A)}$. For $A, B, C,$ and $t$ real numbers we define the hypergeometric function $F$ by:
\begin{equation}
F(A,B,C,t) :=  \sum_{n=0}^{\infty} \frac{(A)_n (B)_n}{n! (C)_n} t^n.
\end{equation}
This function can be used to solve the following hypergeometric equation:
\begin{equation}
\left( t (1-t) \frac{d^2}{d t^2} + ( C - (A +B +1)t) \frac{d}{dt} - AB\right) f(t) =0.
\end{equation}
We have  three bases of solutions corresponding respectively to $t \in (0,1)$, $t \in (1, + \infty)$ and $t \in (- \infty, 0)$.  For   $t \in (0,1)$, under the assumption that $C$  and $ C - A - B $ are not integers,  we have: 
\begin{align}
f(t) &= C_1 F(A,B,C,t) + C_2^+ t^{1 -C} F(1 + A-C, 1 +B - C, 2 -C, t) \\ \nonumber
&= B_1 F(A,B,1+A+B- C, 1 -t)\\ \nonumber
& + B_2^- (1-t)^{C- A - B} F(C- A, C- B, 1 + C - A - B , 1 -t).
\end{align}
For $t \in (- \infty,0)$, under the assumption that $C$ and $ A - B $ are not integers:
\begin{align}
f(t) &= C_1 F(A,B,C,t) + C_2^- t^{1 -C} F(1 + A-C, 1 +B - C, 2 -C, t) \\ \nonumber
&= D_1 t^{-A}F(A,1+A-C,1+A-B,t^{-1})\\  \nonumber
& + D_2^+ t^{-B} F(B, 1 +B - C, 1 +B - A, t^{-1}).
\end{align}
The case $t\in (1,\infty)$ is similar, but we do not need it. For each of the three cases we have four real constants that parametrize the solution space. 
There is an explicit change of basis formula that in particular gives a link between $C_1, C_2^+$, $B_1, B_2^-$, and 
similarly a link between $C_1,C_2^-,D_1,D_2^+$.
These are  the so-called connection formulas:
\begin{equation}\label{hpy1}
\begin{pmatrix}
C_1  \\
C_2^-
\end{pmatrix} 
= 
\begin{pmatrix}
\frac{\Gamma(1 -C ) \Gamma( A - B +1  )  }{\Gamma(A - C +1 )\Gamma(1- B   )  } & \frac{\Gamma(1 -C) \Gamma( B- A +1 )  }{\Gamma(B - C +1 )\Gamma( 1- A )   } \\
\frac{\Gamma(C-1 ) \Gamma( A - B +1  )  }{\Gamma( A  )  \Gamma( C - B  )} & \frac{\Gamma(C-1 )  \Gamma(B - A +1 ) }{ \Gamma(  B ) \Gamma( C - A )  }
\end{pmatrix} 
\begin{pmatrix}
D_1  \\
D_2^+ 
\end{pmatrix},
\end{equation}

\begin{equation}\label{connection1}
\begin{pmatrix}
B_1  \\
B_2^-
\end{pmatrix} 
= 
\begin{pmatrix}
\frac{\Gamma(C)\Gamma(C-A-B)}{\Gamma(C-A)\Gamma(C-B)} & \frac{\Gamma(2-C)\Gamma(C-A-B)}{\Gamma(1-A)\Gamma(1-B)}  \\
\frac{\Gamma(C)\Gamma(A+B-C)}{\Gamma(A)\Gamma(B)} &  \frac{\Gamma(2-C)\Gamma(A+B-C)}{\Gamma(A-C+1)\Gamma(B-C+1)}
\end{pmatrix} 
\begin{pmatrix}
C_1  \\
C_2^+ 
\end{pmatrix}.
\end{equation}

\subsection{Double Gamma and Double Sine functions}\label{sec:double_gamma}
We will now provide some explanations on the functions $\Gamma_{\frac{\gamma}{2}}(x)$ and $S_{\frac{\gamma}{2}}(x)$ that we have introduced. For all $\gamma \in (0,2) $ and for $\Re(x) >0$, $\Gamma_{\frac{\gamma}{2}}(x)$ is defined by the integral formula
\begin{equation}
\log \Gamma_{\frac{\gamma}{2}}(x) = \int_0^{\infty} \frac{dt}{t} \left[ \frac{ e^{-xt} -e^{- \frac{Qt}{2}}   }{(1 - e^{- \frac{\gamma t}{2}})(1 - e^{- \frac{2t}{\gamma}})} - \frac{( \frac{Q}{2} -x)^2 }{2}e^{-t} + \frac{ x -\frac{Q}{2}  }{t} \right],
\end{equation} 
where we have $Q = \frac{\gamma}{2} +\frac{2}{\gamma}$. Since the function $\Gamma_{\frac{\gamma}{2}}(x)$ is continuous it is completely determined by the following two shift equations 
\begin{align}\label{eq:shift_G1}
\frac{\Gamma_{\frac{\gamma}{2}}(x)}{\Gamma_{\frac{\gamma}{2}}(x + \frac{\gamma}{2}) }&= \frac{1}{\sqrt{2 \pi}}
\Gamma(\frac{\gamma x}{2}) ( \frac{\gamma}{2} )^{ -\frac{\gamma x}{2} + \frac{1}{2}
}, \\ \label{eq:shift_G2}
\frac{\Gamma_{\frac{\gamma}{2}}(x)}{\Gamma_{\frac{\gamma}{2}}(x + \frac{2}{\gamma}) }&= \frac{1}{\sqrt{2 \pi}} \Gamma(\frac{2
	x}{\gamma}) ( \frac{\gamma}{2} )^{ \frac{2 x}{\gamma} - \frac{1}{2} },
\end{align}
and by its value at $\frac{Q}{2}$, $\Gamma_{\frac{\gamma}{2}}(\frac{Q}{2} ) =1$. Furthermore $x \mapsto \Gamma_{\frac{\gamma}{2}}(x)$ admits a meromorphic extension to all of $\mathbb{C}$ with single poles at $x = -n\frac{\gamma}{2}-m\frac{2}{\gamma}$  for any $n,m \in \mathbb{N}$ and $\Gamma_{\frac{\gamma}{2}}(x)$ is never equal to $0$.  We also need the double sine function defined by:
\begin{equation} 
S_{\frac{\gamma}{2}}(x) = \frac{\Gamma_{\frac{\gamma}{2}}(x)}{\Gamma_{\frac{\gamma}{2}}(Q -x)}.
\end{equation}
It obeys the following two shift equations:
\begin{align}\label{eq:shift_S1}
\frac{S_{\frac{\gamma}{2}}(x+\frac{\gamma}{2})}{S_{\frac{\gamma}{2}}(x)} = 2\sin(\frac{\gamma\pi}{2}x), \quad\textrm{and}\quad \frac{S_{\frac{\gamma}{2}}(x+\frac{2}{\gamma})}{S_{\frac{\gamma}{2}}(x)} = 2\sin(\frac{2\pi}{\gamma}x).
\end{align}
The double sine function admits a meromorphic extension to $\mathbb{C}$ with poles at $x = -n\frac{\gamma}{2}-m\frac{2}{\gamma}$ and with zeros at $x = Q+n\frac{\gamma}{2}+m\frac{2}{\gamma}$ for any $n,m \in \mathbb{N}$. We also record the following asymptotic for $S_{\frac{\gamma}{2}}(x) $ which can be found in e.g.~\cite[Equation (B.52)]{Lorenz_notes}:
\begin{align}\label{eq:lim_S}
S_{\frac{\gamma}{2}}(x) \sim 
\begin{cases}
e^{-\ii\frac{\pi}{2}x(x-Q)} e^{- \ii \frac{\pi}{12}(Q^2 +1)} & \text{as} \quad \textnormal{Im}(x) \to \infty,\\
e^{\ii\frac{\pi}{2}x(x-Q)} e^{ \ii \frac{\pi}{12}(Q^2 +1)} &\text{as} \quad \textnormal{Im}(x) \to -\infty.
\end{cases}
\end{align}

\subsection{Properties of the Ponsot-Teschner function $\HPT $ and the Hosimichi function $G_{\mathrm{Hos}}$}\label{subsec:PT}

For the purpose of proving Theorem \ref{thm:H}, we establish the following two facts about $\HPT$.

\begin{lemma}\label{lem:secH1}
The function $\HPT$ satisfies the shift equations of Theorem \ref{full_shift_H} satisfied by $H$.
\end{lemma}

\begin{proof}
Recalling \eqref{formule_PT}, we introduce a function $\varphi$ in the following way
\begin{equation}
\HPT
\begin{pmatrix}
\beta_1 , \beta_2, \beta_3 \\
\sigma_1,  \sigma_2,   \sigma_3 
\end{pmatrix} =  \int_{\mathcal{C}} \varphi^{(\beta_1 , \beta_2, \beta_3)}_{(  \sigma_1,  \sigma_2,   \sigma_3)}(r) dr,
\end{equation}
where here $\varphi^{(\beta_1 , \beta_2, \beta_3)}_{(  \sigma_1,  \sigma_2,   \sigma_3)}(r)$ contains the $r$ dependent integrand of the contour integral $\int_{\mathcal{C}}$ present in \eqref{formule_PT} times the prefactor in front of the integral (which does not depend on $r$).

Checking that $\HPT$ satisfies the shift equations of Theorem \ref{full_shift_H} is equivalent to checking the following shift equations

\begin{align}\label{shift three point 3}
\HPT
\begin{pmatrix}
\beta_1 , \beta_2, \beta_3 \\
\sigma_1,  \sigma_2,   \sigma_3 
\end{pmatrix} =& \frac{\Gamma(\chi(\beta_1-\chi)) \Gamma(1 - \chi\beta_2)}{\Gamma(1-\frac{\chi}{2}(\beta_2+\beta_3-\beta_1)) \Gamma(\frac{\chi}{2}(\beta_1+\beta_3-\beta_2-2\chi)) }\HPT
\begin{pmatrix}
\beta_1-\chi , \beta_2 + \chi, \beta_3 \\
\sigma_1,  \sigma_2 + \frac{\chi}{2},   \sigma_3 
\end{pmatrix}  \\ \nonumber 
& + \chi^2 \left( \frac{\pi \Gamma(\frac{\gamma^2}{4})}{\Gamma(1-\frac{\gamma^2}{4})}\right)^{\frac{\chi}{\gamma}} \frac{\Gamma(1-\chi\beta_1) \Gamma(1 - \chi\beta_2 )}{\sin(\pi\chi(\chi-\beta_1)) \Gamma(1 + \chi(Q-\frac{\bar{\beta}}{2})) \Gamma(1-\frac{\chi}{2}(\beta_1+\beta_2-\beta_3))   }\\
&\quad \times 2 \sin(\pi\chi(\frac{\beta_1}{2}+\sigma_1+\sigma_2-Q)) \sin(\pi\chi(\frac{\beta_1}{2}-\sigma_1+\sigma_2)) \HPT
\begin{pmatrix}
\beta_1 + \chi , \beta_2 + \chi, \beta_3 \\
\sigma_1,  \sigma_2 + \frac{\chi}{2},   \sigma_3 
\end{pmatrix} \nonumber,
\end{align}
and:
\begin{align}\label{shift three point 4}
&\frac{\chi^2}{\pi}  \Gamma(1-\chi\beta_2) 2 \sin(\pi\chi(\frac{\beta_2}{2}+\sigma_2+\sigma_3-Q)) \sin(\pi\chi(\frac{\beta_2}{2}+\sigma_2-\sigma_3)) \HPT
\begin{pmatrix}
\beta_1 + \chi , \beta_2 + \chi, \beta_3 \\
\sigma_1,  \sigma_2 + \frac{\chi}{2},   \sigma_3 
\end{pmatrix}\\
&= \left( \frac{\pi  \Gamma(\frac{\gamma^2}{4})}{\Gamma(1-\frac{\gamma^2}{4})}\right)^{-\frac{\chi}{\gamma}}\frac{\Gamma(\chi \beta_1)}{\Gamma( \frac{\chi}{2}(\bar{\beta}-2Q) ) \Gamma(\frac{\chi}{2}(\beta_1+\beta_2-\beta_3) ) } \HPT
\begin{pmatrix}
\beta_1, \beta_2, \beta_3 \\
\sigma_1,  \sigma_2,   \sigma_3 
\end{pmatrix} \nonumber \\
& - \chi^2  \frac{2 \sin(\pi\chi(\frac{\beta_1}{2}-\sigma_1-\sigma_2+Q)) \sin(\pi\chi(\frac{\beta_1}{2}+\sigma_1-\sigma_2)) \Gamma(1-\chi\beta_1 -\chi^2) }{\sin(\pi\chi\beta_1)\Gamma(\frac{\chi}{2}(\beta_2+\beta_3-\beta_1-2\chi) ) \Gamma( 1-\frac{\chi}{2}(\beta_1+\beta_3-\beta_2) ) } \HPT
\begin{pmatrix}
\beta_1 + 2 \chi, \beta_2, \beta_3 \\
\sigma_1,  \sigma_2,   \sigma_3 
\end{pmatrix}. \nonumber
\end{align}
This formulation of the two shift equations has been derived from Theorem \ref{full_shift_H} by shifting $\beta_2$ to $\beta_2 +\chi$ in \eqref{shift three point 1} for the first and shifting $\beta_1$ to $\beta_1 +\chi$ in \eqref{shift three point 2} for the second. We have also used the explicit expression \eqref{eq:g-chi} for $g_{\chi}$ and written the difference of cosines as a product of sines. 
We now compute the following ratios of $\varphi$:
\begin{align}
\frac{\varphi^{(\beta_1-\chi , \beta_2+\chi, \beta_3)}_{(  \sigma_1,  \sigma_2 +\frac{\chi}{2},   \sigma_3)}(r)}{\varphi^{(\beta_1 , \beta_2, \beta_3)}_{(  \sigma_1,  \sigma_2,   \sigma_3)}(r)} =& \frac{\Gamma(\frac{\chi}{2}(\beta_1+\beta_3-\beta_2-2\chi)) \Gamma(1-\frac{\chi}{2}(\beta_2+\beta_3-\beta_1)) \Gamma(1-\chi\beta_1+\chi^2) }{ \pi \Gamma(1-\chi \beta_2) }\\ \nonumber
&\times \sin(\pi \chi (\frac{\beta_1}{2}-\chi+\sigma_1-\sigma_2))\frac{ \sin(\pi \chi (-\frac{\beta_1}{2}+\frac{\beta_2}{2}+\sigma_1-\sigma_3-r))}{\sin(\pi \chi (\frac{\beta_2}{2}+\sigma_2-\sigma_3-r) )},\\
\frac{\varphi^{(\beta_1+\chi , \beta_2+\chi, \beta_3)}_{(  \sigma_1,  \sigma_2 +\frac{\chi}{2},   \sigma_3)}(r)}{\varphi^{(\beta_1 , \beta_2, \beta_3)}_{(  \sigma_1,  \sigma_2,   \sigma_3)}(r)} =&  \chi^{-2}\left(\frac{\pi \Gamma(\frac{\gamma^2}{4})}{\Gamma(1-\frac{\gamma^2}{4})} \right)^{-\frac{\chi}{\gamma}} \frac{\Gamma(1+\chi(Q-\frac{\bar{\beta}}{2})) \Gamma(1-\frac{\chi}{2}(\beta_1+\beta_2-\beta_3))}{\Gamma(1-\chi\beta_1) \Gamma(1-\chi\beta_2)}\\
&\times \frac{\sin(\pi\chi(\frac{\beta_1}{2}+\frac{\beta_2}{2}-\chi+\sigma_1-\sigma_3-r))}{2\sin(\pi\chi(\frac{\beta_1}{2}-\chi+\sigma_1+\sigma_2))\sin(\pi\chi(\frac{\beta_2}{2}+\sigma_2-\sigma_3-r))}. \nonumber
\end{align}
If we plug these expressions into equation \eqref{shift three point 3} and regroup the terms on one side we get:
\begin{align}
\int_{\mathcal{C}} dr \, \varphi^{(\beta_1 , \beta_2, \beta_3)}_{(  \sigma_1,  \sigma_2,   \sigma_3)}(r)\Bigg[\frac{ \sin(\pi\chi(\frac{\beta_1}{2}-\chi+\sigma_1-\sigma_2))\sin(\pi\chi (-\frac{\beta_1}{2}+\frac{\beta_2}{2}+\sigma_1-\sigma_3-r))}{\sin(\pi\chi(\beta_1-\chi))\sin(\pi \chi(\frac{\beta_2}{2}+\sigma_2-\sigma_3-r) )}-1\\ \nonumber
+\frac{\sin(\pi\chi(\frac{\beta_1}{2}-\sigma_1+\sigma_2))\sin(\pi\chi(\frac{\beta_1}{2}+\frac{\beta_2}{2}-\chi+\sigma_1-\sigma_3-r))}{\sin(\pi\chi(\beta_1-\chi))\sin(\pi\chi(\frac{\beta_2}{2}+\sigma_2-\sigma_3-r))} \Bigg].
\end{align}
We can verify with some algebra that the integrand of the above integral equals $0$, hence \eqref{shift three point 3} holds. To check the second shift equation, we will need additionally the ratio:
\begin{align*}
&\frac{\varphi^{(\beta_1+2\chi , \beta_2, \beta_3)}_{(  \sigma_1,  \sigma_2 ,   \sigma_3)}(r)}{\varphi^{(\beta_1 , \beta_2, \beta_3)}_{(  \sigma_1,  \sigma_2,   \sigma_3)}(r)} =  - \frac{\pi}{\chi^2}  \frac{\Gamma(1+\chi(Q-\frac{\bar{\beta}}{2})) \Gamma(1-\frac{\chi}{2}(\beta_1+\beta_2-\beta_3))}{\Gamma(\frac{\chi}{2}(\beta_1+\beta_3-\beta_2)) \Gamma(1-\frac{\chi}{2}(\beta_2+\beta_3-\beta_1-2\chi)) \Gamma(1-\chi\beta_1-\chi^2)) \Gamma(1-\chi\beta_1) }\\
&\quad \times \left(\frac{\pi \Gamma(\frac{\gamma^2}{4})}{\Gamma(1-\frac{\gamma^2}{4})} \right)^{-\frac{\chi}{\gamma}} \frac{\sin(\pi\chi(\frac{\beta_1}{2}+\frac{\beta_2}{2}-\chi+\sigma_1-\sigma_3-r))}{2  \sin(\pi\chi(\frac{\beta_1}{2}+\sigma_1-\sigma_2)) \sin(\pi\chi(\frac{\beta_1}{2}+\sigma_1+\sigma_2-\chi)) \sin(\pi\chi(\frac{\beta_1}{2}-\frac{\beta_2}{2}+\chi+\sigma_3-\sigma_1+r))}.
\end{align*}
Substituting this time into equation \eqref{shift three point 4} and regrouping terms on one side we get:
\begin{small}
\begin{align*}
&\frac{\Gamma(\chi \beta_1) }{\Gamma( \frac{\chi}{2}(\overline{\beta}-2Q) ) \Gamma(\frac{\chi}{2}(\beta_1+\beta_2-\beta_3)) } \int_{\mathcal{C}} dr \,  \varphi^{(\beta_1 , \beta_2, \beta_3)}_{(  \sigma_1,  \sigma_2,   \sigma_3)}(r)\\
 &\Bigg[ \frac{\sin(\pi\chi\beta_1)\sin(\pi\chi(\frac{\beta_2}{2}-\chi+\sigma_2+\sigma_3)) \sin(\pi\chi(\frac{\beta_2}{2}+\sigma_2-\sigma_3))\sin(\pi\chi(\frac{\beta_1}{2}+\frac{\beta_2}{2}-\chi+\sigma_1-\sigma_3-r))}{\sin(\pi\chi(\frac{\bar{\beta}}{2}-\chi)) \sin(\frac{\pi\chi}{2}(\beta_1+\beta_2-\beta_3)) \sin(\pi\chi(\frac{\beta_1}{2}-\chi+\sigma_1+\sigma_2))\sin(\pi\chi(\frac{\beta_2}{2}+\sigma_2-\sigma_3-r))}-1\\
& - \frac{\sin(\frac{\pi\chi}{2}(\beta_1+\beta_3-\beta_2)) \sin(\frac{\pi\chi}{2}(\beta_2+\beta_3-\beta_1-2\chi))\sin(\pi\chi(\frac{\beta_1}{2}+\chi-\sigma_1-\sigma_2)) \sin(\pi\chi(\frac{\beta_1}{2}+\frac{\beta_2}{2}-\chi+\sigma_1-\sigma_3-r))}{\sin(\chi(\frac{\bar{\beta}}{2}-\chi)) \sin(\frac{\pi\chi}{2}(\beta_1+\beta_2-\beta_3)) \sin(\pi\chi(\frac{\beta_1}{2}-\chi+\sigma_1+\sigma_2)) \sin(\pi\chi(\frac{\beta_1}{2}-\frac{\beta_2}{2}+\chi+\sigma_3-\sigma_1+r))} \Bigg].
\end{align*}
\end{small}
After some algebra we can write this quantity in the form:
\begin{small}
\begin{align*}
&\frac{\Gamma(\chi \beta_1) }{\Gamma( \frac{\chi}{2}(\overline{\beta}-2Q) ) \Gamma(\frac{\chi}{2}(\beta_1+\beta_2-\beta_3)) } \frac{\sin(\pi\chi\beta_1)}{\sin(\pi\chi(\frac{\overline{\beta}}{2}-\chi)) \sin(\frac{\pi \chi}{2}(\beta_1+\beta_2-\beta_3)) \sin(\pi\chi(\frac{\beta_1}{2}-\chi+\sigma_1+\sigma_2))} \\
&\times \int_{\mathcal{C}} dr \,  \varphi^{(\beta_1 , \beta_2, \beta_3)}_{(  \sigma_1,  \sigma_2,   \sigma_3)}(r)\Bigg[\frac{\sin(\pi\chi(\frac{\beta_1}{2}+\frac{\beta_2}{2}-\chi+\sigma_1-\sigma_3-r))\sin(\pi\chi r)\sin(\pi\chi(2\sigma_3-\chi+r))}{\sin(\pi\chi(\frac{\beta_2}{2}+\sigma_2-\sigma_3-r))} \\
& \quad \quad - \frac{\sin(\pi\chi(\frac{\beta_3}{2}+\sigma_3-\sigma_1+r)) \sin(\pi\chi(\frac{\beta_3}{2}-\chi+\sigma_1-\sigma_3-r)) \sin(\pi\chi(\frac{\beta_2}{2}-\sigma_2-\sigma_3-r))}{\sin(\pi\chi(\frac{\beta_1}{2}-\frac{\beta_2}{2}+\chi+\sigma_3-\sigma_1+r))} \Bigg]\\
& =\frac{\Gamma(\chi \beta_1) }{\Gamma( \frac{\chi}{2}(\overline{\beta}-2Q) ) \Gamma(\frac{\chi}{2}(\beta_1+\beta_2-\beta_3)) } \\
& \times \frac{\sin(\pi\chi\beta_1)\sin(\pi\chi(\frac{\beta_3}{2}+\sigma_3-\sigma_1)) }{\sin(\pi\chi(\frac{\overline{\beta}}{2}-\chi)) \sin(\frac{\pi \chi}{2}(\beta_1+\beta_2-\beta_3))\sin(\pi\chi(\frac{\beta_1}{2}-\chi+\sigma_1-\sigma_2)) \sin(\pi\chi(\frac{\beta_3}{2}-\sigma_1-\sigma_3))} \\
&\times \int_{\mathcal{C}} dr \, (T_{-\chi}-1) \left( \frac{\sin(\pi\chi(\frac{\beta_1}{2}+\frac{\beta_2}{2}-\chi+\sigma_1-\sigma_3-r)) \sin(\pi\chi(2\sigma_3-\chi+r)) \sin(\pi\chi(2\sigma_3-2\chi+r))}{\sin(\pi\chi(\frac{\beta_2}{2}+\chi-\sigma_2-\sigma_3-r))}\varphi^{(\beta_1 , \beta_2, \beta_3)}_{(  \sigma_1,  \sigma_2+\chi,   \sigma_3+\chi)}(r) \right).
\end{align*}
\end{small}

Here we have used the notation $T_{-\chi}$ for the operator that shifts the argument of the function it is applied to by $-\chi$, the variable we are shifting being $r$. Since we have the combination $(T_{-\chi}-1)$ in the integrand, this corresponds to integrating $r$ over two contours that are separated by a horizontal shift of $\chi$ and which are in the opposite direction. Provided that there are no poles in between these two contours, the whole contour integral will be equal to $0$. This is indeed the case thanks to the way the contour $\mathcal{C}$ has been chosen, which is to the right of the lattice $P^{-}_{\mathrm{PT}}$ of poles extending in the $-\infty$ direction, and to the left of the lattice $P^{+}_{\mathrm{PT}}$ of poles extending in the $+\infty$ direction.
\end{proof}

Let us now perform a residue computation directly for the contour integral part of $\HPT$, namely $\mathcal{J}_{\mathrm{PT}}$, without multiplying by the global prefactor. See equation \eqref{J_PT_app} for a precise definition. This gives the following result.

\begin{lemma}\label{lem:secH2}
Set $\overline{\beta} = \beta_1 + \beta_2 + \beta_3$. The function $\mathcal{J}_{\mathrm{PT}} $ satisfies the following properties:
\begin{equation}\label{eq:PT-property1}
\lim_{\beta_1\to 2Q-\beta_2-\beta_3}(\frac{\bar{\beta}}{2}-Q) \mathcal{J}_{\mathrm{PT}}
\begin{pmatrix}
\beta_1 , \beta_2, \beta_3 \\
\sigma_1,  \sigma_2,   \sigma_3 
\end{pmatrix} = \frac{1}{2\pi } \frac{S_{\frac{\gamma}{2}}(\frac{\beta_1}{2}+\sigma_1+\sigma_2-Q) S_{\frac{\gamma}{2}}(\frac{\beta_1}{2}+\sigma_1-\sigma_2)S_{\frac{\gamma}{2}}(Q-\beta_3)}{S_{\frac{\gamma}{2}}(\beta_1) S_{\frac{\gamma}{2}}(-\frac{\beta_3}{2}+\sigma_1+\sigma_3) S_{\frac{\gamma}{2}}(Q-\frac{\beta_3}{2}+\sigma_1-\sigma_3)}.
\end{equation}
\begin{align}
&\lim_{\beta_1\to 2Q-\beta_2-\beta_3 - \gamma}(\frac{\bar{\beta}}{2}-Q + \frac{\gamma}{2}) \mathcal{J}_{\mathrm{PT}}
\begin{pmatrix}
\beta_1 , \beta_2, \beta_3 \\
\sigma_1,  \sigma_2,   \sigma_3 
\end{pmatrix}\label{eq:PT-property2}\\
&=  - \frac{1  }{ \pi \sin( \frac{\gamma^2 \pi}{4} )  }  \frac{S_{\frac{\gamma}{2}}(\frac{2}{\gamma} -\beta_3 ) S_{\frac{\gamma}{2}}( \frac{2}{\gamma} + \frac{  - \beta_2 - \beta_3 }{2}  + \sigma_1  - \sigma_2) S_{\frac{\gamma}{2}}(\frac{ - \beta_2 - \beta_3 -\gamma}{2}  + \sigma_1  + \sigma_2)}{S_{\frac{\gamma}{2}}( Q + \frac{2}{\gamma} -\beta_2 - \beta_3 ) S_{\frac{\gamma}{2}}( - \frac{\beta_3}{2}  + \sigma_3 + \sigma_1) S_{\frac{\gamma}{2}}(Q - \frac{\beta_3}{2} -\sigma_3 + \sigma_1 )} \nonumber\\
& \times \left( \sin( \frac{\gamma \pi \beta_2}{2}) \cos( \frac{\gamma \pi}{2}( - \frac{\gamma}{2}  +2 \sigma_1 ) ) +  \sin( \frac{\gamma \pi \beta_3}{2}) \cos( \frac{\gamma \pi}{2}( - \frac{\gamma}{2}  +2 \sigma_2 ) ) -  \sin( \frac{\gamma \pi (\beta_2 + \beta_3)}{2}) \cos( \frac{\gamma \pi}{2}( - \frac{\gamma}{2}  +2 \sigma_3 ) ) \right).\nonumber
\end{align}
Furthermore the ratio of the second limit divided by the first is given by
\begin{align*}
& - \frac{ 1   }{ 2 \sin( \frac{\gamma^2 \pi}{4} )  } \frac{\sin( \frac{\pi \gamma}{2}(Q + \frac{2}{\gamma} - \beta_2 - \beta_3 ) )}{ \sin( \frac{\pi \gamma}{2}(\frac{2}{\gamma} - \beta_3 ) ) \sin( \frac{\pi \gamma}{2}(\frac{2}{\gamma} - \frac{\beta_2 + \beta_3 }{2} + \sigma_1 - \sigma_2 ) ) \sin( \frac{\pi \gamma}{2}( - \frac{\beta_2 + \beta_3 +\gamma }{2} + \sigma_1 - \sigma_2 ) ) }  \\
&\times \left( \sin( \frac{\gamma \pi \beta_2}{2}) \cos( \frac{\gamma \pi}{2}( - \frac{\gamma}{2}  +2 \sigma_1 ) ) +  \sin( \frac{\gamma \pi \beta_3}{2}) \cos( \frac{\gamma \pi}{2}( - \frac{\gamma}{2}  +2 \sigma_2 ) ) -  \sin( \frac{\gamma \pi (\beta_2 + \beta_3)}{2}) \cos( \frac{\gamma \pi}{2}( - \frac{\gamma}{2}  +2 \sigma_3 ) ) \right),
\end{align*}
which is a meromorphic function which is not identically $1$.
\end{lemma}

\begin{proof}
The evaluation of the residue is an easy algebra using the shift equations of $\Gamma_{\frac{\gamma}{2}}$ and $S_{\frac{\gamma}{2}}$. When $\beta_1$ approaches $2Q-\beta_2-\beta_3$ from the right hand side, the two poles at $r=-(\frac{\beta_3}{2}+\sigma_3-\sigma_1) $ and $r=-(Q-\frac{\beta_1}{2}-\frac{\beta_2}{2}+\sigma_3-\sigma_1)$ in the contour integral will collapse. To extract the divergent term, we can slightly modify the contour to let it go from the right hand side of $r=-(Q-\frac{\beta_1}{2}-\frac{\beta_2}{2}+\sigma_3-\sigma_1)$, this allows us to pick up the divergent term by the residue theorem. Let us carry this out in detail, we first start by writing the contour integral as:
\begin{align}\label{eq:HPT_values_1}
&\int_{\mathcal{C}} \frac{S_{\frac{\gamma}{2}}(\frac{Q - \beta_2}{2} + \sigma_3 \pm ( \frac{Q}{2} - \sigma_2) +r)  S_{\frac{\gamma}{2}}(\frac{Q}{2} \pm \frac{Q - \beta_3}{2}  +\sigma_3-\sigma_1+r)}{S_{\frac{\gamma}{2}}(\frac{3Q}{2} \pm \frac{Q - \beta_1}{2}-\frac{\beta_2}{2}+\sigma_3-\sigma_1+r) S_{\frac{\gamma}{2}}(2\sigma_3 + r) S_{\frac{\gamma}{2}}(Q+r)} \frac{dr}{i} \nonumber \\
&= \int_{\mathcal{C}} \frac{1}{4 \sin( \frac{2 \pi}{\gamma}( \frac{\beta_3}{2} + \sigma_3 - \sigma_1 +r )) \sin(\frac{2 \pi}{\gamma}(2 Q - \frac{2}{\gamma} - \frac{ \beta_1}{2}-\frac{\beta_2}{2}+\sigma_3-\sigma_1+r) )} f(r)  \frac{dr}{i},
\end{align}
where here:
\begin{align*}
f(r) = \frac{S_{\frac{\gamma}{2}}(\frac{Q - \beta_2}{2} + \sigma_3 \pm ( \frac{Q}{2} - \sigma_2) +r)  S_{\frac{\gamma}{2}}(Q - \frac{\beta_3}{2}  +\sigma_3-\sigma_1+r) S_{\frac{\gamma}{2}}( \frac{2}{\gamma } + \frac{\beta_3}{2}  +\sigma_3-\sigma_1+r)}{S_{\frac{\gamma}{2}}(2 Q - \frac{2}{\gamma} - \frac{ \beta_1}{2}-\frac{\beta_2}{2}+\sigma_3-\sigma_1+r) S_{\frac{\gamma}{2}}(Q + \frac{ \beta_1}{2}-\frac{\beta_2}{2}+\sigma_3-\sigma_1+r) S_{\frac{\gamma}{2}}(2\sigma_3 + r) S_{\frac{\gamma}{2}}(Q+r)}.
\end{align*}
We compute:
\begin{align*}
&f( \frac{\beta_1}{2} + \frac{\beta_2}{2} - \sigma_3 + \sigma_1 - Q) = \frac{S_{\frac{\gamma}{2}}(\frac{ \beta_1 - Q}{2} + \sigma_1 \pm ( \frac{Q}{2} - \sigma_2))  S_{\frac{\gamma}{2}}( \frac{\beta_1}{2} + \frac{\beta_2}{2} - \frac{\beta_3}{2} ) S_{\frac{\gamma}{2}}( \frac{2}{\gamma } + \frac{\overline{\beta}}{2} - Q)}{S_{\frac{\gamma}{2}}(\frac{\gamma}{2}) S_{\frac{\gamma}{2}}( \beta_1) S_{\frac{\gamma}{2}}(- Q + \frac{\beta_1}{2} + \frac{\beta_2}{2} + \sigma_3 + \sigma_1) S_{\frac{\gamma}{2}}(\frac{\beta_1}{2} + \frac{\beta_2}{2} - \sigma_3 + \sigma_1)},\\
&f( \frac{\beta_1}{2} + \frac{\beta_2}{2} - \sigma_3 + \sigma_1 - \frac{2}{\gamma}) = \frac{S_{\frac{\gamma}{2}}(\frac{ \beta_1 - Q}{2} + \sigma_1 \pm ( \frac{Q}{2} - \sigma_2) + \frac{\gamma}{2} )  S_{\frac{\gamma}{2}}( \frac{\beta_1}{2} + \frac{\beta_2}{2} - \frac{\beta_3}{2} + \frac{\gamma}{2} ) S_{\frac{\gamma}{2}}( \frac{\overline{\beta}}{2} )}{S_{\frac{\gamma}{2}}(\gamma) S_{\frac{\gamma}{2}}( \beta_1 + \frac{\gamma}{2})  S_{\frac{\gamma}{2}}(- \frac{2}{\gamma} + \frac{\beta_1}{2} + \frac{\beta_2}{2} + \sigma_3 + \sigma_1) S_{\frac{\gamma}{2}}(\frac{\gamma}{2} +  \frac{\beta_1}{2} + \frac{\beta_2}{2} - \sigma_3 + \sigma_1)}
\end{align*}

Let first start by checking the limit $\beta_1\to 2Q-\beta_2-\beta_3$. By a simple residue computation at the pole $r = -(Q-\frac{\beta_1}{2}-\frac{\beta_2}{2}+\sigma_3-\sigma_1)$ and by using the value of $f(r)$ at that point, one obtains that as $\beta_1\to 2Q-\beta_2-\beta_3$, the right hand side of \eqref{eq:HPT_values_1} is equivalent to:
\begin{align*}
\frac{1}{2\pi (\frac{\overline{\beta}}{2}-Q)} \frac{S_{\frac{\gamma}{2}}(\frac{\beta_1}{2}+\sigma_1+\sigma_2-Q) S_{\frac{\gamma}{2}}(\frac{\beta_1}{2}+\sigma_1-\sigma_2)S_{\frac{\gamma}{2}}(Q-\beta_3)}{S_{\frac{\gamma}{2}}(\beta_1) S_{\frac{\gamma}{2}}(-\frac{\beta_3}{2}+\sigma_1+\sigma_3) S_{\frac{\gamma}{2}}(Q-\frac{\beta_3}{2}+\sigma_1-\sigma_3)}.
\end{align*}
Therefore we obtain~\eqref{eq:PT-property1}.

We now move to the case of the limit $\beta_1 \rightarrow 2Q - \beta_2 - \beta_3 -\gamma$. This time we need to move the contour of integration $\mathcal{C}$ to the right of the poles at $r= -(Q-\frac{\beta_1}{2}-\frac{\beta_2}{2}+\sigma_3-\sigma_1)$ and $r=-(Q-\frac{\beta_1}{2}-\frac{\beta_2}{2}+\sigma_3-\sigma_1) + \frac{\gamma}{2}$.
We get two contributions which in the limit $\beta_1 \rightarrow 2Q - \beta_2 - \beta_3 - \gamma$ are equivalent to:
\begin{align*}
&\frac{1}{ \frac{4}{\gamma} \sin( \frac{2 \pi}{\gamma}( \frac{\overline{\beta}}{2} - Q )) } \frac{S_{\frac{\gamma}{2}}( \frac{2}{\gamma} + \frac{  - \beta_2 - \beta_3 }{2}  + \sigma_1  - \sigma_2) S_{\frac{\gamma}{2}}(\frac{ - \beta_2 - \beta_3 -\gamma}{2}  + \sigma_1  + \sigma_2) S_{\frac{\gamma}{2}}(Q  - \beta_3 - \frac{\gamma}{2} ) S_{\frac{\gamma}{2}}( \frac{2}{\gamma } - \frac{\gamma}{2} )}{S_{\frac{\gamma}{2}}(\frac{\gamma}{2}) S_{\frac{\gamma}{2}}( \frac{4}{\gamma} - \beta_2 -\beta_3 ) S_{\frac{\gamma}{2}}( - \frac{\beta_3}{2} - \frac{\gamma}{2} + \sigma_3 + \sigma_1) S_{\frac{\gamma}{2}}(Q - \frac{\beta_3}{2} - \frac{\gamma}{2} -\sigma_3 + \sigma_1 )} \\
& + \frac{1}{ \frac{4}{\gamma} \sin( \frac{2 \pi}{\gamma}( \frac{\overline{\beta}}{2} - Q  ))  } \frac{S_{\frac{\gamma}{2}}( Q + \frac{  - \beta_2 - \beta_3 }{2}  + \sigma_1  - \sigma_2) S_{\frac{\gamma}{2}}(\frac{ - \beta_2 - \beta_3 }{2}  + \sigma_1  + \sigma_2) S_{\frac{\gamma}{2}}(Q  - \beta_3) S_{\frac{\gamma}{2}}( \frac{2}{\gamma }  )}{S_{\frac{\gamma}{2}}(\gamma) S_{\frac{\gamma}{2}}( \frac{4}{\gamma} - \beta_2 -\beta_3 +\frac{\gamma}{2} ) S_{\frac{\gamma}{2}}( - \frac{\beta_3}{2}  + \sigma_3 + \sigma_1) S_{\frac{\gamma}{2}}(Q - \frac{\beta_3}{2} -\sigma_3 + \sigma_1 )}.
\end{align*}
We obtain:
\begin{align*}
&\frac{ 2 \sin( \frac{\gamma^2 \pi}{4} ) S_{\frac{\gamma}{2}}( \frac{2}{\gamma} - \frac{\gamma}{2} ) S_{\frac{\gamma}{2}}(\frac{2}{\gamma} -\beta_3 ) S_{\frac{\gamma}{2}}( \frac{2}{\gamma} + \frac{  - \beta_2 - \beta_3 }{2}  + \sigma_1  - \sigma_2) S_{\frac{\gamma}{2}}(\frac{ - \beta_2 - \beta_3 -\gamma}{2}  + \sigma_1  + \sigma_2) }{ \frac{4}{\gamma} \sin( \frac{2 \pi}{\gamma}( \frac{\overline{\beta}}{2} - Q )) S_{\frac{\gamma}{2}}( \gamma) S_{\frac{\gamma}{2}}( Q + \frac{2}{\gamma} -\beta_2 - \beta_3 ) S_{\frac{\gamma}{2}}( - \frac{\beta_3}{2}  + \sigma_3 + \sigma_1) S_{\frac{\gamma}{2}}(Q - \frac{\beta_3}{2} -\sigma_3 + \sigma_1 ) }   \\
& \times \Big[ 8 \sin(\frac{\gamma \pi}{2}(\beta_2 + \beta_3)) \sin( \frac{\gamma \pi}{2}( - \frac{\beta_3}{2} - \frac{\gamma }{2} +\sigma_3 + \sigma_1) ) \sin( \frac{\gamma \pi}{2}( \frac{\beta_3}{2}  +\sigma_3 - \sigma_1) )\\
& \quad - 8 \sin( \frac{\gamma \pi \beta_3}{2})   \sin( \frac{\gamma \pi}{2} (- \frac{\beta_2 + \beta_3 +\gamma}{2} + \sigma_1 + \sigma_2 ) ) \sin( \frac{\gamma \pi}{2} ( \frac{\beta_2 + \beta_3 }{2} - \sigma_1 + \sigma_2 ) )\Big].
\end{align*}
Record the trigonometric identity:
\begin{align*}
&  2 \sin(\frac{\gamma \pi}{2}(\beta_2 + \beta_3)) \sin( \frac{\gamma \pi}{2}( - \frac{\beta_3}{2} - \frac{\gamma }{2} +\sigma_3 + \sigma_1) ) \sin( \frac{\gamma \pi}{2}( \frac{\beta_3}{2}  +\sigma_3 - \sigma_1) )\\
&  - 2 \sin( \frac{\gamma \pi \beta_3}{2})   \sin( \frac{\gamma \pi}{2} (- \frac{\beta_2 + \beta_3 +\gamma}{2} + \sigma_1 + \sigma_2 ) ) \sin( \frac{\gamma \pi}{2} ( \frac{\beta_2 + \beta_3 }{2} - \sigma_1 + \sigma_2 ) ) \\
& =  \sin( \frac{\gamma \pi \beta_2}{2}) \cos( \frac{\gamma \pi}{2}( - \frac{\gamma}{2}  +2 \sigma_1 ) ) +  \sin( \frac{\gamma \pi \beta_3}{2}) \cos( \frac{\gamma \pi}{2}( - \frac{\gamma}{2}  +2 \sigma_2 ) ) -  \sin( \frac{\gamma \pi (\beta_2 + \beta_3)}{2}) \cos( \frac{\gamma \pi}{2}( - \frac{\gamma}{2}  +2 \sigma_3 ) ).
\end{align*}
Using the above identity we arrive at~\eqref{eq:PT-property2}.
For the final claim of the lemma, the ratio of the second limit divided by the first can be easily computed using \eqref{eq:shift_S1}.
\end{proof}

We need the following lemma when showing $G =G_{\mathrm{Hos}}$ in Section~\ref{sec:proof_G}.

\begin{lemma}\label{lem:comp1_G}
The following holds:
\begin{align*}
\underset{\alpha \rightarrow Q - \frac{\beta}{2}}{\lim} (\alpha + \frac{\beta}{2} - Q) \int_{ \ii \mathbb{R}} \frac{d \sigma_2}{\ii} e^{2 \ii \pi (Q - 2 \sigma) \sigma_2} \frac{   S_{\frac{\gamma}{2}}( \frac{1}{2}(\alpha + \frac{\beta}{2} - Q) \pm \sigma_2 ) }{  S_{\frac{\gamma}{2}}( \frac{1}{2}(\alpha - \frac{\beta}{2} + Q) \pm \sigma_2 )  } = \frac{1}{ 2 \pi }    \frac{ 1  }{ S_{\frac{\gamma}{2}}(Q - \frac{\beta}{2} )^2 }.
\end{align*}
\end{lemma}

\begin{proof}
Write:
\begin{align*}
& \int_{ \ii \mathbb{R}} \frac{d \sigma_2}{\ii} e^{2 \ii \pi (Q - 2 \sigma) \sigma_2} \frac{   S_{\frac{\gamma}{2}}( \frac{1}{2}(\alpha + \frac{\beta}{2} - Q) \pm \sigma_2 )  }{  S_{\frac{\gamma}{2}}( \frac{1}{2}(\alpha - \frac{\beta}{2} + Q) \pm \sigma_2 ) } \\
& = \int_{ \ii \mathbb{R}} \frac{d \sigma_2}{\ii} e^{2 \ii \pi (Q - 2 \sigma) \sigma_2} \frac{1}{2 \sin(\frac{\gamma \pi }{2} (\frac{1}{2}(\alpha + \frac{\beta}{2} - Q) - \sigma_2 ) )} \frac{   S_{\frac{\gamma}{2}}( \frac{\gamma}{2} +  \frac{1}{2}(\alpha + \frac{\beta}{2} - Q) - \sigma_2 ) S_{\frac{\gamma}{2}}( \frac{1}{2}(\alpha + \frac{\beta}{2} - Q) + \sigma_2 ) }{  S_{\frac{\gamma}{2}}( \frac{1}{2}(\alpha - \frac{\beta}{2} + Q) + \sigma_2 )       S_{\frac{\gamma}{2}}(\frac{1}{2}(\alpha - \frac{\beta}{2} + Q) - \sigma_2)}.
\end{align*}
The residue we are interested in is given by:
\begin{align*}
 2 \pi  \frac{1}{\gamma \pi } e^{ \ii \pi (Q - 2 \sigma) (\alpha + \frac{\beta}{2} - Q)} \frac{   S_{\frac{\gamma}{2}}( \frac{\gamma}{2}  ) S_{\frac{\gamma}{2}}(\alpha + \frac{\beta}{2} - Q) }{  S_{\frac{\gamma}{2}}( \alpha) S_{\frac{\gamma}{2}}(Q - \frac{\beta}{2} ) }.
\end{align*}
We are then interested in the limit:
\begin{align*}
&\underset{\alpha \rightarrow Q - \frac{\beta}{2}}{\lim} (\alpha + \frac{\beta}{2} - Q)   2 \pi  \frac{1}{\gamma \pi } e^{ \ii \pi (Q - 2 \sigma_1) (\alpha + \frac{\beta}{2} - Q)} \frac{   S_{\frac{\gamma}{2}}( \frac{\gamma}{2}  ) S_{\frac{\gamma}{2}}(\alpha + \frac{\beta}{2} - Q) }{  S_{\frac{\gamma}{2}}( \alpha) S_{\frac{\gamma}{2}}(Q - \frac{\beta}{2} ) }\\
& = \frac{2}{\gamma} \frac{   S_{\frac{\gamma}{2}}( \frac{\gamma}{2}  )  }{ S_{\frac{\gamma}{2}}(Q - \frac{\beta}{2} )^2 }  \underset{\alpha \rightarrow Q - \frac{\beta}{2}}{\lim} (\alpha + \frac{\beta}{2} - Q)  S_{\frac{\gamma}{2}}(\alpha + \frac{\beta}{2} - Q)  = \frac{2}{\gamma} \frac{   S_{\frac{\gamma}{2}}( \frac{\gamma}{2}  )  }{ S_{\frac{\gamma}{2}}(Q - \frac{\beta}{2} )^2 }  \frac{\gamma}{4 \pi}  S_{\frac{\gamma}{2}}(\frac{2}{\gamma}) = \frac{1}{ 2 \pi }    \frac{ 1  }{ S_{\frac{\gamma}{2}}(Q - \frac{\beta}{2} )^2 }.
\end{align*}
\end{proof}

\section{Analytic continuation of the reflection coefficient}\label{app:GMC}
In this section we prove Proposition~\ref{prop:R-def} and an easier variant Lemma~\ref{lem-hol-dR}, and along the way supply a technical ingredient also needed in Section~\ref{app:lim_H_R}. The argument for Proposition~\ref{prop:R-def} is similar to that of Section~\ref{subsec-anal-H}, but has two additional complications. Firstly, the reflection coefficient describes the law of a weight $W$ quantum disk, which  does not immediately arise from a Liouville field; instead, the weight $W$ quantum disk weighted by boundary length is described by a  Liouville field \cite[Proposition 2.18]{AHS-SLE-integrability}. Secondly, the GMC moment we need to bound has positive quantum area exponent but negative quantum length exponent (Lemma~\ref{lem-str-near-infty}), and is only finite due to cancellation between the two terms. Thus a naive application of H\"older's inequality fails. We use Lemma~\ref{lem-AL-mixed} which estimates GMC moments conditioned on the field average maximum.

\begin{proposition}\label{prop:analytic-appendix-R}
	Fix $\mu_i$ with $\Re \mu_i > 0$ for $i=1,2,3$, and 
	consider the following three maps from $((Q-\gamma) \vee \frac\gamma2, Q)$ to $\C$:
		\begin{align*}
		\beta &\mapsto \int \frac1{L_2} A e^{-A - \mu_1L_1 - \mu_2L_2} \, \LF_\bbH^{(\beta,0), (\gamma, 1), (\beta, \infty)}(d\phi),\\
		\beta &\mapsto \int \frac1{L_2}A (\mu_1L_1+ \mu_2L_2) e^{-A - \mu_1 L_1 - \mu_2L_2} \, \LF_\bbH^{(\beta,0), (\gamma, 1), (\beta, \infty)}(d\phi),\\
		\beta &\mapsto \int \frac1{L_2} (\mu_1L_1+ \mu_2L_2)^2 e^{-A - \mu_1L_1 - \mu_2L_2} \, \LF_\bbH^{(\beta,0) ,(\gamma, 1), (\beta, \infty)}(d\phi),
	\end{align*}
	where we write $A = \cA_\phi(\bbH)$, $L_1 = \cL_\phi(-\infty,0)$ and $L_2 = \cL_\phi(0,\infty)$. 
	Then each map 	is well-defined in the sense that the integrals converge absolutely, and extends analytically to a neighborhood of $((Q-\gamma) \vee \frac\gamma2, Q)$ in $\C$. 
\end{proposition}

\begin{proof}
The proof is a modification of that of Proposition~\ref{prop:analytic-appendix}, so we will be briefer. 
	We only prove the result for the second function, but the others are proved identically. As before, $(p_1, p_2, p_3) = (-2,0,2)$. We work in $(\bbH, p_1,p_2,p_3)$ rather than $(\bbH, 0, 1, \infty)$. To simplify notation we assume  $\mu_1=\mu_2=1$; the argument works the same way for general $\mu_1, \mu_2$. We need to show that the function 
	\[f(\beta) = \int\frac{ \cA_\phi(\bbH)  \cL_\phi(\R)}{\cL_\phi((p_1,p_3))} e^{-\cA_\phi(\bbH) -  \cL_\phi(\R)} \LF_{\bbH}^{(\beta, p_1), (\gamma, p_2), (\beta, p_3)}(d\phi)\]
	extends analytically to a neighborhood of $((Q-\gamma) \vee \frac\gamma2, Q)$ in $\C$. 
	
	For $r \geq 1$ let $\bbH_r := \bbH\backslash \bigcup_i B_{e^{-r}}(p_i)$, $I_r = \R \backslash (p_1-e^{-r}, p_3 + e^{-r})$ and $J_r = (p_1+ e^{-r}, p_3 - e^{-r})$. For a sample $h \sim P_\bbH$, write $A_r$ (resp. $K_r, L_r$) for the quantum area of $\bbH_r$ (resp. quantum length of $I_r \cup J_r$, $J_r$) with respect to the field $h - 2Q \log |\cdot|_+ + \frac\gamma2 G_\bbH(\cdot, p_2)$. Write $h_r(x)$ for the average of $h$ on $\partial B_{e^{-r}}(x) \cap \bbH$. Then define for $r \geq 1$ the function
	\[f_r(\beta) = \int_\R \, dc e^{(\beta + \frac\gamma2 - Q)c} \E \left[e^{\frac\beta2 (h_r(p_1) + h_r(p_3))  - \frac{\beta^2}2 r } \frac{(e^{\gamma c}A_r)(e^{\frac\gamma2c} K_r )}{e^{\frac\gamma2 c} L_r} e^{-A_r - K_r}  \right]\]
	on a neighborhood of $((Q-\gamma) \vee \frac\gamma2, Q)$ in $\C$ to be specified later. Clearly $f_r$ is analytic.

	\medskip 
	\noindent \textbf{Claim 1: Each point in $((Q-\gamma)\vee \frac\gamma2, Q)$ has an open neighborhood in $\C$ on which $f_r$ converges uniformly as $r \to \infty$.}

 Pick a point in $((Q - \gamma) \vee \frac\gamma2, Q)$ and let $O$ be a neighborhood of the point in $\C$, which we will choose sufficiently small as described below.
	We have the alternative expression 
	\[f_r(\beta) = \int_\R  e^{(\beta + \frac\gamma2 - Q)c} \E \left[e^{\frac\beta2 (h_{r+1}(p_1) + h_{r+1}(p_3))  - \frac{\beta^2}2 (r+1) } \frac{(e^{\gamma c}A_r)(e^{\frac\gamma2c} K_r)}{e^{\frac\gamma2 c} L_r} e^{-e^{\gamma c}A_r -  e^{\gamma c/2}K_r}  \right] \, dc.\]
	Write $\beta = u + \ii v$, and let $\wt A_r = \cA_{\wt h} (\bbH_r), \wt K_r = \cL_{\wt h}(I_r \cup J_r), \wt L_r = \cL_{\wt h}(J_r)$ where $\wt h = h - 2Q \log |\cdot|_+ + \frac\beta2 (G_\bbH (\cdot, p_1) + G_\bbH(\cdot, p_3))+ \frac\gamma2 G_\bbH(\cdot,p_2)$. 	 As before, with $g(a,k,\ell) = \frac{ak}\ell e^{-a -  k}$, for $\beta \in O$ we have 
	\[|f_{r+1} (\beta) - f_r(\beta)| \leq C_0 e^{\frac r2 v^2} \int_\R  e^{(u + \frac\gamma2-Q)}\E\left[ | g(e^{\gamma c} \wt A_{r+1}, e^{\frac\gamma2c}\wt K_{r+1}, e^{\frac\gamma2c}  \wt L_{r+1}) - g(e^{\gamma c} \wt A_{r}, e^{\frac\gamma2c}\wt K_{r}, e^{\frac\gamma2c} \wt L_{r})| \right] \, dc,\]
 where, as in~\eqref{eq-C0}, $C_0$ is a constant depending $O$ but not on $r$.
	As before, we can bound
	\begin{align*}
	|g(a_{r+1}, &k_{r+1}, \ell_{r+1}) - g(a_r, k_r, \ell_r)| \\
	&< C_1 \left( \frac{a_{r+1} e^{-a_{r+1}}}{\ell_{r+1}} (k_{r+1} - k_r) + \frac{(a_{r+1}-a_r)e^{-a_{r+1}}}{\ell_{r+1}} + \frac{a_r (e^{-a_r} - e^{a_{r+1}})}{\ell_{r+1}} + \frac{a_re^{-a_r} (\ell_{r+1}-\ell_r)}{\ell_r\ell_{r+1}}\right)
	\end{align*}
 where $C_1$ is a universal constant.
This gives us four terms to bound. Writing $s = u + \frac\gamma2 - Q \in ((-\frac\gamma2) \vee (\frac\gamma2 - \frac2\gamma), \frac\gamma2)$,
\begin{align*}
	\int dc \, e^{(s + \gamma)c} \E[\frac{\wt A_{r+1} (\wt K_{r+1} - \wt K_r)}{\wt L_{r+1}} e^{-e^{\gamma c} \wt A_{r+1}}] &= \frac1\gamma \Gamma(\frac s\gamma + 1) \E[\frac{\wt K_{r+1} - \wt K_r}{\wt L_{r+1}} \wt A_{r+1}^{-\frac s\gamma}], \\
	\int dc \, e^{(s + \frac\gamma2)c} \E[\frac{\wt A_{r+1}  - \wt A_r}{\wt L_{r+1}} e^{-e^{\gamma c} \wt A_{r+1}}] &= \frac1\gamma \Gamma(\frac s\gamma + \frac12) \E[\frac{\wt A_{r+1} - \wt A_r}{\wt L_{r+1}} \wt A_{r+1}^{-\frac s\gamma - \frac12}], \\
		\int dc \, e^{(s + \frac\gamma2)c} \E[\frac{\wt A_r}{\wt L_{r+1}} (e^{-e^{\gamma c} \wt A_{r}} - e^{-e^{\gamma c} \wt A_{r+1}})] &= \frac1\gamma \Gamma(\frac s\gamma + \frac12) \E[\frac{\wt A_r}{\wt L_{r+1}} (\wt A_r^{-\frac s\gamma - \frac12} - \wt A_{r+1}^{-\frac s\gamma - \frac12})], \\
		\int dc \, e^{(s + \frac\gamma2)c} \E[\frac{\wt A_r (\wt L_{r+1} - \wt L_r)}{\wt L_{r+1} \wt L_r} e^{-e^{\gamma c} \wt A_{r}}] &= \frac1\gamma \Gamma(\frac s\gamma + \frac12) \E[\frac{\wt L_{r+1} - \wt L_r}{\wt L_{r+1} \wt L_r} \wt A_r^{-\frac s\gamma - \frac12}].
\end{align*}
All four integrals follow from the identity $\int  e^{tc} e^{-e^{\gamma c} a}\, \mathrm dc = \frac1\gamma \Gamma(\frac {t}\gamma) a^{-\frac {t}\gamma}$ which holds for $t,a >0$. 

By Lemma~\ref{lem-many-moments}, for each point in $((Q-\gamma) \vee \frac\gamma2, Q)$ there is a neighborhood $O$ and a constant $C>0$ such that these four terms are uniformly bounded by $Ce^{-r/C}$ on $O$. Choosing $O$ such that $v$ is small gives the desired uniform convergence. 

	\medskip 
\noindent \textbf{Claim 2: $\lim_{r\to\infty} f_r(\beta) = f(\beta)$  when $\beta \in ((Q-\gamma)\vee \frac\gamma2, Q)$. }

The proof is the same as that of Claim 1 with $(r+1)$ replaced by $\infty$. 

\medskip
\noindent \textbf{Conclusion.} By Claim 1, there is a neighborhood of $((Q-\gamma)\vee \frac\gamma2, Q)$ on which $\lim_{r \to \infty} f_r$ exists and is holomorphic. By Claim 2, this limit agrees with $f$ on $((Q-\gamma)\vee \frac\gamma2, Q)$, hence is the desired analytic continuation of $f$. 
\end{proof}

\begin{lemma} \label{lem-many-moments}
	For each point in $((Q-\gamma) \vee \frac\gamma2, Q)$ there is a neighborhood $U \subset \R$ and a constant $C>0$ such that the following holds. For $\beta \in U$ and $r \geq 1$, writing $s = \beta + \frac\gamma2 - Q$, the expectations 
	\[\E[\frac{\wt K_{r+1} - \wt K_r}{\wt L_{r+1}} \wt A_{r+1}^{-\frac s\gamma}], \quad  \E[\frac{\wt A_{r+1} - \wt A_r}{\wt L_{r+1}} \wt A_{r+1}^{-\frac s\gamma - \frac12}], \quad \E[\frac{\wt L_{r+1} - \wt L_r}{\wt L_{r+1} \wt L_r} \wt A_r^{-\frac s\gamma - \frac12}], \quad \E[\frac{\wt A_r}{\wt L_{r+1}} (\wt A_r^{-\frac s\gamma - \frac12} - \wt A_{r+1}^{-\frac s\gamma - \frac12})]\]
	 are each bounded by $C e^{-r/C}$.  Moreover, the statement still holds when all instances of $(r+1)$ are replaced by $\infty$. 
	 
	 Here, we let $(p_1, p_2, p_3)= (-2,0,2)$, 
	 $\bbH_r := \bbH\backslash \bigcup_i B_{e^{-r}}(p_i)$, $I_r = \R \backslash (p_1-e^{-r}, p_3 + e^{-r})$ and $J_r = (p_1+ e^{-r}, p_3 - e^{-r})$. For a sample $h \sim P_\bbH$, let $\wt h = h - 2Q \log |\cdot|_+ + \frac\beta2 (G_\bbH(\cdot, p_1) + G_\bbH(\cdot, p_3))+ \frac\gamma2 G_\bbH(\cdot, p_2)$, and define $\wt A_r = \cA_{\wt h}(\bbH_r)$, $\wt K_r = \cL_{\wt h}(I_r \cup J_r)$ and $\wt L_r = \cL_{\wt \phi}(J_r)$.

\end{lemma}
\begin{proof}
	We explain the exponential bounds for each $\beta$ (rather than uniformly in $U$); all inputs in this argument vary continuously so we get uniform bounds in neighborhoods $U$. We only explain the case where the subscripts are $(r+1)$ rather than $\infty$ since the argument is the same. We again write $x \lesssim y$ to denote that $x \leq ky$ for some $k$ depending only on $U$.
	
	Let $\eps > 0$ and let $p,q>1$ satisfy $\frac1p + \frac1q = 1$. By the same trick used for the first inequality of Lemma~\ref{lem-exp-bound-holo}, for the first three terms we have 
	\[\E[\frac{\wt K_{r+1} - \wt K_r}{\wt L_{r+1}} \wt A_{r+1}^{-\frac s\gamma}] \lesssim  \eps \E[\frac{\wt A_{r+1}^{-\frac s\gamma}}{\wt L_{r+1}}] + \E[(\frac{\wt K_{r+1}}{\wt L_{r+1}} \wt A_{r+1}^{-\frac s\gamma})^p]^{1/p} \P[\wt K_{r+1} - \wt K_r > \eps]^{1/q},\]
	\[\E[\frac{\wt A_{r+1} - \wt A_r}{\wt L_{r+1}} \wt A_{r+1}^{-\frac s\gamma-\frac12}] \lesssim  \eps \E[\frac{\wt A_{r+1}^{-\frac s\gamma - \frac12}}{\wt L_{r+1}}] + \E[(\frac{\wt A_{r+1}^{-\frac s\gamma + \frac12}}{\wt L_{r+1}})^p]^{1/p} \P[\wt A_{r+1} - \wt A_r > \eps]^{1/q},\]
	\[\E[\frac{\wt L_{r+1} - \wt L_r}{\wt L_{r+1} \wt L_r} \wt A_r^{-\frac s\gamma - \frac12}] \lesssim \eps \E[\frac{\wt A_r^{-\frac s\gamma + \frac12}}{\wt L_{r+1}\wt L_r}] + \E[(\frac{\wt A_r^{-\frac s\gamma + \frac12}}{\wt L_r})^p]^{1/p} \P[\wt L_{r+1} - \wt L_r > \eps]^{1/q} .\]
	For the fourth term, we use the same argument as in the second inequality of Lemma~\ref{lem-exp-bound-holo} to get
	\[\E[\frac{\wt A_r}{\wt L_{r+1}} (\wt A_r^{-\frac s\gamma - \frac12} - \wt A_{r+1}^{-\frac s\gamma - \frac12})] \lesssim \eps \E[\frac{\wt A_r^{-\frac s\gamma - \frac12}}{\wt L_{r+1}}] + \E[(\frac{\wt A_r^{-\frac s\gamma + \frac12}}{\wt L_{r+1}})^p]^{1/p} \P[A_{r+1} - A_r > \eps]^{1/q} \]
	Choosing $p>1$ small, all the expectations in the above inequalities are uniformly bounded in $r$ by Lemma~\ref{lem-str-near-infty} since $\beta > \frac\gamma2$ implies $-\frac {2s}\gamma + 1 < \frac4{\gamma^2}$, so we can ignore those terms. The choice of $\eps>0$ and the bounds on the probabilities can be done as in Lemma~\ref{lem-exp-bound-holo} to complete the proof.
\end{proof}

\begin{lemma}\label{lem-str-near-infty}
	Let $\beta < Q$. 
	Let  $x < \frac2{\gamma^2}$ and $y, z < \frac{4}{\gamma^2}$ satisfy $\frac\gamma2(2x+y+z)< \min(Q-\beta, Q - \gamma)$ and $\max(2x,0) + \max (y, 0) + \max (z,0) < \frac4{\gamma^2}$. 	In the setting of Lemma~\ref{lem-many-moments}, there is a continuous function $C(\beta, x, y, z)$ such that for all $r>1$,
	\[ \E[\wt A_r^x \wt K_r^y \wt L_r^z ] < C(\beta, x, y, z).\]
\end{lemma}
\begin{proof}
	We will just show that this expectation is finite; all the inputs in the proof vary continuously so the statement about $C(\beta,x,y,z)$ follows. We write $x \lesssim y$ to mean $x < ky$ holds for some constant $k$ not depending on $r$. 
 
For $i = 1,3$ let $M_i = \sup_{0 \leq t \leq r} h_r(p_i) - Qr$, and let $M_2 = \sup_{t \geq 0} h_r(p_2) - Qr$. Let $M = \max_i M_i$.
 We first show that for any $p < \frac 2{\gamma^2}$ we have 
 \begin{equation}\label{eq-combined-moments-M}
\max(\E[\wt A_r^p \mid M],  \E[\wt K_r^{2p} \mid M], \E[\wt L_r^{2p} \mid M] )\lesssim e^{p \gamma M}.
 \end{equation}
	For $i=1,2,3$ let $A_i = \cA_{\wt h} (\bbH_r \cap B_1(p_i))$ and let $A_0 = \cA_{\wt h} (\bbH_r \backslash \bigcup_i B_1(p_i))$, so $\wt A_r = \sum_{i=0}^3 A_i$.
Write $(M_j)_j = (M_1, M_2, M_3)$. 
	  If $p \geq 0$, then by Lemma~\ref{lem-AL-mixed} and the Markov property of the GFF we have $\E[A_i^p \mid (M_j)_j] \lesssim e^{p \gamma M_i}$ so $\E[\wt A_r^p \mid (M_j)_j] \lesssim \E[A_0^p \mid (M_j)_j] + \sum_{i=1}^3 \E[A_i^p \mid (M_j)_j] \lesssim e^{p \gamma M}$. If instead $p < 0$ we get the desired bound using $\wt A_r^p \lesssim \min_i A_i^p$. Arguing similarly for $K$ and $L$ gives~\eqref{eq-combined-moments-M}. 

   Next, we show that $\E[\wt A_r^x \wt K_r^y \wt L_r^z] \lesssim  \E[ e^{(2x + y + z) \frac\gamma2 M}]$. We explain the argument in the case $x, y > 0$ and $z<0$; the other cases are similar. Let $\lambda_1, \lambda_2, \lambda_3>1$   satisfy $\frac1{\lambda_1} + \frac1{\lambda_2} + \frac1{\lambda_3} = 1$. We can choose these parameters so that 
   $\max(\lambda_1 \cdot 2x, \lambda_2 \cdot y, \lambda_3 \cdot z) < \frac4{\gamma^2}$, e.g.\ by taking $\lambda_3 \gg 0$ and choosing $\lambda_1, \lambda_2$ such that $\frac{\lambda_1}{2x} = \frac{\lambda_2}y$. H\"older's inequality then gives the desired
   \[ \E[\wt A_r^x \wt K_r^y \wt L_r^z \mid M ]\leq \E[ \wt A_r^{\lambda_1 x} \mid M]^{1/\lambda_1} \E[\wt K_r^{\lambda_2 y} \mid M]^{1/\lambda_2} \E[\wt L_r^{\lambda_3 z} \mid M]^{1/\lambda_3} \lesssim e^{(2x+y+z)\frac\gamma2 M}.\]
Finally, for $b>0$, if $(B_t)_{t \geq 0}$ is standard Brownian motion, the law of $\sup_{t\geq0} B_{2t} - bt$ is the exponential distribution with rate $b$. Thus $M_1$ and $M_3$ are exponential with rate $Q-\beta$, and $M_2$ is exponential with rate $Q -\gamma$. We conclude
that $\E[e^{(2x+y+z) \frac\gamma2 M}]<\infty$  if and only if $\frac\gamma2(2x+y+z)< \min(Q-\beta, Q - \gamma)$. 
  \end{proof}

\begin{lemma}\label{lem-AL-mixed}
Let $a>0$ and $r\in(1, \infty]$. Let $\cS = \R\times (0,\pi)$ be the horizontal strip. Let $h$ be a GFF on $\cS$ normalized so $(h, \rho) = 0$ for some $\rho$ supported on $(-\infty, 0) \times (0,\pi)$, and let $\wt h = h - a \Re \cdot$. Let $X_t$ be the average of $\wt h$ on $\{t\} \times (0,\pi)$ and let $M = \sup_{0<t<r-1} X_t$. Then for any $p < \frac2{\gamma^2}$ and some constant $C = C(p, a, \rho)$ we have 
\[\max(\E[A^p \mid M], \E[K^{2p} \mid M], \E[L^{2p} \mid M]) \leq C e^{p\gamma M}\]
where $A := \cA_{\wt h} ((0, r)\times (0,\pi))$, $K := \cL_{\wt h} ((0, r)\times \{0,\pi\})$ and $L := \cL_{\wt h}((0,r) \times \{0\})$. 
\end{lemma}
\begin{proof}
	We prove the inequality for $A$ and $r < \infty$; the others are identically proved. We write $x \lesssim y$ to mean $x < ky$ for some constant $k$ not depending on $M$. 
	Suppose $p>0$. 
	A variant of \cite[Lemma A.5]{wedges} gives, with $j^*$ the value of $t \in [j,j+1]$  maximizing $X_t$, we have 
	\[\E[(\sum_{j=0}^{\lfloor r \rfloor}\gamma X_{j^*})^p \mid M] \lesssim e^{p \gamma M}.\]
	Then the argument of \cite[(A.11)]{wedges} gives the desired result. 
	
	Now suppose $p \leq 0$. Conditon on $M$ and on the time $t_* \in (0,r-1)$ at which $X_{t_*}$ is maximal. The conditional law of $(X_t)_{[t_*, r-1]}$ is variance 2 Brownian motion with drift $-a$ conditioned to stay below $M$, and $(X_t)_{[r-1,r]}$ then evolves as unconditioned variance 2 Brownian motion with drift $-a$. Thus $Y:= M - \inf_{[t_*, t_*+1]} X_t$ is stochastically dominated by $\hat Y := -\inf_{0<s<1} (B_{2s} - as)$ where $B$ is standard Brownian motion. Let $h_2$ be the projection of $h$ to $\mathcal{H}_2(\cS)$, then $h_2$ is independent of $(X_t)$ and hence $(M_r, t_*)$, so writing $S = (t_*, t_*+1) \times (0,\pi)$, 
	\begin{align*}
		\E[A^p \mid M, t_*] &\leq \E[\cA_{\wh h}(S)^{p} \mid M, t_*] \leq \E[\cA_{h_2 + M - Y}(S)^{p} \mid M_{r}, t_*] \\
		&\leq e^{p\gamma M} \E[ e^{p \gamma \hat Y}] \E[\cA_{h_2}(0,1)^{p}] \lesssim e^{p \gamma M}.
	\end{align*}
\end{proof}

\begin{proof}[Proof of Proposition~\ref{prop:R-def}]
		Morera's theorem gives analyticity in $\mu$. Next, 
	\cite[Proposition 2.18]{AHS-SLE-integrability} relates the two-pointed quantum disk weighted by the quantum length of a side to the three-point Liouville field: writing $A = \cA_\phi(\bbH), L_1 = \cL_\phi(-\infty,0), L_2 = \cL_\phi(0,\infty)$ and $s=\beta - Q$, we have 
	\[R_{\mu_1, \mu_2}(\beta) = \frac1{Q-\beta}\int (\frac\gamma s A + \frac{\gamma}{2s} \frac\gamma{s+\frac\gamma2} A (\sum_i \mu_i L_i) + \frac\gamma{2s} \frac{\frac\gamma2}{s+\frac\gamma2} (\sum_i \mu_i L_i)^2 ) \frac{e^{-A - \sum_i \mu_i L_i}}{L_2} \LF_{\bbH}^{(\beta, 0),(\gamma, 1),(\beta, \infty)}(d\phi). \]
	Therefore Proposition~\ref{prop:analytic-appendix-R} gives the analyticity of $R_{\mu_1,\mu_2}(\beta)$ in $\beta$. 
\end{proof}

Finally, we give a holomorphicity result for an observable of $\cMtwo(\beta)$.

\begin{lemma}\label{lem-hol-dR}
    Let $n > \frac4{\gamma^2}$ be an integer and let $\mu_1, \mu_2 > 0$, then the function 
    \[ \beta \mapsto  \cMtwo(\beta)[L_1^n e^{-A-\mu_1L_1-\mu_2 L_2}] \quad \textrm{for }\beta\in (\frac{\gamma}2,Q)\]
    extends holomorphically to a neighborhood of $(\frac\gamma2,Q)$ in $\C$. Here, $A$ and $(L_1, L_2)$ are the quantum area and boundary arc lengths respectively of the quantum disk. 
\end{lemma}
\begin{proof}
The proof  is similar to that of  Proposition~\ref{prop:analytic-appendix-R} but is substantially easier. We only point out the differences. 
As before, we assume $\mu_1 = \mu_2 = 1$, but the general case follows the same argument. 
    \cite[Proposition 2.18]{AHS-SLE-integrability} explains that the law of a sample from $\cMtwo(\beta)$ weighted by $L_1$ can be described by a three-point Liouville field with insertions $\beta, \beta, \gamma$; we choose the three-pointed domain $(\bbH, p_1, p_2, p_3)$ where $(p_1, p_2, p_3) = (-2, 0, 2)$. It thus suffices to show that
	\[f(\beta) = \int \cL_\phi((p_1, p_3))^{n-1} e^{-\cA_\phi(\bbH) - \mu \cL_\phi(\R)} \LF_{\bbH}^{(\beta, p_1), (\gamma, p_2), (\beta, p_3)}(d\phi)\]
	extends analytically to a neighborhood of $(\frac\gamma2, Q)$ in $\C$. 

 We use the setup and notation of the proof of Proposition~\ref{prop:analytic-appendix-R}. Instead of $g$, we use $\hat g(a,k,\ell) = \ell^{n-1}e^{-a - k}$ for $a > 0$ and $k > \ell > 0$, for which we have the bound 
 \[|\hat g(a_{r+1}, k_{r+1}, \ell_{r+1}) - \hat g(a_r, k_r, \ell_r)|<
C_1(\ell_{r+1}^{n-2} (\ell_{r+1} - \ell_r) e^{- \ell_{r+1}} + \ell_r^{n-1} (e^{-k_r} - e^{-k_{r+1}}) + \ell_r^{n-1} (e^{-a_r} - e^{a_{r+1}}))\]
for $(a_{r+1}, k_{r+1}, \ell_{r+1})$ and $(a_{r}, k_{r}, \ell_{r})$ satisfying $a_{r+1} > a_r, k_{r+1} > k_r, \ell_{r+1} > \ell_r$. 

Our assumption $n > \frac4{\gamma^2}$ implies $(n-1) + \frac{2s}\gamma > 0$. So we can obtain 
\[ \int e^{(s + (n-1)\frac\gamma2)c} \E [\wt L_{r+1}^{n-2}(\wt L_{r+1} - \wt L_r)  e^{-e^{\frac\gamma2c} \wt L_{r+1}} ]  \, dc = \frac2\gamma\Gamma((n-1) + \frac{2s}\gamma) \E[\wt L_{r+1}^{-\frac{2s}\gamma - 1} (\wt L_{r+1} - \wt L_r)]\]
and two similar identities. These identities and the bound for $\hat g$ reduce the problem to proving 
\[\E[\wt L_{r+1}^{-\frac{2s}\gamma-1} (\wt L_{r+1} - \wt L_r)],   \E[\wt L_r^{n-1} ( \wt K_r^{-(n-1) - \frac{2s}\gamma}  - \wt K_{r+1}^{-(n-1) - \frac{2s}\gamma} )] ,  \E[\wt L_r^{n-1} (
\wt A_r^{-\frac{n-1}2 - \frac s\gamma}
 - \wt A_{r+1}^{-\frac{n-1}2 - \frac s\gamma} )]  < C e^{-r/C}\]
for sufficiently large $C$ for all $r$. Using Lemma~\ref{lem-str-near-infty}, this can be done exactly as in Lemma~\ref{lem-exp-bound-holo}. 
\end{proof}

\section{Operator product expansion}\label{sec:ope}

In this appendix we take a look at the Operator Product Expansion (OPE) lemmas that are used in Section \ref{sec:reff}. These proofs were first done for LCFT on the sphere in \cite{DOZZ_proof} and then in the boundary case in \cite{rz-boundary}.
The main novelty in our case is that we must handle a Liouville functional containing both area and boundary GMC measures and therefore it cannot be reduced to a simple moment of GMC as in \cite{DOZZ_proof, rz-boundary}. Nonetheless we will be quite brief in several places of the proofs below when the arguments are similar to those of \cite{rz-boundary}. 

Recall the functions $H_\chi$ and $\tilde H_\chi$ defined in and below~\eqref{eq:def_H_chi}. We first give the OPE result without reflection which is valid only for $\chi = \frac{\gamma}{2}$.

\begin{lemma}(OPE without reflection)\label{lem:ope1} Let $\chi =\frac{\gamma}{2}$. Assume that the constraints of \eqref{eq:para_H_chi} hold, meaning that  $ \sigma_1, \sigma_2, \sigma_3, \sigma_1 - \frac{\gamma}{4}, \sigma_2 -\frac{\gamma}{4} \in  [-\frac{1}{2 \gamma} + \frac{Q}{2}, \frac{1}{2 \gamma} + \frac{Q}{2}] \times \mathbb{R} $,  $\beta_1, \beta_2, \beta_3 < Q$, $ \sum_{i=1}^3 \beta_i > 2Q + \frac{\gamma}{2}$. Assume also that $\beta_1 \in (\frac{\gamma}{2}, \frac{2}{\gamma})$. Then as $ t \rightarrow 0$, one has
\begin{align} 
H_{\frac{\gamma}{2}}(t) - H_{\frac{\gamma}{2}}(0) \underset{t>0}{=}  C_2^{+} t^{1-C} + o(t^{1-C}), \quad H_{\frac{\gamma}{2}}(t) - H_{\frac{\gamma}{2}}(0) \underset{t<0}{=}  C_2^{-} t^{1-C} + o(t^{1-C}),
\end{align}
where:
\begin{align*} 
C_2^+ &=  - \frac{\Gamma(-1+\frac{\gamma \beta_1}{2} - \frac{\gamma^2}{4} ) \Gamma(1 - \frac{\gamma \beta_1}{2})}{\Gamma(- \frac{\gamma^2}{4})}  \left( g_{\frac{\gamma}{2}}(\sigma_1-\frac{\gamma}{4} ) - g_{\frac{\gamma}{2}}(\sigma_2+\frac{\beta_1}{2}-\frac{\gamma}{4}) \right)
H
\begin{pmatrix}
\beta_1 + \frac{\gamma}{2}, \beta_2, \beta_3 \\
\sigma_1 - \frac{\gamma}{4},  \sigma_2,   \sigma_3 
\end{pmatrix},\\
C_2^- &=  - \frac{\Gamma(-1+\frac{\gamma \beta_1}{2} - \frac{\gamma^2}{4} ) \Gamma(1 - \frac{\gamma \beta_1}{2})}{\Gamma(- \frac{\gamma^2}{4})}  \left( g_{\frac{\gamma}{2}}(\sigma_1 ) - g_{\frac{\gamma}{2}}(\sigma_2+\frac{\beta_1}{2}) \right) H
\begin{pmatrix}
\beta_1 + \frac{\gamma}{2}, \beta_2, \beta_3 \\
\sigma_1 - \frac{\gamma}{4},  \sigma_2,   \sigma_3 
\end{pmatrix}.
\end{align*}
A similar statement holds for $\tilde H_{\frac{\gamma}{2}}(t)$. Assume this time that \eqref{eq:para_H_chi2} hold, and also $\beta_1 \in (\frac{\gamma}{2}, \frac{2}{\gamma})$. Then as $ t \rightarrow 0$, one has
\begin{align} 
\tilde H_{\frac{\gamma}{2}}(t) - \tilde H_{\frac{\gamma}{2}}(0) \underset{t>0}{=}  \tilde C_2^{+} t^{1-C} + o(t^{1-C}), \quad \tilde H_{\frac{\gamma}{2}}(t) - \tilde  H_{\frac{\gamma}{2}}(0) \underset{t<0}{=}  \tilde C_2^{-} t^{1-C} + o(t^{1-C}),
\end{align}
where:
\begin{align}
\tilde{C}_2^+ &=- \frac{\Gamma(-1+\frac{\gamma \beta_1}{2} - \frac{\gamma^2}{4} ) \Gamma(1 - \frac{\gamma \beta_1}{2})}{\Gamma(- \frac{\gamma^2}{4})} \left( g_{\frac{\gamma}{2}}(\sigma_1+\frac{\gamma}{4} ) -g_{\frac{\gamma}{2}}(\sigma_2 - \frac{\beta_1}{2} + \frac{\gamma}{4}) \right)  H
\begin{pmatrix}
\beta_1 + \frac{\gamma}{2}, \beta_2, \beta_3 \\
\sigma_1 + \frac{\gamma}{4},  \sigma_2,   \sigma_3 
\end{pmatrix},\\
\tilde{C}_2^- &=- \frac{\Gamma(-1+\frac{\gamma \beta_1}{2} - \frac{\gamma^2}{4} ) \Gamma(1 - \frac{\gamma \beta_1}{2})}{\Gamma(- \frac{\gamma^2}{4})} \left( g_{\frac{\gamma}{2}}(\sigma_1) -g_{\frac{\gamma}{2}}(\sigma_2 - \frac{\beta_1}{2} ) \right) H
\begin{pmatrix}
\beta_1 + \frac{\gamma}{2}, \beta_2, \beta_3 \\
\sigma_1 + \frac{\gamma}{4},  \sigma_2,   \sigma_3 
\end{pmatrix}.
\end{align}
In a similar fashion one can perform expansions as $t \rightarrow 1$, we give the following one, under \eqref{eq:para_H_chi2} and $\beta_2 \in (\frac{\gamma}{2}, \frac{2}{\gamma})$, one has $\tilde H_{\frac{\gamma}{2}}(t) - \tilde H_{\frac{\gamma}{2}}(1) \underset{t<1}{=}  \tilde{B}_2^- (1-t)^{C-A-B} + o((1-t)^{C-A-B})$ where:

\begin{equation}
    \tilde{B}_2^- =  - \frac{\Gamma(-1+ \frac{ \gamma \beta_2}{2}-\frac{\gamma^2}{4}) \Gamma(1- \frac{\gamma \beta_2}{2})}{\Gamma(-\frac{\gamma^2}{4})} \left( g_{\frac{\gamma}{2}}(\sigma_3) - g_{\frac{\gamma}{2}}(\sigma_2+\frac{\beta_2}{2}) \right) H \begin{pmatrix}
\beta_1 , \beta_2 + \chi, \beta_3 \\
\sigma_1 + \frac{\gamma}{4},  \sigma_2 + \frac{\gamma}{4},    \sigma_3 
\end{pmatrix}.
\end{equation}

\end{lemma}

\begin{proof}
This proof follows exactly the steps of \cite[Equation (3.15)]{rz-boundary}. When computing the asymptotic of the difference $H_{\frac{\gamma}{2}}(t) - H_{\frac{\gamma}{2}}(0)$, even if in our case there is an area GMC term in the functional, there will be no contribution at order $t^{1-C}$ from this area GMC term. We will only write out the proof for the first asymptotic as the others follow from a similar argument. Now starting from the expression \eqref{eq:def_H_chi} of $H_{\frac{\gamma}{2}}(t)$ one can write:
\begin{align*}
&H_{\frac{\gamma}{2}}(t) - H_{\frac{\gamma}{2}}(0) \\
&= - \int_{\mathbb{R}} dc  \,e^{(\frac{\gamma p}{2}-\frac{\gamma}{4})c}\E \Bigg[ e^{\frac{\gamma c}{2}} \int_{\mathbb{R}} \frac{ ( |r_1-t|^{\frac{\gamma^2}{4}} - |r_1|^{\frac{\gamma^2}{4}} ) \hat g(r_1)^{\frac{\gamma^2 q}{8}}}{|r_1|^{\frac{\gamma\beta_1}{2}}|r_1-1|^{\frac{\gamma\beta_2}{2}}} e^{\frac{\gamma}{2} h(r_1)}  d\mu_{\partial}^t (r_1)  \\ 
& \times \exp\Bigg( - e^{\gamma c}\int_{\mathbb{H}} \frac{|x|^{\frac{\gamma^2}{2}} \hat g(x)^{\frac{\gamma^2}{4}(q +1 )} }{ |x|^{\gamma\beta_1}|x-1|^{\gamma\beta_2}}e^{\gamma h(x)} d^2x
 -e^{\frac{\gamma c}{2}}\int_{\mathbb{R}} \frac{|r|^{\frac{\gamma^2}{4}} \hat g(r)^{\frac{\gamma^2 q}{8}}}{|r|^{\frac{\gamma\beta_1}{2}}|r-1|^{\frac{\gamma\beta_2}{2}}}e^{\frac{\gamma}{2} h(r)}  d\mu_{\partial}^t (r)  \Bigg) \Bigg] + o(t^{1-C}) \\
&= -t^{1-C}  \sqrt{\frac{1}{\sin(\pi\frac{\gamma^2}{4})}} \Bigg(  \cos(\pi \gamma(\sigma_1 - \frac{\gamma}{4}-\frac{Q}{2}))\int_{\mathbb{R}_+} d u  \frac{(1+u)^{\frac{\gamma^2}{4}}- u^{\frac{\gamma^2}{4}}}{ u^{\frac{\gamma \beta_1}{2}} } \nonumber\\
& \quad + \cos(\sigma_2 - \frac{Q}{2} - \frac{\gamma}{4}) \int_0^1 du  \frac{
|u-1|^{\frac{\gamma^2}{4}} -|u|^{ \frac{\gamma^2}{4} } }{|u|^{\frac{\gamma \beta_1}{2}}} + \cos(\sigma_2 - \frac{Q}{2} ) \int_1^{+\infty} du  \frac{
|u-1|^{\frac{\gamma^2}{4}} -|u|^{ \frac{\gamma^2}{4} } }{|u|^{\frac{\gamma \beta_1}{2}}}  \Bigg) \nonumber\\
& \quad \times H
\begin{pmatrix}
\beta_1 + \frac{\gamma}{2}, \beta_2, \beta_3 \\
\sigma_1 - \frac{\gamma}{4},  \sigma_2,   \sigma_3 
\end{pmatrix}  + o(t^{1-C}) \nonumber.
\end{align*}
As in \cite{rz-boundary}, in the last equality above we have used the change of variable $r_1 = ut$ and applied the Girsanov theorem to the boundary GMC measure $e^{\frac{\gamma}{2}h(r_1)}  d\mu_{\partial}^t (r_1)$. To get the final answer we can simply use the following integral identity,
\begin{align*}
& \cos(\sigma_2 - \frac{Q}{2} - \frac{\gamma}{4}) \int_0^1 du  \frac{
|u-1|^{\frac{\gamma^2}{4}} -|u|^{ \frac{\gamma^2}{4} } }{|u|^{\frac{\gamma \beta_1}{2}}} + \cos(\sigma_2 - \frac{Q}{2} ) \int_1^{+\infty} du  \frac{
|u-1|^{\frac{\gamma^2}{4}} -|u|^{ \frac{\gamma^2}{4} } }{|u|^{\frac{\gamma \beta_1}{2}}}\\
& = -\frac{1}{2}e^{ \ii \pi\gamma(\sigma_2-\frac{Q}{2} - \frac{\gamma}{4} + \frac{\beta_1}{2})} \int_{\mathbb{R}_+ e^{\ii \pi}}d u  \frac{(1+u)^{\frac{\gamma^2}{4}}- u^{\frac{\gamma^2}{4}}}{ u^{\frac{\gamma \beta_1}{2}} }  -\frac{1}{2}e^{- \ii \pi\gamma(\sigma_2-\frac{Q}{2} - \frac{\gamma}{4} +\frac{\beta_1}{2}) } \int_{\mathbb{R}_+ e^{- \ii \pi}}d u  \frac{(1+u)^{\frac{\gamma^2}{4}}- u^{\frac{\gamma^2}{4}}}{ u^{\frac{\gamma \beta_1}{2}} }\\
& = - \frac{\Gamma(-1+\frac{\gamma \beta_1}{2} - \frac{\gamma^2}{4} ) \Gamma(1 - \frac{\gamma \beta_1}{2})}{\Gamma(- \frac{\gamma^2}{4})}   \cos(\sigma_2 - \frac{Q}{2} -\frac{\gamma}{4} +\frac{\beta_1}{2}), 
\end{align*}
where in the second line $\mathbb{R}_+ e^{ \ii \pi}$ (resp. $\mathbb{R}_+ e^{ - \ii \pi}$) should be understood as a $(-\infty,0)$ contour passing above (resp. below) $-1$.
The computations for other cases are similar.
\end{proof}
Let us now turn to the OPE with reflection. The following result holds for both values of $\chi$.

\begin{lemma}\label{lem_reflection_ope2} (OPE with reflection) For $\chi = \frac{\gamma}{2}$ or $\frac{2}{\gamma}$, assume again the constraints of \eqref{eq:para_H_chi} are satisfied. Then there exists a parameter $\beta_0 >0$ small enough so that under the assumption that $\beta_1 \in (Q - \beta_0,Q)$, as $t \rightarrow 0$ the following asymptotic holds
\begin{align}
H_{\chi}(t) - H_{\chi}(0) \underset{t >0}{=}  C_2^{+} t^{1 - \chi \beta_1 + \chi^2 } + o(|t|^{1 - \chi \beta_1 + \chi^2}), \: \:  H_{\chi}(t) - H_{\chi}(0) \underset{t<0}{=}  C_2^{-} t^{1 - \chi \beta_1 + \chi^2 } + o(|t|^{1 - \chi \beta_1 + \chi^2}),
\end{align}
where:
\begin{align}
 &C_2^+ = R(\beta_1, \sigma_1 - \frac{\chi}{2}, \sigma_2-\frac{\chi}{2}) H
\begin{pmatrix}
2Q - \beta_1 - \chi, \beta_2, \beta_3 \\
\sigma_1 - \frac{\chi}{2},  \sigma_2,   \sigma_3 
\end{pmatrix},\\
&C_2^- = e^{- \ii \pi(1 - \chi \beta_1 + \chi^2 )} R(\beta_1, \sigma_1, \sigma_2) H
\begin{pmatrix}
2Q - \beta_1 - \chi, \beta_2, \beta_3 \\
\sigma_1 - \frac{\chi}{2},  \sigma_2,   \sigma_3 
\end{pmatrix}.
\end{align}
Similarly for $\tilde H_{\chi}$, still for $\beta_1 \in (Q - \beta_0,Q)$ and assuming this time \eqref{eq:para_H_chi2} holds, the following asymptotic holds
\begin{align}
\tilde H_{\chi}(t) - \tilde H_{\chi}(0) \underset{t >0}{=}  \tilde C_2^{+} t^{1 - \chi \beta_1 + \chi^2 } + o(|t|^{1 - \chi \beta_1 + \chi^2}), \: \:  \tilde H_{\chi}(t) - \tilde H_{\chi}(0) \underset{t<0}{=} \tilde C_2^{-} t^{1 - \chi \beta_1 + \chi^2 } + o(|t|^{1 - \chi \beta_1 + \chi^2}),
\end{align}
with this time:
\begin{align}
 &\tilde C_2^+ = R(\beta_1, \sigma_1 + \frac{\chi}{2}, \sigma_2 + \frac{\chi}{2}) H
\begin{pmatrix}
2Q - \beta_1 + \chi, \beta_2, \beta_3 \\
\sigma_1 + \frac{\chi}{2},  \sigma_2,   \sigma_3 
\end{pmatrix},\\
& \tilde C_2^- = e^{- \ii \pi(1 - \chi \beta_1 + \chi^2 )} R(\beta_1, \sigma_1, \sigma_2) H
\begin{pmatrix}
2Q - \beta_1 + \chi, \beta_2, \beta_3 \\
\sigma_1 + \frac{\chi}{2},  \sigma_2,   \sigma_3 
\end{pmatrix}.
\end{align}
\end{lemma}
\begin{proof} We will only give the proof for $\chi = \frac{2}{\gamma}$, as the case of $\chi = \frac{\gamma}{2}$ can also be deduced similarly by adapting the proof strategy of \cite{rz-boundary}. Throughout this proof we will use the notation $q = \frac{1}{\gamma}(2 Q  - \beta_1 - \beta_2 - \beta_3 + \frac{2}{\gamma})$. For a Borel set $I \subseteq [0,\infty)$, consider the symmetric interval $\hat{I} = I \cup -I$ and the domain of the upper-half plane $D_I = \{ e^{ \ii \theta}t \: \vert \: t \in I, \: \theta \in [0, \pi]   \}$. We introduce the notation:
\begin{equation}
K_{I}(t) : = e^{\gamma c} \int_{D_I} \frac{|x-t |^{2} \hat g(x)^{\frac{\gamma^2}{4}(q+1)}}{ |x|^{\gamma \beta_1 }  |x-1|^{\gamma \beta_2 }  }   e^{\gamma h(x)} d^2x + e^{\frac{\gamma c}{2}}  \int_{\hat{I}} \frac{|r-t | \hat g(r)^{\frac{\gamma^2 q}{8}}}{|r|^{\frac{\gamma \beta_1 }{2}} |r-1|^{\frac{\gamma \beta_2 }{2}}}   e^{\frac{\gamma}{2} h(r)} d \mu^t_{\partial}(r).
\end{equation}
Now we want to study the asymptotic of
\begin{equation}
\int_{\mathbb{R}} e^{-\frac{\gamma}{2} q c} \E[\exp(-K_{\mathbb{R}_+}(t))] dc - \int_{\mathbb{R}} e^{-\frac{\gamma}{2} q c} \E[\exp(-K_{\mathbb{R}_+}(0))] dc =: T_1+T_2,
\end{equation}
where we defined:
\begin{align*}
&T_1: = \int_{\mathbb{R}} e^{-\frac{\gamma}{2} q c} \E \left[\exp( -K_{(|t|, + \infty)}(t) ) - \exp( - K_{\mathbb{R}_+}(0) ) \right] dc,\\
&T_2:= \int_{\mathbb{R}} e^{-\frac{\gamma}{2} q c} \E \left[ \exp( -K_{\mathbb{R}_+}(t) ) - \exp( -K_{(|t|, + \infty)}(t) ) \right] dc.
\end{align*}
Following the same inequalities as \cite[Equations (5.21) - (5.24)]{rz-boundary} one obtains $T_1 =o(|t|^{1 - \frac{ 2 \beta_1}{\gamma} + \frac{4}{\gamma^2}}) = o(|t|^{\frac{2}{\gamma}(Q-\beta_1)})$.
Now we focus on $T_2$. The goal is to restrict $K_I$ to the choice $I = (0, |t|^{1+\eta})\cup(|t|,\infty)$, with $\eta>0$ a small positive constant to be fixed, and then the area and boundary GMCs on the three disjoint parts will be weakly correlated.  
By using the inequalities of \cite[Equation (5.26)]{rz-boundary} and taking $\eta$ satisfying the condition
\begin{equation}
\eta < \frac{1+(\frac{2}{\gamma}-\frac{\gamma}{2})\beta_1 -\frac{4}{\gamma^2}}{\frac{\gamma\beta_1}{2}-1}, 
\end{equation}
we have $\int_{\mathbb{R}} e^{-\frac{\gamma}{2} q c} \; \E \left[\exp(-K_{\mathbb{R}_+}(t)) - \exp (-K_{(0, |t|^{1+\eta})\cup(|t|,\infty)}(t)) \right]  dc = o(|t|^{\frac{2}{\gamma}(Q-\beta_1)}).$
It remains to evaluate $\int_{\mathbb{R}} e^{-\frac{\gamma}{2} q c}  \E \left[\exp( - K_{(0, |t|^{1+\eta})\cup(|t|,\infty)}(t)) - \exp( - K_{(|t|,+\infty)}(t)) \right]  dc$. We now introduce the radial decomposition of the field $h$, $h(x) = B_{-2\log|x|} + Y(x)$,
where $B$, $Y$ are independent Gaussian processes with $(B_{s})_{s \in \mathbb{R}}$ a Brownian motion starting from $0$ for $s \geq 0$, $B_s= 0$ when $s<0$, and $Y$ is a centered Gaussian process with covariance,
\begin{equation}\label{equation correlation Y}
\E[Y(x)Y(y)] = 
\begin{cases}
2\log \frac{|x|\lor |y|}{|x-y|}, \quad & |x|,|y| \le 1,\\
2\log \frac{1}{|x-y|} -\frac{1}{2}\log \hat g(x) -\frac{1}{2}\log \hat g(y), \quad &\text{else.}
\end{cases}
\end{equation}
Now with this decomposition one can write:
\begin{align}
K_{I}(t)  &= e^{\gamma c} \int_{D_I} \frac{|x-t|^2 \hat g(x)^{\frac{\gamma^2}{4}(q + 1)} e^{\gamma B_{-2\log|x|}} }{ |x|^{\gamma \beta_1 } |x-1|^{\gamma \beta_2  } }  e^{ \gamma Y(x)}  d^2 x + e^{\frac{\gamma c}{2}}  \int_{\hat{I}} \frac{|r-t| \hat g(r)^{\frac{\gamma^2 q}{8}} e^{\frac{\gamma}{2} B_{-2\log|r|}} }{|r|^{\frac{\gamma \beta_1 }{2}} |r-1|^{\frac{\gamma \beta_2 }{2}}}    e^{\frac{\gamma}{2}Y(r)} d \mu_{\partial}^t(r).
\end{align}
Define the processes $P(x) = Y(x)\mathbf{1}_{|x| \le t^{1+\eta}} + Y(x)\mathbf{1}_{|x| \ge t}$, $\tilde{P}(x) = \tilde{Y}(x)\mathbf{1}_{|x| \le t^{1+ \eta}} + Y(x)\mathbf{1}_{|x| \ge t}$
where $\tilde{Y}$ is an independent copy of $Y$. 
By using the inequalities as in \cite[Equation (5.39)]{rz-boundary}
\begin{align}
& \left| \int_{\mathbb{R}} e^{-\frac{\gamma}{2} q c}  \E \left[\exp( - K_{(0, |t|^{1+\eta})\cup(|t|,\infty)}(t))  - \exp(- \tilde{K}_{(0,|t|^{1+\eta})}(t) - K_{(|t|,+\infty)}(t)) \right]  dc \right| \le c_3\, t^\eta,
\end{align}
for some constant $c_3 >0$, and where in $\tilde{K}_{(0, |t|^{1+\eta})}(t)$ we simply use the field $\tilde{Y}$ instead of $Y$. When $\eta > \frac{2}{\gamma}(Q-\beta_1)$, we can bound the previous term by $o(|t|^{\frac{2}{\gamma}(Q-\beta_1)})$.
Consider now the change of variable $x = t^{1+\eta}e^{-s/2}$. By the Markov property of the Brownian motion and stationarity of
\begin{equation*}
 N_{\tilde{Y}}(ds d \theta) := e^{\gamma \tilde{Y}( e^{- \frac{s}{2} } e^{ \ii \theta} ) - \frac{\gamma^2}{2} \mathbb{E}[\tilde{Y}( e^{- \frac{s}{2} } e^{\ii \theta})^2 ]} ds d\theta,  \quad d\mu_{\tilde{Y}}(s): = \mu_1 e^{ \frac{\gamma}{2} \tilde{Y}(-e^{-s/2}) }ds + \mu_2 e^{ \frac{\gamma}{2} \tilde{Y}(e^{-s/2}) }ds,
\end{equation*}
we can express $\tilde{K}_{(0, |t|^{1+\eta})}(t)$ by the formula
\begin{align*}
\tilde{K}_{(0, |t|^{1+\eta})}(t) &= e^{\frac{\gamma c}{2}} \frac{1}{2} t^{1+(1+\eta)(1-\frac{\gamma\beta_1}{2}+\frac{\gamma^2}{4})}e^{\frac{\gamma}{2}B_{2(1+h)\log (1/t)}}\int_{0}^{\infty} \frac{|1 - t^\eta e^{-s/2}|}{|t^{1+\eta}e^{-s/2}-1|^{\frac{\gamma\beta_2}{2}}}  e^{ \frac{\gamma}{2}(\tilde{B}_s - \frac{s}{2}(Q-\beta_1) ) } d\mu_{\tilde{Y}}(s)\\
&+ e^{\gamma c} \frac{1}{2} \sigma_t^2  \int_{0}^{\infty}   e^{\gamma (\tilde{B}_s- \frac{s}{2}(Q -\beta_1))}  \int_0^{\pi}   \frac{|e^{\ii \theta} t^\eta e^{-\frac{s}{2}} -1 |^2}{|e^{\ii \theta} - e^{- \ii \theta} |^{\frac{\gamma^2}{2}} |e^{\ii \theta} t^{1+\eta} e^{-\frac{s}{2}} - 1 |^{ \gamma \beta_2} } N_{\tilde{Y}}(ds d \theta),
\end{align*}
with $\tilde{B}$ an independent Brownian motion and $\sigma_t:= t^{1+(1+\eta)(1-\frac{\gamma\beta_1}{2}+\frac{\gamma^2}{4})}e^{\frac{\gamma}{2}B_{2(1+\eta)\log (1/t)}}$. We denote:
\begin{equation*}
 V_1: =\frac{1}{2} e^{\frac{\gamma c}{2}} \int_{0}^{\infty} e^{ \frac{\gamma}{2}(\tilde{B}_s - \frac{s}{2}(Q-\beta_1) ) } d\mu_{\tilde{Y}}(s), \: \:
V_2: = \frac{1}{2} e^{ \gamma c} \int_{0}^{\infty}   e^{\gamma (\tilde{B}_s- \frac{s}{2}(Q -\beta_1))}  \int_0^{\pi}   \frac{1}{|e^{\ii \theta} - e^{- \ii  \theta} |^{\frac{\gamma^2}{2}}} N_{\tilde{Y}}(ds d \theta).
\end{equation*}
One can then prove that
\begin{align*}
\left| \int_{\mathbb{R}} e^{-\frac{\gamma}{2} q c}  \E \left[\exp(- \tilde{K}_{(0,|t|^{1+\eta})}(t) - K_{(|t|,+\infty)}(t)) - \exp(- \sigma_t V_1 - \sigma_t^2 V_2 - K_{(|t|,+\infty)}(t)) \right]  dc \right| \leq c_4\,|t|^{(1+\eta)(2-\frac{\gamma\beta_1}{2})},
\end{align*}
for some $c_4>0$. This upper bound is also a $o(|t|^{\frac{2}{\gamma}(Q-\beta_1)})$. 
Now let $ M = \sup_{s \geqslant 0} (\tilde{B}_s -\frac{Q-\beta_1}{2} s) $ and let $L_M$ be the last time $\left( \mathcal{B}^{\frac{Q-\beta_1}{2}}_{-s} \right)_{s\ge 0} $ hits $-M$. Recall that the law of $M$ is known, for $v\ge 1$, $\mathbb{P}( e^{\frac{\gamma}{2} M} >v ) = v^{ -\frac{2}{\gamma}(Q-\beta_1)}$.
For simplicity, we introduce the notation:
\begin{align*}
&\rho_1(\beta_1) := \frac{1}{2}   \int_{- \infty}^{\infty}    e^{ \gamma \mathcal{B}^{\frac{Q-\beta_1}{2}}_s}  \int_0^{\pi}     \frac{1}{|e^{\ii \theta} - e^{- \ii \theta} |^{\frac{\gamma^2}{2}} }    N_{\tilde{Y}}(ds d \theta), \quad \rho_2(\beta_1) := \frac{1}{2}\int_{-\infty}^{\infty}  e^{ \frac{\gamma}{2} \mathcal{B}^{\frac{Q-\beta_1}{2}}_s} \mu_{\tilde{Y}}(ds).
\end{align*}
We can finally show that:
\begin{align*}
T_2 &=  \int_{\mathbb{R}} dc e^{ \frac{q \gamma c}{2} }  \E \left[\exp( -K_{(|t|, +\infty)}(t) - \sigma^2_t e^{\gamma M + \gamma c} \rho_1(\beta_1) - \sigma_t e^{\frac{\gamma}{2} M +  \frac{\gamma}{2} c} \rho_2(\beta_1)) - \exp(-K_{(|t|, +\infty)}(t)) \right]\\
&+ o(|t|^{\frac{2}{\gamma}(Q-\beta_1)}).
\end{align*}

We now evaluate the above difference at first order explicitly using the fact that density of $e^{\frac{\gamma}{2}M}$ is known:
\begin{align*}
& \int_{\mathbb{R}} dc e^{\frac{q\gamma c}{2}} \left( \E[\exp( -K_{(|t|, +\infty)}(t) - \sigma^2_t e^{\gamma M + \gamma c} \rho_1(\beta_1) - \sigma_t e^{\frac{\gamma}{2} M +  \frac{\gamma}{2} c} \rho_2(\beta_1))] -   \E[\exp(-K_{(|t|, +\infty)}(t))] \right) \\ \nonumber
& = \int_{\mathbb{R}} dc e^{\frac{q\gamma c}{2}} \E\left[\int_1^{\infty} \frac{dv}{v^{\frac{2}{\gamma}(Q-\beta_1)+1}}\left(\exp \left(- K_{(|t|, +\infty)}(t) - \sigma^2_t v^2 e^{\gamma c} \rho_1(\beta_1)  - \sigma_t v e^{\frac{\gamma}{2} c} \rho_2(\beta_1) \right) - \exp(-K_{(|t|, +\infty)}(t)) \right) \right]\\
& = \frac{\gamma}{2} \int_{\mathbb{R}} dc e^{\frac{q\gamma c}{2}} \E\Bigg[\int_{c + \frac{2}{\gamma} \log \sigma_t }^{\infty} dw \sigma_t^{\frac{2}{\gamma}(Q -\beta_1)} e^{(c-w)(Q-\beta_1)} \\
& \qquad \qquad \qquad \qquad \qquad \times  \Big(\exp \left(- K_{(|t|, +\infty)}(t) - e^{\gamma w} \rho_1(\beta_1)  - e^{\frac{\gamma}{2} w} \rho_2(\beta_1) \right) -  \exp(-K_{(|t|, +\infty)}(t)) \Big) \Bigg]\\
&  = \frac{\gamma}{2} \int_{\mathbb{R}} dc e^{\frac{q\gamma c}{2}} e^{c(Q-\beta_1)}  \E\left[ \sigma_t^{\frac{2}{\gamma}(Q -\beta_1)} \exp \left(- K_{(|t|, +\infty)}(t)  \right)  \right]\\
&  \times   \E\left[\int_{ \mathbb{R} } dw  e^{-w(Q-\beta_1)} \left(\exp \left( - e^{\gamma w} \rho_1(\beta_1)  - e^{\frac{\gamma}{2} w} \rho_2(\beta_1) \right) - 1 \right) \right] + o(|t|^{\frac{2}{\gamma}(Q-\beta_1)}). 
\end{align*}
By apply the Girsanov theorem to the power of $\sigma_t$  in the first expectation  and taking the limit $t \rightarrow 0$, the first expectation will become $t^{\frac{2}{\gamma}(Q-\beta_1)}$ times an $H$ function. The second expectation is the reflection coefficient $R$. Therefore we arrive at the final claim for the $t >0$ limit:
\begin{align}
& \int_{\mathbb{R}} dc e^{q\gamma c/2} \left( \E[\exp( -K_{(|t|, +\infty)}(t) - \sigma^2_t e^{\gamma M + \gamma c} \rho_1(\beta_1) - \sigma_t e^{\frac{\gamma}{2} M +  \frac{\gamma}{2} c} \rho_2(\beta_1))] -   \E[\exp(-K_{(|t|, +\infty)}(t))] \right) \\ \nonumber
&= t^{\frac{2}{\gamma}(Q-\beta_1)} R(\beta_1, \sigma_1 - \frac{1}{\gamma}, \sigma_2- \frac{1}{\gamma}) H
\begin{pmatrix}
2Q - \beta_1 - \frac{2}{\gamma}, \beta_2, \beta_3 \\
\sigma_1 - \frac{1}{\gamma},  \sigma_2,   \sigma_3 
\end{pmatrix}  + o(|t|^{\frac{2}{\gamma}(Q-\beta_1)}).
\end{align}
The $t< 0 $ case is obtained in a similar fashion. 
\end{proof}

\bibliographystyle{alpha}
\bibliography{cibib}

\def\cprime{$'$}
\begin{thebibliography}{GKRV21}

\bibitem[AG21]{ag-disk}
Morris {Ang} and Ewain {Gwynne}.
\newblock {Liouville quantum gravity surfaces with boundary as matings of
  trees}.
\newblock {\em Annales de l'Institut Henri Poincar{\'e}, Probabilit{\'e}s et
  Statistiques}, 57(1):1 -- 53, 2021.

\bibitem[AHS23a]{ahs-disk-welding}
Morris Ang, Nina Holden, and Xin Sun.
\newblock Conformal welding of quantum disks.
\newblock {\em Electron. J. Probab.}, 28:Paper No. 52, 50, 2023.

\bibitem[AHS23b]{AHS-SLE-integrability}
Morris Ang, Nina Holden, and Xin Sun.
\newblock {Integrability of SLE via conformal welding of random surfaces}.
\newblock {\em Communications on Pure and Applied Mathematics}, pages 1--57,
  2023.

\bibitem[{Ang}23]{ang-lcft-zip}
Morris {Ang}.
\newblock {Liouville conformal field theory and the quantum zipper}.
\newblock {\em arXiv e-prints}, page arXiv:2301.13200, January 2023.

\bibitem[ARS23]{ARS-FZZ}
Morris Ang, Guillaume Remy, and Xin Sun.
\newblock {FZZ formula of boundary Liouville CFT via conformal welding}.
\newblock {\em Journal of the European Mathematical Society}, pages 1--58,
  2023.

\bibitem[ASY22]{ASY-triangle}
Morris {Ang}, Xin {Sun}, and Pu~{Yu}.
\newblock {Quantum triangles and imaginary geometry flow lines}.
\newblock {\em arXiv e-prints}, page arXiv:2211.04580, November 2022.

\bibitem[Ber17]{Ber_GMC}
Nathanael Berestycki.
\newblock {An elementary approach to Gaussian multiplicative chaos}.
\newblock {\em Electronic Communications in Probability}, 22(none):1 -- 12,
  2017.

\bibitem[BHSY23]{borga2023baxter}
Jacopo Borga, Nina Holden, Xin Sun, and Pu~Yu.
\newblock Baxter permuton and liouville quantum gravity.
\newblock {\em Probability Theory and Related Fields}, pages 1--49, 2023.

\bibitem[BK01]{Bakalov-Kirillov}
Bojko Bakalov and Alexander Kirillov, Jr.
\newblock {\em Lectures on tensor categories and modular functors}, volume~21
  of {\em University Lecture Series}.
\newblock American Mathematical Society, Providence, RI, 2001.

\bibitem[Bor23]{borga2021skew}
Jacopo Borga.
\newblock The skew brownian permuton: a new universality class for random
  constrained permutations.
\newblock {\em Proceedings of the London Mathematical Society}, 2023.

\bibitem[BP]{berestycki-lqg-notes}
Nathana{\"e}l {Berestycki} and Ellen {Powell}.
\newblock {I}ntroduction to the {G}aussian {F}ree {F}ield and {L}iouville
  {Q}uantum {G}ravity.
\newblock {A}vailable at
  \url{https://homepage.univie.ac.at/nathanael.berestycki/articles.html}.

\bibitem[BPZ84]{BPZ1984}
A.A. Belavin, A.M. Polyakov, and A.B. Zamolodchikov.
\newblock Infinite conformal symmetry in two-dimensional quantum field theory.
\newblock {\em Nuclear Physics B}, 241(2):333--380, 1984.

\bibitem[BW23]{bw-bpz}
Guillaume Baverez and Baojun Wu.
\newblock {Higher equations of motion at level 2 in Liouville CFT}.
\newblock {\em arXiv preprint arXiv:2312.13900}, 2023.

\bibitem[{Cer}22]{cercle_toda_reflection}
Baptiste {Cercl{\'e}}.
\newblock {Three-point correlation functions in the $\mathfrak{sl}_3$ Toda
  theory I: Reflection coefficients}.
\newblock {\em arXiv e-prints}, page arXiv:2206.06902, June 2022.

\bibitem[{Cer}24]{Cercle-BPZ}
Baptiste {Cercl{\'e}}.
\newblock {Higher equations of motion for boundary Liouville Conformal Field
  Theory from the Ward identities}.
\newblock {\em arXiv e-prints}, page arXiv:2401.13271, January 2024.

\bibitem[CEZ23]{Lorentz2023}
Scott {Collier}, Lorenz {Eberhardt}, and Mengyang {Zhang}.
\newblock {Solving 3d Gravity with Virasoro TQFT}.
\newblock {\em arXiv e-prints}, page arXiv:2304.13650, April 2023.

\bibitem[DKRV16]{dkrv-lqg-sphere}
Fran\c{c}ois David, Antti Kupiainen, R\'emi Rhodes, and Vincent Vargas.
\newblock Liouville quantum gravity on the {R}iemann sphere.
\newblock {\em Comm. Math. Phys.}, 342(3):869--907, 2016.

\bibitem[{\relax DLMF}]{NIST:DLMF}
{\it NIST Digital Library of Mathematical Functions}.
\newblock http://dlmf.nist.gov/, Release 1.1.0 of 2020-12-15.
\newblock F.~W.~J. Olver, A.~B. {Olde Daalhuis}, D.~W. Lozier, B.~I. Schneider,
  R.~F. Boisvert, C.~W. Clark, B.~R. Miller, B.~V. Saunders, H.~S. Cohl, and
  M.~A. McClain, eds.

\bibitem[DMS21]{wedges}
Bertrand Duplantier, Jason Miller, and Scott Sheffield.
\newblock Liouville quantum gravity as a mating of trees.
\newblock {\em Ast\'{e}risque}, (427):viii+257, 2021.

\bibitem[DO94]{DOZZ1}
Harald Dorn and H.~J. Otto.
\newblock {Two and three point functions in Liouville theory}.
\newblock {\em Nucl. Phys. B}, 429:375--388, 1994.

\bibitem[DRV16]{LQG_tori}
Fran{\c c}ois David, R{\'e}mi Rhodes, and Vincent Vargas.
\newblock Liouville quantum gravity on complex tori.
\newblock {\em Journal of Mathematical Physics}, 57(2):022302, 2016.

\bibitem[Dub09]{dubedat-coupling}
Julien Dub{\'e}dat.
\newblock S{LE} and the free field: partition functions and couplings.
\newblock {\em J. Amer. Math. Soc.}, 22(4):995--1054, 2009.

\bibitem[Ebe23]{Lorenz_notes}
Lorenz Eberhardt.
\newblock Notes on crossing transformations of virasoro conformal blocks.
\newblock 2023.

\bibitem[FZZ00]{FZZ}
V.~{Fateev}, A.~{Zamolodchikov}, and Al. {Zamolodchikov}.
\newblock {Boundary Liouville Field Theory I. Boundary State and Boundary
  Two-point Function}.
\newblock {\em arXiv e-prints}, pages hep--th/0001012, January 2000.

\bibitem[GKRV20]{GKRV_bootstrap}
Colin {Guillarmou}, Antti {Kupiainen}, R{\'e}mi {Rhodes}, and Vincent {Vargas}.
\newblock {Conformal bootstrap in Liouville Theory}.
\newblock {\em arXiv e-prints}, page arXiv:2005.11530, May 2020.

\bibitem[GKRV21]{GKRV-Segal}
Colin {Guillarmou}, Antti {Kupiainen}, R{\'e}mi {Rhodes}, and Vincent {Vargas}.
\newblock {Segal's axioms and bootstrap for Liouville Theory}.
\newblock {\em arXiv e-prints}, page arXiv:2112.14859, December 2021.

\bibitem[GRV19]{LQG_higher_genus}
Colin Guillarmou, R{\'e}mi Rhodes, and Vincent Vargas.
\newblock Polyakov's formulation of 2d bosonic string theory.
\newblock {\em Publications math{\'e}matiques de l'IH{\'E}S}, 130(1):111--185,
  2019.

\bibitem[Hos01]{bulk-boundary}
Kazuo Hosomichi.
\newblock {Bulk boundary propagator in Liouville theory on a disc}.
\newblock {\em JHEP}, 11:044, 2001.

\bibitem[HRV18]{hrv-disk}
Yichao Huang, R\'{e}mi Rhodes, and Vincent Vargas.
\newblock Liouville quantum gravity on the unit disk.
\newblock {\em Ann. Inst. Henri Poincar\'{e} Probab. Stat.}, 54(3):1694--1730,
  2018.

\bibitem[KRV20]{DOZZ_proof}
Antti Kupiainen, Rémi Rhodes, and Vincent Vargas.
\newblock Integrability of {L}iouville theory: proof of the {DOZZ} formula.
\newblock {\em Annals of Mathematics}, 191(1):81--166, 2020.

\bibitem[MS89]{Moore-Seiberg}
Gregory~W. Moore and Nathan Seiberg.
\newblock {Classical and Quantum Conformal Field Theory}.
\newblock {\em Commun. Math. Phys.}, 123:177, 1989.

\bibitem[MS17]{ig4}
Jason Miller and Scott Sheffield.
\newblock Imaginary geometry {IV}: interior rays, whole-plane reversibility,
  and space-filling trees.
\newblock {\em Probab. Theory Related Fields}, 169(3-4):729--869, 2017.

\bibitem[NQSZ23]{backbone}
Pierre {Nolin}, Wei {Qian}, Xin {Sun}, and Zijie {Zhuang}.
\newblock {Backbone exponent for two-dimensional percolation}.
\newblock {\em arXiv e-prints}, page arXiv:2309.05050, September 2023.

\bibitem[Pol81]{Polyakov1981}
A.M. Polyakov.
\newblock Quantum geometry of bosonic strings.
\newblock {\em Physics Letters B}, 103(3):207--210, 1981.

\bibitem[PT99]{PT-fusion}
B.~{Ponsot} and J.~{Teschner}.
\newblock {Liouville bootstrap via harmonic analysis on a noncompact quantum
  group}.
\newblock {\em arXiv e-prints}, pages hep--th/9911110, November 1999.

\bibitem[PT01]{PT-CMP}
B.~Ponsot and J.~Teschner.
\newblock Clebsch-{G}ordan and {R}acah-{W}igner coefficients for a continuous
  series of representations of {$ U_q(\mathrm{sl}(2,\mathbb R))$}.
\newblock {\em Comm. Math. Phys.}, 224(3):613--655, 2001.

\bibitem[PT02]{PT_boundary_3pt}
B.~Ponsot and J.~Teschner.
\newblock Boundary {L}iouville field theory: boundary three-point function.
\newblock {\em Nuclear Physics B}, 622(1):309--327, 2002.

\bibitem[Rem18]{Remy_annulus}
Guillaume Remy.
\newblock {Liouville quantum gravity on the annulus}.
\newblock {\em J. Math. Phys.}, 59(8):082303, 2018.

\bibitem[Rem20]{remy-fb-formula}
Guillaume Remy.
\newblock The {F}yodorov-{B}ouchaud formula and {L}iouville conformal field
  theory.
\newblock {\em Duke Math. J.}, 169(1):177--211, 2020.

\bibitem[Run99]{Runkel}
Ingo Runkel.
\newblock {Boundary structure constants for the A series Virasoro minimal
  models}.
\newblock {\em Nucl. Phys. B}, 549:563--578, 1999.

\bibitem[RV14]{rhodes-vargas-review}
R{\'e}mi Rhodes and Vincent Vargas.
\newblock Gaussian multiplicative chaos and applications: {A} review.
\newblock {\em Probab. Surv.}, 11:315--392, 2014.

\bibitem[RV19]{RV_reflection}
R{\'e}mi Rhodes and Vincent Vargas.
\newblock {The tail expansion of Gaussian multiplicative chaos and the
  Liouville reflection coefficient}.
\newblock {\em The Annals of Probability}, 47(5):3082 -- 3107, 2019.

\bibitem[RZ20]{RZ_interval}
Guillaume Remy and Tunan Zhu.
\newblock {The distribution of Gaussian multiplicative chaos on the unit
  interval}.
\newblock {\em The Annals of Probability}, 48(2):872 -- 915, 2020.

\bibitem[RZ22]{rz-boundary}
Guillaume Remy and Tunan Zhu.
\newblock Integrability of boundary liouville conformal field theory.
\newblock {\em Communications in Mathematical Physics}, 395(1):179--268, 2022.

\bibitem[She07]{shef-gff}
Scott Sheffield.
\newblock Gaussian free fields for mathematicians.
\newblock {\em Probab. Theory Related Fields}, 139(3-4):521--541, 2007.

\bibitem[Tes01]{Teschner-revisited}
J.~Teschner.
\newblock {Liouville theory revisited}.
\newblock {\em Class. Quant. Grav.}, 18:R153--R222, 2001.

\bibitem[Tes03]{Teschner-modular}
J.~Teschner.
\newblock {From Liouville theory to the quantum geometry of Riemann surfaces}.
\newblock In {\em {14th International Congress on Mathematical Physics}}, 8
  2003.

\bibitem[Ver88]{Verlinde1989}
Erik~P. Verlinde.
\newblock {Fusion Rules and Modular Transformations in 2D Conformal Field
  Theory}.
\newblock {\em Nucl. Phys. B}, 300:360--376, 1988.

\bibitem[VT13]{TesVar2013}
G.~Vartanov and J.~Teschner.
\newblock Supersymmetric gauge theories, quantization of moduli spaces of flat
  connections, and conformal field theory, 2013.

\bibitem[{Wu}22]{Wu-annulus}
Baojun {Wu}.
\newblock {Conformal Bootstrap on the Annulus in Liouville CFT}.
\newblock {\em arXiv e-prints}, page arXiv:2203.11830, March 2022.

\bibitem[{Wu}23]{wu2023sle}
Da~{Wu}.
\newblock {The SLE Bubble Measures via Conformal Welding of Quantum Surfaces}.
\newblock {\em arXiv e-prints}, page arXiv:2302.13992, February 2023.

\bibitem[{Zha}22]{Zhan-bubble}
Dapeng {Zhan}.
\newblock {SLE$_\kappa(\rho)$ bubble measures}.
\newblock {\em arXiv e-prints}, page arXiv:2206.04481, June 2022.

\bibitem[ZZ96]{DOZZ2}
Alexander~B. Zamolodchikov and Alexei~B. Zamolodchikov.
\newblock {Structure constants and conformal bootstrap in Liouville field
  theory}.
\newblock {\em Nucl. Phys. B}, 477:577--605, 1996.

\end{thebibliography}
 
\end{document}